\documentclass [11pt,oneside,a4paper,mathscr]{amsart}

\usepackage{amscd,amsmath,amssymb,euscript}
\usepackage[frame,cmtip,curve,arrow,matrix,line,graph]{xy}
\usepackage{tikz-cd}
\usepackage{hyperref}
\usepackage{stmaryrd}
\usepackage{bm}
\usepackage{here}
\usepackage{enumitem}

\title[Dolbeault geometric Langlands conjecture]{The Dolbeault geometric Langlands conjecture via limit categories}
\author{Tudor P\u adurariu and Yukinobu Toda}

\newtheorem{thm}{Theorem}[section]
\newtheorem{cor}[thm]{Corollary}
\newtheorem{prop}[thm]{Proposition}
\newtheorem{propdef}[thm]{Proposition-Definition}
\newtheorem{consdef}[thm]{Construction-Definition}
\newtheorem{conj}[thm]{Conjecture}
\newtheorem{lemma}[thm]{Lemma}

\theoremstyle{definition}
\newtheorem{defn}[thm]{Definition}

\newtheorem{thm*}[thm]{Theorem$^*$}
\newtheorem{remark}[thm]{Remark}

\newtheorem{example}[thm]{Example}

\newtheorem{quest}[thm]{Question}
\newtheorem{step}{Step}

\newcommand{\comment}[1]{}

\renewcommand{\leq}{\leqslant}
\renewcommand{\geq}{\geqslant}

\newcommand{\PP}{\operatorname{\mathbb P}}

\newcommand{\E}{\mathcal{E}}

\newcommand{\Y}{\mathcal{Y}}
\newcommand{\sS}{\mathcal{S}}
\newcommand{\zZ}{\mathcal{Z}}
\newcommand{\U}{\mathcal{U}}
\newcommand{\V}{\mathcal{V}}
\newcommand{\LL}{\operatorname{L}}
\newcommand{\MM}{\operatorname{M}}

\newcommand{\ox}{\overline{x}}

\newcommand{\QCoh}{\operatorname{QCoh}}

\newcommand{\X}{\mathcal{X}}

\newcommand{\bgm}{B\mathbb{G}_m}
\newcommand{\fg}{\mathfrak{g}}
\newcommand{\fp}{\mathfrak{p}}
\newcommand{\fm}{\mathfrak{m}}
\newcommand{\diasquare}{\ar@{}[rd]|\square}

\newcommand{\rank}{\operatorname{rank}}

\newcommand{\Coh}{\operatorname{Coh}}
\newcommand{\Aut}{\operatorname{Aut}}

\newcommand{\Ker}{\operatorname{Ker}}
\newcommand{\id}{\operatorname{id}}

\newcommand{\Ind}{\operatorname{Ind}}
\newcommand{\IndCoh}{\operatorname{IndCoh}}
\newcommand{\IndL}{\operatorname{IndL}}
\newcommand{\Hom}{\operatorname{Hom}}

\newcommand{\Spec}{\operatorname{Spec}}

\newcommand{\GL}{\operatorname{GL}}

\newcommand{\ol}{\overline}

\renewcommand{\Im}{\operatorname{Im}}

\newcommand{\Ad}{\mathrm{Ad}}

\newcommand{\shear}{{\mathbin{\mkern-6mu\fatslash}}}

\newcommand{\inclusion}{\ar@<-0.3ex>@{^{(}->}[r]}
\newcommand{\iinclusion}{\ar@<-0.3ex>@{^{(}->}[rr]}
\newcommand{\linclusion}{\ar@<0.3ex>@{_{(}->}[l]}
\newcommand{\dinclusion}{\ar@<-0.3ex>@{^{(}->}[d]}
\newcommand{\uinclusion}{\ar@<-0.3ex>@{^{(}->}[u]}

\newcommand{\ssslash}{/\!\!/}

\newcommand{\wt}{\mathrm{wt}}
\newcommand{\Hig}{\mathrm{Higgs}}
\newcommand{\Bun}{\mathrm{Bun}}

\makeatletter
\newcommand{\colim@}[2]{%
  \vtop{\m@th\ialign{##\cr
    \hfil$#1\operator@font colim$\hfil\cr
    \noalign{\nointerlineskip\kern1.5\ex@}#2\cr
    \noalign{\nointerlineskip\kern-\ex@}\cr}}%
}
\newcommand{\colim}{%
  \mathop{\mathpalette\colim@{}}\nmlimits@
}
\makeatother

\usetikzlibrary{positioning,fit,calc}
\tikzstyle{block}=[draw=black, width=1cm, minimum height=2cm, align=center] 
\tikzstyle{block2}=[draw=black, text width=2cm, minimum height=1cm, align=center] 
\tikzstyle{block3}=[draw=black, text width=2cm, minimum height=1cm, align=center] 

\makeatletter

\@addtoreset{equation}{section}	
\makeatother

\begin{document}

\begin{abstract}
We introduce limit categories for cotangent stacks of smooth stacks as an 
effective version of classical limits of categories of D-modules on them. 
We develop their general theory and pursue their relation with categories of D-modules. 
In particular, we establish the functorial properties of limit categories such as the smooth 
pull-back and projective push-forward. 

Using the notion of limit categories, we propose a precise formulation of the Dolbeault geometric Langlands conjecture, proposed by Donagi–Pantev as the classical limit of the de Rham geometric Langlands equivalence. It states an equivalence between the derived categories of moduli stacks of semistable Higgs bundles and 
limit categories of moduli stacks of all Higgs bundles.

We prove the existence of a semiorthogonal decomposition of the limit category into quasi-BPS categories, which are categorical versions of BPS invariants on a non-compact Calabi–Yau 3-fold.
This semiorthogonal decomposition is interpreted as a Langlands dual to the semiorthogonal decomposition constructed in our previous work on the category of coherent sheaves on the moduli stack of semistable Higgs bundles.

We also construct Hecke operators on limit categories for Higgs bundles. They are expected to be compatible 
with Wilson operators under our formulation of Dolbeault geometric Langlands conjecture. 

The conjectured equivalence implies an equivalence between BPS categories for semistable Higgs bundles, which we expect to be a categorical version of the topological mirror symmetry conjecture for Higgs bundles by Hausel--Thaddeus. 

\end{abstract}

\maketitle

 \setcounter{tocdepth}{2}
\tableofcontents

\section{Introduction}
\subsection{Overview}
In this paper, we propose a new and precise formulation of the \textit{Dolbeault geometric Langlands conjecture} proposed by Donagi--Pantev~\cite{DoPa}; see also the work of Kapustin--Witten~\cite{KapWit}. It asserts an equivalence between certain derived categories of coherent sheaves on moduli stacks of Higgs bundles on smooth projective curves.

We prove several foundational results supporting our conjecture, including compact generation and the construction of Hecke operators. We construct semiorthogonal decompositions of both sides into \textit{quasi-BPS categories}, linking the geometric Langlands correspondence with categorical Donaldson--Thomas theory~\cite{T}. 
In the subsequent papers~\cite{TodaGL2, TodaA}, the second author proves our formulation of the Dolbeault geometric Langlands correspondence for type $A$ groups over an open locus of the Hitchin base beyond the elliptic locus; see Theorem~\ref{thm:BA}. This overcomes some of the main difficulties in extending the Dolbeault geometric Langlands correspondence beyond the elliptic locus, and the formulation and results developed in this paper are essential for those works.

In what follows, we work over an algebraically closed field $k$ of characteristic zero.

On the automorphic side, we introduce and investigate novel categories, called \textit{limit categories}. They may be regarded as classical limits of the categories of D-modules on moduli stacks of vector bundles over curves. 
Our formulation, Conjecture~\ref{conj:intro:higgs}, states an equivalence between the derived categories of moduli stacks of semistable Higgs bundles and the limit category associated with moduli stacks of Higgs bundles \textit{without stability condition}.
\begin{conj}\emph{(Conjecture~\ref{conj:DL0})}\label{conj:intro:higgs}
    Let $C$ be a smooth projective curve, $G$ a reductive group and $^{L}G$ its Langlands dual. 
    Let $\Hig_G$ be the derived moduli stack of $G$-Higgs bundles. 
    For $(\chi, w) \in \pi_1(G) \times Z(G)^{\vee}$, there is an equivalence
    \begin{align}\label{intro:cohL}
        \IndCoh(\Hig_{^{L}G}(w)^{\mathrm{ss}})_{-\chi}\simeq \IndL(\Hig_{G}(\chi))_w.
    \end{align}
\end{conj}

As we mention later, the above equivalence is expected to preserve the subcategories with nilpotent singular supports, giving a classical limit of the geometric Langlands correspondence formulated by Beilinson--Drinfeld~\cite{BD0} and Arinkin--Gaitsgory~\cite{AG}, and proved in~\cite{GLC1, GLC2, GLC3, GLC4, GLC5}. It is further expected to be compatible with parabolic induction, with the action of Wilson/Hecke operators, Whittaker normalized, and be linear over the Hitchin base. 

The construction of the right-hand side in (\ref{intro:cohL}), the (ind-)limit category, 
is one of the key ingredients of this paper. Note that in the left-hand side
of (\ref{intro:cohL}) we consider the moduli stack of \textit{semistable} $^{L}G$-Higgs bundles (which is quasi-compact), 
while in the right-hand side we consider the moduli stack of \textit{all} $G$-Higgs bundles (which is not quasi-compact in general). The main contributions in this paper are summarized as follows: 

\vspace{2mm}

{\bf (i) Definition of limit categories.} For a large class of smooth stacks $\X$ and $\delta \in \mathrm{Pic}(\X)_{\mathbb{R}}$, we define the dg-category $\IndL(\Omega_{\X})_{\delta}$ and its subcategory of compact objects $\LL(\Omega_{\X})_{\delta}$ (Proposition-Definition~\ref{intro:propdef}). 
Particular examples of limit categories are the `quasi-BPS categories' of~\cite{PThiggs}, and some of the `non-commutative resolutions' of~\cite{SvdB},
the `magic window subcategories' of~\cite{hls}, and the `intrinsic window subcategories' of~\cite{T}.
Our new perspective is the conjectural interpretation of limit categories as 
an effective version of the classical limit of the categories of D-modules on $\X$ (Subsection~\ref{subsec:limit} and Section~\ref{sec:Dmod}). 

    \vspace{2mm}

{\bf (ii) Functorial properties of limit categories.}
We show that limit categories satisfy functorial properties with 
respect to Lagrangian correspondences associated with smooth pull-backs and projective push-forwards
(Theorem~\ref{intro:thmfunct}). Moreover, limit categories admit functors which exist for D-modules, 
but do not exist for ind-coherent sheaves or quasi-coherent sheaves (e.g. ${!}$-push-forward for some open immersions). These properties allow us to import ideas from algebraic analysis to commutative algebraic geometry.

\vspace{2mm}

   {\bf (iii) Compact generation.} We prove that $\IndL(\Hig_G(\chi))_w$ is compactly generated (Theorem~\ref{intro:thm:gen}), whereas $\IndCoh(\Hig_G(\chi))_w$ is not (Proposition~\ref{intro:prop:cpt}). This result is regarded as an analogue of compact generation of $\text{D-mod}(\Bun_G)$, proved by Drinfeld-Gaitsgory~\cite{DGbun}. Thus, when working with the limit category, one avoids
   the technical subtleties 
   present when working with categories of coherent sheaves on non-quasi-compact stacks. 

    \vspace{2mm}

{\bf (iv) Semiorthogonal decomposition.} We show that $\LL(\Hig_G(\chi))_w$ admits a semiorthogonal decomposition
    into \textit{quasi-BPS categories} introduced and studied in~\cite{PThiggs} (Theorem~\ref{conj:intro}), interpreted as the Langlands-dual counterpart of the semiorthogonal decomposition for semistable Higgs bundles in~\cite{PThiggs} (proved there for $G \in \{\GL_r, \mathrm{PGL}_r, \mathrm{SL}_r\}$, but expected to hold in general), and implying the K-theoretic version of~\eqref{intro:cohL} for genus-zero curves (Corollary~\ref{intro:corK}). This result gives a 
    close relationship between the (Dolbeault) geometric Langlands equivalence and categorical 
    Donaldson-Thomas theory. 

\vspace{2mm}
    
    {\bf (v) Hecke operators on limit categories.} We construct Hecke operators on $\LL(\Hig_G(\chi))_w$ (Theorem~\ref{intro:cor2}):
    for $G=\GL_r$, it is an action of quasi-BPS categories of zero-dimensional sheaves;
    for general $G$, it is given by Hecke correspondence corresponding to a minuscule cocharacter\footnote{the miniscule condition is needed at this moment due to singularities of correspondences, see Remark~\ref{rmk:mini}.}. They also induce Hecke operators on BPS categories, which resolve 
    a long-standing obstacle in the Higgs-bundle setting due to the non-preservation of (semi)stability under Hecke correspondences. The Hecke operators do not preserve the semiorthogonal decomposition in (iv), but we can control 
    their amplitude (Proposition~\ref{prop:Hecke:sod}). 
    
\vspace{2mm}

Let us mention two ways in which the equivalence~\eqref{conj:intro:higgs} is analogous to the de Rham Langlands equivalence~\cite{BD0, AG} and different from the Dolbeault equivalence of Donagi--Pantev~\cite{DoPa}: both categories of \eqref{conj:intro:higgs} are compactly generated, and we expect that the topological K-theory (and also periodic cyclic homology) of their categories of compact objects are isomorphic to their counterparts in the de Rham Langlands equivalence.

\subsection{Semiorthogonal decomposition of the limit category}
We first explain one of our main results, the existence of semiorthogonal decomposition of
the limit category, which gives a strong motivation toward Conjecture~\ref{conj:intro:higgs}. We postpone the definition of the limit category to Proposition-Definition~\ref{intro:propdef}, since it is fairly technical.

Let $C$ be a smooth projective curve over an algebraically closed field 
$k$ of characteristic zero. Let $G$ be a reductive algebraic group, and 
$\mathfrak{g}$ its Lie algebra. 
For $\chi \in \pi_1(G)$, let $\Bun_G(\chi)$ be the 
corresponding connected component of the moduli stack of $G$-bundles on $C$. 
Its cotangent stack
\begin{align}\label{intro:higgs}
    \Hig_G(\chi)=\Omega_{\Bun_G(\chi)}
\end{align}
is a connected component of the derived moduli stack of $G$-Higgs bundles
\begin{align*}
(E_G, \phi_G), \ \phi_G \in H^0(C, \Ad_{\fg}(E_G)\otimes \Omega_C)
\end{align*}
where $E_G$ is a $G$-bundle. 
Consider the quasi-compact open substack of semistable $G$-Higgs bundles 
\begin{align*}
    \Hig_G(\chi)^{\mathrm{ss}} \subset \Hig_G(\chi). 
\end{align*}
Note that much of this paper is concerned with the moduli stack of Higgs bundles (\ref{intro:higgs})
\textit{without imposing the stability condition}, which fails to be quasi-compact in most cases, e.g. when $G=\mathrm{GL}_r$ with $r\geq 2$.

Let $Z(G)$ be the center of $G$ and $Z(G)^{\vee}$ its character group. 
For $w\in Z(G)^{\vee}$, we will define and investigate a subcategory (Proposition-Definition~\ref{intro:propdef})
\begin{align}\notag
    \IndL(\Hig_G(\chi))_w \subset \IndCoh(\Hig_G(\chi))_w.
\end{align}
Its subcategory of compact objects 
\begin{align}\label{intro:limit}
\LL(\Hig_G(\chi))_w
:=\IndL(\Hig_G(\chi))_w^{\mathrm{cp}}\subset \Coh(\Hig_G(\chi))_w
\end{align}
is called \textit{the limit category}.

We regard the limit category (\ref{intro:limit}) as the classical limit of the (twisted) D-module category on the stack $\Bun_G(\chi)$, which plays a central role in the geometric Langlands correspondence~\cite{BD0, AG}.
One of the main results of this paper is the following:
\begin{thm}\emph{(Theorem~\ref{thm:mainLG})}\label{conj:intro}
There exist subcategories (which we refer to as quasi-BPS categories)
\begin{align}\label{intro:qbps0}
    \mathbb{T}_G(\chi)_w \subset \Coh(\Hig_G(\chi)^{\mathrm{ss}})_w
\end{align}
such that there is a semiorthogonal decomposition 
\begin{align}\label{intro:sod:L}
    \LL(\Hig_G(\chi))_w=\left\langle \mathbb{T}_{M}(\chi_{M})_{w_{M}} : \begin{array}{ll}
     \mu_M(\chi_{M})\in N(T)_{\mathbb{Q}+}^{W_M} \\ \mu_M(w_M)=\mu_G(w) \end{array}\right\rangle.
\end{align}
    Here, the right-hand side is after all standard parabolic subgroups $P\subset G$ with Levi 
    quotient $M$, 
    satisfying
    \begin{align}\label{intro:chiw}
    (\chi_{M}, w_{M}) \in \pi_1(M) \times Z(M)^{\vee}, \mathbb{}a_{M}(\chi_{M}, w_{M})=(\chi, w).
    \end{align}
\end{thm}
We refer to Subsection~\ref{subsec:SOD0} for the notation in Theorem~\ref{conj:intro}. 
In the case of $G=\mathrm{GL}_r$, then $(\chi, w)\in \mathbb{Z}^2$, and 
the semiorthogonal decomposition 
in Theorem~\ref{conj:intro} becomes 
\begin{align}\label{sod:intro1}
   \LL(\Hig_{\mathrm{GL}_r}(\chi))_w=\left\langle \bigotimes_{i=1}^k \mathbb{T}_{\mathrm{GL}_{r_i}}(\chi_i)_{w_i} : \frac{\chi_1}{r_1}>\cdots>\frac{\chi_k}{r_k}, \frac{w_i}{r_i}=\frac{w}{r}  \right\rangle,  
\end{align}
where the right-hand side is after all partitions 
\begin{align}\label{intro:part}
    (r, \chi, w)=(r_1, \chi_1, w_1)+\cdots+(r_k, \chi_k, w_k).
\end{align}
The functor from summands of the right hand side to the left hand side is given by parabolic induction, and it is fully faithful. The order of the summands is induced by a natural order on Harder-Narasimhan strata. 

The geometric explanation of the semiorthogonal decomposition (\ref{intro:sod:L}) is the Harder-Narasimhan stratification of the stack $\Hig_G(\chi)$, and each stratum 
contributes to a semiorthogonal summand exactly once. 
Indeed, the inequalities which appear in the right-hand side 
of (\ref{sod:intro1}) are nothing but the inequalities of the Harder-Narasimhan filtration 
of (Higgs) bundles. Such a 
semiorthogonal decomposition is natural in the setting of D-modules or constructible sheaves 
(for example, see the work of McGerty--Nevins~\cite{McNe} for D-modules), but it is novel in the setting of coherent sheaves
and the proof of Theorem~\ref{conj:intro} is quite intricate. 
We will first prove a version of Theorem~\ref{conj:intro} in a local model in Section~\ref{sec:magic} involving \textit{magic categories} of smooth stacks, and in Section~\ref{sec:sodlim} we will reduce the problem to the local setting, see also the discussion in Subsection \ref{subsec:compactgeneration}. 

Note that there are semiorthogonal decompositions for certain smooth or quasi-smooth stacks due to Halpern--Leistner~\cite{halp, HalpTheta, HalpK32}, but those involve certain choices of weights and stability conditions, whereas (\ref{intro:sod:L}) and the decomposition from~\cite{McNe} do not, and each unstable stratum appears infinitely many times as opposed to once in~\eqref{sod:intro1}.

\subsection{The Dolbeault geometric Langlands conjecture via limit categories}
In our earlier work~\cite{PThiggs}, we studied the dg-category of coherent sheaves on 
the quasi-compact stack $\Hig_G(\chi)^{\mathrm{ss}}$ for $G\in \{\GL_r, \mathrm{PGL}_r, \mathrm{SL}_r\}$, and obtained the following semiorthogonal decomposition, see Theorem~\ref{thm:PThiggs}:
\begin{align}\label{intro:sod:general}
    \Coh(\Hig_G(\chi)^{\mathrm{ss}})_w=
    \left\langle \mathbb{T}_{M}(\chi_{M})_{w_{M}} : \begin{array}{ll}\mu_M(\chi_M)=\mu_G(\chi) \\
     \mu_M(w_{M}) \in 
   -M(T)_{\mathbb{Q}+}^{W_{M}}\end{array}
    \right\rangle. 
\end{align}
Here $(\chi_M, w_M)$ satisfies (\ref{intro:chiw}). 
In the case of $G=\GL_r$, it says that 
\begin{align}\label{sod:intro2}
   \Coh( \Hig_{\mathrm{GL}_r}(\chi)^{\mathrm{ss}})_w=\left\langle \bigotimes_{i=1}^k \mathbb{T}_{\mathrm{GL}_{r_i}}( \chi_i)_{w_i} : \frac{w_1}{r_1}<\cdots<\frac{w_k}{r_k}, \frac{\chi_i}{r_i}=\frac{\chi}{r}  \right\rangle,  
\end{align}
where the right-hand side is after all partitions (\ref{intro:part}). The functor from the summands of the right hand side to the left hand side is induced by parabolic induction (in other words, the Hall product~\cite{PoSa, PThiggs}), and it is fully faithful. 
We expect that the semiorthogonal decomposition~\eqref{intro:sod:general} holds for general reductive groups $G$, and we plan to prove it in the work in progress~\cite{BPT}.

The semiorthogonal decompositions (\ref{sod:intro1}) and (\ref{sod:intro2}) 
are of similar forms, but the roles of the characters/cocharacters are exchanged. 
However, whereas (\ref{sod:intro1}) has a clear geometric explanation, namely the Harder-Narasimhan stratification of $\Hig_G$, the semiorthogonal decomposition (\ref{intro:sod:general}) does not. The semiorthogonal decomposition (\ref{intro:sod:general}) can be explained based on combinatorics of weight 
polytopes of some (self-dual) representations of reductive groups, see also~\cite{SpVdB}, and is motivated by categorification of the PBW/cohomological integrality theorem in cohomological DT theory, see Subsection~\ref{subsec:intro:catDT}. 

We expect that the semiorthogonal decompositions in (\ref{sod:intro1})
and (\ref{sod:intro2}) are Langlands dual to each other. 
Based on the above observations, we propose the following formulation of the \textit{Dolbeault geometric Langlands conjecture (DGLC)}, which, given the decompositions (\ref{sod:intro1}) and (\ref{sod:intro2}), is a weak form of compatibility of the equivalence~\eqref{intro:cohL} with parabolic induction\footnote{At this moment, parabolic inductions for limit categories are only defined partially, due to the non-properness of the maps $\Hig_P \to \Hig_G$. The result of Theorem~\ref{conj:intro} implies the existence of parabolic induction in the case of a Harder-Narasimhan stratum.}: 
\begin{conj}\label{conj:Higgs}
There is an equivalence 
\begin{align}\label{intro:eq:CohL}
    \Coh(\Hig_{^{L}G}(w)^{\mathrm{ss}})_{-\chi} \simeq \LL(\Hig_{G}(\chi))_{w}
\end{align}
under which the semiorthogonal decompositions (\ref{sod:intro1}) and (\ref{sod:intro2}) into quasi-BPS categories are 
compatible.
\end{conj}

Later we discuss further compatibilities of the above equivalence. 
Note that Conjecture~\ref{conj:Higgs} obviously holds when $T$ is a torus, see~Remark~\ref{rmk:torus}. It is more interesting and difficult for other cases. 
At this moment, even the case of $G=\GL_2$ is open. 

\medskip

For a dg-category $\mathscr{D}$, we denote by $K^{\text{top}}(\mathscr{D})$ the topological K-theoretic spectrum defined by Blanc~\cite{Blanc}. For the relation between Blanc's topological K-theory and the usual (equivariant) topological K-theory, see~\cite{Blanc, HLP}.
We will use topological K-theory to get an indication of the `size' of the dg-categories discussed in this paper. These computations will be expressed (sometimes conjecturally) in terms of BPS cohomology, which we mention in Subsection~\ref{subsec:catDT}. We prefer using topological K-theory instead of periodic cyclic homology because of the computations with the former in~\cite{PTtop, PThiggs2}.

By combining Theorem~\ref{conj:intro} and the results on (topological) K-theories of quasi-BPS 
categories from~\cite{PThiggs2}, we obtain a version of Conjecture~\ref{conj:Higgs} 
in K-theory. It gives evidence toward Conjecture~\ref{conj:Higgs}, see 
Subsection~\ref{subsec:SODLang}.

\begin{cor}\label{intro:corK}
If $G \in \{\GL_r, \mathrm{PGL}_r, \mathrm{SL}_r\}$,  
there is an isomorphism of rational topological K-theories respecting the filtration induced by the semiorthogonal decompositions~(\ref{sod:intro1}) and (\ref{sod:intro2}): 
\begin{align*}
    K^{\mathrm{top}}(\Coh(\Hig_{^{L}G}(w)^{\mathrm{ss}})_{-\chi})_{\mathbb{Q}} \cong K^{\mathrm{top}}(\LL(\Hig_{G}(\chi))_w)_{\mathbb{Q}}
\end{align*}
in the following cases: 
(i) $g(C)=0$; (ii) $r=2$; (iii) $(r, w)$ are coprime. 

In the case of (i) $g(C)=0$, the above isomorphism holds with integral coefficients, 
and also holds for algebraic K-theory. 
\end{cor}

\subsection{Relation to topological mirror symmetry}\label{subsec:intro:catDT}

In the case of $G=\mathrm{GL}_r$, the quasi-BPS category (\ref{intro:qbps0})
\begin{align}\label{intro:qbps}
    \mathbb{T}_{\mathrm{GL}_r}(\chi)_w \subset \Coh(\Hig_{\mathrm{GL}_r}(\chi)^{\mathrm{ss}})_w
\end{align}
was introduced and studied in~\cite{PThiggs}. As we mention in the next subsection, the above subcategory 
is motivated by categorical Donaldson-Thomas theory. 

If $(r, \chi)$ are coprime, then (\ref{intro:qbps}) is nothing but the dg-category of coherent sheaves 
on the (classical) moduli space of stable Higgs bundles $\text{H}_{\GL_r}(\chi)^{\text{ss}}$, which is a holomorphic symplectic manifold. 

The case that $(r, \chi)$ are non-coprime, but $\gcd(r, \chi, w)=1$, is more interesting. 
In this case, the category (\ref{intro:qbps})\footnote{More precisely, we need to take its reduced version by removing a trivial derived structure.} is smooth, proper, and Calabi–Yau over the Hitchin 
base, and recovers the 
BPS invariants of~\eqref{def:Xtotal}, see~\cite{PThiggs}.
It shares the same categorical properties with the category of coherent sheaves on 
moduli spaces of stable Higgs bundles, and therefore we regard it as a `\textit{non-commutative Hitchin moduli space}' (in this case, it is
\textit{a BPS category}). In particular, it is indecomposable. 
Recall that, for general $(r,\chi)$, the good moduli space $\text{H}_{\GL_r}(\chi)^{\text{ss}}$ of the classical truncation of $\Hig_{\mathrm{GL}_r}(\chi)^{\text{ss}}$ does not admit crepant resolutions~\cite{KaLeSo}, alternatively a Calabi--Yau resolution of singularities over the Hitchin base. A BPS category may be regarded as a twisted non-commutative crepant resolution of singularities of $\text{H}_{\GL_r}(\chi)^{\text{ss}}$.

A consequence of the equivalence from Conjecture~\ref{conj:Higgs} is the following conjectural equivalence of quasi-BPS categories of \textit{semistable} Higgs bundles:
\begin{conj}\label{intro:conj:qbps}
For $(\chi, w)\in \pi_1(G)\times Z(G)^{\vee}$, there is an equivalence 
\begin{align}\label{intro:equiv:T}
    \mathbb{T}_{^{L}G}(w)_{-\chi} \simeq \mathbb{T}_{G}(\chi)_{w}.
\end{align}
\end{conj}
An equivalence (\ref{intro:equiv:T}) is the main conjecture in~\cite{PThiggs} when $G \in \{\GL_r, \mathrm{PGL}_r, \mathrm{SL}_r\}$, and the above conjecture extends it 
to arbitrary $G$.
In type A, 
the equivalence (\ref{intro:equiv:T}) is a categorical version of the topological mirror symmetry for
Higgs bundles proposed by Hausel–Thaddeus~\cite{HauTha}, and proved in~\cite{GrWy, MSendscopic}. 
The above conjecture is also motivated by the D-equivalence conjecture~\cite{B-O2, MR1949787} and SYZ mirror symmetry~\cite{SYZ}. 
Indeed, it states an equivalence between two non-commutative Calabi--Yau compactifications (over the Hitchin base) of dual abelian fibrations, namely the relative Picard schemes of smooth spectral curves~\cite[Theorem A]{DoPa}. For more details, see the introduction of~\cite{PThiggs}. 

For general groups $G$, a conjectural topological mirror symmetry was proposed in~\cite[Conjecture~10.3.18]{BDIKP}, formulated using BPS cohomology. We expect the topological K-theoretic version of (\ref{intro:equiv:T}) to be closely related to the 2-periodic version of this conjecture.

For $G\in \{\GL_r, \mathrm{PGL}_r, \mathrm{SL}_r\}$, Conjecture~\ref{intro:conj:qbps} is 
known to be true for $g=0$ (obvious), should follow from the work in progress~\cite{CauPT} for $g=1$, and is
true (for any $g$) at the level of rational topological K-theories by~\cite{PThiggs2} if $(r, \chi, w)$ is primitive.

\subsection{Relation to categorical Donaldson-Thomas theory}\label{subsec:catDT}
Let us first say a few words about Donaldson-Thomas (DT) theory. For introductions to the subject, see~\cite{MR3221298, Sz, TodaSpringer}.

\medskip

Initially, DT theory studied integer invariants of complex threefolds~\cite{Thom} associated to moduli of ideal sheaves, giving a mathematical definition of BPS invariants (Bogomol'nyi-Prasad-Sommerfield). Later, for Calabi--Yau $3$-folds, it was realized that these invariants can be defined more generally, for moduli of semistable sheaves~\cite{K-S, JS}, or even more for $(-1)$-shifted symplectic stacks (with orientation)~\cite{PTVV, BDIKP}. 

Furthermore they can be refined, for example to a perverse sheaf~\cite{Beh, MR2851153, BBJ, BBBJ} such that the Euler characteristic of its hypercohomology equals the integer valued BPS invariant. These sheaves are the main object of study in \textit{cohomological DT theory}. 
A moduli stack of semistable sheaves on a Calabi--Yau $3$-fold is locally the critical locus of a regular function $f\colon \Y\to\mathbb{A}^1_\mathbb{C}$ on a smooth stack with a good moduli space $\Y\to Y$. In such a local model, the BPS perverse sheaf is the vanishing cycle sheaf $\varphi_f(\text{IC}_Y)$, see~\cite{DM, BDIKP}. It is a direct summand of the Joyce perverse sheaf $\varphi_{f}(\text{IC}_{\Y})$, after pushing forward to $Y$. 

We now mention two results in cohomological DT theory, relevant to the present paper.
The first one is the PBW/cohomological integrality theorem of Davison--Meinhardt~\cite{DM, BDIKP}, which says that the cohomology of the Joyce sheaf decomposes into cohomologies of BPS sheaves, using the cohomological Hall products.  

The second is the \textit{$\chi$-independence}~\cite[Conjecture 1.2]{TodGV} phenomenon. It says that, for any Calabi--Yau $3$-fold $X$, the BPS invariants of sheaves supported on curves in $Y$ do not depend on the Euler characteristic $\chi$ of the parametrized sheaves. The $\chi$-independence was proved for numerical DT invariants of local K3 surfaces~\cite{MTK3}, and for the cohomological DT spaces of the local curve~$X=\text{Tot}_C(N)$ in~\cite{KinjoKoseki, DMJS}, where $N$ is a rank two vector bundle on 
$C$ with $\det N=\omega_C$.

To summarize these two results, one decomposes a possibly infinite dimensional space (the cohomology of the Joyce sheaf) into finite dimensional ones (the cohomology of the BPS sheaf), and the latter has surprising (numerical) symmetries ($\chi$-independence). 

\medskip

In \textit{categorical DT theory}~\cite{T, T4, HHR1, HHR2}, one studies a further refinement of the Joyce perverse sheaf, which is a dg-category locally described using categories of matrix factorizations. Here recall that the periodic cyclic homology of a dg-category of matrix factorizations is a 2-periodic version of the cohomology of the vanishing cycle sheaf~\cite{Eff, MR3877165}. For an overview of the scope to the subject, see the introduction of~\cite{T}. In our previous work, see for example~\cite{PTquiver, PTtop, PThiggs}, we studied a categorical analogue of BPS invariants, called \textit{quasi-BPS categories}, and we constructed semiorthogonal decompositions of categories of coherent sheaves or matrix factorizations on a stack where the summands are quasi-BPS categories.  We regard such decompositions as a categorical version of the PBW theorem in cohomological DT theory.

Note that quasi-BPS categories are not yet defined for compact Calabi--Yau $3$-folds. Nevertheless, they can be indirectly defined~\cite{T} (using the expected dimensional reduction property~\cite{I, Hirano, Kinjo}) for certain non-compact Calabi--Yau $3$-folds, including the one relevant to the current paper:
\begin{equation}\label{def:Xtotal}
    X:=\mathrm{Tot}_C(\mathcal{O}_C\oplus \Omega_C). 
\end{equation}

The quasi-BPS category (\ref{intro:qbps}) is regarded 
as a categorical version of the BPS invariants of $X$. 
In this context, the semiorthogonal decomposition (\ref{sod:intro2}) is regarded as 
 a categorical analogue of 
the PBW/cohomological integrality theorem in DT theory. 
The equivalence (\ref{intro:equiv:T}) is a categorical version of the $\chi$-independence phenomenon for 
the Calabi--Yau 3-fold (\ref{def:Xtotal}), as explained in~\cite{PThiggs}.

This paper may be viewed as establishing a connection between Donaldson–Thomas theory on the non-compact Calabi–Yau 3-fold $X$ and the geometric Langlands correspondence through categorical DT theory,
see the following diagram. 
\begin{figure}[H]
\caption{Relation between DT theory and geometric Langlands}
\begin{align*}
\xymatrix{
\mbox{Donaldson-Thomas theory } \ar[rrr]^-{{\mbox{\tiny{categorification}}}} & & &\mbox{Categorical DT theory} \ar@<1ex>[d]^-{{\mbox{\tiny{formulation}}}} \\
\mbox{Geometric Langlands } \ar[rrr]_-{{\mbox{\tiny{classical limit}}}} & & &\mbox{Dolbeault Langlands} \ar@<1ex>[u]^-{{\mbox{\tiny{symmetry}}}}
}  
\end{align*}
\end{figure}

There are several follow-up questions inspired by this connection, which we do not pursue in this paper. First, one may wonder whether quasi-BPS categories have natural deformation quantizations, or, related, whether there are natural categorifications of BPS invariants using categories of D-modules or constructible sheaves.
Second, it is interesting to investigate a categorical version of the $\chi$-independence phenomenon for Calabi--Yau $3$-folds beyond the Calabi--Yau 3-fold (\ref{def:Xtotal}), for example for other local curves or for local K3 surfaces~\cite{PTK3}.

\subsection{Classical limit of the (de Rham) geometric Langlands equivalence}

We now recall the proposal of Donagi--Pantev~\cite{DoPa} for a Dolbeault version of the geometric Langlands equivalence.

Let $C$ be a smooth projective curve over $k$, and let $G$ be a reductive algebraic group. The (de Rham) geometric Langlands correspondence is formulated as an equivalence of dg-categories~\cite{BD0, AG}: 
\footnote{
More precisely, the right-hand side should be the category of $1/2$-twisted D-modules.
It is equivalent to the untwisted one uncanonically, which depends on a choice of $\omega_C^{1/2}\in \mathrm{Pic}(C)$.}
\begin{align}\label{intro:GLC} \IndCoh_{\mathcal{N}}(\mathrm{Locsys}_{^{L}G}) \simeq \mathrm{D\text{-}mod}(\Bun_G). \end{align} 
On the left-hand side (the spectral side), $\mathrm{Locsys}_{^{L}G}$ denotes the moduli stack of $^{L}G$-bundles with flat connections, and $\IndCoh_{\mathcal{N}}(-)$ is the dg-category of ind-coherent sheaves with nilpotent singular supports~\cite{AG}.
On the right-hand side (the automorphic side), $\Bun_G$ denotes the moduli stack of $G$-bundles, and $\mathrm{D\text{-}mod}(-)$ is the dg-category of D-modules~\cite{GaiCry} on it.
The equivalence~\eqref{intro:GLC} has been established in a series of papers~\cite{GLC1, GLC2, GLC3, GLC4, GLC5}, and satisfies several compatibilities, for example it respects parabolic induction, the action of Wilson/Hecke operators, and Whittaker normalization.

Before the precise formulation~\eqref{intro:GLC} of the GLC appeared in~\cite{AG}, Donagi–Pantev~\cite{DoPa} proposed a classical limit version for Higgs bundles. The stack $\mathrm{Locsys}_{^{L}G}$ admits a degeneration to the stack of $^{L}G$-Higgs bundles, \begin{align}\notag \mathrm{Locsys}_{^{L}G} \leadsto \Hig_{^{L}G}, \end{align} via the moduli stack of Deligne $\lambda$-connections.
On the automorphic side, the sheaf of differential operators may be regarded as a deformation quantization of the cotangent bundle. Thus, there is a degeneration of categories \begin{align*} \mathrm{D\text{-}mod}(\Bun_G) \leadsto \mathrm{QCoh}(\Hig_G).
\end{align*} Based on these observations, Donagi–Pantev~\cite{DoPa} (also see~\cite[Conjecture~1.3]{ZN}) conjectured an equivalence, often called the \emph{Dolbeault geometric Langlands conjecture} (DGLC)\footnote{This version of conjectural equivalence (\ref{intro:DL}) is due to~\cite[Conjecture~1.3]{ZN}.}: \begin{align}\label{intro:DL} \mathrm{QCoh}(\Hig_{^{L}G}) \simeq \mathrm{QCoh}(\Hig_{G}). \end{align} 
One may then attempt to deform it into the geometric Langlands equivalence~\eqref{intro:GLC}, see~\cite{BezBra}.

The two moduli stacks of Higgs bundles in~\eqref{intro:DL} admit Hitchin maps \begin{align*} \xymatrix{ \Hig_{^{L}G}\ar[rd] & & \Hig_{G} \ar[ld] \\ & \mathrm{B}. & } \end{align*} 
Generically over the Hitchin base $\mathrm{B}$, they are dual abelian fibrations ~\cite[Theorem~A]{DoPa}.
It is expected that the Fourier–Mukai equivalence~\cite{Mu1} over the generic point of $\mathrm{B}$ extends to an equivalence~\eqref{intro:DL} over the entire base $\mathrm{B}$. Both equivalences~\eqref{intro:equiv:T} and~\eqref{intro:cohL} are such conjectural extensions.

A remarkable feature of the DGLC is that both sides 
are mathematical objects in commutative algebraic geometry, and thus is an interesting duality in algebraic geometry. 
There have been several efforts proving an equivalence (\ref{intro:DL}), especially over some 
dense open subset of the Hitchin base $\mathrm{B}$, see~\cite{Ardual, ArFe, MRVF2, MLi, PThiggs}. Among them, in Arinkin's remarkable paper~\cite{Ardual},
an equivalence (\ref{intro:DL}) for $G=\GL_r$ is proved over the elliptic locus in $\mathrm{B}$, where the spectral curves of Higgs bundles are irreducible.

\subsection{Local systems and semistable Higgs bundles}\label{subsec:Dolb}
However, so far there is no known example for $G$ non-abelian where an equivalence (\ref{intro:DL}) holds over the full base $\mathrm{B}$. Indeed, it seems not even clear whether an equivalence (\ref{intro:DL}) is a correct formulation. Due to the non-quasi-compactness of $\Hig_G$, there 
are several obstruction and difficulities in dealing with (quasi)coherent sheaves on it. 
For example, in~\cite[Remark~1.4]{ZN}, it is mentioned that the formulation DGLC needs 
to be modified. 

Since the stack $\Hig_G$ is singular, one might consider replacing $\mathrm{QCoh}$ with $\IndCoh$ or $\IndCoh_{\mathcal{N}}$ in the formulation, for example inquire whether:
\begin{equation}\label{equiv:dolb3}
    \IndCoh(\Hig_{^LG})\cong \IndCoh(\Hig_{G}).
\end{equation}
However, the following issue highlights a significant difference between the categories in~\eqref{intro:DL} and those in~\eqref{intro:GLC}:

\begin{prop}\emph{(Proposition~\ref{prop:cgen}, Remark~\ref{rmk:compact})}\label{intro:prop:cpt} For $G = \mathrm{GL}_r$ with $r \geq 2$, the dg-categories
$\mathrm{QCoh}(\Hig_G(\chi))$, $\IndCoh(\Hig_G(\chi))$, and
$\IndCoh_{\mathcal{N}}(\Hig_G(\chi))$ are not compactly generated. \end{prop}

In contrast, both sides of~\eqref{intro:GLC} are known to be compactly generated~\cite{MR3037900, DGbun}, and this property plays a crucial role in the proof of~\eqref{intro:GLC} in~\cite{GLC1, GLC2, GLC3, GLC4, GLC5}.
It is therefore desirable to formulate a version of the DGLC in terms of compactly generated dg-categories, which is what we pursue in this paper. 

\begin{remark}\label{remark:newDolb}
Note, however, that there are conjectural versions of the Dolbeault geometric Satake analogous to~\eqref{equiv:dolb3}, for example~\cite[Section 7.9]{BFM} and~\cite[Conjecture 1.6]{CW}. 
The Wilson/Hecke operators we discuss later are different slope categories of the monoidal category in~\cite[Conjecture 1.6]{CW}.
One may then hope for a Dolbeault equivalence compatible with these equivalences, see~\cite{KapWit}\footnote{We thank Sabin Cautis and Eric Vasserot for discussions about this.}. We do not pursue such a version in the current paper, but, as we note in Subsection \ref{subsec:KHA}, the equivalence~\eqref{intro:cohL} may be regarded as an equivalence between two slope categories of a modified equivalence~\eqref{equiv:dolb3}.\end{remark}

The failure of compact generation in Proposition~\ref{intro:prop:cpt} is a consequence of the fact that the connected components of $\Hig_G$ are not quasi-compact.
However, in the spectral side, there exists an alternative (and `smaller') degeneration in the case $G = \mathrm{GL}_r$ due to Simpson~\cite[Proposition~4.1]{Simp},~\cite[Theorem~2.5]{Simp2}: \begin{align}\label{intro:degss} \mathrm{Locsys}_{\mathrm{GL}_r} \leadsto \Hig_{\GL_r}(0)^{\mathrm{ss}}, \end{align} where $\Hig_{\GL_r}(0)^{\mathrm{ss}}$ denotes the moduli stack of semistable Higgs bundles of degree zero.
This stack is quasi-compact, and moreover, we have an identity of motivic classes~\cite[Theorem~1.2.1]{FeSoSo} in the Grothendieck ring of Artin stacks:
\begin{align*} [\mathrm{Locsys}_{\mathrm{GL}_r}] = [\Hig_{\GL_r}(0)^{\mathrm{ss}}]. \end{align*} 
There is also an isomorphism of the corresponding cohomological Hall algebras, see~\cite[Theorem~1.8]{Henn}. We expect the analogous isomorphism holds for topological K-theoretic Hall algebras.
Hence, it is natural to replace the left-hand side of~\eqref{intro:DL} with a suitable dg-category of coherent sheaves on $\Hig^{\mathrm{ss}}_{^{L}G}$.
This leads to the following natural question:

\begin{quest}\label{question} What is the counterpart of $\Hig^{\mathrm{ss}}_{^{L}G}$ on the automorphic side under Langlands duality? \end{quest}

In particular, such a category is expected to have the same topological K-theory as the category of compact objects of $\text{D-mod}(\Bun_G)$.

\subsection{Classical limits of the categories of D-modules}\label{subsec:limit}
In this paper, we propose an ansatz for Question~\ref{question} via a limit category (\ref{intro:limit}), which we regard as a version of the classical limit 
of the category of coherent D-modules. 

For a smooth variety $X$, let $\text{D-mod}(X)$ be the dg-category of 
D-modules on $X$ and $\text{D-mod}_{\mathrm{coh}}(X)$ its subcategory of 
coherent D-modules. It is well-known that it admits a degeneration 
to the dg-category of coherent sheaves of the cotangent space $\Omega_X$
\begin{align*}
    \text{D-mod}_{\mathrm{coh}}(X) \leadsto \Coh(\Omega_X). 
\end{align*}
In other words, the sheaf of differential operators $D_X$ on $X$ is a non-commutative 
deformation of $\mathcal{O}_{\Omega_X}$. Accordingly, the category $\text{D-mod}_{\mathrm{coh}}(X)$ may be viewed as a deformation of the category $\Coh(\Omega_X)$.

We will explore the above picture in the case of a smooth 
stack $\X$. In this case, we still have a degeneration of 
$\text{D-mod}_{\mathrm{coh}}(\X)$ into $\Coh(\Omega_{\X})$, 
although the latter is often significantly larger than the former.
For example, we have the following (see Subsection~\ref{subsec:ExD:red} for details):
\begin{example}
Let $\X=BG$ for a reductive group $G$
and let $\Omega_{\X} \to \X$ be its cotangent stack. Let $0\colon \X \to \Omega_{\X}$ be the 
zero section. Then 
$\Coh(\Omega_{\X})$ is generated by objects $0_{*}(\Gamma\otimes \mathcal{O}_{\X})$ for irreducible 
$G$-representations $\Gamma$. On the other hand, one can show an equivalence 
\begin{align}\label{intro:equivD}
    \text{D-mod}_{\mathrm{coh}}(\X)\simeq \langle 0_{*}\mathcal{O}_{\X} \rangle \subset 
    \Coh(\Omega_{\X})
\end{align}
where $\langle 0_{*}\mathcal{O}_{\X}\rangle$ is the subcategory generated by $0_{*}\mathcal{O}_{\X}$. 
In particular, $\text{D-mod}_{\mathrm{coh}}(\X)$ is much smaller than $\Coh(\Omega_{\X})$.
\end{example}

For a smooth stack $\X$, 
instead of $\Coh(\Omega_{\X})$, we will consider subcategories 
which generalize the subcategory $\langle 0_{*}\mathcal{O}_{\X}\rangle$ in (\ref{intro:equivD}), 
and behave similarly to categories of coherent D-modules, called \textit{limit categories}. 
Briefly, limit categories and categories of coherent D-modules have the following in common:
\begin{itemize}
    \item existence of projective pushforward and smooth pullback (Theorem~\ref{intro:thmfunct}),
    \item existence of $j_!$ for \textit{certain} open immersions $j$ (Equation~\eqref{eq:jlowershriek}),
    \item conjecturally, the singular support map has image in the K-theory of the limit category, and further the induced map is an isomorphism in topological K-theory (Conjecture~\ref{conj:L2} and Remark~\ref{rem:singularsupport}).
\end{itemize}

The results of our construction of limit categories are summarized as follows: 
\begin{propdef}\emph{(Definition~\ref{def:Lcat}, Lemma~\ref{lem:emb})}\label{intro:propdef}
For a pair $(\mathfrak{M}, \delta)$ of a quasi-smooth QCA derived stack $\mathfrak{M}$
with $\mathbb{L}_{\mathfrak{M}}\simeq \mathbb{T}_{\mathfrak{M}}$ and $\delta \in \mathrm{Pic}(\mathfrak{M})_{\mathbb{R}}$, 
there exists a subcategory 
\begin{align*}
     \LL(\mathfrak{M})_{\delta} \subset \Coh(\mathfrak{M})
\end{align*}
    satisfying the following. Suppose that there is a diagram 
   \begin{align*}
\xymatrix{
& \mathcal{V} \ar[d] \\
\mathfrak{M}\simeq s^{-1}(0) \inclusion^-{i} & \mathcal{Y} \ar@/_18pt/[u]_-{s}
}
\end{align*}
where $\mathcal{Y}$ is a smooth stack, $\mathcal{V} \to \mathcal{Y}$ is a vector bundle 
with a section $s$, and 
$s^{-1}(0)$ is the derived zero locus of $s$. Then 
an object $\mathcal{E} \in \Coh(\mathfrak{M})$ lies in 
$\LL(\mathfrak{M})_{\delta}$ if and only if, for any field extension $k'/k$ and 
map $\nu \colon (\bgm)_{k'} \to \mathfrak{M}$, we have 
\begin{align}\notag
    \wt(\nu^{\ast}i^{\ast}i_{\ast}\mathcal{E})
    \subset \left[\frac{1}{2} c_1 (\nu^{\ast}\mathbb{L}_{\mathcal{V}}^{<0}), 
    \frac{1}{2} c_1 (\nu^{\ast}\mathbb{L}_{\mathcal{V}}^{>0})
    \right]+c_1(\nu^{\ast}\delta). 
\end{align} 
Here, $\nu^* \mathbb{L}_{\mathcal{V}}^{>0}$ (resp.~$\nu^* \mathbb{L}_{\mathcal{V}}^{<0}$)
is the sum of positive (resp.~negative) $\mathbb{G}_m$-weights, see Subsection~\ref{subsec:Gmwt}. Further for $\mathcal{F}\in\mathrm{Coh}(\bgm)$, the set $\wt(\mathcal{F}) \subset \mathbb{Z}$ is defined to be 
$w\in \mathbb{Z}$ with $\mathcal{F}_w \neq 0$, where $\mathcal{F}_w$ is the direct summand of $\mathcal{F}$ with $\mathbb{G}_m$-weight $w$. 

For a general (not necessarily quasi-compact) quasi-smooth derived stack 
$\mathfrak{M}$ with $\mathbb{L}_{\mathfrak{M}}\simeq \mathbb{T}_{\mathfrak{M}}$, we define 
\begin{align*}
    \IndL(\mathfrak{M})_{\delta}:=\lim_{\mathcal{U}\subset \mathfrak{M}} \Ind(\LL(\mathcal{U})_{\delta})
\end{align*}
where the limit is after all the QCA open substacks $\mathcal{U}\subset \mathfrak{M}$, and define 
\begin{align*}
    \LL(\mathfrak{M})_{\delta} \subset \IndL(\mathfrak{M})_{\delta}
\end{align*}
to be the subcategory of compact objects. 
For a smooth stack $\mathcal{X}$, we 
call $\LL(\Omega_{\X})_{\delta}$ the limit category. 
\end{propdef}
\begin{remark}
The construction of categories of coherent sheaves on (certain smooth or quasi-smooth) stacks $\mathcal{X}$ using weight conditions for maps $\bgm\to\mathcal{X}$ was previously used by Halpern-Leistner~\cite{HalpTheta, HalpK32}.
In the case of affine GIT stack $\X=\Spec R/G$ for a reductive group $G$, the construction 
of the limit category essentially appeared in~\cite{hls, PTquiver} with different names, and a (very) closely related construction is that of non-commutative resolutions of singularities for $\text{Spec}(R^G)$ of \v{S}penko--Van den Bergh~\cite{SvdB}. 
In this paper, we extend this construction to an arbitrary smooth stack $\X$ that does not necessarily admit 
good moduli space~\cite{MR3237451}, using the weight conditions with respect to 
all maps $
    \nu \colon \bgm \to \Omega_{\X}$. 
    We refer to Definition~\ref{def:Lcat} for its precise definition. 
    \end{remark}
    
    When $\X$ is a smooth QCA stack and $\delta=\omega_{\X}^{1/2}$, we 
expect that $\LL(\Omega_{\X})_{\delta}$ provides a better degeneration of 
$\text{D-mod}_{\mathrm{coh}}(\X)$ than $\Coh(\Omega_\X)$:
\begin{align*}
    \text{D-mod}_{\mathrm{coh}}(\X) \stackrel{?}{\leadsto} \LL(\Omega_{\X})_{\frac{1}{2}}:=\LL(\Omega_{\X})_{\delta=\omega_{\X}^{1/2}}. 
\end{align*}
In Section~\ref{sec:Dmod}, we discuss a precise formulation of this expectation 
as follows: 
\begin{conj}\emph{(Conjecture~\ref{conj:L})}\label{conj:L2}
    For any $\mathcal{E}\in \mathrm{D}\text{-}\mathrm{mod}_{\mathrm{coh}}(\X)$, there is a good filtration on 
    it whose associated graded lies in $\LL(\Omega_{\X})_{\frac{1}{2}}$.
\end{conj}

\begin{remark}\label{rem:singularsupport}
    In particular, the above conjecture implies that the singular support map on algebraic K-theory
    \[\text{ss}\colon K(\mathrm{D}\text{-}\mathrm{mod}_{\mathrm{coh}}(\X))\to K\left(\Coh(\Omega_\X)\right)\] factors through the image of the limit category
    \[\text{ss}'\colon K(\mathrm{D}\text{-}\mathrm{mod}_{\mathrm{coh}}(\X))\to K\left(\LL(\Omega_{\X})_{\frac{1}{2}}\right).\]
    We expect that $\text{ss}'$ is an isomorphism for the stacks of representations of an arbitrary quiver. Further, we expect that the topological K-theory version of $\text{ss}'$ is an isomorphism for a large class of stacks, including $\Bun_G(\chi)$. 
    
    As mentioned in Subsection~\ref{subsection:DLL}, we expect that Conjecture~\ref{conj:L2} and the above claims may be proved, at least for some smooth QCA stacks, by developing the theory of limit categories for the loop stack $\mathcal{L}\X$, and even more for a category of loop-equivariant coherent sheaves on the formal completion $\hat{\mathcal{L}}\X$ of $\mathcal{L}\X$ along the constant loops $\X\hookrightarrow\mathscr{L}\X$. One expects to construct an $\mathbb{A}^1/\mathbb{G}_m$-linear category with generic fiber $\text{D-mod}_{\mathrm{coh}}(\X)$ (using the Koszul duality from~\cite{HChen}) and special fiber a $\mathbb{G}_m$-equivariant version of $\text{L}(\Omega_{\X})_{\frac{1}{2}}$. 
    We plan to come back to these problems in future work.  
\end{remark}

As another relation to D-modules, we show the following functorial properties of limit 
categories. They are regarded as classical limits of pull-back/push-forward 
functors for D-modules: 
\begin{thm}\emph{(Proposition~\ref{prop:pback}, Theorem~\ref{thm:proj})}\label{intro:thmfunct}
For a morphism $f\colon \X \to \Y$ of smooth QCA stacks
and the associated Lagrangian correspondence
\begin{align}\label{intro:lag}
\Omega_{\X} \stackrel{\beta}{\leftarrow} f^{*}\Omega_{\Y} \stackrel{\alpha}{\to} \Omega_{\Y}
\end{align}
we have the following: 
\begin{enumerate}
    \item If $f$ is smooth, the Lagrangian correspondence (\ref{intro:lag})
induces the functor 
\begin{align}\label{intro:pback}
    f^{\Omega!}=\beta_{*}\alpha^! \colon \LL(\Omega_{\Y})_{\delta \otimes \omega_{\Y}^{1/2}}
    \to \LL(\Omega_{\X})_{f^*\delta \otimes \omega_{\X}^{1/2}}.
\end{align}
\item If $f$ is projective, the Lagrangian 
correspondence (\ref{intro:lag}) induces the functor
\begin{align}\label{intro:push}
    f^{\Omega}_{*}=\alpha_{*}\beta^* \colon \LL(\Omega_{\X})_{f^*\delta \otimes \omega_{\X}^{1/2}}
    \to \LL(\Omega_{\Y})_{\delta \otimes \omega_{\Y}^{1/2}}.
\end{align}
\end{enumerate}
\end{thm}
The functors (\ref{intro:pback}), (\ref{intro:push}) are regarded 
as classical limits of the corresponding functors 
for D-modules 
\begin{align*}
  &f^! \colon \text{D-mod}_{\mathrm{coh}}(\Y) \to \text{D-mod}_{\mathrm{coh}}(\X), \\ 
   & f_{\ast} \colon \text{D-mod}_{\mathrm{coh}}(\X) \to \text{D-mod}_{\mathrm{coh}}(\Y).
\end{align*}
In Section~\ref{sec:Dmod}, we will discuss more about the relationship between limit categories 
and D-modules.

\subsection{Compact generation of the limit category}\label{subsec:compactgeneration}
We will apply the construction of the limit category 
for the smooth (not quasi-compact in general) stack 
\begin{align*}\X=\Bun_G(\chi).
\end{align*}
For $w\in Z(G)^{\vee}$, we construct an element $\delta_w \in \mathrm{Pic}(\Bun_G(\chi))_{\mathbb{R}}$ with $Z(G)$-weight $w$, and define the category
\footnote{The subscript $(\delta=\delta_w, w)$ means taking a direct summand corresponding to $Z(G)$-weights $w$. See Definition~\ref{def:indlim:Higgs} for details.} \begin{align*}
\IndL(\Hig_G(\chi))_w := \lim_{U \subset \Bun_{G}(\chi)} \Ind(\LL(\Omega_{U})_{\delta = \delta_w, w}), \end{align*} where the limit is taken over all quasi-compact open substacks $U \subset \Bun_G(\chi)$.
Despite the failure of compact generation in Proposition~\ref{intro:prop:cpt}, we have the following result
which provides an analogue of the compact generation of $\text{D-mod}(\Bun_G)$ established in~\cite{DGbun}.

\begin{thm}\label{intro:thm:gen}\emph{(Theorem~\ref{thm:mainLG})}
The dg-category $\IndL(\Hig_G(\chi))_w$ is compactly generated
by the subcategory of compact objects $\LL(\Hig_G(\chi))_w$. 
\end{thm}

One of the important properties of the limit categories is that, 
for some class of open immersions $j \colon U \subset \X$,\footnote{More generally, we expect the existence of such a left adjoint functor for the complement of a
$\Theta$-stratum (\cite{HalpTheta}) in $\Omega_{\X}$. See the proof of Proposition~\ref{prop:sod:KZ}. We also expect the existence of $j_!$ for co-truncative open immersions in~\cite{DGbun}.} 
the pull-back $j^*$ admits a left adjoint
\begin{align}\label{eq:jlowershriek}
    j_{!} \colon \LL(\Omega_{\mathcal{U}})_{\delta} \to \LL(\Omega_{\X})_{\delta}.
\end{align}
For example, see the case of $\X=\mathbb{A}^1/\mathbb{G}_m$ in Subsection~\ref{subsub:A1}.
Note that $j_{!}$ does not exist for (quasi-)coherent sheaves, which is a long-standing issue in algebraic geometry. For D-modules, however, $j_{!}$ exists for co-truncative open 
immersions~\cite{DGbun}; this fact is crucial in proving the compact generation of $\text{D-mod}(\Bun_G)$. The same feature of limit categories is essential to Theorem~\ref{intro:thm:gen}; conversely, the lack of a functor $j_{!}$ for (quasi-)coherent sheaves is a source of the failure of the compact generation in Proposition~\ref{intro:prop:cpt}. 

The existence of~\eqref{eq:jlowershriek} is also needed to formulate a version of Whittaker normalization for the equivalence~\eqref{intro:cohL} when $\chi=0$, namely that the structure sheaf of $\Hig_{{}^LG}^{\text{ss}}$ is sent by~\eqref{intro:cohL} to $j_!s_*\mathcal{O}_{\mathrm{B}}$, where $s\colon \mathrm{B}\hookrightarrow \Hig_{G}^{\mathrm{ss}}$ is the Hitchin section and $j\colon \Hig_{G}^{\mathrm{ss}}\hookrightarrow \Hig_{G}$ is the natural open immersion, see Remark~\ref{rmk:Hitchin}.

As a corollary of Theorem~\ref{intro:thm:gen}, we have the following: 
\begin{cor}\label{intro:cor:eq}
There is an equivalence (\ref{intro:cohL}) if and only if there is an equivalence (\ref{intro:eq:CohL}). 
\end{cor}

Note that we prove both Theorem~\ref{conj:intro} and Theorem~\ref{intro:thm:gen} at the same time, as Theorem~\ref{thm:mainLG}. Its proof is given in Section~\ref{sec:sodlim}. We include in the beginning of Section~\ref{sec:sodlim} a flowchart for the proof, and now we only mention some important steps.
As mentioned above, the crucial step is the existence of $j_!$ for the open immersion of the complement of a Harder-Narasimhan (alternatively, $\Theta$-)stratum, where $\mu$ is a dominant coweight:
\[j\colon \Hig_G(\chi)_{\preceq \mu}\setminus \Hig_G(\chi)_{(\mu)}\hookrightarrow \Hig_G(\chi)_{\preceq \mu}.\]
The stack $\Hig_G(\chi)_{\preceq \mu}$ may be written as the zero locus of a section of a vector bundle on the \textit{smooth} stack 
$\Hig^L_G(\chi)_{\preceq \mu}$ of $L$-twisted Higgs bundles, where $\deg(L)\gg 0$. The existence of $j_!$ 
for limit categories of $\Hig_G(\chi)_{\preceq \mu}$ follows, using matrix factorizations and the Koszul equivalence, from an analogous statement about \textit{magic categories} of the smooth stack $\Hig^L_G(\chi)_{\preceq \mu}$. Magic categories, which we discuss in Section~\ref{sec:magic}, are also defined using weight conditions for maps from the stack $B\mathbb{G}_m$, but they are considered for \textit{smooth} stacks, and thus are more tractable than limit categories.

\subsection{Limit categories with nilpotent singular supports}
We will also consider the subcategory of the limit category with nilpotent singular 
supports 
\begin{align*}
    \IndL_{\mathcal{N}}(\Hig_G(\chi))_w \subset \IndL(\Hig_G(\chi))_w.
\end{align*}
For $w=0$, it is expected to be a suitable classical limit of the category of 
$1/2$-twisted dg-category of D-modules on $\Bun_G(\chi)$, 
\begin{align*}
    \text{D-mod}(\Bun_G(\chi))_{\frac{1}{2}}\leadsto \IndL_{\mathcal{N}}(\Hig_G(\chi))_0.
\end{align*}
We expect that an equivalence in
Conjecture~\ref{conj:intro:higgs} restricts to an equivalence of subcategories with nilpotent 
singular supports: 

\begin{conj}\label{conj:Higgs2} 
An equivalence (\ref{intro:cohL}) 
restricts to an equivalence:
\begin{align}\label{intro:equiv:ind2} \IndCoh_{\mathcal{N}}(\Hig_{^{L}G}(w)^{\mathrm{ss}})_{-\chi} \simeq \IndL_{\mathcal{N}}(\Hig_{G}(\chi))_{w}. 
\end{align} 
\end{conj}

\begin{remark}\label{rmk:conj} 
The nilpotent singular support condition for limit categories is discussed in Subsection~\ref{subsec:ssupport}.
Note that the corresponding condition for D-modules is unconditionally 
satisfied. 
As we will mention in Remark~\ref{rmk:limit2}, an equivalence (\ref{intro:equiv:ind2}) is expected
to be a genuine classical limit of GLC~(\ref{intro:GLC}). 

On the other hand, a version of GLC without the nilpotent singular support condition is briefly discussed in~\cite[Remark~1.6.9]{GLC1}, under the name of the \textit{renormalized} geometric Langlands conjecture:
\begin{align*}
    \IndCoh(\mathrm{Locsys}_{^LG}) \simeq \text{D-mod}(\Bun_G)^{\mathrm{ren}}
\end{align*}
where the compact objects of the right hand side are locally compact D-modules on $\Bun_G$, denoted by 
\begin{align*}
    \text{D-mod}(\Bun_G)^{\rm{lcp}} \subset \text{D-mod}(\Bun_G).
\end{align*}
An equivalence (\ref{intro:cohL}) is regarded as a classical limit of renormalized GLC. 
A proof of the renormalized version is not currently available in the literature.
\end{remark}

\subsection{Hecke operators on limit categories for Higgs bundles}\label{subsec:Heckeoperators}
Recall that the GLC is an equivalence 
\begin{align}\label{intro:GL}
    \IndCoh_{\mathcal{N}}(\mathrm{Locsys}_{^{L}G}) \simeq \text{D-mod}(\Bun_{G}). 
\end{align}
On the right-hand side, there are Hecke operators defined 
using the stack of Hecke modifications at points in $C$. Hecke operators are a main ingredient in defining the functor realizing the equivalence~\eqref{intro:GL}, see~\cite{GLC1}. 

Under the equivalence (\ref{intro:GL}), the Hecke operators are compatible with Wilson operators 
in the left-hand side. 
Using functorial properties of limit categories, we will construct Hecke operators on limit categories
via an action of 
quasi-BPS categories of zero-dimensional sheaves on the local surface $S=\mathrm{Tot}(\Omega_C)$. 

Let $\mathcal{S}(m)$ be the derived moduli stack of coherent sheaves $Q\in \Coh^{\heartsuit}(S)$
such that $\dim Q=0$ and $\chi(Q)=m$. The quasi-BPS category for zero-dimensional sheaves
\begin{align*}
    \mathbb{T}(m)_{w} \subset \Coh(\mathcal{S}(m))
\end{align*}
is considered in~\cite{P2, PT0, PT2}, and appears as a wall-crossing term in the 
categorical DT/PT correspondence for local Calabi–Yau 3-folds~\cite{PT2}. 
The direct sum
\begin{align*}
    \mathbb{H}:=\bigoplus_{m\in \mathbb{Z}_{\geq 0}}\mathbb{T}(m)_0
\end{align*}
admits a monoidal structure via categorical Hall products, see
Subsection~\ref{subsec:qbps:zero} and also~\cite{P2, PT0, PT2}.
Using the functors (\ref{intro:pback}), (\ref{intro:push}), we will 
construct the Hecke operators via the action of $\mathbb{H}$ on 
limit categories. 
\begin{thm}\emph{(Corollary~\ref{cor:1}, Corollary~\ref{cor:1prime})}\label{intro:cor2}
For a fixed $(r, w)\in \mathbb{Z}^2$, there exist right and left monoidal actions on 
     the following direct sum 
    \begin{align*}
   \mathbb{H} \curvearrowright     \bigoplus_{\chi \in \mathbb{Z}}\LL(\Hig_{\mathrm{GL}_r}
   \mathbb(\chi))_w \curvearrowleft\mathbb{H}.
    \end{align*}
\end{thm}
As a corollary, we have the following:
\begin{cor}\emph{(Corollary~\ref{cor:1}, Corollary~\ref{cor:1prime})}\label{intro:cor3}
For coprime $(r, w)$, there exist right and left monoidal actions on 
     the following direct sum of BPS categories 
    \begin{align*}
 \mathbb{H} \curvearrowright     \bigoplus_{\chi \in \mathbb{Z}}
 \mathbb{T}_{\GL_r}(\chi)_w \curvearrowleft\mathbb{H}.
    \end{align*} 
\end{cor}
If $(r, \chi)$ are coprime, then $\mathbb{T}_{\GL_r}(\chi)$ is equivalent to the derived 
category of moduli space of stable Higgs bundles of rank $r$ and degree $\chi$. 
Note that the construction of Hecke operators on moduli spaces of stable Higgs bundles 
have suffered from an issue with stability conditions: a Hecke 
modification of a (Higgs) bundle does not preserve the (semi)stability in general
\footnote{In references such as~\cite[Section~3.4]{HauICM}, the Hecke operators have been discussed by ignoring 
this issue.}\footnote{This issue has been resolved in an alternative way in cohomology~\cite{HMMS, MMSV}, but the techniques of loc.cit. do not apply in our context.}.
The result of Corollary~\ref{intro:cor3} resolves this issue, and 
considering the limit category for the full stack of 
Higgs bundles (without stability) is essential in constructing the Hecke operator. 

For a general reductive group $G$, we construct a Hecke functor associated with a minuscule coweight $\mu$:\footnote{As we discuss in Subsection~\ref{subsec:heckeG}, there are still obstructions to define a Hecke functor if $\mu$ is not minuscule, either due to infinite-dimensionality of Hecke correspondences or to singularities of Schubert cycles in the affine Grassmannians.}
\begin{align*}
    H_x^{\mu} \colon \LL(\Hig_G(\chi))_w \to \LL(\Hig_G(\chi+\overline{\mu}))_{w},
\end{align*}
see Subsection~\ref{subsec:heckeG}. 

Using the above Hecke functors, we will also formulate a compatibility of an equivalence 
in Conjecture~\ref{conj:Higgs} with Wilson/Hecke operators, see Conjecture~\ref{conj:Hecke}. For previous results related to the Wilson/Hecke compatibility in the DGLC (for complexes supported in the locus of stable Higgs bundles), see~\cite{HauICM, HauHit, HaMePe, DFang}. 

We also remark that these operators do not preserve the semiorthogonal decompositions (\ref{intro:sod:general}), (\ref{intro:sod:L}). However, we can control the amplitude,
see Proposition~\ref{prop:Hecke:sod}. 
Also, see Example~\ref{exam:Hecke} for the case of $G=\GL_2$ and $g=0$ and its comparison 
with the Wilson operator in the spectral side.

\subsection{Limit categories and Hall algebras}\label{subsec:KHA}

The purpose of this subsection is two-fold. We make some remarks on a (possibly modified) equivalence~\eqref{equiv:dolb3} compatible with the action of the Dolbeault version of geometric Satake~\cite{BFM, CW}, that is, with actions of both Hecke and Wilson operators on both sides. At the same time, we explain the connection (conjectural at this point) between the main equivalence~\eqref{intro:cohL} and Hall algebras~\cite{MR2944034, OSc, PoSa}. Indeed, Hall algebras capture several striking numerical consequences of the GLC~\cite{MR2944034, SV, OSc, KSV}, and it is instructive to look at the consequences of Conjecture~\ref{conj:intro:higgs} for K-theoretic Hall algebras. Details will appear elsewhere. For simplicity, we will assume in this subsection that $G=\GL_r$ and $\chi=0$.

A brief explanation is as follows. A (possibly modified) equivalence~\eqref{equiv:dolb3} which satisfies certain natural conditions may be alternatively regarded as a categorical Hall algebra $\text{HA}(C)$ with an $\text{SL}(2,\mathbb{Z})$-symmetry. Further, the Grothendieck group of $\text{HA}(C)$ is a \textit{toroidal} algebra of a Lie algebra $\mathfrak{g}$ (the BPS Lie algebra of semistable degree zero Higgs bundles). Even more, $\text{HA}(C)$ has an action of the double $\text{DHA}_0(C)$ of the categorical Hall algebra of points on $\Omega_C$, which factors through the action of Hecke and Wilson operators. 
Passing to two (equivalent) slope subcategories (whose Grothendieck groups are \textit{affine} algebras of $\mathfrak{g}$), we obtain the equivalence~\eqref{intro:cohL} for $\chi=0$. It is compatible with the action of Wilson and Hecke operators, each corresponding to two (equivalent) slope subcategories of $\text{DHA}_0(C)$, and compatible with parabolic induction (called, in the case $G=\GL_r$, the Hall product)\footnote{It is natural to inquire for a $\mathbb{G}_m$-equivariant version of the above, where $\mathbb{G}_m$ scales the Higgs field, but we do not discuss it in this paper.}.

\medskip



Going back to Remark~\ref{remark:newDolb}, we may inquire for an equivalence~\eqref{equiv:dolb3}:
 \begin{equation}\label{equiv:dolb2}
  \Coh(\Hig_{^LG})^\dagger\simeq \Coh(\Hig_{G})^\dagger,
  \end{equation} 
  which is a modification of~\eqref{equiv:dolb3}. 
  Here, the category $\Coh(-)^{\dag}$ is a (yet to be defined) modification of $\Coh(-)$, which keeps track of 
 the non-quasi-compactness of $\Hig_G$. The `best case scenario' is that this equivalence is:
 \begin{itemize}
     \item compatible with parabolic induction,
     \item compatible with the action of the `double' of the categorical Hall algebra of points on the local surface $S=\Omega_C$, which factors through the action of Hecke and Wilson operators~\cite{CW}, 
     \item is an \textit{affinization} of the de Rham Langlands equivalence (which is a statement about the dimension of its topological K-theory).
 \end{itemize}
For $G=\GL_r$, the equivalence~\eqref{equiv:dolb2}, together with the equivalence obtained by tensoring the determinant line bundle, generates an $\text{SL}(2,\mathbb{Z})$-action on the $\mathbb{Z}^{\oplus 2}$-graded category: \[\text{HA}_r(C):=\bigoplus_{\chi,w\in\mathbb{Z}} \mathrm{Coh}(\Hig_{\mathrm{GL}_r}(\chi))^\dagger_w.\] 

In what follows, we assume that the equivalence~\eqref{equiv:dolb2} satisfies all the properties listed above. We further assume that some (topological) K-theoretic Hall algebras satisfy PBW-type theorems. 
Partial progress towards such results were obtained in~\cite{PT1, PTtop, PThiggs2}.
We also assume the existence of a `double' of the categorical Hall algebra with an action of $\text{SL}(2,\mathbb{Z})$-action, for which the results in~\cite{CauPT} provide partial evidence.
We finally assume a comparison between slope subcategories and the categories in~\eqref{conj:DLequiv}.
 
\medskip


 Above, \textit{affinization} refers to a comparison between the topological K-theories of $\text{D-mod}(\Bun_G(\chi))^{\text{lcp}}$ and $\mathrm{Coh}(\Hig_G(\chi))^\dagger$, see the isomorphism~\eqref{iso:KHAC}, and it is motivated by the following observation due to Grojnowski~\cite{GrojnowskiU}.

Let $Q$ be a finite type quiver with set of vertices $I$ and corresponding nilpotent Lie algebra $\mathfrak{n}$. For a dimension vector $d\in\mathbb{N}^I$, let $\X(d)$ be the stack of representations of dimension $d$ of $Q$. There are PBW isomorphisms for $\mathbb{N}^I$-graded and $\mathbb{N}^I\times\mathbb{Z}$-graded vector spaces, respectively, where $q$ has degree $(0,1)$:
\begin{align}\label{intro:isom:PBW}
\bigoplus_{d\in\mathbb{N}^I}K_0\left(\text{D-mod}_{\mathrm{coh}}(\X(d))\right)_\mathbb{Q}&\cong \mathrm{Sym}\left(\mathfrak{n}\right),\\
\notag
\bigoplus_{d\in\mathbb{N}^I}K_0\left(\mathrm{Coh}(\Omega_{\X(d)})\right)_\mathbb{Q}&\cong \operatorname{Sym}\left(\mathfrak{n}\otimes \mathbb{Q}[q^{\pm 1}]\right).
\end{align}
The details of the isomorphisms (\ref{intro:isom:PBW}), and a discussion of more general quivers and of Remark~\ref{rem:singularsupport}, will appear elsewhere. We remark that the two left-hand sides are Hall algebras, the first one analogous to the Ringel--Lusztig construction, while the second one is the K-theoretic Hall algebra~\cite{VaVa}. 

\medskip

For a dg-category $\mathscr{D}$, denote by $K^{\text{top}}_\bullet(\mathscr{D})_\mathbb{Q}$ the (rational) homotopy groups of the topological K-theory~\cite{Blanc} of $\mathscr{D}$, which is a 2-periodic $\mathbb{Q}$-vector space. 
We denote by $\mathrm{BPS}(r)$ the 2-periodic version of the BPS cohomology~\cite{DM, KinjoKoseki} of $\Hig_{\GL_r}(\chi)^{\text{ss}}$, which is independent of $\chi\in \mathbb{Z}$ by~\cite{KinjoKoseki}. In particular, there are isomorphisms of $2$-periodic $\mathbb{Q}$-vector spaces:
\[\text{BPS}(r)\cong K^{\text{top}}_\bullet(\text{H}_{\mathrm{GL}_r}(1)^{\text{ss}})_\mathbb{Q},\] where $\text{H}_{\mathrm{GL}_r}(1)^{\text{ss}}$ is the usual (classical, smooth, proper over the Hitchin base) moduli space of semistable Higgs bundles of degree one. 
We expect that there is a PBW isomorphism for $\mathbb{N}\times\mathbb{Z}$-graded vector spaces:
\[\bigoplus_{r\in \mathbb{N}, w\in\mathbb{Z}}K_\bullet^{\text{top}}\left(\mathrm{Coh}(\Hig_{\GL_r}(0)^{\text{ss}})_w\right)_\mathbb{Q}\cong \mathrm{Sym}\left(\bigoplus_{r\in\mathbb{N}}\mathrm{BPS}(r)\otimes \mathbb{Q}[q^{\pm 1}]\right),\] where $q$ has degree $(0,1)$. For partial progress towards such an isomorphism, see~\cite{PThiggs2}.
A similar PBW isomorphisms is expected to hold for the stack of local systems of rank $r$, whose ($2$-periodic) BPS cohomology is isomorphic to the BPS cohomology of $\Hig_{\mathrm{GL}_r}(0)^{\text{ss}}$, see~\cite{Henn}. 
We also expect the following Langlands-dual PBW isomorphism for the category of locally compact D-modules on $\Bun_{\mathrm{GL}_r}$:
\[\bigoplus_{r\in \mathbb{N}, \chi\in\mathbb{Z}}K_\bullet^{\text{top}}\left(\text{D-mod}(\Bun_{\mathrm{GL}_r}(\chi)\right)^{\mathrm{lcp}})_\mathbb{Q}\cong \mathrm{Sym}\left(\bigoplus_{r\in\mathbb{N}}\mathrm{BPS}(r)\otimes \mathbb{Q}[z^{\pm 1}]\right),\] where $z$ has degree $(0,1)$. See Remark~\ref{rmk:conj} for the definition of the left hand side. 

Recall that we assume that the equivalences~\eqref{equiv:dolb2}
 are compatible with parabolic induction. 
  Consider the monoidal category 
  \begin{equation}\label{def:Hall}
  \text{HA}(C):=\bigoplus_{r\in\mathbb{N}} \text{HA}_r(C)
  \end{equation}
  and its corresponding K-theoretic Hall algebra:
  \[\mathrm{KHA}(C):=K^{\text{top}}_\bullet(\mathrm{HA}(C)).\] 
  Following Grojnowski's observation for quivers, we assume that the following PBW isomorphism of  $\mathbb{N}\times\mathbb{Z}^{\oplus 2}$-graded vector spaces holds, where $z$ has degree $(0,1,0)$ and $q$ has degree $(0,0,1)$:
  \begin{equation}\label{iso:KHAC}
  \mathrm{KHA}(C)_\mathbb{Q}\cong \mathrm{Sym}\left(\bigoplus_{r\in\mathbb{N}}\mathrm{BPS}(r)\otimes\mathbb{Q}[z^{\pm 1}, q^{\pm 1}]\right).
  \end{equation}


\medskip

We next discuss the action of Hecke and Wilson operators via the `double' $\mathrm{DHA}_0(C)$ of the categorical Hall algebra $\mathrm{HA}_0(C)$ of points on $T^*C$. We postpone precise conjectures about the construction of $\mathrm{DHA}_0(C)$ to future work. 
The category $\mathrm{DHA}_0(C)$ is $\mathbb{Z}^{\oplus 2}$-graded, has an action of $\text{SL}(2,\mathbb{Z})$-action (for which we provide partial evidence in~\cite{CauPT}), and acts on the $\mathbb{Z}^{\oplus 2}$-graded category $\text{HA}_r(C)$ compatibly with the $\text{SL}(2,\mathbb{Z})$-actions. Further, the action of $\text{DHA}_0(C)$ on $\text{HA}_r(C)$ factors through the action of Hecke and Wilson operators.

Recall from Subsection~\ref{subsec:Heckeoperators} that $\mathcal{S}(\chi)$ is the stack of length $\chi$ sheaves on $S=\Omega_C$. 
 Further, $\text{DHA}_0(C)$ contains the categorical Hall algebra of points on $C$:
  \[\text{DHA}_0(C)\supset \text{HA}_0(C):=\bigoplus_{\chi>0,w\in\mathbb{Z}}\Coh(\mathcal{S}(\chi))_w,\] which acts by correspondences on $\text{HA}_r(C)$. 
We expect that there is an isomorphism of $\mathbb{Z}^{\oplus 2}$-graded vector spaces, where $z$ has degree $(1,0)$ and $q$ has degree $(0,1)$:
\[K^{\text{top}}_\bullet(\text{DHA}_0(C))_\mathbb{Q}\cong \text{Sym}\left(K_\bullet(S)\otimes \mathbb{Q}[z^{\pm 1}, q^{\pm 1}]\right).\]

  Recall the slope subalgebras of quantum affine algebras~\cite{OkSm, Negutthesis}. We expect natural constructions of slope subcategories of categorical Hall algebras of quivers and curves (and their doubles, when defined); for an example, see~\cite{PT1}. 
Consider slope subcategories of $\text{HA}(C)$ and $\text{DHA}_0(C)$ for the $\text{slope}:=w/\chi$. 
  We expect the action of $\text{DHA}_0(C)$ on $\text{HA}_r(C)$ to specialize, when $\text{slope}=\infty$, to the action of the Wilson operators:
  \[\text{DHA}_0(C)_{\text{slope}=\infty}\curvearrowright \text{HA}_r(C)_{\text{slope}=\infty}\simeq \bigoplus_{w\in\mathbb{Z}}\mathrm{Coh}(\Hig_{\mathrm{GL}_r}(0)^{\text{ss}})_w. \] 
  When $\text{slope}=0$, we expect the action of $\text{DHA}_0(C)$ on $\text{HA}_r(C)$ to specialize to the action of Hecke operators from Subsection~\ref{subsec:Heckeoperators}:
  \[\text{DHA}_0(C)_{\text{slope}=0}\curvearrowright \text{HA}_r(C)_{\text{slope}=0}\simeq \bigoplus_{\chi\in\mathbb{Z}}\text{L}(\Hig_{\mathrm{GL}_r}(\chi))_0. \]
The $\mathrm{SL}(2,\mathbb{Z})$-action induces a commutative diagram:
\begin{align*} \xymatrix{ \text{DHA}_0(C)_{\text{slope}=0} \ar[d]^{\simeq} & \curvearrowright & \text{HA}_r(C)_{\text{slope}=0} \ar[d]_{\Phi}^{\simeq} \\ 
\text{DHA}_0(C)_{\text{slope}=\infty} & \curvearrowright & \text{HA}_r(C)_{\text{slope}=\infty}. } \end{align*} 
      The equivalence $\Phi\colon \text{HA}_r(C)_{\text{slope}=0}\simeq \text{HA}_r(C)_{\text{slope}=\infty}$ is then the same as the equivalences~\eqref{intro:cohL}, compatible with parabolic induction, for $G=\mathrm{GL}_r$ and $\chi=0$ and for all $w\in \mathbb{Z}$.

\subsection{Future directions}

Besides the conjectures already made in the introduction, we already mentioned some future directions of research, see Subsections~\ref{subsec:catDT} and~\ref{subsec:KHA} and Remark~\ref{rem:singularsupport}. Let us mention a few more natural questions which we do not address in this paper, but which we hope to discuss in future work. 

\begin{enumerate}[leftmargin=*]
\item One may inquire for a $\mathbb{G}_m$-equivariant version of the equivalence~\eqref{intro:cohL}. Note that such an action is also natural from the perspective of Hall algebras discussed in Subsection~\ref{subsec:KHA}. A feature of such an action is the existence of further semiorthogonal decomposition of $\mathbb{G}_m$-equivariant BPS categories. Here, note that the BPS categories~\eqref{intro:qbps} (so those with $\gcd(r,\chi,w)=1$) are indecomposable~\cite{PThiggs}. 

\item It is interesting to pursue a Dolbeault version of the spectral decomposition from~\cite{AGKRRV}. For irreducible local systems, the corresponding summand on the automorphic (limit category) side is a classical limit of the Hecke eigensheaves~\cite{FGV}. For the trivial local system, one expects the summand to be a category whose Grothendieck group is the top nilpotent cohomological Hall algebra of the curve~\cite{SalaSchif}. Langlands duality (as in~\cite{AGKRRV}) implies the equivalence between such a category and a categorical preprojective (nilpotent) Hall algebra of the $g$-loop quiver, where $g$ is the genus of the curve. These categories are further expected to be related to a categorification of the spherical (Ringel) Hall algebra of the curve~\cite{MR2944034, OSc}\footnote{We thank Olivier Schiffmann for discussions about this.}.

\item One may inquire for the limit category analogue of the category of tempered D-modules, and for a description of it in terms of the singular support analogous to that of Fægerman--Raskin for categories of sheaves with nilpotent singular support~\cite{FaeRas}.  

\item It is interesting to construct semiorthogonal decompositions (analogous to the one from Theorem~\ref{conj:intro}) directly for the automorphic category in both the de Rham and Betti~\cite{ZN} settings. In the de Rham setting, such a decomposition was announced in~\cite{McNe}.
As a consequence, we obtain equivalences between the quotient of the automorphic category by the subcategory of object with singular support in the unstable locus $\Hig^{\text{uns}}_{G}:=\Hig_{G}\setminus \Hig^{\text{ss}}_{G}$ and quasi-BPS categories of the spectral categories.
\end{enumerate}

\subsection{Updates on recent developments}

Since the appearance of the first version of this paper on arXiv, 
there have been several further developments concerning the Dolbeault geometric Langlands conjecture proposed in this paper. 
Here we describe some of them. 

First, the second author proved Conjecture~\ref{conj:intro:higgs} (with nilpotent singular support) for $G=\GL_r, \mathrm{SL}_r, \mathrm{PGL}_r$ over the locus $\mathrm{B}^{A} \subset \mathrm{B}$ of the Hitchin base where the spectral curves have at worst type A singularities, i.e. formally locally of the form 
$\{y^2=x^m\}$ for some $m$. 
More precisely, we have the following: 
\begin{thm}\emph{(\cite{TodaGL2, TodaA})}\label{thm:BA}
For $G=\GL_r, \mathrm{SL}_r$ or $\mathrm{PGL}_r$, 
there is a $\mathrm{B}^{A}$-linear equivalence 
\begin{align*}
    \IndCoh_{\mathcal{N}}(\Hig_{\GL_2}(w)^{\mathrm{ss}}\times_{\mathrm{B}}\mathrm{B}^{A})_{-\chi} \stackrel{\sim}{\to}
    \IndL_{\mathcal{N}}(\Hig_{\GL_2}(\chi)\times_{\mathrm{B}}\mathrm{B}^{A})_{w}
\end{align*}
which is compatible with Wilson/Hecke operators. 
\end{thm}
In the case of $G=\GL_2, \mathrm{SL}_2$ or $\mathrm{PGL}_2$, we have $\mathrm{B}^A=\mathrm{B}^{\mathrm{red}}$, where
$\mathrm{B}^{\mathrm{red}} \subset \mathrm{B}$ corresponds to reduced spectral curves. The result of Theorem~\ref{thm:BA}
was first proved for $G=\GL_2$ in~\cite{TodaGL2}, and is then proved for $G=\GL_r, \mathrm{SL}_r, \mathrm{PGL}_r$
by applying and extending the general strategy of~\cite{TodaGL2}.

Note that such an equivalence had previously been known only over the elliptic locus 
$\mathrm{B}^{\mathrm{ell}} \subset \mathrm{B}^{\mathrm{red}}$, where the spectral curves are irreducible, by Arinkin~\cite{Ardual}. 
Over the elliptic locus, the stack 
$\Hig_G(\chi)$ is quasi-compact and consists of stable Higgs bundles, so there is no need to use limit categories. 
However, over $\mathrm{B}^{A}$, the stack $\Hig_G(\chi)$ is not 
quasi-compact, and there are strictly semistable Higgs bundles on the 
left-hand side of the correspondence. This is the first case 
in which the theory of limit categories developed in this paper is essential both for 
the formulation and for the proof. 

Secondly, in the recent joint work with Chenjing Bu, we extend the 
semiorthogonal decomposition (\ref{intro:sod:general})
to any reductive group $G$. 
\begin{thm}\emph{(\cite{BPT})}
For any reductive group $G$, there is a semiorthogonal decomposition (\ref{intro:sod:general}). 
\end{thm}
The main result of~\cite{BPT} is more general, and applies to arbitrary 
quasi-smooth derived stacks with good moduli spaces. It recovers the semiorthogonal decompositions in~\cite{PTK3, PThiggs}. 
In this generality, the semiorthogonal decomposition is described 
in terms of component lattices, which are a central object in ``the intrinsic DT theory" started in~\cite{bhik}.

\subsection{Acknowledgements}
We would like to thank Sabin Cautis, Eric Chen, Tasuki Kinjo, Davesh Maulik, Olivier Schiffmann, Francesco Sala, Junliang Shen, Philsang Yoo and Eric Vasserot for discussions related to this work. 
T.P. thanks CNRS (grant PEPS JCJC) and FSI Sorbonne Université (grant Tremplins) for providing financial support to visit IPMU between 14-25 October 2024 and 24 February-7 March 2025. T.P. thanks IPMU for its hospitality and excellent working conditions during these visits. 
Y.~T.~is supported by World Premier International Research Center
	Initiative (WPI initiative), MEXT, Japan, 
    and Inamori Research Institute for Science, 
    and JSPS KAKENHI Grant Number JP24H00180.

\subsection{Materials and Methods}
The AI tool (ChatGPT) is only used to check typo and English corrections.

\begin{figure}
	\centering
\scalebox{0.7}{
	\begin{tabular}{|l|l|l|}
		\hline
	Notation & Description & Location defined \\\hline
 $\LL(\mathfrak{M})_{\delta}, \LL(\Omega_{\mathcal{X}})_{\delta}$ & limit categories  & Subsection~\ref{subsec:deflim}
\\ \hline
 $\IndL(\mathfrak{M})_{\delta}, \IndL(\Omega_{\mathcal{X}})_{\delta}$ & ind-limit categories  & Subsection~\ref{subsec:deflim}
\\ \hline
$\fg_x$, $T_x$, $\mathfrak{o}_x$ & cohomologies of tangent complex &
Subsection~\ref{subsec:deflim}
\\ \hline
$\nu$, $\nu^{\mathrm{reg}}$ & map from $\bgm$ and its regularization &
Subsection \ref{subsec:deflim}
\\ \hline
$\mu$ & moment map &
Subsection \ref{subsec:deflim}
\\ \hline 
$\delta$ & $\mathbb{R}$-line bundle on a stack &
Subsection \ref{subsec:deflim}
\\ \hline
$\mathrm{Supp}^{\mathrm{AG}}, \mathrm{Supp}^L$ & singular supports &
Subsection~\ref{subsec:singular}
\\ \hline
$(-)_{\mathcal{N}}$ & nilpotent singular supports &
Subsection \ref{subsec:limsupp}
\\ \hline
$f^{\Omega!}, f_{*}^{\Omega}$ & induced functors on cotangents &
Subsection \ref{subsec:funccotan}
\\ \hline
$\MM(\Y)_{\delta}$ & magic category & Subsection~\ref{subsec:magic}
\\ \hline
$n_{\Y, \nu}$ & weight of positive part of cotangent complex & Definition~\ref{defn:QBY}
\\ \hline
$\Theta$ & $\Theta$-stack $\Theta=\mathbb{A}^1/\mathbb{G}_m$ & Subsection~\ref{subsec:thetast}
\\ \hline
$\rho$ & half sum of positive weights & Proposition~\ref{prop:QB:weight}
\\ \hline
$\beta_{>i}, \beta_{\leq i}$ & weight truncation & Equation (\ref{triangle:beta})
\\ \hline
$\Delta_i$ & the functor between magic categories & Lemma~\ref{lem:delta}
\\ \hline
$\mathscr{E}_G$ & universal $G$-bundle & Equation (\ref{univ:E})
\\ \hline
$\Ad_{\fg}(E_G)$ & adjoint vector bundle & Subsection~\ref{subsec:limHiggs}
\\ \hline
$Z(G)$ & center of $G$ & Subsection~\ref{subsec:limHiggs}
\\ \hline
$\delta_w$ & $\mathbb{Q}$-line bundle associated with $w$ & Equation (\ref{deltaw})
\\ \hline
$\Hig_G$, $\Hig_G^{L}$ & moduli stack of ($L$-twisted) $G$-Higgs bundles & Subsection~\ref{subsec:sHiggs}
\\ \hline
$\Hig_G(\chi)^{\mathrm{ss}}$, $\Hig_G^{L}(\chi)^{\mathrm{ss}}$ & moduli stack of ($L$-twisted) semistable $G$-Higgs bundles & Subsection~\ref{subsec:ssH}
\\ \hline
$(E_G, \phi_G), (E_G, \phi_G^{L})$ & ($L$-twisted) $G$-Higgs bundle & Subsection~\ref{subsec:sHiggs}
\\ \hline
$(\mathscr{F}_G^{L}, \varphi_G^{L})$ & universal ($L$-twisted) $G$-Higgs bundle & Subsection~\ref{subsec:sHiggs}
\\ \hline
$\IndL(\Hig_G(\chi))_w$, $\LL(\Hig_G(\chi))_w$ & (ind)-limit category for Higgs bundles & Subsection~\ref{subsec:limHiggs}
\\ \hline
$N(T), M(T)$ &cocharacter/character lattice & Subsection~\ref{subsec:SOD0}
\\ \hline
$N(T)_{+}, M(T)_{+}$ & dominant chambers & Subsection~\ref{subsec:SOD0}
\\ \hline
$N(T)_{\mathbb{Q}+}^{W_M}, M(T)_{\mathbb{Q}+}^{W_M}$ & strictly $W$-dominant chambers & Subsection~\ref{subsec:SOD0}
\\ \hline
$\mathbb{T}_{G}(\chi)_w$ & quasi-BPS category & Definition~\ref{def:qbpsG}
\\ \hline
$\mu_M$ & slope maps & Equation (\ref{def:slopemaps})
\\ \hline
$a_M$ & an induced map on $\pi_1(-)\times Z(-)^{\vee}$ & Equation (\ref{def:aM})
\\ \hline
$\mathbb{T}^i_{(E_G, \phi_G^{L})}$ & cohomology of cotangent complex & Equation (\ref{def:Ti})
\\ \hline
$N(T)_{\mathrm{pos}}, N(T)_{\mathbb{Q}, \mathrm{pos}}$ & positive cones & Equation (\ref{pcones})
\\ \hline
$\leq_G$ & partial order on $N(T)_{\mathbb{Q}}$ & Subsection~\ref{subsec:ssH}
\\ \hline
$\Hig_G^{L}(\chi)_{(\mu)}$ & Harder-Narasimhan stratum & Subsection~\ref{subsec:HNhiggs}
\\ \hline
$\Hig_G^{L}(\chi)_{\preceq \mu}$ & union of HN strata & Subsection~\ref{subsec:HNhiggs}
\\ \hline
$\preceq$ & total order on $N(T)_{\mathbb{Q}+}$ & Subsection~\ref{subsec:HNhiggs}
\\ \hline
$\iota_D$ & closed embedding to smooth stack & Equation (\ref{iota:DC})
\\ \hline
$l(\mu)$ & lower bound of $\deg D$ & Lemma~\ref{lem:smooth}
\\ \hline
$\mathcal{V}^L$ & vector bundle on $\Hig_G^{L}(\chi)_{\preceq \mu}$ & Diagram (\ref{section:sL})
\\ \hline
$s^L$ & section of $\mathcal{V}^L$ & Diagram (\ref{section:sL})
\\ \hline
$\mathcal{S}_{\mathcal{V}}, \mathcal{Z}_{\mathcal{V}}$ & $\Theta$-stratum of $(\mathcal{V}^L)^{\vee}$ and its center & Subsection~\ref{subsec:thetashift}
\\ \hline
$\mathcal{C}(m)$ & moduli stack of zero-dimensional sheaves on curve & Subsection~\ref{subsec:qbps:zero}
\\ \hline
$\mathcal{S}(m)$ & moduli stack of zero-dimensional sheaves on surface & Subsection~\ref{subsec:qbps:zero}
\\ \hline $\mathbb{T}(m)_w$ & quasi-BPS category of zero-dimensional sheaves & Subsection~\ref{subsec:qbps:zero}
\\ \hline
$\mathcal{C}(m_1, m_2)$ & stack of exact sequences of zero-dimensional sheaves & Subsection~\ref{subsec:qbps:zero}
\\ \hline
$\mathcal{S}(m_1, m_2)$ & stack of exact sequences of zero-dimensional sheaves & Subsection~\ref{subsec:qbps:zero}
\\ \hline
$\mathbb{H}, \Ind\mathbb{H}$ & direct sum of (ind-)quasi-BPS categories & Subsection~\ref{subsec:qbps:zero}
\\ \hline $\mathrm{Hecke}_{\GL_r}$ & Hecke stack & Subsection~\ref{subsec:Hecke}
\\ \hline
$\mathrm{Hecke}_{\GL_r}^{\mathrm{Hig}}$ & Hecke stack for Higgs bundles& Subsection~\ref{subsec:Hecke}
\\ \hline
$\mathrm{H}_x^{\mu}$ & Hecke functor & Equation (\ref{funct:Hmu})
\\ \hline
	\end{tabular}
}
	\vspace{.5cm}
	 \caption{Notation used in the paper}
	 \notag
\end{figure}

\section{Preliminaries}\label{sec:QBcat}
In this section, we discuss the preliminary background necessary for this paper. 

\subsection{Conventions on stacks}
We work over an algebraically closed field $k$ of characteristic $0$.
Unless stated otherwise, all (derived) stacks are locally Noetherian and quasi-separated over $k$.
For a stack $\X$ over $k$, we write $\X(k)$ for the set of $k$-valued points of $\X$.
For a field extension $k'/k$, we denote by $\X_{k'}$ the base change to $k'$.

A stack is called \textit{QCA} if it is quasi-compact with affine geometric stabilizer groups.

For a scheme $Y$ with an action of an algebraic group $G$, we write $Y/G$ for the associated quotient stack.
We also use the notion of good moduli spaces of stacks, see~\cite{MR3237451}.

For a stack $\X$, let $\mathbb{L}_{\X}$ and $\mathbb{T}_{\X}$ denote its cotangent and tangent complexes.
For $m \ge 0$, its $m$-shifted cotangent stack is
\[
  \Omega_{\X}[m] := \Spec_{\X}\,\mathrm{Sym}\big(\mathbb{T}_{\X}[-m]\big) \longrightarrow \X .
\]
For a finite-dimensional $k$-vector space $V$ and $m \ge 0$, set
\[
  V[-m] := \Spec \mathrm{Sym}\big(V^{\vee}[m]\big),
\]
where $\mathrm{Sym}\big(V^{\vee}[m]\big)$ has zero differentials.

\subsection{Notations for dg-categories}
For a dg-category $\mathcal{C}$, denote by $\mathcal{C}^{\mathrm{cp}} \subset \mathcal{C}$ the subcategory of compact objects. The dg-category $\mathcal{C}$ is called \textit{compactly generated} if $\mathcal{C}=\Ind(\mathcal{C}^{\mathrm{cp}})$, where $\Ind(-)$ is the ind-completion. 

A dg-functor $F \colon \mathcal{C}_1 \to \mathcal{C}_2$ is called \textit{continuous} 
if it commutes with taking direct sums. 

For compactly generated dg-categories $\mathcal{C}_1$ and $\mathcal{C}_2$, their 
tensor product is defined by $\mathcal{C}_1 \otimes \mathcal{C}_2=\Ind(\mathcal{C}_1^{\mathrm{cp}}\boxtimes \mathcal{C}_2^{\mathrm{cp}})$, 
where $\mathcal{C}_1^{\mathrm{cp}}\boxtimes \mathcal{C}_2^{\mathrm{cp}}$ consist
of objects $A_1\boxtimes A_2$ for $A_i \in \mathcal{C}_i^{\mathrm{cp}}$ with 
$\Hom(A_1\boxtimes A_2, B_1\boxtimes B_2)=\Hom(A_1, B_1)\otimes \Hom(A_2, B_2)$.
We denote by $\mathcal{C}_1^{\mathrm{cp}}\otimes \mathcal{C}_2^{\mathrm{cp}}$ the subcategory 
of $\mathcal{C}_1\otimes \mathcal{C}_2$ split generated by $\mathcal{C}_1^{\mathrm{cp}}\boxtimes \mathcal{C}_2^{\mathrm{cp}}$. It equals to the subcategory of compact objects in $\mathcal{C}_1 \otimes \mathcal{C}_2$.

\subsection{Notations for (ind)coherent sheaves}
We use the notation of~\cite{MR3701352} for (ind)coherent sheaves. 
For a derived stack $\X$, we denote by 
$\Coh(\X)$ the dg-category of coherent sheaves,
by $\QCoh(\X)$ the dg-category of quasi-coherent sheaves,
and by $\IndCoh(\X)$ the dg-category of ind-coherent sheaves. 
For an affine derived scheme $\X=\Spec R$, we have that 
$\IndCoh(\X)=\Ind(\Coh(\X))$. For a stack $\X$, it is defined by 
\begin{align*}
    \IndCoh(\X)=\lim_{U\to \X}\IndCoh(U)
\end{align*}
for affine derived schemes $U \to \X$. If $\X$ is QCA, it is compactly generated with 
compact objects $\Coh(\X)$ by~\cite{MR3037900}.
Moreover, for QCA stacks $\mathcal{X}$ and $\mathcal{Y}$, there is an equivalence, 
see~\cite[Corollary~4.2.3]{MR3037900}:
\begin{align}\label{indtensor}
    \IndCoh(\mathcal{X})\otimes \IndCoh(\mathcal{Y}) \stackrel{\sim}{\to}\IndCoh(\mathcal{X}\times \mathcal{Y}).
\end{align}
We also refer to~\cite{MR3037900} for functorial properties of ind-coherent sheaves, e.g. 
for a morphism $f\colon \X \to \Y$, there are functors 
\begin{align*}
    f_{*}\colon \IndCoh(\X) \to \IndCoh(\Y), \ f^! \colon \IndCoh(\Y) \to \IndCoh(\X).
\end{align*}
When $\Y=\Spec k$, the functor $f_{*}$ is the (derived) global section functor 
$\Gamma$, and we write $H^i(-):=\mathcal{H}^i(\Gamma(-))$. 

If $\X=\Spec R$, we often denote $\Coh(\X)$ by $\Coh(R)$ (and so on) for simplicity. 
The heart of a standard t-structure on $\Coh(\X)$ is denoted by 
$\Coh^{\heartsuit}(\X)$. 
We also denote by $K(\mathcal{X})$ the Grothendieck group of 
perfect complexes on $\mathcal{X}$. 

For a morphism $f\colon \X \to \Y$, we denote by $\mathbb{L}_f$ the relative 
cotangent complex. If $\mathbb{L}_f$ is perfect, denote by $\omega_f:=\det \mathbb{L}_f$ the 
relative canonical line bundle. 
When $\Y=\Spec k$, denote $\mathbb{L}_f=:\mathbb{L}_{\X}$ and $\omega_f=:\omega_{\X}$. 

A morphism $f$ is called \textit{quasi-smooth} 
if $\mathbb{L}_f$ is perfect of cohomological amplitude $[-1, 1]$, and $\X$ is called 
\textit{quasi-smooth} if $\X \to \Spec k$ is quasi-smooth. 
If $f$ is quasi-smooth, we also have the $*$-pull-back for coherent sheaves, 
see~\cite[Section~4.2]{PoSa}
\begin{align*}
    f^* \colon \Coh(\Y) \to \Coh(\X).
\end{align*}
If $\mathcal{X}$ is quasi-smooth, then 
there are duality equivalences: 
\begin{align*}
\mathbb{D}_{\mathcal{X}} \colon \Coh(\mathcal{X})^{\mathrm{op}} \stackrel{\sim}{\to}
\Coh(\mathcal{X}), \ (-)\mapsto \mathcal{H}om(-, \omega_{\mathcal{X}}[\dim \mathcal{X}]).
\end{align*}
We often simply write $\mathbb{D}=\mathbb{D}_{\mathcal{X}}$ when $\mathcal{X}$ is 
obvious in the context.

\subsection{\texorpdfstring{Notations for $\mathbb{G}_m$-weights}{Notation on Gm-weights}}\label{subsec:Gmwt}
For a $\mathbb{G}_m$-gerbe $\X \to X$, there is an orthogonal weight 
decomposition 
\begin{align*}
    \Coh(\X)=\bigoplus_{m\in \mathbb{Z}}\Coh(\X)_m,
\end{align*}
where $\Coh(\X)_m$ is the subcategory of $\mathbb{G}_m$-weights $m$. 

For an object $A \in \Coh(\bgm)$,
let $A=\oplus_m A_m$ 
be its weight decomposition. We write $A^{>0}=\oplus_{m>0} A_m$. 
We denote by $\wt(A)\subset \mathbb{Z}$ the set of $m\in \mathbb{Z}$ such that $A_m\neq 0$, 
and $\wt^{\mathrm{max}}(A)\in \mathbb{Z}$ is the maximal element in $\wt(A)$. 
We will use the following map 
\begin{align*}
c_1 \colon K(\bgm)_{\mathbb{Q}} \to H^2(\bgm, \mathbb{Q})=\mathbb{Q}, 
\end{align*}
which equals $\sum_m m \cdot \dim A_m$. 

For a stack $\mathcal{X}$, a map $\nu \colon \bgm \to \mathcal{X}$, and an object
$A \in \mathrm{Perf}(\mathcal{X})$, write 
$\langle \nu, A \rangle:=c_1 (\nu^{\ast}A)$ and 
$\langle \nu, A^{\nu>0}\rangle :=c_1 (\nu^{\ast}A^{>0})$. 
The notations for negative weights $A^{<0}$ and $c_1(\nu^* A^{<0})$ are similarly defined. 

\subsection{Attracting and fixed loci}\label{notation:attracting}
For a reductive group $G$, a $G$-representation $Y$, and a cocharacter $\lambda \colon \mathbb{G}_m \to G$, denote by $Y^{\lambda \geq 0} \subset Y$ the subspace 
of $y\in Y$ such that $\lim_{t\to 0} \lambda(t)(y)$ exists (or equivalently the 
direct sum of $\mathbb{G}_m$-subrepresentations whose weights pair non-negatively with 
$\lambda$). Denote by $Y^{\lambda} \subset Y$ the $\lambda$-fixed subspace. 
Further denote by $G^{\lambda}\subset G^{\lambda\geq 0} \subset G$ the 
associated Levi and parabolic subgroups. 
Denote by $G^{\vee}=\Hom(G, \mathbb{G}_m)$ the character group. 

For the quotient stack $\Y=Y/G$ and a cocharacter $\lambda \colon \mathbb{G}_m \to G$, we write \[\Y^{\lambda \geq 0}:=Y^{\lambda \geq 0}/G^{\lambda \geq 0}\text{ and }\mathcal{Y}^{\lambda}:=Y^{\lambda}/G^{\lambda}.\] 
Let $\mathcal{V} \to \mathcal{Y}$ be a vector bundle with a section $s$, 
and $\mathfrak{M}=s^{-1}(0)$ the derived zero locus of $s$. 
It induces a section $s^{\lambda \geq 0}$ on $\mathcal{V}^{\lambda \geq 0} \to \Y^{\lambda \geq 0}$, and we denote by $\mathfrak{M}^{\lambda \geq 0}$ the derived zero locus of 
$s^{\lambda \geq 0}$.

\subsection{Notations for reductive groups}\label{subsec:notation:reductive}
For a reductive group $G$, fix a Borel subgroup $B\subset G$
and a maximal torus $T\subset B$.  
We denote by $\fg$ the Lie algebra of $G$ and $\mathfrak{h} \subset \fg$ the Cartan 
subalgebra. We denote by $W$ the Weyl group. 
We denote by $G^{\vee}:=\Hom(G, \mathbb{G}_m)$ the character group. 

We also use the notation $M(T)=T^{\vee}$ for the character 
lattice of $T$ and $N(T)$ the cocharacter lattice. 
A choice of $B$ determines dominant chambers 
$M(T)_{+}\subset M(T)$ and $N(T)_{+}\subset N(T)$.
We set $M(T)_{\mathbb{R}}:=M(T)\otimes_{\mathbb{Z}}\mathbb{R}$ etc. 
There is a perfect pairing 
\begin{align*}
    \langle -, - \rangle \colon M(T)_{\mathbb{R}}\times N(T)_{\mathbb{R}} \to \mathbb{R}. 
\end{align*}
For a $T$-representation $Y$, we often write $Y \in M(T)$ for the sum of its 
$T$-weights (equivalently, it is the $T$-character $\det Y$).

\subsection{The Koszul equivalence}\label{subsec:Koszuleq}
Consider a quotient stack $\Y=Y/G$, where $Y$ is a smooth quasi-projective scheme with 
an action of a reductive algebraic group $G$. 
Let $\mathcal{V} \to \mathcal{Y}$ be a vector 
bundle with a section $s$, and 
consider its derived zero locus
\begin{align}\label{funct:fV}
\mathfrak{M}=s^{-1}(0) \hookrightarrow \mathcal{Y}.
\end{align}
The section $s$ also induces the regular function \begin{equation}\notag
f\colon \mathcal{V}^{\vee} \to \mathbb{A}^1, \ 
(y, v)\mapsto \langle s(y), v\rangle
\end{equation}
for $y\in \mathcal{Y}$ and $v\in \mathcal{V}^{\vee}|_y$.
Consider the category of graded matrix factorizations \begin{align*}\text{MF}^{\text{gr}}\left(\mathscr{V}^{\vee}, f\right)
\end{align*}
with respect to the weight two $\mathbb{G}_m$-action on the fibers of $\mathcal{V}^{\vee} \to \mathcal{Y}$. 
It has objects pairs $(P, d_P)$, where $P$ is a $\mathbb{G}_m$-equivariant 
coherent sheaf on $\mathcal{V}^{\vee}$ and $d_P \colon P \to P(1)$ is a $\mathbb{G}_m$-equivariant morphism of weight one, satisfying $d_P \circ d_P=f$.
We refer to~\cite[Section~2.6.2]{PT0} for more details on the definition of graded matrix factorizations. 
The \textit{Koszul equivalence}, also called dimensional reduction in the literature, says the following:

\begin{thm}\emph{(\cite{I, Hirano, T})}\label{thm:Kduality}
There is an equivalence 
\begin{align}\notag
	\Psi \colon \Coh(\mathfrak{M})\stackrel{\sim}{\to}
	\mathrm{MF}^{\mathrm{gr}}(\mathcal{V}^{\vee}, f) 
	\end{align}
	given by $\Psi(-)=\mathcal{K}\otimes_{\mathcal{O}_{\mathfrak{M}}}(-)$. Here $\mathcal{K}$ is the following Koszul factorization, see~\cite[Theorem~2.3.3]{T}
    \begin{align}\notag
        \mathcal{K}=(\mathcal{O}_{\V^{\vee}}\otimes_{\mathcal{O}_{\Y}}\mathcal{O}_{\mathfrak{M}}, d_{\mathcal{K}}).
    \end{align}
	\end{thm}

Let $\Coh_{\mathbb{G}_m}(\V^{\vee})$ be the dg-category of $\mathbb{G}_m$-equivariant coherent 
sheaves on $\V^{\vee}$. 
    For a full subcategory $\mathbb{M} \subset \Coh(\V^{\vee})$, we denote by 
    $\mathbb{M}_{\mathbb{G}_m}$ the subcategory of $\Coh_{\mathbb{G}_m}(\V^{\vee})$ 
    consisting of objects which lie in $\mathbb{M}$ after forgetting the $\mathbb{G}_m$-action. 
    We denote by 
    \begin{align*}
        \mathrm{MF}^{\mathrm{gr}}(\mathbb{M}, f) \subset \mathrm{MF}^{\mathrm{gr}}(\V^{\vee}, f)
    \end{align*}
    the subcategory of matrix factorizations with factors in $\mathbb{M}_{\mathbb{G}_m}$, 
    see~\cite[Subsection~2.6.2]{PT0}. 

When $s=0$, we have $f=0$ and we use the notation 
\begin{align*}
    \Coh^{\shear}(\mathcal{V}^{\vee}):=
    \mathrm{MF}^{\mathrm{gr}}(\mathcal{V}^{\vee}, 0). 
\end{align*}
The notations $\mathrm{Perf}^{\shear}(\mathcal{V}^{\vee})$, 
$\mathrm{QCoh}^{\shear}(\mathcal{V}^{\vee})$ are similarly used. 
For a full subcategory $\mathbb{M}\subset \Coh(\mathcal{V}^{\vee})$, 
we also denote by $\Coh^{\shear}(\mathbb{M}):=\mathrm{MF}^{\mathrm{gr}}(\mathbb{M}, 0)$.

If $\mathcal{V}^{\vee} \to M$ is a good moduli space, we have the 
induced $\mathbb{G}_m$-action on $M$ 
by the universal property of good moduli spaces, see~\cite[Main properties (4)]{MR3237451}, and we also 
write $\Coh^{\shear}(M):=\mathrm{MF}^{\mathrm{gr}}(M, 0)$ and so on. 

\subsection{Limits and colimits of dg-categories}
    Here we review the relation of limits and colimits of dg-categories 
    following~\cite[Section~1.7]{DGbun}. 
Let 
\begin{align*}
    i \mapsto \mathcal{C}_i, \ (i\to j) \mapsto (\phi_{ij} \colon \mathcal{C}_i \to \mathcal{C}_j)
\end{align*}
be a diagram of dg-categories indexed by a category $I$, 
where each $\mathcal{C}_i=\Ind(\mathcal{C}_i^{\mathrm{cp}})$ is compactly generated 
by the subcategory of compact objects $\mathcal{C}_i^{\mathrm{cp}} \subset \mathcal{C}_i$, 
and $\phi_{ij}$ is a continuous functor. 
We have the limit of dg-categories $\mathcal{C}_i$ together with
the evaluation functor 
\begin{align}\label{C:lim}
    \mathcal{C}^{\mathrm{lim}}=\lim_{i\in I} \mathcal{C}_i \stackrel{\mathrm{ev}_i}{\to} \mathcal{C}_i. 
\end{align}
Informally, an object of $\mathcal{C}^{\mathrm{lim}}$ consists of 
\begin{align*}
    \{c_i, \alpha_{ij}\}, \ c_i \in \mathcal{C}_i, \ \alpha_{\phi_{ij}} \colon 
    \phi_{ij}(c_i) \stackrel{\sim}{\to} c_j
\end{align*}
with additional higher homotopical data, and $\mathrm{ev}_i$ sends 
the above data to $c_i$. 

Suppose that each $\phi_{ij}$ admits a continuous left adjoint 
\begin{align}\label{funct:phij}
\psi_{ji} \colon \mathcal{C}_j \to \mathcal{C}_i.     
\end{align}
Note that $\psi_{ji}$ must satisfy $\psi_{ji}(\mathcal{C}_j^{\mathrm{cp}}) \subset 
\mathcal{C}_i^{\mathrm{cp}}$
since it admits a continuous right adjoint $\phi_{ij}$. In this 
situation, by~\cite[Proposition~1.7.5]{DGbun}
there is a colimit 
\begin{align*}
    \mathcal{C}^{\mathrm{colim}}:=\colim_{i\in I^{\mathrm{op}}}\mathcal{C}_i
\end{align*}
with respect to functors (\ref{funct:phij}), and 
there are functors
\begin{align}\notag
\psi_i \colon \mathcal{C}_i \to \mathcal{C}^{\mathrm{colim}}
\end{align}
which satisfy
$\psi_i \circ \psi_{ji} \cong \psi_j$. 
The following proposition is given in~\cite{DGbun}: 
\begin{prop}\emph{(\cite{DGbun})}\label{prop:dgcat}
 (i) There is a canonical equivalence $\mathcal{C}^{\mathrm{lim}} \simeq \mathcal{C}^{\mathrm{colim}}$. 

 (ii) Under the equivalence in (i), the functor $\psi_i$ is a left adjoint of $\mathrm{ev}_i$. 

 (iii) We have 
 $\mathcal{C}^{\mathrm{colim}}=\Ind(\colim \mathcal{C}_i^{\mathrm{cp}})$. In particular 
 an object $\mathcal{E} \in \mathcal{C}^{\mathrm{colim}}$ is compact if and only 
 if it is isomorphic to $\psi_i(\mathcal{E}_i)$ for some $i\in I$ and 
 $\mathcal{E}_i \in \mathcal{C}_i^{\mathrm{cp}}$. 

 (iv) If each $\psi_{ji}$ is fully-faithful, then $\psi_i$ is also fully-faithful. 
 In this case, for $i \to j$ in $I$ 
 the natural morphism $\psi_{ji} \to \mathrm{ev}_i \psi_j$ is an isomorphism. \end{prop}
\begin{proof}
    (i), (ii) are given in~\cite[Proposition~1.7.5]{DGbun}. 
    (iii) is given in~\cite[Corollary~1.9.4]{DGbun}. The first part of (iv) is given in~\cite[Remark~1.7.6]{DGbun}. For the second part, 
    for $x \in \mathcal{C}_j$ there is a natural morphism 
    $\psi_{ji}(x) \to \mathrm{ev}_i\psi_j(x)$ by the adjunction 
    and $\psi_i \circ \psi_{ji} \cong \psi_j$. 
    For any $y \in \mathcal{C}_j$, the induced map 
    \begin{align*}
        \Hom(y, \psi_{ji}(x)) \to \Hom(y, \mathrm{ev}_i \psi_j(x))=\Hom(\psi_i(y), \psi_i \circ \psi_{ji}(x))
    \end{align*}
    is an isomorphism as $\psi_i$ is fully-faithful. Therefore 
    $\psi_{ji} \stackrel{\cong}{\to} \mathrm{ev}_i \psi_j$.     
\end{proof}

\section{Limit categories}
In this section, we introduce limit categories for smooth stacks.
For a smooth stack $\X$, limit categories are subcategories of $\Coh(\Omega_\X)$ defined in terms of
weight conditions with respect to all maps $\bgm \to \Omega_\X$. 
Limit categories include many of the `quasi-BPS categories' studied in our previous work~\cite{PTquiver, PThiggs}, and they are closely related to the `non-commutative resolutions' of~\cite{SvdB},
the `magic windows' of~\cite{hls}, and the `intrinsic windows' of~\cite{T}. 

We discussed the comparison between limit categories and these previous definitions in the introduction. Let us emphasize that the use of weights conditions with respect to maps $\bgm\to\X$ was previously used by Halpern-Leistner~\cite{HalpTheta, HalpK32}, and categories defined using weights conditions for \textit{all} such maps were used previously by \v{S}penko--Van den Bergh~\cite{SvdB}. 
Limit categories have properties analogous to categories of coherent D-modules, for example semiorthogonal decompositions where each unstable stratum appears once, as opposed to infinitely many, see Theorem~\ref{conj:intro}; for more such similarities, see Subsection~\ref{subsec:limit}.

We also note that an advantage of limit categories is that they are defined 
for stacks $\X$ which may not have good moduli spaces, for example they may not be quasi-compact.

\subsection{Motivation toward limit categories}
Let $\X$ be a smooth stack, let $\delta \in \mathrm{Pic}(\X)_{\mathbb{R}}$, and let $\Omega_{\X}$ be
its cotangent stack. We define a subcategory
\[
  \LL(\Omega_{\X})_\delta \subset \Coh(\Omega_{\X}),
\]
which is meant to play the role of a classical limit of the
category of coherent D-modules on $\X$.

In particular, for a suitable class of open immersions $j\colon U \hookrightarrow \X$,
we require an adjoint pair
\[
  j_! \colon \LL(\Omega_U)_\delta \;\rightleftarrows\; \LL(\Omega_{\X})_\delta \colon j^* .
\]
In the above, $j_!$ is a left adjoint of $j^*$. 
Note that an adjoint pair as above exists for $\text{D-mod}(\X)$ for 
co-truncative open substacks in~\cite{DGbun}, but it does not exists for $\QCoh(\X)$ or $\Coh(\X)$, unless
$U$ is open and closed. 

\medskip

We explain the idea of the construction in the following simpler situation, of \textit{a magic category}, see Definition~\ref{defn:QBY}. 
Magic categories have a similar definition to limit categories, but they are defined for smooth stacks, as opposed to cotangent of smooth stacks.

Let $Y$ be a finite-dimensional $\mathbb{G}_m$-representation, and set $\Y=Y/\mathbb{G}_m$. 
Consider the following diagram for the fixed and attracting stacks 
\begin{align*}
    \mathcal{Z}=Y^0/\mathbb{G}_m \leftarrow \mathcal{S}=Y^{\geq 0}/\mathbb{G}_m \hookrightarrow \mathcal{Y}.
\end{align*}
As we will review in Theorem~\ref{thm:window}, the window theorem~\cite{HL} implies that there is a 
semiorthogonal decomposition 
\begin{align*}
    \Coh(\Y)=\langle \cdots, \Coh(\mathcal{Z})_{0}, \mathcal{W}, \Coh(\mathcal{Z})_1, \cdots\rangle.
\end{align*}
Here, the subcategory $\mathcal{W}$
is generated by $\mathcal{O}_{\Y}(a)$ for $a\in (0, m]$, where 
$m$ is the $\mathbb{G}_m$-weight of $\det (Y^{<0})^{\vee}$. The subcategory $\mathcal{W}$ is an example of a \textit{window subcategory}, and there is an equivalence~\cite{HL}: 
\begin{align*}
    j^* \colon \mathcal{W} \stackrel{\sim}{\to} \Coh(\mathcal{Y}^{\circ}), \ 
    \mathcal{Y}^{\circ}:=\mathcal{Y}\setminus \mathcal{S}. 
\end{align*}

The embedding $\mathcal{W} \subset \Coh(\Y)$ is neither the left nor the right adjoint of 
the pullback $j^* \colon \Coh(\Y) \to \Coh(\Y^{\circ})$. However, if we define the following subcategory slightly larger than $\mathcal{W}$:
\begin{align}\label{sod:LY}
    \MM(\mathcal{Y}) \subset \Coh(\mathcal{Y})
\end{align}
generated by $\mathcal{O}_{\mathcal{Y}}(a)$
with $a\in [0, m]$ (i.e. including weights of the left boundary), then 
there is a semiorthogonal decomposition 
\begin{align*}
    \MM(\mathcal{Y})=\langle \Coh(\mathcal{Z})_0, \mathcal{W} \rangle. 
\end{align*}
Then the embedding $\mathcal{W} \subset \MM(\mathcal{Y})$ gives a fully-faithful left adjoint of the pullback
$j^* \colon \MM(\mathcal{Y}) \to \Coh(\Y^{\circ})$: 
\begin{align*}
    j_!\colon \Coh(\mathcal{Y}^{\circ})\hookrightarrow \MM(\mathcal{Y}).
\end{align*}

We are going to construct an analogue of the subcategory (\ref{sod:LY}) for cotangent 
stacks, or, more generally, for quasi-smooth derived stacks with self-dual cotangent complexes. 
Since we do not assume that they are global quotient stacks, we construct such subcategories using weight conditions for all maps from $\bgm$.

\subsection{Definition of limit categories}\label{subsec:deflim}
Let $\mathfrak{M}$ be a QCA (quasi-compact and with affine geometric stabilizers) and quasi-smooth 
derived stack over $k$, 
satisfying 
\begin{align}\label{sym:L}
    \mathbb{L}_{\mathfrak{M}} \simeq \mathbb{T}_{\mathfrak{M}},
\end{align}
e.g. a zero-shifted symplectic stack~\cite{PTVV}. 
For each field extension $k'/k$ and $x \in \mathfrak{M}(k')$, 
we set 
\begin{align*}
\mathfrak{o}_x:=
\mathcal{H}^{1}(\mathbb{T}_{\mathfrak{M}}|_{x}), \ 
T_x:=
\mathcal{H}^{0}(\mathbb{T}_{\mathfrak{M}}|_{x}), \ 
\mathfrak{g}_x:=
\mathcal{H}^{-1}(\mathbb{T}_{\mathfrak{M}}|_{x}).
\end{align*}
Note that $\mathfrak{g}_x$ is the Lie algebra of 
$G_x:=\mathrm{Aut}(x)$. By the self-duality (\ref{sym:L}), 
we have 
\begin{align*}
T_x \cong T_x^{\vee}, \ \mathfrak{o}_x \cong \mathfrak{g}_x^{\vee}
    \end{align*}
    as $\mathrm{Aut}(x)$-representations and 
    \begin{align*}
\mathcal{H}^{-1}(\mathbb{L}_{\mathfrak{M}}|_{x})=\mathfrak{g}_x, \ 
\mathcal{H}^0(\mathbb{L}_{\mathfrak{M}}|_{x})=T_x, \ 
\mathcal{H}^1(\mathbb{L}_{\mathfrak{M}}|_{x})=\mathfrak{o}_x. 
\end{align*}

Given a map 
\begin{align}\label{map:nu}\nu \colon (\bgm)_{k'} \to \mathfrak{M}, \ 
\nu(0)=x\in \mathfrak{M}(k'),
\end{align}
we have the corresponding cocharacter 
$(\mathbb{G}_m)_{k'} \to G_x$. 
We let $(\mathbb{G}_m)_{k'}$ act on $\mathfrak{g}_x, T_x$
through the above cocharacter, and regard them as elements of $K((\bgm)_{k'})$. 
The map (\ref{map:nu}) may not be quasi-smooth, but we can modifiy 
it to a quasi-smooth map in the following way:
\begin{consdef}\emph{(\cite[Construction~3.2.2]{HalpK32})}\label{def:construction}
A map $\nu$ extends to a map 
\begin{align}\label{nu:reg}
\nu^{\mathrm{reg}} \colon \mathfrak{g}_x^{\vee}[-1]/(\mathbb{G}_m)_{k'} \to \mathfrak{M}
\end{align}
called \textit{regularization}. It satisfies that the map 
\begin{align*}
    \mathbb{L}_{\mathfrak{M}}|_{x} \to 
    \mathbb{L}_{\mathfrak{g}_x^{\vee}[-1]}|_{0}=\mathfrak{g}_x[1]
\end{align*}
induces an isomorphism on $\mathcal{H}^{-1}(-)$. 

The regularization map is quasi-smooth, unique up to 2-isomorphism, 
and determined by a choice of splitting of 
$\mathcal{H}^{-1}(\mathbb{L}_{\mathfrak{M}}|_{x})[1] 
\to \mathbb{L}_{\mathfrak{M}}|_{x}$.
\end{consdef}

We also denote by $\iota$ the natural 
morphism 
\begin{align*}
    \iota \colon \mathfrak{g}_x^{\vee}[-1]/(\mathbb{G}_m)_{k'} \to 
    (\bgm)_{k'}.
\end{align*}
When it is necessary to indicate $x\in \X(k')$, we often write 
$(\nu, \nu^{\mathrm{reg}}, \iota)$ with $(\nu_x, \nu_x^{\mathrm{reg}}, \iota_x)$.
We refer to~\cite[Construction~3.2.2]{HalpK32} for a general construction of a regularization (\ref{nu:reg}). 
Here we only describe its construction in the case of cotangent 
of quotient stacks. 

\begin{example}\label{exam:reg}
Let an algebraic group $G$ act on a smooth variety $X$ 
and consider the moment map 
$\mu \colon \Omega_X \to \mathfrak{g}^{\vee}$, where $\mathfrak{g}$ is the Lie algebra of $G$. 
Let $\mathfrak{M}$ be given by 
\begin{align*}
    \mathfrak{M}=\mu^{-1}(0)/G =\Omega_{X/G}. 
\end{align*}
    Then a map \[\nu \colon (\bgm)_{k'} \to \mathfrak{M}\] is given by 
    $x\in \Omega_X(k')$ with $\mu(x)=0$ together with 
    a cocharacter $\lambda \colon \mathbb{G}_m \to G_x$,
    where $G_x \subset G_{k'}$ is the stabilizer subgroup of $x$. 
    Let $\mathfrak{g}_x \subset \mathfrak{g}$ be the Lie algebra of $G_x$ 
    and choose a splitting $\mathfrak{g}\twoheadrightarrow \mathfrak{g}_x$ 
    as $(\mathbb{G}_m)_{k'}$-representations. 
    Let $\tilde{\nu}=a\circ\nu$, where $a$ is the natural inclusion $a\colon \mathfrak{M}\hookrightarrow \Omega_X/G$.
    There is a commutative diagram 
    \begin{align*}
        \xymatrix{
        (\bgm)_{k'} \ar[r]^-{0} \ar[d]_-{\tilde{\nu}} & \mathfrak{g}_x^{\vee}/(\mathbb{G}_m)_{k'} \ar[d] \\
        \Omega_X/G \ar[r]^-{\mu} & \mathfrak{g}^{\vee}/G.
        }
    \end{align*}
    The induced map on fibers of horizontal arrows give a regularization (\ref{nu:reg}).
\end{example}

Let $\delta \in \mathrm{Pic}(\mathfrak{M})_{\mathbb{R}}$. 
Recall the notations about $\mathbb{G}_m$-weights in Subsection~\ref{subsec:Gmwt}.
\begin{defn}\label{def:Lcat}
We define the subcategory 
\begin{align}\label{def:L}
    \mathrm{L}(\mathfrak{M})_{\delta} \subset \Coh(\mathfrak{M})
\end{align}
to be consisting of objects $\mathcal{E}$ 
such that, for any field extension $k'/k$ and a map
$\nu \colon (\bgm)_{k'} \to \mathfrak{M}$ 
with $\nu(0)=x$ and regularization (\ref{nu:reg}),
we have 
\begin{align}\label{wt:nu}
\wt(\iota_{\ast}\nu^{\mathrm{reg}\ast}(\mathcal{E}))
\subset \left[\frac{1}{2} c_1 (T_x^{<0}),
\frac{1}{2} c_1 (T_x^{>0}) \right]+
\frac{1}{2}c_1(\mathfrak{g}_x)+c_1(\nu^*\delta).
\end{align}
\end{defn}
\begin{remark}\label{rmk:reg}
The definition of $\LL(\mathfrak{M})_{\delta}$ is independent of a choice of a regularization, 
since it is unique up to 2-isomorphisms. 
\end{remark}
\begin{remark}\label{rmk:nureg}
    When $\mathfrak{M}$ admits a good moduli space, we can show that 
    (using~\cite[Corollary~3.53]{PTquiver} and Lemma~\ref{lem:emb} below)
    that the subcategory (\ref{def:L}) equals to the subcategory 
    defined in~\cite[Definition~4.3.5]{HalpK32}.  
    However, note that the definition above (\ref{def:L}) applies to stack without 
    good moduli spaces.
\end{remark}

\begin{remark}\label{rmk:k'}
    By taking an algebraic closure $\overline{k}'$ of $k'$ and 
    composing the map $\nu$ with $(\bgm)_{\overline{k}'}\to (\bgm)_{k'}$, it is enough 
    to check the condition for $k'/k$ such that $k'$ is algebraically closed. Moreover, if 
    $\mathfrak{M}$ is locally of finite presentation over $k$, it is enough to check 
    the condition for $k'=k$. In this case, we omit $k'$ and just write a map $\nu$ as 
    $\bgm\to \mathfrak{M}$.
\end{remark}

Since $\mathfrak{M}$ is QCA, by~\cite{MR3037900} 
the category $\IndCoh(\mathfrak{M})$ is compactly 
generated with compact objects $\Coh(\mathfrak{M})$, i.e. 
\begin{align*}
    \IndCoh(\mathfrak{M})=\Ind(\Coh(\mathfrak{M})). 
\end{align*}
We then define 
\begin{align*}
    \IndL(\mathfrak{M})_{\delta} :=\Ind(\LL(\mathfrak{M})_{\delta}) \subset \IndCoh(\mathfrak{M}). 
\end{align*}

\subsection{Characterization in a local model}
Suppose that there is a
following diagram 
\begin{align}\label{dia:zeros}
\xymatrix{
& \mathcal{V} \ar[d] \\
\mathfrak{M} \simeq s^{-1}(0) \inclusion^-{i} & \mathcal{Y} \ar@/_18pt/[u]_-{s}
}
\end{align}
where $\mathcal{Y}$ is a smooth QCA stack, $\mathcal{V} \to \mathcal{Y}$ is a vector bundle 
with a section $s$, and 
$s^{-1}(0)$ is the derived zero locus of $s$.
In this case, we have another description of $\LL(\mathfrak{M})_{\delta}$.
    \begin{lemma}\label{lem:emb}
An object $\mathcal{E} \in \Coh(\mathfrak{M})$ lies in 
$\LL(\mathfrak{M})_{\delta}$ if and only if, for any map $\nu \colon (\bgm)_{k'} \to \mathfrak{M}$, we have 
\begin{align}\notag
    \wt(\nu^{\ast}i^{\ast}i_{\ast}\mathcal{E})
    \subset \left[\frac{1}{2} c_1 (\nu^{\ast}\mathbb{L}_{\mathcal{V}}^{<0}), 
    \frac{1}{2} c_1 (\nu^{\ast}\mathbb{L}_{\mathcal{V}}^{>0})
    \right]+c_1(\nu^{\ast}\delta). 
\end{align}
    \end{lemma}
\begin{proof}
Let $\nu \colon (\bgm)_{k'}\to \mathfrak{M}$
be a map such that $\nu(0)=x$. 
Let $\mathcal{V}_x$ be the fiber of $\mathcal{V} \to \mathcal{Y}$
at $j(x)$, which is a $G_x=\mathrm{Aut}(x)$-representation. 
Below we regard the $G_x$-representation 
as a $\mathbb{G}_m$-representation 
via the cocharacter 
$(\mathbb{G}_m)_{k'} \to G_x$
induced by $\nu$. 
We have the Cartesian square 
\begin{align*}
\xymatrix{
\mathfrak{M} \inclusion^-{i}\diasquare & \mathcal{Y} \\
\mathcal{V}_x[-1]/(\mathbb{G}_m)_{k'}
\ar[u]^-{\widetilde{\nu}} 
\ar[r]^-{\widetilde{\iota}} & (\bgm)_{k'} \ar[u]_-{j \circ \nu} 
}
\end{align*}

Then we have the isomorphisms
\begin{align}\notag
    \nu^{\ast}i^{\ast}i_{\ast}\mathcal{E} \cong (i\circ \nu)^{*}i_{*}\mathcal{E}\cong \widetilde{\iota}_{\ast}\widetilde{\nu}^{\ast}\mathcal{E}. 
\end{align}
From the distinguished triangle 
\begin{align*}
    \mathcal{V}^{\vee}|_{\mathfrak{M}} \to \mathbb{L}_{\mathcal{Y}}|_{\mathfrak{M}} \to \mathbb{L}_{\mathfrak{M}}
\end{align*}
we have the exact sequence 
of $G_x$-representations
\begin{align}\label{exact:Vx}
    0 \to \mathfrak{g}_x \to \mathcal{V}_x^{\vee}
    \to \mathcal{H}^0(\mathbb{L}_{\mathcal{Y}}|_{x})
    \to T_x \to 0
\end{align}
and $\mathcal{H}^{1}(\mathbb{L}_{\mathcal{Y}}|_{x}) \stackrel{\cong}{\to} \mathfrak{g}_x^{\vee}$. 
The first inclusion in (\ref{exact:Vx}) determines a
morphism 
\begin{align*}
    \eta \colon \mathcal{V}_x[-1]/(\mathbb{G}_m)_{k'} \to \mathfrak{g}_x^{\vee}[-1]/(\mathbb{G}_m)_{k'}.
\end{align*}
We have the commutative diagram 
\begin{align*}
    \xymatrix{
(\bgm)_{k'} & \mathfrak{g}_x^{\vee}[-1]/(\mathbb{G}_m)_{k'} \ar[r]^-{\nu^{\mathrm{reg}}} \ar[l]_-{\iota} 
& \mathfrak{M} \\
& \mathcal{V}_x[-1]/(\mathbb{G}_m)_{k'}
\ar[lu]^-{\widetilde{\iota}} \ar[ru]_-{\widetilde{\nu}} \ar[u]_-{\eta}&    
    }
\end{align*}
Then we have 
\begin{align*}
    \widetilde{\iota}_{\ast}\widetilde{\nu}^{\ast}\mathcal{E}
\cong \iota_{\ast}\eta_{\ast}\eta^{\ast}\nu^{\mathrm{reg}\ast}\mathcal{E} \cong \iota_{\ast}(\nu^{\mathrm{reg}\ast}\mathcal{E}\otimes 
\eta_{\ast}\mathcal{O}_{\mathcal{V}_x[-1]}).
\end{align*}
We divide the exact sequence (\ref{exact:Vx}) into 
two short exact sequences
\begin{align}\label{exact:gW}
    0 \to \mathfrak{g}_x \to \mathcal{V}_x^{\vee} \to 
    W \to 0, \ 
    0 \to W \to \mathcal{H}^0(\mathbb{L}_{\mathcal{Y}}|_{x})
    \to T_x \to 0. 
\end{align}

Then we have 
\begin{align*}
    \eta_{\ast}\mathcal{O}_{\mathcal{V}_x[-1]}
    \cong \left(\bigwedge^{\bullet}W\right) \otimes 
    \mathcal{O}_{\mathfrak{g}_x^{\vee}[-1]}
\end{align*}
where the complex in the right-hand side has 
zero differentials. 
Therefore 
\begin{align}\notag
    \widetilde{\iota}_{\ast}\widetilde{\nu}^{\ast}\mathcal{E}
    \cong \iota_{\ast}\nu^{\mathrm{reg}\ast}\mathcal{E} \otimes\bigwedge^{\bullet}W.
\end{align}

It follows that 
$\wt(\iota_{\ast}\nu^{\mathrm{reg}\ast}\mathcal{E})$
lies in $[a, b]$ if and only if 
\begin{align*}
\wt(\nu^{\ast}i^{\ast}i_{\ast}\mathcal{E})=
\wt(\widetilde{\iota}_{\ast}\widetilde{\nu}^{\ast}\mathcal{E})
\subset [a+c_1 (W^{<0}), b+c_1 (W^{>0})].
\end{align*}
Using the exact sequences (\ref{exact:gW}), we have 
\begin{align*}
    \frac{1}{2} c_1 (\mathbb{L}_{\mathcal{V}}|_{x}^{>0}) &=\frac{1}{2} c_1(\mathcal{V}_x^{\vee >0})+
    \frac{1}{2}c_1(\mathbb{L}_{\mathcal{Y}}|_{x}^{>0}) \\
    &=\frac{1}{2} c_1(\mathcal{V}_x^{\vee>0})+
    \frac{1}{2}c_1(\mathcal{H}^0(\mathbb{L}_{\mathcal{Y}}|_{x})^{>0})
    -\frac{1}{2} c_1((\mathfrak{g}_x^{\vee})^{>0}) \\
    &=\frac{1}{2}c_1(\mathfrak{g}_x^{>0})+\frac{1}{2}c_1(T_x^{>0})+\frac{1}{2}c_1(\mathfrak{g}_x^{<0})
    +c_1 (W^{>0}) \\
    &=\frac{1}{2}c_1(T_x^{>0})+\frac{1}{2}c_1(\mathfrak{g}_x)+c_1 (W^{>0}). 
\end{align*}
Similarly we have 
\begin{align*}
 \frac{1}{2} c_1 (\mathbb{L}_{\mathcal{V}}|_{x}^{<0})
 =\frac{1}{2}c_1(T_x^{<0})+\frac{1}{2}c_1(\mathfrak{g}_x)+c_1 (W^{<0}).
\end{align*}
Therefore the lemma holds. 
\end{proof}

\subsection{Dualities}
Since $\mathfrak{M}$ is quasi-smooth satisfying~\eqref{sym:L}, there is a dualizing equivalence
\begin{align}\label{dual:M}
    \mathbb{D}:=\mathcal{H}om(-, \omega_{\mathfrak{M}}[\dim \mathfrak{M}]) \colon \Coh(\mathfrak{M}) \stackrel{\sim}{\to} \Coh(\mathfrak{M})^{\mathrm{op}}
\end{align}
where $\omega_{\mathfrak{M}}=\det \mathbb{L}_{\mathfrak{M}}$. 
Note that $\omega_{\mathfrak{M}}^{\otimes 2} \cong \mathcal{O}_{\mathfrak{M}}$ 
by the self-duality (\ref{sym:L}). 
\begin{lemma}\label{lem:duality}
The functor (\ref{dual:M}) restricts to an equivalence 
\begin{align*}
\mathbb{D} \colon \LL(\mathfrak{M})_{\delta}^{\mathrm{op}} \stackrel{\sim}{\to} \LL(\mathfrak{M})_{\delta^{-1}}. 
\end{align*}
\end{lemma}
\begin{proof}
    For $\mathcal{E} \in \Coh(\mathfrak{M})$, we have 
    \begin{align*}
        \wt(\iota_{\ast}\nu^{\mathrm{reg}\ast}\mathbb{D}(\mathcal{E}))
        =\wt(\iota_{\ast}\mathbb{D}(\nu^{\mathrm{reg}\ast}\mathcal{E})) 
        =\wt((\iota_{\ast}\nu^{\mathrm{reg}\ast}(\mathcal{E}))^{\vee})+c_1 (\mathfrak{g}_x). 
    \end{align*}
    Here, the second equality follows from the Grothendieck duality for the closed immersion 
    $\mathfrak{g}_x^{\vee}[-1] \hookrightarrow \Spec k'$. 
    The lemma follows from the above identities. 
\end{proof}

The following lemma is obvious from the definition of $\LL(\mathfrak{M})_{\delta}$:
\begin{lemma}\label{lem:tensor}
For $\mathcal{L} \in \mathrm{Pic}(\mathfrak{M})$, there is an equivalence 
\begin{align*}
    \otimes \mathcal{L} \colon \LL(\mathfrak{M})_{\delta} \stackrel{\sim}{\to}
    \LL(\mathfrak{M})_{\delta \otimes \mathcal{L}}.
\end{align*}
\end{lemma}

\subsection{Relation with magic categories for matrix factorizations}\label{subsec:magic}
In the case that $\mathfrak{M}$ is a global zero locus as in the diagram (\ref{dia:zeros}), 
we can relate the category $\LL(\mathfrak{M})_{\delta}$ with a subcategory 
of graded matrix factorizations under Koszul equivalence in Theorem~\ref{thm:Kduality}.

Let $\mathcal{Y}$ be a smooth stack over $k$, and $K(\mathcal{Y})$ the Grothendieck group of 
$\Coh(\Y)$. 
For purposes of this paper, we use the following notion of a symmetric stack:
\begin{defn}\label{def:sstack}
A smooth stack $\Y$ is called 
\textit{symmetric} 
if we have 
\begin{align*}
	[\mathbb{L}_{\mathcal{Y}}]=[\mathbb{L}_{\mathcal{Y}}^{\vee}] \in K(\Y).
	\end{align*}
    \end{defn}
\begin{example}\label{exam:quiver}
	Let $G$ be a reductive algebraic group and let $Y$ be a self-dual $G$-representation. 
	Then any open substack of $Y/G$ is symmetric. 
	\end{example}

 We define the following \textit{magic category} as a subcategory of $\Coh(\mathcal{Y})$:
 \begin{defn}\label{defn:QBY}
For a smooth symmetric stack $\mathcal{Y}$ and $\delta \in \mathrm{Pic}(\mathcal{Y})_{\mathbb{R}}$, we 
define 
the subcategory 
\begin{align}\label{MY:delta}
\MM(\mathcal{Y})_{\delta} \subset \Coh(\mathcal{Y})
\end{align}
to be consisting of objects $\mathcal{E} \in \Coh(\mathcal{Y})$ 
such that, for any field extension $k'/k$ and a
map $\nu \colon (\bgm)_{k'} \to \mathcal{Y}$, the following holds: 
\begin{align*}
    \mathrm{wt}(\nu^{\ast}\mathcal{E}) \subset 
    \left[-\frac{1}{2}n_{\mathcal{Y}, \nu}, \frac{1}{2}n_{\mathcal{Y}, \nu}   \right]
    +\langle \delta, \nu\rangle. 
\end{align*}
Here, $n_{\mathcal{Y}, \nu}$ is defined by 
\begin{align}\label{def:nV}
	n_{\mathcal{Y}, \nu}:=c_1(\nu^{\ast}\mathbb{L}_{\mathcal{Y}}^{>0})=-c_1(\nu^{\ast}\mathbb{L}_{\mathcal{Y}}^{<0}).
	\end{align}
The second equality follows from the definition of a symmetric stack. 
\end{defn}
\begin{remark}
In the case of $\Y=Y/G$ for a self-dual representation $Y$ of a reductive group $G$, 
the subcategory (\ref{MY:delta}) is also defined by some explicit generators, 
see Proposition~\ref{prop:QB:weight}. 
In this case, the category (\ref{MY:delta}) was 
essentially first considered by \v{S}penko--Van den Bergh~\cite{SvdB} to construct 
 non-commutative crepant resolutions, 
 and was used by Halpern-Leistner--Sam~\cite{hls} to define the \textit{magic window 
 subcategory} and prove the D-equivalence/K-equivalence conjecture for some 
 variations of GIT quotients.\footnote{Let us explain the use of the name ``magic categories" by Halpern-Leistner--Sam~\cite{hls}. In the setting of GIT, there are \textit{window categories} defined in terms of weight conditions for \textit{certain} maps cocharacters corresponding to the Kempf-Ness (unstable) strata. In general, it is hard to find explicit generators of window categories. Also, it is not straightforward to compare window categories for two different stability conditions. However, for self-dual (or, more generally, quasi-symmetric) representations, window categories are equivalent to \textit{magic categories}, defined in terms of weight conditions for \textit{all} cocharacters. Such categories have explicit generators and they are independent of generic stability conditions.}
\end{remark}

Let $\mathfrak{M}$ be a quasi-smooth derived stack satisfying (\ref{sym:L}), given as in 
Subsection~\ref{subsec:Koszuleq}.
Let $f\colon \mathcal{V}^{\vee} \to \mathbb{A}^1$ be the function as in (\ref{funct:fV}). 
Recall the Koszul equivalence from Theorem~\ref{thm:Kduality}:
\begin{align}\label{Kdual:magic}
    \Psi \colon \Coh(\mathfrak{M}) \stackrel{\sim}{\to} \mathrm{MF}^{\mathrm{gr}}(\mathcal{V}^{\vee}, f). 
\end{align}
Suppose that $\mathcal{V}^{\vee}$ is symmetric. 
Let $\delta \in \mathrm{Pic}(\Y)_{\mathbb{R}}$. We use the same symbol $\delta$ 
for its pull-back to $\mathfrak{M}$ and $\mathcal{V}^{\vee}$. 
The following proposition compares the category $\LL(\mathfrak{M})_{\delta}$ with 
the magic category for graded matrix factorizations. 
\begin{prop}\label{prop:comparemagic}
    The equivalence (\ref{Kdual:magic}) restricts to an equivalence 
    \begin{align*}
        \Psi \colon \LL(\mathfrak{M})_{\delta} \stackrel{\sim}{\to} \mathrm{MF}^{\mathrm{gr}}(\MM(\mathcal{V}^{\vee})_{\delta \otimes \omega_{\Y}^{1/2}}, f).
    \end{align*}
\end{prop}
\begin{proof}
The proposition follows from the characterization of the category $\LL(\mathfrak{M})_{\delta}$ in Lemma~\ref{lem:emb} and the construction 
of Koszul equivalence in Theorem~\ref{thm:Kduality}. See 
the arguments of~\cite[Lemma~4.5]{PT0}, ~\cite[Lemma~2.6]{PTquiver}
\end{proof}

\subsection{The case of cotangent stacks}\label{subsec:cotan}
Let $\X$ be a smooth and QCA stack and consider 
its cotangent stack $\mathfrak{M}=\Omega_{\X}$, where recall that:
\begin{align*}
\Omega_{\X} :=
\Spec_{\X} \mathrm{Sym}(\mathbb{L}_{\X}). 
\end{align*}
It is a quasi-smooth derived stack with zero-shifted
symplectic structure, in particular it satisfies (\ref{sym:L}). 
Let $\delta \in \mathrm{Pic}(\X)_{\mathbb{R}}$, and 
regard it as $\mathbb{R}$-line bundle on $\Omega_{\X}$ 
via pull-back of the projection 
\begin{align*}
\pi \colon \Omega_{\X} \to \X.
\end{align*}
We consider 
the categories 
\begin{align*}
\LL(\Omega_{\X})_{\delta}
\subset \Coh(\Omega_{\X}), \ \IndL(\Omega_{\X})_{\delta}
\subset \IndCoh(\Omega_{\X}).     
\end{align*}
Let $\omega_{\X}$ be the canonical line bundle on 
$\X$. 
We are particularly interested in 
\begin{align}\label{L:half}
    \LL(\Omega_{\X})_{\pm \frac{1}{2}} :=
    \LL(\Omega_{\X})_{\delta=\omega_{\X}^{\pm 1/2}}, \ 
    \IndL(\Omega_{\X})_{\pm \frac{1}{2}}:=
    \IndL(\Omega_{\X})_{\delta=\omega_{\X}^{\pm 1/2}}.
\end{align}
There are equivalences, see Lemma~\ref{lem:duality} and Lemma~\ref{lem:tensor}
\begin{align}\label{equiv:Domega}
    \mathbb{D} \colon \LL(\Omega_{\X})^{\mathrm{op}}_{-\frac{1}{2}}
    \stackrel{\sim}{\to} \LL(\Omega_{\X})_{\frac{1}{2}}, \ 
    \otimes \omega_{\X} \colon 
 \LL(\Omega_{\X})_{-\frac{1}{2}}
    \stackrel{\sim}{\to} \LL(\Omega_{\X})_{\frac{1}{2}}.   
\end{align}

In the case of cotangents, we give another useful criterion for objects in 
$\Coh(\Omega_{\X})$ to lie in $\LL(\Omega_{\X})_{\delta}$.
Let $\nu \colon (\bgm)_{k'} \to \Omega_{\X}$ be a map with 
$\nu(0)=\overline{x}$, and set $x=\pi(\overline{x})$. 
By abuse of notation, we set 
\begin{align*}
    &T_{\X, x}:=\mathcal{H}^0(\mathbb{T}_{\X}|_{x}), \ 
    \mathfrak{g}_x :=\mathcal{H}^{-1}(\mathbb{T}_{\X}|_{x}), \\ 
    &T_{\Omega, \overline{x}}:=\mathcal{H}^0(\mathbb{T}_{\Omega_{\X}}|_{\overline{x}}), \ 
    \mathfrak{g}_{\overline{x}}:=\mathcal{H}^{-1}(\mathbb{T}_{\Omega_{\X}}|_{\overline{x}}). 
\end{align*}

Since $\pi$ is representable, we have 
$\mathfrak{g}_{\overline{x}} \subset \mathfrak{g}_x$,
and this induces a quasi-smooth map 
\begin{align*}
    \eta \colon \mathfrak{g}_x^{\vee}[-1]/(\mathbb{G}_m)_{k'} \to 
    \mathfrak{g}_{\ox}^{\vee}[-1]/(\mathbb{G}_m)_{k'}
\end{align*}
where the $(\mathbb{G}_m)_{k'}$-actions are given by
$\nu \colon (\mathbb{G}_m)_{k'} \to \Aut(\ox) \subset \Aut(x)$. 
We have the following diagram 
\begin{align}\label{dia:nuxreg}
    \xymatrix{
\Omega_{\X} \ar[d]_-{\pi} & \mathfrak{g}_{\ox}^{\vee}[-1]/(\mathbb{G}_m)_{k'} 
\ar[l]_-{\nu_{\ox}^{\mathrm{reg}}} \ar[r]^-{\iota_{\ox}} & (\bgm)_{k'} \\
\X & \ar[lu]^-{\overline{\nu}_x^{\mathrm{reg}}} \mathfrak{g}_x^{\vee}[-1]/(\mathbb{G}_m)_{k'}.
\ar[u]_-{\eta}  \ar[ru]_-{\iota_{x}}&  
    }
\end{align}
Here $\nu_{\ox}^{\mathrm{reg}}$ is a regularization of $\nu$ and $\overline{\nu}_x^{\mathrm{reg}}$ is 
defined by $\overline{\nu}_x^{\mathrm{reg}}:=\nu_{\ox}^{\mathrm{reg}} \circ \eta$. 
\begin{lemma}\label{lem:wtOmega}
For an object $\mathcal{E} \in \Coh(\Omega_{\X})$, we have the condition 
\begin{align*}
   \wt( \iota_{\ox\ast}\nu_{\ox}^{\mathrm{reg}\ast}(\mathcal{E})) \subset 
    \left[\frac{1}{2}c_1(T_{\Omega, \ox}^{<0}), \frac{1}{2} c_1(T_{\Omega, \ox}^{>0})  \right]
    +\frac{1}{2}c_1(\mathfrak{g}_{\overline{x}})+c_1(\nu^*\delta)
\end{align*}
if and only if 
\begin{align*}
   \wt( \iota_{x\ast}\overline{\nu}_x^{\mathrm{reg}\ast}(\mathcal{E}))
    \subset 
    \left[\frac{1}{2} c_1(T_{\X, x}^{<0}-T_{\X, x}^{>0}), 
    \frac{1}{2} c_1(T_{\X, x}^{>0}-T_{\X, x}^{<0})\right]
    +\frac{1}{2}c_1(\mathfrak{g}_x)+c_1(\nu^*\delta). 
\end{align*}
\end{lemma}
\begin{proof}
    We have the isomorphisms 
    \begin{align*}
\iota_{x\ast}\overline{\nu}_x^{\mathrm{reg}\ast}(\mathcal{E}) \cong 
\iota_{\overline{x}\ast}\eta_{\ast}\eta^{\ast}\nu_{\ox}^{\mathrm{reg}\ast}(\mathcal{E}) 
\cong \iota_{\overline{x}\ast}\nu_{\ox}^{\mathrm{reg}\ast}(\mathcal{E}) \otimes \bigwedge^{\bullet}W
    \end{align*}
    where $W=\mathfrak{g}_{x}/\mathfrak{g}_{\ox}$. 
    From the distinguished triangle 
    \begin{align*}
        \pi^{\ast}\mathbb{L}_{\X} \to \mathbb{T}_{\Omega_{\X}} \to \pi^{\ast}
        \mathbb{T}_{\X}
    \end{align*}
    we have the short exact sequences of $\mathrm{Aut}(\ox)$-representations
    \begin{align*}
    0 \to W \to T_{\X, x}^{\vee} \to V \to 0, \ 
    0 \to V \to T_{\Omega, \ox} \to V^{\vee} \to 0
    \end{align*}
    for some $\mathrm{Aut}(\ox)$-representation $V$. 
    
    Then we have the following identities
    \begin{align*}
        c_1\left(\frac{1}{2}T_{\Omega, \ox}^{>0}+\frac{1}{2}\mathfrak{g}_{\ox}+W^{>0}\right) &=c_1\left(
        \frac{1}{2}V^{>0}-\frac{1}{2}V^{<0}+\frac{1}{2}W^{>0}-\frac{1}{2}W^{<0}+\frac{1}{2}\mathfrak{g}_x\right) \\
        &=c_1\left(\frac{1}{2}(T_{\X, x}^{\vee})^{>0}-\frac{1}{2}(T_{\X, x}^{\vee})^{<0}+\frac{1}{2}\mathfrak{g}_x\right) \\
        &=c_1\left(-\frac{1}{2}T_{\X, x}^{<0}+\frac{1}{2}T_{\X, x}^{>0}+\frac{1}{2}\mathfrak{g}_x\right). 
    \end{align*}
    
    Similarly we have 
    \begin{align*}
        c_1\left(\frac{1}{2}T_{\Omega, \ox}^{<0}+\frac{1}{2}\mathfrak{g}_{\ox}+W^{<0}\right)
        =c_1\left(-\frac{1}{2}T_{\X, x}^{>0}+\frac{1}{2}T_{\X, x}^{<0}+\frac{1}{2}\mathfrak{g}_x\right).
    \end{align*}
    Therefore the lemma holds. 
\end{proof}

\begin{example}\label{exam:O}
Let $0 \colon \X \to \Omega_{\X}$ be the zero section of the projection 
$\Omega_{\mathcal{X}}\to \mathcal{X}$. 
Then we have 
\begin{align*}
    0_{\ast}\mathcal{O}_{\mathcal{X}} \in \LL(\Omega_{\mathcal{X}})_{-\frac{1}{2}}, \ 
   0_{\ast}\omega_{\mathcal{X}} \in \LL(\Omega_{\mathcal{X}})_{\frac{1}{2}}.
\end{align*}
Indeed for a map $\nu \colon (\bgm)_{k'} \to \Omega_{\mathcal{X}}$
with $x=\nu(0) \in 0(\mathcal{X})$, we have the Cartesian square 
\begin{align*}
    \xymatrix{
\mathfrak{g}_x^{\vee}[-1]/(\mathbb{G}_m)_{k'} \ar[r]^-{\overline{\nu}_x^{\mathrm{reg}}} \diasquare& \Omega_{\mathcal{X}} \\
T_{\mathcal{X}, x}^{\vee}[-1]/(\mathbb{G}_m)_{k'} \ar[r] \ar[u] & \mathcal{X} \ar[u]_-{0}.  
}
\end{align*}
It follows that $\iota_{x\ast}\overline{\nu}_x^{\mathrm{reg}\ast}0_{\ast}\mathcal{O}_{\mathcal{X}}=
\wedge^{\bullet}T_{\mathcal{X}, x}$, and the latter condition 
in Lemma~\ref{lem:wtOmega} is satisfied. 
\end{example}

As in Example~\ref{exam:reg}, let $\X=X/G$ 
for a smooth scheme $X$ with an action of an affine algebraic 
group $G$ and consider its cotangent 
\begin{align*}
    \Omega_{\X} \simeq \mu^{-1}(0)/G
    \stackrel{j}{\hookrightarrow} \Omega_X/G,
\end{align*}
where $\mu \colon \Omega_X \to \fg^{\vee}$ is the moment map. 
In this case, we have the following characterization of 
$\LL(\Omega_{\X})_{\delta}$. 
\begin{lemma}\label{lem:char2}
    An object $\mathcal{E} \in \Coh(\Omega_{\X})$ lies in 
    $\LL(\Omega_{\X})_{\delta}$
    if and only if, for any $\nu \colon (\bgm)_{k'} \to \Omega_{\X}$ such that 
    $\pi \circ \nu(0)=x \in X/G$, 
    we have
    \begin{align*}
        &\wt(\nu^{\ast}j^{\ast}j_{\ast}\mathcal{E})
        \subset \\
        &\left[ -\frac{1}{2}c_1((T_X|_{x} \oplus \Omega_X|_{x})^{>0}), 
        \frac{1}{2}c_1((T_X|_{x} \oplus \Omega_X|_{x})^{>0})
        \right]+\frac{1}{2}c_1(\mathfrak{g}_x)+c_1(\nu^*\delta).
    \end{align*}
    In particular, we have 
    \begin{enumerate}
\item for $\delta=\omega_{\X}^{1/2}$, the above condition is 
\begin{align*}
        \wt(\nu^{\ast}j^{\ast}j_{\ast}\mathcal{E})
        \subset \left[c_1 (\Omega_X|_{x}^{<0}), c_1(\Omega_X|_{x}^{>0}) \right]
        +c_1(\nu^{*}\mathfrak{g}). 
        \end{align*}
\item for $\delta=\omega_{\X}^{-1/2}$, the above condition is 
 \begin{align*}
        \wt(\nu^{\ast}j^{\ast}j_{\ast}\mathcal{E})
        \subset \left[c_1 (T_X|_{x}^{<0}), c_1(T_X|_{x}^{>0}) \right].
        \end{align*}    
        \end{enumerate}
    \end{lemma}
    \begin{proof}
        The lemma immediately follows from Lemma~\ref{lem:emb}
        by setting 
        \begin{align*}\mathfrak{M}=\Omega_{\X},     \mathcal{Y}=\Omega_X/G, \ \mathcal{V}=(\Omega_X \times \mathfrak{g}^{\vee})/G, \ s=\mu. 
        \end{align*}
            \end{proof}

\subsection{The non-quasi-compact case}\label{subsec:nonqcp}
Now suppose that $\mathfrak{M}$ is a quasi-smooth derived stack satisfying (\ref{C:lim}), which is 
not necessarily quasi-compact, but admits a Zariski open covering 
\begin{align*}
    \mathfrak{M}=\bigcup_{\mathcal{U}\subset \mathfrak{M}}\mathcal{U},
\end{align*}
where each $\mathcal{U}\subset \mathfrak{M}$ is a QCA stack. 
By the Zariski descent of ind-coherent sheaves~\cite[Proposition~4.2.1]{MR3136100}, 
we have 
\begin{align*}
    \IndCoh(\mathfrak{M})=\lim_{\mathcal{U}\subset \mathfrak{M}}\IndCoh(\mathcal{U})
\end{align*}
where the limit is after quasi-compact open substacks $\mathcal{U}\subset \mathfrak{M}$ 
and pull-backs
\begin{align}\label{pull:j0}
j^{*}\colon 
    \IndCoh(\mathcal{U}_2) \to \IndCoh(\mathcal{U}_1)
\end{align}
for open immersions $j \colon \mathcal{U}_1 \hookrightarrow \mathcal{U}_2$. 
The above functor restricts to the functor of compact objects
\begin{align}\notag
    j^* \colon \Coh(\mathcal{U}_2) \to \Coh(\mathcal{U}_1)
\end{align}
which, by Definition~\ref{def:Lcat}, 
restricts to the functor 
\begin{align}\label{pback:L0}
    j^* \colon \LL(\mathcal{U}_2)_{\delta}\to \LL(\mathcal{U}_1)_{\delta}. 
\end{align}
By taking the ind-completion, we obtain the functor 
\begin{align}\label{pback:L}
    j^* \colon \IndL(\mathcal{U}_2)_{\delta} \to \IndL(\mathcal{U}_1)_{\delta},
\end{align}
which is the restriction of the functor (\ref{pull:j0}). 
We then define 
\begin{align}\label{def:Ltilde0}
    \IndL(\mathfrak{M})_{\delta}:=\lim_{\mathcal{U}\subset \mathfrak{M}}\IndL(\mathcal{U})_{\delta}. 
\end{align}
Here the limit is after quasi-compact open substacks $\mathcal{U}\subset \mathfrak{M}$ and 
pull-backs (\ref{pback:L}). 
We define the category $\LL(\mathfrak{M})_{\delta}$ to be 
the subcategory of compact objects in $\IndL(\mathfrak{M})_{\delta}$
\begin{align*}
    \LL(\mathfrak{M})_{\delta}:=(\IndL(\mathfrak{M})_{\delta})^{\mathrm{cp}}. 
\end{align*}

\begin{remark}\label{rmk:Ltilde}
We can also consider the limit 
\begin{align}\label{def:Ltilde}
\widetilde{\LL}(\mathfrak{M})_{\delta}:=\lim_{\mathcal{U}\subset \mathfrak{M}}\LL(\mathcal{U})_{\delta} 
\end{align}
where the limit is taken with respect to pull-backs (\ref{pback:L0}). 
We have $\LL(\mathfrak{M})_{\delta}=\widetilde{\LL}(\mathfrak{M})_{\delta}$ if $\mathfrak{M}$ is 
QCA, but this is not necessarily the case if $\mathfrak{M}$ is not QCA. 
The category (\ref{def:Ltilde}) will appear in Section~\ref{sec:hecke}. 
\end{remark}

\subsection{Singular supports of ind-coherent sheaves}\label{subsec:singular}
Let $\mathfrak{M}$ be as in the previous subsection. 
Let $\Omega_{\mathfrak{M}}[-1]$ be its $(-1)$-shifted cotangent 
\begin{align}\label{-1:shift}
    \Omega_{\mathfrak{M}}[-1]:=\Spec_{\mathfrak{M}} \mathrm{Sym}(\mathbb{T}_{\mathfrak{M}}[1]) \to \mathfrak{M}. 
\end{align}
For each $x\in \mathfrak{M}(k)$, a fiber of the above 
map is 
\begin{align*}
    \mathcal{H}^1(\mathbb{T}_{\mathfrak{M}}|_x)^{\vee} \cong 
    \mathcal{H}^{-1}(\mathbb{T}_{\mathfrak{M}}|_{x})=\mathfrak{g}_x
\end{align*}
where $\mathfrak{g}_x$ is the Lie algebra of $\mathrm{Aut}(x)$. We denote by 
\begin{align}\label{nilp:N}
    \mathcal{N}\subset \Omega_{\mathfrak{M}}[-1]
\end{align}
the closed substack consisting of nilpotent elements on fibers of (\ref{-1:shift}). 

For an object $\mathcal{E} \in \IndCoh(\mathfrak{M})$, Arinkin-Gaitsgory~\cite{AG} 
associate the conical closed subset, called \textit{singular support}
\begin{align}\notag
    \mathrm{Supp}^{\mathrm{AG}}(\mathcal{E}) \subset \Omega_{\mathfrak{M}}[-1]. 
\end{align}
Locally, it equals to the support of the matrix factorization under Koszul 
equivalence in Theorem~\ref{thm:Kduality}, see~\cite[Proposition~2.3.9]{T}.  
As in~\cite{AG}, we define the subcategory of nilpotent singular supports
\begin{align*}
    \IndCoh_{\mathcal{N}}(\mathfrak{M}) \subset \IndCoh(\mathfrak{M})
\end{align*}
to be consisting of objects whose singular supports are contained in $\mathcal{N}$. 
We define 
\begin{align}\label{nilp:L}
\LL_{\mathcal{N}}(\mathfrak{M})_{\delta}\subset \LL(\mathfrak{M})_{\delta}, \ 
\IndL_{\mathcal{N}}(\mathfrak{M})_{\delta} \subset \IndL(\mathfrak{M})_{\delta}
\end{align}
as the subcategories whose objects have nilpotent singular supports. 

\subsection{Limit categories for cotangent stacks}\label{subsec:limcot}
Let $\X$ be a smooth stack, which is
not necessarily quasi-compact, but admits a Zariski open covering 
\begin{align*}
    \X=\bigcup_{U\subset \X} U
\end{align*}
where each $U$ is QCA. 
Then we have the open covering
\begin{align}\label{open:cot}
    \Omega_{\X}=\bigcup_{U\subset \mathcal{X}} \Omega_{U}.
\end{align}
Note that each $\Omega_U$ is also QCA. 
We take 
$\delta \in \mathrm{Pic}(\X)_{\mathbb{R}}$, 
and regard it as an element of $\mathrm{Pic}(\Omega_{\X})_{\mathbb{R}}$ 
via pull-back the projection 
$\Omega_{\X} \to \X$. 
By applying the above construction for the cotangent stack
$\mathfrak{M}=\Omega_{\X}$, 
we define the \textit{ind-limit category} 
\begin{align*} \IndL(\Omega_{\X})_{\delta} \subset 
    \IndCoh(\Omega_{\X}). 
\end{align*}
By the open covering (\ref{open:cot}), it is also written as 
\begin{align*}
    \IndL(\Omega_{\X})_{\delta}=\lim_{U\subset \X} \IndL(\Omega_U)_{\delta}. 
\end{align*}

The subcategory of compact objects 
\begin{align*}
    \LL(\Omega_{\X})_{\delta}:=\IndL(\Omega_{\X})_{\delta}^{\mathrm{cp}}
\end{align*}
is called \textit{limit category}. The (ind-)limit categories of nilpotent singular supports 
\begin{align}\label{limit:nilp}
    \LL_{\mathcal{N}}(\Omega_{\X})_{\delta} \subset \LL(\Omega_{\X})_{\delta}, \ 
    \IndL_{\mathcal{N}}(\Omega_{\X})_{\delta} \subset \IndL(\Omega_{\X})
\end{align}
are also defined similarly to (\ref{nilp:L}). As before we use the subscript $(-)_{\frac{1}{2}}$ for $\delta=\omega_{\X}^{1/2}$.
\begin{remark}\label{rmk:limit}
As we will discuss in Section~\ref{sec:Dmod}, 
the category $\LL(\Omega_{\X})_{\frac{1}{2}}$ should be related to
the classical limit of the category of coherent D-modules 
$\text{D-mod}_{\mathrm{coh}}(\X)$:
\begin{align*}
    \text{D-mod}_{\mathrm{coh}}(\X) \leadsto \LL(\Omega_{\X})_{\frac{1}{2}}.
\end{align*}
For example, when $\X=BG$ for a reductive group $G$, the category $\LL(\Omega_{\X})_{\frac{1}{2}}$ is equivalent to 
$\text{D-mod}_{\mathrm{coh}}(\X)$, see Subsection~\ref{subsec:exam}. 
Similarly, as we discuss in Subsection~\ref{subsec:limsupp}, we expect $\LL_{\mathcal{N}}(\Omega_{\X})_{\frac{1}{2}}$ to be related to the classical limit of the category of compact D-modules: 
\begin{align*}
    \text{D-mod}(\X)^{\mathrm{cp}} \leadsto \LL_{\mathcal{N}}(\Omega_{\X})_{\frac{1}{2}}.
\end{align*}
\end{remark}

\subsection{Examples of limit categories}\label{subsec:Exam}
\subsubsection{$\mathcal{X}$ a smooth Deligne-Mumford stack}
If $\X$ is a smooth Deligne-Mumford stack, we have 
\begin{align*}
\LL_{\mathcal{N}}(\Omega_{\X})_{\delta}=
    \LL(\Omega_{\X})_{\delta}=\Coh(\Omega_{\X}). 
\end{align*}
Indeed, there is no map $\nu \colon \bgm \to \Omega_{\X}$ 
which is non-trivial on stabilizer groups, so the condition (\ref{wt:nu}) 
is vacuous. Moreover, the cotangent $\Omega_{\X}$ is smooth, 
so the nilpotent singular support condition is also vacuous. 

\subsubsection{$\mathcal{X}=BG$ for a reductive $G$}\label{subsub:BG}
Suppose that $\X=BG$ for an affine algebraic group $G$ with Lie algebra $\mathfrak{g}$.
Then we have
\begin{align*}
    \Omega_{\X}=\mathfrak{g}^{\vee}[-1]/G\stackrel{j}{\hookrightarrow}BG 
\end{align*}
where $j$ is the closed immersion. 
Suppose that $G$ is reductive. Then, 
for any $\mathcal{E}\in \Coh(\Omega_{\X})$, 
we have the decomposition \begin{align*}j_{*}\mathcal{E}=\bigoplus V(\chi)\otimes \mathcal{O}_{BG},\end{align*}
where $V(\chi)$ is a direct sum of 
irreducible $G$-representations with highest weight $\chi$. 
By Lemma~\ref{lem:emb},
we have $\mathcal{E} \in \LL(\Omega_{\X})_{\frac{1}{2}}$ if and only if 
$V(\chi) =0$ for $\chi \neq 0$, or, equivalently, 
$j_{*}\mathcal{E}$ is a direct sum of $\mathcal{O}_{BG}$. 

By Koszul equivalence in Theorem~\ref{thm:Kduality}, there is 
an equivalence
\begin{align}\label{equiv:BG}
    \Coh(\Omega_{\X})\simeq\Coh^{\shear}(\mathfrak{g}/G).
\end{align}
From the construction of Koszul equivalence in Theorem~\ref{thm:Kduality}, 
the equivalence (\ref{equiv:BG}) restricts to an equivalence 
\begin{align*}
    \LL(\Omega_{\X})_{\frac{1}{2}}=\langle 0_{*}\mathcal{O}_{BG}\rangle \simeq \langle \mathcal{O}_{\mathfrak{g}/G}\rangle
    \subset \Coh^{\shear}(\mathfrak{g}/G),
\end{align*}
where the left-hand side is the subcategory 
of $\Coh(\Omega_{\X})$ generated by $0_{*}\mathcal{O}_{\X}$ for the 
zero section $0 \colon \X \to \Omega_{\X}$ 
, and the 
right-hand side is the subcategory of $\Coh^{\shear}(\fg/G)$
generated by $\mathcal{O}_{\mathfrak{g}/G}$. 
We have the equivalence 
\begin{align*}
\langle \mathcal{O}_{\mathfrak{g}/G}\rangle    \stackrel{\sim}{\leftarrow} \Coh^{\shear}(\mathfrak{g}\ssslash G)=\Coh^{\shear}(\mathfrak{h}\ssslash W),
\end{align*}
where $\mathfrak{g}/G \to \mathfrak{g}\ssslash G=\mathfrak{h}\ssslash W$ is the good moduli space morphism,
$\mathfrak{h} \subset \fg$ is the Cartan subalgebra, and $W$ is the Weyl group.

It is known that 
\begin{align}\label{cartan}
k[\mathfrak{h}[-2]]^W=k[y_2, y_4, \ldots, y_{2n}], \ \deg y_{2i}=2m_i, \ n=\dim \mathfrak{h}
\end{align}
for some $m_i\in \mathbb{Z}_{>0}$. 
From the above argument, we have an equivalence 
\begin{align*}
    \LL(\Omega_{BG})_{\frac{1}{2}} \simeq \Coh(k[y_2, y_4, \ldots, y_{2n}]). 
\end{align*}
The subcategory $\LL_{\mathcal{N}}(\Omega_{BG})_{\frac{1}{2}}$ in the left-hand side 
corresponds to nilpotent modules in the right-hand side. By Koszul duality, we have an equivalence
\begin{align*}
    \LL_{\mathcal{N}}(\Omega_{BG})_{\frac{1}{2}} \simeq \mathrm{Perf}(k[x_{-1}, x_{-3}, \ldots, x_{1-2n}]), 
\end{align*}
where $\deg x_{1-2i}=1-2m_i$.


\subsubsection{$\mathcal{X}=\mathbb{A}^1/\mathbb{G}_m$}\label{subsub:A1}
Let $\mathbb{G}_m$ act on $\mathbb{A}^1$ by weight one, 
and consider $\X=\mathbb{A}^1/\mathbb{G}_m$. 
It consists of two points 
\begin{align}\label{2points}
    \bgm \stackrel{i}{\hookrightarrow} \X \stackrel{j}{\hookleftarrow} \mathrm{pt}
\end{align}
where $i$ is a closed immersion and $j$ is an open immersion. 
Its cotangent is 
\begin{align*}
    \Omega_{\X} = \{xy=0\}/\mathbb{G}_m \subset \mathbb{A}^2/\mathbb{G}_m
\end{align*}
where $\mathbb{G}_m$ acts on $\mathbb{A}^2$ by weight $(1, -1)$. 
Let $\lambda \colon \mathbb{G}_m \to \mathbb{G}_m$ be given by $z\mapsto z^{-1}$. 
Then we have the diagram, see Subsection~\ref{notation:attracting} for the notation  
\begin{align*}
    \xymatrix{
    \Omega_{\X}^{\lambda \geq 0}=\Spec k[y, \epsilon]/\mathbb{G}_m \inclusion^-{p} \ar[d]^-{q} & \Omega_{\X} \\
    \Omega_{\X}^{\lambda}=\Spec k[\epsilon]/\mathbb{G}_m &
    }
\end{align*}
where $\deg \epsilon =-1$ with $\mathbb{G}_m$-weight zero. 

By~\cite[Theorem~3.3.1]{HalpK32}, 
there is a semiorthogonal decomposition 
\begin{align*}
    \Coh(\Omega_{\X})=\langle \ldots, \Coh(k[\epsilon])_{\lambda\text{-}\wt=-1}, \langle \mathcal{O}_{\Omega_{\X}} \rangle,  \Coh(k[\epsilon])_{\lambda\text{-}\wt=0}, \Coh(k[\epsilon])_{\lambda\text{-}\wt=1}, \ldots\rangle. 
\end{align*}
Here $\langle\mathcal{O}_{\Omega_{\X}} \rangle$ is the subcategory generated by $\mathcal{O}_{\Omega_{\X}}$, which 
is equivalent to 
\begin{align*}
\langle\mathcal{O}_{\Omega_{\X}}\rangle \stackrel{\sim}{\to} \Coh(\Omega_{\X} \setminus \mathrm{Im}(p)) 
=\Coh(\mathrm{pt})
\end{align*}
and each other component is
\begin{align*}
    \Coh(k[\epsilon])_{\lambda\text{-}\wt=i}=p_{*}q^* \Coh(k[\epsilon])_{\lambda\text{-}\wt=i}
    \simeq \Coh(k[\epsilon]). 
\end{align*}
Note that $\omega_{\X}$ has $\lambda$-weight one. 
Using Lemma~\ref{lem:char2}, 
one can check that 
\begin{align}\label{check:L}
&\LL(\Omega_{\X})_{-\frac{1}{2}}
    =\langle \Coh(k[\epsilon])_{\lambda\text{-}\wt=-1}, 
    \langle\mathcal{O}_{\Omega_{\X}} \rangle \rangle, \\ \notag
&\LL(\Omega_{\X})_{\frac{1}{2}}
    =\langle \langle \mathcal{O}_{\Omega_{\X}} \rangle,  \Coh(k[\epsilon])_{\lambda\text{-}\wt=0} \rangle.
\end{align}
Using the second equivalence in (\ref{equiv:Domega}), we have 
\begin{align}
    \label{sod:LA1}
    \LL(\Omega_{\X})_{\frac{1}{2}}&=\langle \langle \mathcal{O}_{\Omega_{\X}}\rangle, \Coh(k[\epsilon])_{\lambda\text{-}\wt=0}\rangle \\
   \notag &=\langle \Coh(k[\epsilon])_{\lambda\text{-}\wt=0}, \langle \mathcal{O}_{\Omega_{\X}}(-1)\rangle\rangle.
\end{align}
Note that $\LL(\Omega_{\mathrm{pt}})_{\frac{1}{2}}=\Coh(\mathrm{pt})$
and $\LL(\Omega_{\bgm})_{\frac{1}{2}}=\Coh(k[\epsilon])$. 

From the semiorthogonal decompositions (\ref{check:L})
and the second equivalence in (\ref{L:half}), we 
have the following recollement diagram 
\begin{align}\label{dia:recoll}
\xymatrix{
\LL(\Omega_{\bgm})_{\frac{1}{2}} \ar[r]^-{i_{*}^{\Omega}} & \LL(\Omega_{\X})_{\frac{1}{2}} \ar[r]^-{j^*} \ar@/^18pt/[l]_-{i^!} \ar@/_18pt/[l]_-{i^*}& 
\LL(\Omega_{\mathrm{pt}})_{\frac{1}{2}}. \ar@/^18pt/[l]_-{j_*} \ar@/_18pt/[l]_-{j_!}
}
\end{align}
That is, there are exact triangles of functors
\begin{align*}
   i_{\ast}^{\Omega} i^! \to \mathrm{id}_{\LL(\Omega_{\X})_{1/2}} \to j_{*}j^*, \ 
    j_{!}j^* \to \mathrm{id}_{\LL(\Omega_{\X})_{1/2}} \to i_{\ast}^{\Omega}i^*
\end{align*}
so that $j_{*} \circ i_{\ast}^{\Omega}=0$ and the functors $j_!$, $j_*$ and $i_{*}^{\Omega}$ are 
fully-faithful. 
The functor $i_{*}^{\Omega}$ is the functor induced by $i\colon \bgm \hookrightarrow \X$. It will be constructed in general in Proposition~\ref{prop:push}, but see below for this particular example. The semiorthogonal decompositions in (\ref{sod:LA1}) are 
\begin{align*}
\LL(\Omega_{\X})_{\frac{1}{2}}&=\langle j_{*} \LL(\Omega_{\mathrm{pt}})_{\frac{1}{2}},i_{\ast}^{\Omega}\LL(\Omega_{\bgm})_{\frac{1}{2}}\rangle \\
\notag&=\langle i_{\ast}^{\Omega}\LL(\Omega_{\bgm})_{\frac{1}{2}},  j_{!}\LL(\Omega_{\mathrm{pt}})_{\frac{1}{2}}\rangle.
\end{align*}

The recollement diagram (\ref{dia:recoll}) restricts to the recollement diagram for limit 
categories of nilpotent singular supports. In particular, there are semiorthogonal decompositions
\begin{align}\notag
\LL_{\mathcal{N}}(\Omega_{\X})_{\frac{1}{2}}&=\langle j_{*} \LL_{\mathcal{N}}(\Omega_{\mathrm{pt}})_{\frac{1}{2}},i_{\ast}^{\Omega}\LL_{\mathcal{N}}(\Omega_{\bgm})_{\frac{1}{2}}\rangle\\
\notag&=\langle i_{\ast}^{\Omega}\LL_{\mathcal{N}}(\Omega_{\bgm})_{\frac{1}{2}},  j_{!}\LL_{\mathcal{N}}(\Omega_{\mathrm{pt}})_{\frac{1}{2}}\rangle.
\end{align}
Note that $\LL_{\mathcal{N}}(\Omega_{\mathrm{pt}})_{\frac{1}{2}}=\LL(\Omega_{\mathrm{pt}})_{\frac{1}{2}}\simeq\Coh(\mathrm{pt})$ and $\LL_{\mathcal{N}}(\Omega_{\bgm})_{\frac{1}{2}}\simeq\mathrm{Perf}(k[\epsilon])$.


\subsubsection{$\mathcal{X}=BP$ for a parabolic $P$}\label{subsub:GA1}
Let $G$ be a reductive group and $\lambda \colon \mathbb{G}_m \to G$ a cocharacter. 
Let $P=G^{\lambda \geq 0}$ be the parabolic subgroup and $M=G^{\lambda}$ the associated Levi
quotient, with Lie algebras $\fp$, $\fm$ respectively. 
Let $\X$ be the stack $\X=BP$. Note that $\X$ does not admit 
a good moduli space. 
We have the Koszul equivalence 
\begin{align}\notag
    \Coh(\Omega_{\X}) \simeq \Coh^{\shear}(\mathfrak{p}/P). 
\end{align}
The diagram of attacting loci
\begin{align*}
\fg^{\lambda}/G^{\lambda} \leftarrow \fg^{\lambda \geq 0}/G^{\lambda \geq 0} \to \fg/G
\end{align*}
is identified with 
\begin{align*}
    \fm/M \stackrel{\pi}{\leftarrow} \fp/P \to \fg/G.
\end{align*}

We have the semiorthogonal decomposition, see~\cite[Amplification~3.18]{halp}
\begin{align*}
    \Coh(\mathfrak{p}/P)=\langle \pi^{*}\Coh(\fm/M)_{\lambda\text{-}\wt=j} : j\in \mathbb{Z}\rangle. 
\end{align*}
Then one can show that there is an equivalence
\begin{align*}
\pi^*(-)\otimes \omega_{BP} \colon \MM(\fm/M)_0 \stackrel{\sim}{\to}
    \MM(\fp/P)_{\omega_{BP}}. 
\end{align*}
By Proposition~\ref{prop:comparemagic} and the computation in Subsection~\ref{subsub:BG}, we have equivalences 
\begin{align*}
    &\LL(\Omega_{BP})_{\frac{1}{2}} \simeq \Coh^{\shear}(\MM(\fp/P)_{\omega_{BP}}) \simeq 
    \mathrm{Perf}^{\shear}(\mathfrak{h}_M\ssslash W_M), \\
    &\LL_{\mathcal{N}}(\Omega_{BP})_{\frac{1}{2}} \simeq \Coh^{\shear}_{\mathrm{nilp}}(\MM(\fp/P)_{\omega_{BP}}) \simeq 
    \mathrm{Perf}^{\shear}_{\mathrm{nilp}}(\mathfrak{h}_M\ssslash W_M).
\end{align*}
Here $\mathfrak{h}_M \subset \fm$ is the Cartan subalgebra and $W_M$ the Weyl group of $M$.

\section{Limit categories and categories of D-modules}\label{sec:Dmod}
In this section, we discuss the relationship between limit categories 
and categories of D-modules on stacks. Most of the relations discussed in this section 
are still conjectural.

\subsection{The category of D-modules}\label{subsec:catD}
For a smooth stack $\mathcal{X}$, let $\text{D-mod}(\mathcal{X})$
be the dg-category of (right) D-modules on $\mathcal{X}$. 
It is defined by 
\begin{align*}
\text{D-mod}(\mathcal{X}):=
\text{D-mod}^r(\mathcal{X}):=\IndCoh(\mathcal{X}_{\mathrm{dR}})
\end{align*}
see~\cite{MR3701353}. 
Here $\mathcal{X}_{\mathrm{dR}}$ is the de-Rham stack of $\mathcal{X}$,
whose $S$-valued points for an affine derived scheme $S$ is 
given by 
\begin{align*}
\mathrm{Map}(X, \mathcal{X}_{\mathrm{dR}})=\mathrm{Map}(S^{cl, red}, \mathcal{X}). 
\end{align*}
Here, $S^{cl, red}$ is the reduced scheme of the classical truncation of $S$.
We also denote by 
\begin{align*}
    \text{D-mod}_{\mathrm{coh}}(\mathcal{X}):=\Coh(\X_{\mathrm{dR}}) \subset \text{D-mod}(\mathcal{X})
\end{align*}
the subcategory of coherent D-modules. 
\begin{remark}\label{rmk:D-mod1}
In the case that $X$ is a smooth scheme, 
the category $\text{D-mod}(X)$ is the usual (dg) category 
of quasi-coherent right $\mathrm{D}_{X}$-modules, where 
$\mathrm{D}_{X}$ is the sheaf of differential operators on $X$, and $\text{D-mod}_{\mathrm{coh}}(\mathcal{X})$ is the category of coherent right $\mathrm{D}_{X}$-modules.
\end{remark}
\begin{remark}\label{rmk:D-mod2}
In the case that $\mathcal{X}$ is a smooth quotient stack 
$\mathcal{X}=X/G$, the category $\text{D-mod}(\mathcal{X})$ is 
the category of (strongly) $G$-equivariant derived category of $\mathrm{D}_X$-modules 
as in~\cite[Section~5]{McNe}.  In general, one can equivalently 
define $\text{D-mod}(\mathcal{X})$ as 
\begin{align*}
    \text{D-mod}(\mathcal{X})=\lim_{U \to \mathcal{X}}\text{D-mod}(U)
\end{align*}
where $U$ is a scheme with smooth morphism $U\to \mathcal{X}$, and further one has that 
\[\text{D-mod}_{\mathrm{coh}}(\mathcal{X})=\lim_{U \to \mathcal{X}}\text{D-mod}_{\mathrm{coh}}(U).\]
\end{remark}

\begin{remark}\label{rmk:D-mod3}
There is also a category of left D-modules, defined by 
\begin{align*}
    \text{D-mod}^l(\mathcal{X}):=\QCoh(\mathcal{X}_{\mathrm{dR}}). 
\end{align*}
There is a canonical equivalence, see~\cite{MR3701353}:
\begin{align}\label{equiv:rl}
    \Upsilon \colon \text{D-mod}^l(\mathcal{X}) \stackrel{\sim}{\to} 
    \text{D-mod}(\mathcal{X}), \ \mathcal{E} \mapsto \mathcal{E}\otimes \omega_{\mathcal{X}_{\mathrm{dR}}}. 
\end{align}
\end{remark}
By~\cite[Theorem~8.1]{MR3037900}, if $\X$ is a QCA stack, then $\text{D-mod}(\X)$ is compactly 
generated by the subcategory of compact objects $\text{D-mod}(\X)^{\mathrm{cp}}$. 
We have an inclusion
\begin{align*}
    \text{D-mod}(\X)^{\mathrm{cp}} \subset \text{D-mod}_{\mathrm{coh}}(\X)
\end{align*}
which is an equality if $\X$ is a scheme, but it is strict in general. 

\subsection{Examples of comparisons between limit categories and categories of D-modules}\label{subsec:exam}

Here we discuss the relationship between limit categories 
$\LL(\Omega_{\X})_{\frac{1}{2}}$
and 
the categories of coherent D-modules in the examples of Subsection~\ref{subsec:Exam}. 
\subsubsection{$\X=BG$ for a finite group $G$}
Let $\X=BG$ for a finite group $G$, which is a Deligne-Mumford stack. 
Then we have $\Omega_{\X}=\X=BG$ and 
\begin{align*}
\LL_{\mathcal{N}}(\Omega_{\X})_{\frac{1}{2}}=
    \LL(\Omega_{\X})_{\frac{1}{2}}=\Coh(\Omega_{\X})=\mathrm{Rep}(G)=
    \text{D-mod}_{\mathrm{coh}}(BG)=\text{D-mod}(BG)^{\mathrm{cp}}.
\end{align*}
Here $\mathrm{Rep}(G)$ is the dg-category of finite-dimensional $G$-representations. 

\subsubsection{$\X=BG$ for a reductive $G$}\label{subsec:ExD:red}
Let $\X=BG$ for a reductive algebraic group $G$. 
By~\cite[Section~7.2]{MR3037900}, we have
\begin{align*}
    \text{D-mod}(BG) \simeq \QCoh(A)
\end{align*}
where $A=\Gamma_{\mathrm{dR}}(G)^{\vee}$. 
It is known that 
\begin{align*}
A=k[x_{-1}, x_{-3}, \ldots, x_{1-2n}], \ \deg x_{1-2i}=1-2m_i
\end{align*}
for $n=\dim \mathfrak{h}$. By (\ref{cartan}), $A$ is Koszul dual to 
$k[\mathfrak{h}[-2]]^W$, hence we have equivalences
\begin{align*}
   & \text{D-mod}_{\mathrm{coh}}(BG) \simeq \Coh(A)
    \simeq \Coh^{\shear}(\mathfrak{h}\ssslash W), \\ 
    & \text{D-mod}(BG)^{\mathrm{cp}}\simeq \mathrm{Perf}(A)\simeq \Coh^{\shear}_{\mathrm{nilp}}(\mathfrak{h}\ssslash W).
\end{align*}
By the computation in Subsection~\ref{subsub:BG}
we have equivalences 
\begin{align*}
    \LL(\Omega_{BG})_{\frac{1}{2}} \simeq \text{D-mod}_{\mathrm{coh}}(BG), \ 
    \LL_{\mathcal{N}}(\Omega_{BG})_{\frac{1}{2}} \simeq \text{D-mod}(BG)^{\mathrm{cp}}.
\end{align*}

\subsubsection{$\X=\mathbb{A}^1/\mathbb{G}_m$}
Let $\X=\mathbb{A}^1/\mathbb{G}_m$, where $\mathbb{G}_m$ acts on 
$\mathbb{A}^1$ by weight one, and let $i, j$ be as in (\ref{2points}). 
We have the recollements diagram
\begin{align}\notag
\xymatrix{\text{D-mod}(\bgm) \ar[r]^-{i_{*}} & \text{D-mod}(\X) \ar[r]^-{j^*} \ar@/^18pt/[l]_-{i^!} \ar@/_18pt/[l]_-{i^*}& 
\text{D-mod}(\mathrm{pt}). \ar@/^18pt/[l]_-{j_*} \ar@/_18pt/[l]_-{j_!}
}
\end{align}
By~\cite[Theorem~1.1]{McNe}, the above recollement diagram restricts to the 
recollement diagram
\begin{align}\label{dia:recoll3}
\xymatrix{\text{D-mod}_{\mathrm{coh}}(\bgm) \ar[r]^-{i_{*}} & \text{D-mod}_{\mathrm{coh}}(\X) \ar[r]^-{j^*} \ar@/^18pt/[l]_-{i^!} \ar@/_18pt/[l]_-{i^*}& 
\text{D-mod}_{\mathrm{coh}}(\mathrm{pt}). \ar@/^18pt/[l]_-{j_*} \ar@/_18pt/[l]_-{j_!}
}
\end{align}
Note that we have 
$\LL(\Omega_{\bgm})_{\frac{1}{2}}=\text{D-mod}_{\mathrm{coh}}(\bgm)$, 
$\LL(\Omega_{\mathrm{pt}})_{\frac{1}{2}}=\text{D-mod}_{\mathrm{coh}}(\mathrm{pt})$. 
One can check that 
\begin{align*}
    \LL(\Omega_{\mathbb{A}^1/\mathbb{G}_m})_{\frac{1}{2}} \simeq \text{D-mod}_{\mathrm{coh}}(\mathbb{A}^1/\mathbb{G}_m)
\end{align*}
and the recollement (\ref{dia:recoll3}) for D-modules is identified with the recollement
(\ref{dia:recoll}) for limit categories. 

The recollement (\ref{dia:recoll3}) also restricts to the subcategories of compact objects. 
One can also check that 
\begin{align*}
    \LL_{\mathcal{N}}(\Omega_{\mathbb{A}^1/\mathbb{G}_m})_{\frac{1}{2}} \simeq \text{D-mod}(\mathbb{A}^1/\mathbb{G}_m)^{\mathrm{cp}}.
\end{align*}

\subsubsection{$\mathcal{X}=BP$ for a parabolic $P$}
Let $\X=BP$ for a parabolic subgroup $P\subset G$ as in Subsection~\ref{subsub:GA1}.
The quotient $P\twoheadrightarrow M$ induces the morphism
$\pi \colon \X \to BM$, 
which is a $\mathbb{G}_a^n$-gerbe for some $n$. By~\cite[Lemma~10.3.6]{MR3037900}, 
we have the equivalence
\begin{align*}
    \pi^! \colon \text{D-mod}(BM) \stackrel{\sim}{\to} \text{D-mod}(\X),
\end{align*}
which restricts to an equivalences of coherent D-modules from the 
proof of~\cite[Lemma~10.3.6]{MR3037900}. 
Therefore, from the computation in Subsection~\ref{subsub:GA1}, we have equivalences
\begin{align*}
&\text{D-mod}_{\mathrm{coh}}(\X)\simeq \text{D-mod}_{\mathrm{coh}}(BM)
\simeq\mathrm{Perf}^{\shear}(\mathfrak{h}_M\ssslash W_M) \simeq
    \LL(\Omega_{\X})_{\frac{1}{2}}, \\
& \text{D-mod}(\X)^{\mathrm{cp}} \simeq \text{D-mod}(BM)^{\mathrm{cp}}
\simeq \mathrm{Perf}^{\shear}_{\mathrm{nilp}}(\mathfrak{h}_M\ssslash M) \simeq
    \LL_{\mathcal{N}}(\Omega_{\X})_{\frac{1}{2}}.
\end{align*}

\begin{remark}\label{rmk:eqD}
In the examples in Subsection~\ref{subsec:exam},
we have that the cardinality of  $\X(k)$ is finite, and there are equivalences
\begin{align*}
    \LL(\Omega_{\X})_{\frac{1}{2}} \simeq \text{D-mod}_{\mathrm{coh}}(\X), \ 
     \LL_{\mathcal{N}}(\Omega_{\X})_{\frac{1}{2}} \simeq \text{D-mod}(\X)^{\mathrm{cp}}.
\end{align*}
By taking the ind-completion, the latter in particular implies that 
\begin{align*}
    \IndL_{\mathcal{N}}(\Omega_{\X})_{\frac{1}{2}} \simeq \text{D-mod}(\X).
\end{align*}
The above equivalences are not always true for smooth QCA stacks with finite $\X(k)$, for example, for
$\X=B\mathbb{G}_a$, we have:
\begin{align*}
    \text{D-mod}(B\mathbb{G}_a)\simeq \mathrm{QCoh}(\mathrm{pt}), \ \IndL_{\mathcal{N}}(\Omega_{B\mathbb{G}_a})_{\frac{1}{2}}\simeq \mathrm{QCoh}_{\mathbb{G}_a}(k[\epsilon]).
\end{align*}
\end{remark}

\subsection{The category of filtered D-modules}\label{subsec:filt}
For a stack $\mathcal{X}$, let
$\text{D-mod}^{\mathrm{filt}}(\mathcal{X})$ be 
the dg-category of filtered (right) D-modules on 
$\mathcal{X}$. It is given by 
\begin{align}\label{def:Dfilt}
    \text{D-mod}^{\mathrm{filt}}(\X):=\IndCoh(\mathcal{X}_{\mathrm{Hod}}).
\end{align}
Here $\mathcal{X}_{\mathrm{Hod}}$ is the Hodge stack, 
given by the formal deformation to the normal cone
with respect to $\mathcal{X}\to \mathcal{X}_{\mathrm{dR}}$ as in~\cite[Section~9.2.4]{MR3701353}. 
It is a 
prestack over $\mathbb{A}^1/\mathbb{G}_m$
\begin{align}\notag
    \mathcal{X}_{\mathrm{Hod}} \to \mathbb{A}^1/\mathbb{G}_m.
\end{align}
Its generic fiber is $\mathcal{X}_{\mathrm{dR}}$. Its special fiber is
$\mathcal{X}_{\mathrm{Dol}}/\mathbb{G}_m$, where $\mathcal{X}_{\mathrm{Dol}}$
is the Dolbeault stack 
\begin{align*}
    \mathcal{X}_{\mathrm{Dol}}=\mathcal{X}/\widehat{T}_{\mathcal{X}}.
\end{align*}
Here, $\widehat{T}_{\mathcal{X}}$ is the formal completion 
of the zero section of the tangent stack $T_{\mathcal{X}} \to \mathcal{X}$. 
The functor of points of $\X_{\mathrm{Hod}}$ is described 
in~\cite[Definition~2.3.5]{Bhattnote}).

\begin{remark}\label{rmk:hodge}
In~\cite{HChen}, the category $\text{D-mod}^{\mathrm{filt}}(\X)$ is as 
the limit 
\begin{align*}
    \text{D-mod}^{\mathrm{filt}}(\X)=\lim_{U\to \X}\IndCoh(U_{\mathrm{Hod}}),
\end{align*}
where the limit is after all the schemes $U$ with smooth maps $U\to \X$. Tasuki Kinjo~\cite{Kinjonote} informed us that the above definition is equivalent to (\ref{def:Dfilt}). 
\end{remark}

By a version of Koszul duality, 
we have an equivalence, see~\cite[Theorem~2.19]{HChen}
\begin{align}\label{kequ:dol}
\IndCoh_{\mathbb{G}_m}(\mathcal{X}_{\mathrm{Dol}}) \stackrel{\sim}{\to}
\QCoh_{\mathbb{G}_m}(\Omega_{\mathcal{X}}). 
\end{align}
We have the following diagram 
\begin{align}\notag
\xymatrix{
\mathcal{X}_{\mathrm{Dol}}/\mathbb{G}_m \inclusion^-{i} \ar[d]\diasquare & 
\mathcal{X}_{\mathrm{Hod}} \ar[d] \diasquare& \mathcal{X}_{\mathrm{dR}} \ar[d] \linclusion_-{j} \\
\bgm \inclusion & \mathbb{A}^1/\mathbb{G}_m & \mathrm{pt}. \linclusion. 
}
\end{align}
Here the left arrows are closed immersions and the right arrows are 
open immersions. 
We have the diagram 
\begin{align*}
    \IndCoh(\mathcal{X}_{\mathrm{dR}}) \stackrel{j^*}{\leftarrow} \IndCoh(\mathcal{X}_{\mathrm{Hod}})
    \stackrel{i^!}{\to} \IndCoh_{\mathbb{G}_m}(\X_{\mathrm{Dol}}). 
\end{align*}
By the equivalence (\ref{kequ:dol}), the above diagram 
is identified with 
\begin{align}\label{funct:filt}
    \text{D-mod}(\mathcal{X}) \stackrel{j^{\ast}}{\leftarrow} 
    \text{D-mod}^{\mathrm{filt}}(\mathcal{X}) \stackrel{i^!}{\to}
    \QCoh_{\mathbb{G}_m}(\Omega_{\mathcal{X}}). 
\end{align}

\begin{remark}\label{rmk:filtD}
If $X$ is a smooth scheme, the category 
$\text{D-mod}^{\mathrm{filt}}(X)$ is the derived category of filtered
$D_X$-modules, or equivalently quasi-coherent sheaves over the 
Rees algebra 
\begin{align*}
\widetilde{\mathrm{D}}_X=\bigoplus_{k\geq 0}\mathrm{D}_X^{\leq k} t^k,
\end{align*}
where $\mathrm{D}_X^{\leq \bullet} \subset \mathrm{D}_X$ is the order filtration, 
see~\cite[Proposition~3.3.14]{HChen}. 
The functor $j^{*}$ forgets the filtration and the functor $i^!$
is taking the associated graded, see~\cite{HChen}. 
\end{remark}

We denote by 
\begin{align*}
    \text{D-mod}^{\mathrm{filt}}_{\mathrm{coh}}(\mathcal{X}) \subset \text{D-mod}^{\mathrm{filt}}(\mathcal{X})
\end{align*}
the subcategory of coherent filtered D-modules. The functors (\ref{funct:filt})
restrict to the functors
\begin{align}\notag
    \text{D-mod}_{\mathrm{coh}}(\mathcal{X}) \stackrel{j^*}{\leftarrow} \text{D-mod}^{\mathrm{filt}}_{\mathrm{coh}}(\mathcal{X}) \stackrel{i^!}{\to} \Coh_{\mathbb{G}_m}(\Omega_{\mathcal{X}}). 
\end{align}
We also define 
\begin{align*}
    \mathrm{gr} \colon \text{D-mod}^{\mathrm{filt}}(\mathcal{X})
    \stackrel{i^!}{\to} \Coh_{\mathbb{G}_m}(\Omega_{\mathcal{X}}) \to
    \Coh(\Omega_{\mathcal{X}})
\end{align*}
where the last functor is forgetting the grading. 

We formulate a conjecture which provides evidence that the limit category $\LL(\Omega_{\mathcal{X}})_{\frac{1}{2}}$
may be regarded as a classical limit of coherent D-modules. We discuss a strategy of proving it in Subsection~\ref{subsection:DLL}.
\begin{conj}\label{conj:L}
For any $\mathcal{E} \in \mathrm{D}\text{-}\mathrm{mod}_{\mathrm{coh}}(\mathcal{X})$, there is 
an object $\widetilde{\mathcal{E}} \in \mathrm{D}\text{-}\mathrm{mod}^{\mathrm{filt}}_{\mathrm{coh}}(\mathcal{X})$
(i.e. good filtration of $\mathcal{E}$) such that 
\begin{align*}
    \mathrm{gr}(\widetilde{\mathcal{E}}) \in \LL(\Omega_{\mathcal{X}})_{\frac{1}{2}}. 
\end{align*}
\end{conj}

\begin{remark}\label{rmk:check}
It is straightforward to check that Conjecture~\ref{conj:L} holds for the examples 
in Subsection~\ref{subsec:exam}. For example, in the case of $\X=BG$ for reductive $G$, 
the category $\text{D-mod}_{\mathrm{coh}}(BG)$ is generated by a trivial one-dimensional $G$-representation 
$k \in \text{D-mod}_{\mathrm{coh}}(BG)$. For a trivial choice of filtration, we have 
\begin{align*}
    \mathrm{gr}(\widetilde{k})=0_{*}\mathcal{O}_{BG} \in \LL(\Omega_{BG})_{\frac{1}{2}}.
\end{align*}
\end{remark}
\begin{remark}\label{rmk:Kth}
We have the following exact sequence 
\begin{align}\notag
\xymatrix{
    G_{\mathbb{G}_m}(\Omega_{\mathcal{X}}) \ar[r]^-{i_*} &
    K(\text{D-mod}^{\mathrm{filt}}_{\mathrm{coh}}(\mathcal{X})) \ar[d]_-{\mathrm{gr}} \ar[r]^-{j^*} &
    K(\text{D-mod}_{\mathrm{coh}}(\mathcal{X})) \ar[r] & 0 \\
    &  K(\Coh(\Omega_{\mathcal{X}})). & &
    }
\end{align}
Since $\mathrm{gr}\circ i_{\ast}=0$, 
the map $j^*$ factors through the singular support map
\begin{align}\label{ss:map}
    \mathrm{ss} \colon K(\text{D-mod}_{\mathrm{coh}}(\mathcal{X})) \to 
    K(\Coh(\Omega_{\mathcal{X}})). 
\end{align}
Conjecture~\ref{conj:L} in particular implies that 
the map (\ref{ss:map}) factors through 
    \begin{align*}
        \mathrm{ss} \colon K(\mathrm{D}\text{-}\mathrm{mod}_{\mathrm{coh}}(\mathcal{X})) \to 
        K\left(\LL(\Omega_{\mathcal{X}})_{\frac{1}{2}}\right).
    \end{align*}
\end{remark}

\subsection{Functoriality of categories of D-modules}
We recall functorial properties of the categories of 
filtered D-modules and its relation with taking the associated 
graded. These relations were proved by Laumon~\cite{Laufilt}, also see~\cite{CDK}. 

Let $f\colon \mathcal{X}\to \mathcal{Y}$ be a morphism of 
smooth stacks.
We have the induced functors, see~\cite[Section~6]{MR3037900}
\begin{align}\label{dmod:funct}
 \xymatrix{
    \text{D-mod}(\mathcal{Y}) \ar@<0.5ex>[r]^-{f^{!}} & \ar@<0.5ex>[l]^-{f_{*}} 
    \text{D-mod}(\mathcal{X}).   }
\end{align}
The functor $f^!$ preserves coherence if $f$ is smooth, and $f_{*}$ preserves coherence if $f$ is proper.
The map $f$ induces 
the map 
$\mathcal{X}_{\mathrm{Hod}}\to \mathcal{Y}_{\mathrm{Hod}}$, hence 
we have 
the functors
\begin{align*}
 \xymatrix{
    \text{D-mod}^{\mathrm{filt}}(\mathcal{Y}) \ar@<0.5ex>[r]^-{f^{!}} & \ar@<0.5ex>[l]^-{f_{*}} 
    \text{D-mod}^{\mathrm{filt}}(\mathcal{X}).   }
\end{align*}
We also have the induced functors on cotangent stacks
\begin{align*}
    \xymatrix{
    \IndCoh(\Omega_{\Y}) \ar@<0.5ex>[r]^-{f^{\Omega!}} & \ar@<0.5ex>[l]^-{f^{\Omega}_{*}} \IndCoh(\Omega_{\X})
    }
\end{align*}
which are defined using the pull-back/push-forward for the Lagrangian correspondence
$\Omega_{\X} \leftarrow f^{*}\Omega_{\Y} \to \Omega_{\Y}$, see Subsection~\ref{subsec:funccotan}.

\begin{lemma}\label{lem:pullD}
(i) 
    If $f$ is smooth, we have the commutative diagram 
      \begin{align*}
\xymatrix{\mathrm{D}\text{-}\mathrm{mod}_{\mathrm{coh}}(\mathcal{Y}) \ar[d]_-{f^!}&
\mathrm{D}\text{-}\mathrm{mod}_{\mathrm{coh}}^{\mathrm{filt}}(\mathcal{Y}) \ar[l] \ar[r]^-{\mathrm{gr}} \ar[d]_-{f^!}
&  \Coh(\Omega_{\mathcal{Y}}) \ar[d]_-{f^{\Omega !}} \\
\mathrm{D}\text{-}\mathrm{mod}_{\mathrm{coh}}(\mathcal{X}) & \ar[l] \mathrm{D}\text{-}\mathrm{mod}_{\mathrm{coh}}^{\mathrm{filt}}(\mathcal{X})\ar[r]^{\mathrm{gr}} &
\Coh(\Omega_{\mathcal{X}}). 
}
\end{align*}

(ii)
If $f$ is proper, we have the commutative diagram 
    \begin{align*}
\xymatrix{\mathrm{D}\text{-}\mathrm{mod}_{\mathrm{coh}}(\mathcal{X}) \ar[d]_-{f_{*}}&
\mathrm{D}\text{-}\mathrm{mod}_{\mathrm{coh}}^{\mathrm{filt}}(\mathcal{X}) \ar[l] \ar[r]^-{\mathrm{gr}} \ar[d]_-{f_{*}}
&  \Coh(\Omega_{\mathcal{X}}) \ar[d]_-{f_{*}} \\
\mathrm{D}\text{-}\mathrm{mod}_{\mathrm{coh}}(\mathcal{Y}) & \ar[l] \mathrm{D}\text{-}\mathrm{mod}_{\mathrm{coh}}^{\mathrm{filt}}(\mathcal{Y})\ar[r]^{\mathrm{gr}} &
\Coh(\Omega_{\mathcal{Y}}). 
}
\end{align*}
\end{lemma}

\begin{remark}
    The commutative diagrams in Lemma~\ref{lem:pullD} 
    differ from those in~\cite{Laufilt, CDK} by the tensor product with $\omega_{\mathcal{X}/\mathcal{Y}}$. 
    This is because we consider right D-modules, while the 
    diagrams in~\cite{Laufilt, CDK} are for left D-modules. By applying the 
    equivalence (\ref{equiv:rl}), we obtain the diagrams in Lemma~\ref{lem:pullD}. 
\end{remark}

In Section~\ref{sec:functL}, we show that the functors $f^{\Omega!}$, $f_{*}^{\Omega}$ preserve limit 
categories under the assumption of Lemma~\ref{lem:pullD}. 
By Proposition~\ref{prop:pback} and Theorem~\ref{thm:proj}, 
we have the following: 
\begin{cor}\label{cor:funcL}
(i) Suppose that $f\colon \mathcal{X} \to \mathcal{Y}$ is smooth. 
Then, if $\mathcal{E} \in \mathrm{D}\text{-}\mathrm{mod}_{\mathrm{coh}}(\mathcal{Y})$
satisfies Conjecture~\ref{conj:L}, then 
$f^! \mathcal{E} \in \mathrm{D}\text{-}\mathrm{mod}_{\mathrm{coh}}(\mathcal{X})$ satisfies 
Conjecture~\ref{conj:L}. 

(ii) Suppose that $f\colon \mathcal{X} \to \mathcal{Y}$ is projective. 
Then, if $\mathcal{E} \in \mathrm{D}\text{-}\mathrm{mod}_{\mathrm{coh}}(\mathcal{X})$
satisfies Conjecture~\ref{conj:L}, then 
$f_{\ast} \mathcal{E} \in \mathrm{D}\text{-}\mathrm{mod}_{\mathrm{coh}}(\mathcal{Y})$ satisfies 
Conjecture~\ref{conj:L}. 
\end{cor}


\subsection{Singular supports for D-modules}\label{sec:ssupport}
In what follows, we explain how the nilpotent singular support condition (\ref{limit:nilp}) appears 
for the classical limit of the dg-categories of compact D-modules. 

We first discuss the absence of a nilpotent singular support condition on the automorphic side of the de Rham Langlands equivalence. 
Let $C$ be a smooth projective curve over an algebraically closed field of characteristic zero and let $G$ be a reductive algebraic 
group with Langlands dual $^{L}G$. 
The GLC is an equivalence 
\begin{align*}
\IndCoh_{\mathcal{N}}(\mathrm{Locsys}_{^{L}G}(C)) \simeq 
\text{D-mod}(\Bun_{G}). 
\end{align*}
In the left-hand side, we impose nilpotent singular 
support condition in the sense of Arinkin-Gaitsgory~\cite{AG}. 
As explained in~\cite[Remark 12.8.8]{AG}, the corresponding property for 
$\text{D-mod}(\Bun_{G})$ is \textit{unconditionally} satisfied. Namely, if 
we consider $\text{D-mod}(\Bun_{G})$ linear over 
\begin{align*}
    \IndCoh((^{L}\mathfrak{g})^{\vee}[-1]/^{L}G) \simeq \QCoh^{\shear}(^{L}\mathfrak{g}/^{L}G), 
\end{align*}
it is supported over 
\begin{align}\label{nil:supp}
    \mathrm{Nilp}(^{L}\mathfrak{g})/^{L}G \subset ^{L}\mathfrak{g}/^{L}G. 
\end{align}
Equivalently, if we consider $\text{D-mod}(\Bun_{G})$ linear over 
\begin{align*}
\QCoh^{\shear}(^{L}\mathfrak{g}\ssslash ^{L}G)=\QCoh^{\shear}(\mathfrak{g}\ssslash G)
\end{align*}
via pull-back $^{L}\mathfrak{g}/^{L}G \to ^{L}\mathfrak{g}\ssslash ^{L}G$ and the identification 
$^{L}\mathfrak{g}\ssslash ^{L}G=\mathfrak{g}\ssslash G$, it is supported over 
$0 \in \mathfrak{g}\ssslash G$. 
Here the monoidal structure on $\QCoh^{\shear}(\mathfrak{g}\ssslash G)$
is given by the usual tensor product. 

Below we explain how the above property is unconditionally satisfied 
in a more general setting. For a smooth stack $\mathcal{X}$ with a map 
$f\colon \mathcal{X}\to BG$, we have the monoidal action of
$\QCoh^{\shear}(\mathfrak{g}\ssslash G)$ on $\text{D-mod}(\X)$ in the following way. 
We have the following functor 
\begin{align}\label{f!}
    f^! \colon \text{D-mod}_{\mathrm{coh}}(BG) \hookrightarrow \text{D-mod}(BG)
    \stackrel{f^!}{\to} \text{D-mod}(\mathcal{X}). 
\end{align}
We have the monoidal structure on 
$\text{D-mod}(\mathcal{X})$ by the $\otimes^!$-tensor product, i.e. 
\begin{align*}
    A \otimes^! B=\Delta^!(A\boxtimes B)
\end{align*}
where $\Delta \colon \X\to \mathcal{X} \times \mathcal{X}$ is the diagonal, 
see~\cite[Section~5.1.5]{MR3037900}. 
The above monoidal structure induces the monoidal structure on 
$\text{D-mod}_{\mathrm{coh}}(BG)$, and the functor (\ref{f!}) is a monoidal functor. 

By Subsection~\ref{subsec:ExD:red}, we have
\begin{align*}
    \text{D-mod}(BG) \simeq \QCoh(A)
\end{align*}
where $A=\Gamma_{\mathrm{dR}}(G)^{\vee}$, which is given by 
\begin{align*}
A=k[x_{-1}, x_{-3}, \ldots, x_{1-2n}], \ \deg x_i=i. 
\end{align*}
The Koszul equivalence
\begin{align*}
    \text{D-mod}_{\mathrm{coh}}(BG) \simeq \Coh(A)
    \simeq \mathrm{Perf}^{\shear}(\mathfrak{g}\ssslash G)
\end{align*}
is a monoidal equivalence, where 
the monoidal structure on the right-hand side is the usual tensor product, and 
\begin{align*}
    \mathcal{O}_{\mathfrak{g}\ssslash G}=\mathcal{O}_{\mathfrak{h}\ssslash W}=k[y_2, y_4, \ldots, y_{2n}]=H^{*}_{\mathrm{dR}}(BG), \ \deg y_i=i. 
\end{align*}
By taking the ind-completion of the functor (\ref{f!}), we obtain 
the monoidal functor 
\begin{align*}
f^! \colon
    \QCoh^{\shear}(\mathfrak{g}\ssslash G) \to \text{D-mod}(\mathcal{X})
\end{align*}
which gives the monoidal action of $\QCoh^{\shear}(\mathfrak{g}\ssslash G)$
on $\text{D-mod}(\mathcal{X})$, i.e. 
\begin{align}\label{act:Qcoh}
   \QCoh^{\shear}(\mathfrak{g}\ssslash G) \otimes 
 \text{D-mod}(\mathcal{X}) \to   \text{D-mod}(\mathcal{X}), \ 
 (A, B)\mapsto f^! A\otimes^! B. 
\end{align}
\begin{lemma}\label{lem:supp:dmod}
Any compact object $M\in \mathrm{D}\text{-}\mathrm{mod}(\mathcal{X})^{\mathrm{cp}}$ has 
support $0\in \mathfrak{g}\ssslash G$ with respect to the monoidal action (\ref{act:Qcoh}). 
\end{lemma}
\begin{proof}
As explained in~\cite[Remark~12.8.8]{AG}, for a compact object 
$M\in \text{D-mod}(\X)^{\mathrm{cp}}$ the natural map 
\begin{align*}
    H_{\mathrm{dR}}^{*}(\X)\otimes M \to M
\end{align*}
vanishes on a sufficiently high power of the augmentation ideal of 
of $H_{\mathrm{dR}}^{*}(\X)$. Therefore, the composition 
\begin{align*}
    \mathcal{O}_{\mathfrak{g}\ssslash G}\otimes M=H_{\mathrm{dR}}^{*}(BG) \otimes M
    \stackrel{f^! \otimes \id}{\to} H_{\mathrm{dR}}^{*}(\X)\otimes M \to M
\end{align*}
vanishes on a sufficiently high power of the augmentation ideal of 
$\mathcal{O}_{\mathfrak{g}\ssslash G}$. It follows that the inner homomorphism 
$\mathcal{H}om(M, M) \in \QCoh^{\shear}(\mathfrak{g}\ssslash G)$
is supported over $0 \in \mathfrak{g}\ssslash G$. 
\end{proof}

For $\mathcal{X}=\Bun_G$, 
we take a point $x\in C$ and consider the map 
\begin{align*}
    i_x \colon \Bun_G \to BG, \ \mathscr{P}\mapsto \mathscr{P}|_{x}. 
\end{align*}
By Lemma~\ref{lem:supp:dmod} and the compact generation of 
$\text{D-mod}(\Bun_G)$ proved in~\cite{DGbun}, objects in 
$\text{D-mod}(\Bun_G)$ has nilpotent 
supports (\ref{nil:supp}) unconditionally. 

\begin{remark}\label{rmk:action}
The dg-category $\QCoh^{\shear}(\mathfrak{g}\ssslash G)$
is generated by $\mathcal{O}_{\mathfrak{g}\ssslash G}$ and the action
of $\mathcal{O}_{\mathfrak{g}\ssslash G}$
on $\text{D-mod}(\mathcal{X})$ is trivial on the isomorphism classes of objects, 
since $f^!(\mathcal{O}_{\mathfrak{g}\ssslash G})=\omega_{\mathcal{X}_{\mathrm{dR}}}$ is a
monoidal unit. 
In particular, the monoidal action (\ref{act:Qcoh}) restricts to the monoidal action 
\begin{align}\label{act:perfD}
     \mathrm{Perf}^{\shear}(\mathfrak{g}\ssslash G) \otimes 
 \text{D-mod}_{\mathrm{coh}}(\mathcal{X}) \to   \text{D-mod}_{\mathrm{coh}}(\mathcal{X}). 
\end{align}
\end{remark}

\subsection{Singular supports for limit categories}\label{subsec:ssupport}
Let $\X=X/G$ be a quotient stack where $X$ is a smooth  
scheme and $G$ reductive. Let $f$ be the map 
\begin{align}\label{map:f}
    f\colon \mathcal{X} \to BG
\end{align}
given by the $G$-torsor $X\to \mathcal{X}$. Below
we construct a monoidal action of $\mathrm{Perf}^{\shear}(\mathfrak{g}\ssslash G)$
on $\LL(\Omega_{\mathcal{X}})_{\frac{1}{2}}$ which we regard as a classical limit 
of the monoidal action on $\text{D-mod}_{\mathrm{coh}}(\mathcal{X})$ 
as in Remark~\ref{rmk:action}. 

We first construct a monoidal structure on $\IndCoh(\Omega_{\mathcal{X}})$ 
as a classical limit of the tensor monoidal structure on $\text{D-mod}(\mathcal{X})$, 
which is different from the usual monoidal structure on ind-coherent sheaves for derived stacks. 
\begin{defn}\label{def:tensor}
We define the following functor 
\begin{align}\label{funct:delta}
\otimes^{\Omega!}:=
    \Delta^{\Omega !} \colon \IndCoh(\Omega_{\mathcal{X}})\otimes 
 \IndCoh(\Omega_{\mathcal{X}}) \to  \IndCoh(\Omega_{\mathcal{X}})  
\end{align}
where $\Delta \colon \mathcal{X} \to \mathcal{X} \times \mathcal{X}$ is the diagonal. 
\end{defn}
Namely, it is given by $\Delta^{\Omega!}=\alpha_{*}\beta^!$ in the diagram 
\begin{align*}
    \xymatrix{
    \Omega_{\X} \ar[d] & \Delta^{\ast}(\Omega_{\X} \times \Omega_{\X})
    \ar[l]_-{\beta} \ar[r]^-{\alpha} \ar[d]\diasquare & \Omega_{\X}\times \Omega_{\X} \ar[d] \\
    \X \ar@{=}[r] & \X \ar[r]^-{\Delta} & \X \times \X.  
    }
\end{align*}
\begin{lemma}\label{lem:monoidal}
The functor (\ref{funct:delta}) gives a monoidal structure on 
$\IndCoh(\Omega_{\mathcal{X}})$ with monoidal unit $0_{*}\omega_{\mathcal{X}}[\dim \X]$. 
Here $0 \colon \mathcal{X} \to \Omega_{\X}$ is the zero section. 
\end{lemma}
\begin{proof}
    The lemma is easily proved by the base change formula for ind-coherent sheaves
    in~\cite[Corollary~3.7.14]{MR3037900}. 
\end{proof}

We have the functor 
\begin{align}\label{funct:f2}
    f^{\Omega!} \colon \IndCoh(\Omega_{BG}) \to \IndCoh(\Omega_{\mathcal{X}})
\end{align}
which is a monoidal equivalence with respect to the monoidal structure in Lemma~\ref{lem:monoidal}.
In particular, we have the monoidal action 
\begin{align}\label{maction:lim}
    \IndCoh(\Omega_{BG})\otimes \IndCoh(\Omega_{\X}) \to \IndCoh(\Omega_{\X}), \ 
    (A, B)\mapsto f^{\Omega!}(A)\otimes^{\Omega!}B.
\end{align}

Note that $\Omega_{BG}\simeq \mathfrak{g}^{\vee}[-1]/G$
and we have the Koszul equivalence 
\begin{align*}
    \IndCoh(\Omega_{BG}) \simeq \QCoh^{\shear}(\mathfrak{g}/G)
\end{align*}
which is a monoidal equivalence where the monoidal structure on the right-hand side is 
the usual tensor product. 
In particular, $\IndCoh(\Omega_{\mathcal{X}})$ is linear over $\QCoh^{\shear}(\mathfrak{g}/G)$
and $\Coh(\Omega_{\mathcal{X}})$ is linear over $\mathrm{Perf}^{\shear}(\mathfrak{g}/G)$. 
For $\mathcal{E} \in \IndCoh(\Omega_{\X})$, we define 
\begin{align*}
    \mathrm{Supp}^L(\mathcal{E}) \subset \mathfrak{g}/G
\end{align*}
to be the support of $\mathcal{E}$ with respect to the above 
$\QCoh^{\shear}(\mathfrak{g}/G)$-action, i.e. 
the support of the inner homomorphism 
\begin{align*}\mathcal{H}om(\mathcal{E}, \mathcal{E})\in \QCoh^{\shear}(\mathfrak{g}/G).
\end{align*}

On the other hand, recall the 
subcategory of nilpotent Arinkin-Gaitsgory singular supports in Subsection~\ref{subsec:singular}:
\begin{align*}
    \IndCoh_{\mathcal{N}}(\Omega_{\X}) \subset \IndCoh(\Omega_{\X}). 
\end{align*}
\begin{lemma}\label{lem:ssupport}
For an object $\mathcal{E}\in \IndCoh(\Omega_{\X})$, we have 
\begin{align*}
    \mathrm{Supp}^L(\mathcal{E}) \subset \mathrm{Nilp}(\mathfrak{g})/G
\end{align*}
if and only if $\mathcal{E}\in \IndCoh_{\mathcal{N}}(\Omega_{\X})$. 
Here $\mathrm{Nilp}(\fg)\subset \fg$ is the nilpotent cone.
\end{lemma}
\begin{proof}
The monoidal action (\ref{maction:lim}) is given by     
\begin{align*}
    f^{\Omega!}(A)\otimes^{\Omega !} B &= \Delta^{\Omega !}(f^{\Omega!}(A) \boxtimes B) \\
    &=\Delta^{\Omega!}(f\times \id)^{\Omega !}(A\boxtimes B) \\
    &=g^{\Omega!}(A\boxtimes B)
\end{align*}
where $g=(f, \id) \colon \X\to BG\times \X$. 
It is then given by $\beta_* \alpha^!$ for the diagram 
\begin{align}\label{dia:mon1}
    \xymatrix{
    \Omega_{\X} \ar[d] & \Y
    \ar[l]_-{\beta} \ar[r]^-{\alpha} \ar[d] \diasquare& \Omega_{BG}\times \Omega_{\X} \ar[d] \\
    \X \ar@{=}[r] & \X \ar[r]^-{g} & BG \times \X.  
    }
\end{align}
Here, the right square is Cartesian. 

Note that we have 
\begin{align*}
    \Omega_{BG} \simeq \mathfrak{g}^{\vee}[-1]/G \simeq 0/G \times_{\mathfrak{g}^{\vee}/G}0/G
\end{align*}
where $0\in \mathfrak{g}^{\vee}$ is the origin. We also have equivalences
\begin{align*}
    \Y &\simeq  \Omega_{\X}\times_{BG}\Omega_{BG} \\
    &\simeq \Omega_{\X}\times_{BG}(0/G \times_{\mathfrak{g}^{\vee}/G}0/G) \\
    &\simeq \Omega_{\X}\times_{\mathfrak{g}^{\vee}/G}\times 0/G \\
    &\simeq \Omega_{\X}\times_{\Omega_{X/G}}(\Omega_{X/G}\times_{\mathfrak{g}^{\vee}/G}0/G)\\
    &\simeq \Omega_{\X}\times_{\Omega_{X/G}}\times \Omega_{\X}.
\end{align*}

Under the above equivalences, the monoidal action $g^{\Omega!}=\beta_{*}\alpha^!$
in the diagram (\ref{dia:mon1}) is identified with the monoidal action 
induced by the diagram 
\begin{align}\label{dia:mon2}
    \xymatrix{
    \Omega_{\X} \ar[d] & \Omega_{\X}\times_{\Omega_X/G} \Omega_{\X}
    \ar[l] \ar[r] \ar[d] &  \Omega_{\X} \ar[d] \\
    BG & \ar[l] 0/G\times_{\mathfrak{g}^{\vee}/G}0/G \ar[r] & BG.  
    }
\end{align}
Here the horizontal arrows are projections onto the corresponding 
factors. The monoidal action of $\IndCoh(\Omega_{BG})$ induced by 
the diagram (\ref{dia:mon2}) is used to characterize Arinkin-Gaitsgory singular 
supports in~\cite[Section~9.2]{AG}. The lemma now follows 
from~\cite[Lemma~9.2.6]{AG}. 
\end{proof}

\subsection{Nilpotent singular supports at classical limits}\label{subsec:limsupp}
Since the map $f$ in (\ref{map:f}) is smooth, by Proposition~\ref{prop:pback}
the functor (\ref{funct:f2}) restricts to the functor 
\begin{align}\label{ind:fL}
    f^{\Omega!}\colon \LL(\Omega_{BG})_{\frac{1}{2}} \to \LL(\Omega_{\X})_{\frac{1}{2}}. 
\end{align}
From the computation in Subsection~\ref{subsub:BG}, the 
subcategory 
\begin{align*}
    \LL(\Omega_{BG})_{\frac{1}{2}} \subset \Coh(\Omega_{\mathcal{X}})
\end{align*}
is identified with the monoidal subcategory 
\begin{align*}
    \pi^{*}\colon \mathrm{Perf}^{\shear}(\mathfrak{g}\ssslash G) 
    \hookrightarrow \mathrm{Perf}^{\shear}(\mathfrak{g}/ G). 
\end{align*}
Since $\mathcal{O}_{\mathfrak{g}\ssslash G}$ acts on the isomorphism classes of objects in 
$\IndCoh(\Omega_{\X})$ trivially, 
the functor (\ref{ind:fL})
induces the monoidal action of $\mathrm{Perf}^{\shear}(\mathfrak{g}\ssslash G)$
\begin{align}\label{action:L}
    \mathrm{Perf}^{\shear}(\mathfrak{g}\ssslash G) \otimes \LL(\Omega_{\X})_{\frac{1} {2}} \to \LL(\Omega_{\X})_{\frac{1}{2}}.
\end{align}
It extends to the monoidal action 
\begin{align}\label{action:L2}
\QCoh^{\shear}(\mathfrak{g}\ssslash G) \otimes 
\IndL(\Omega_{\X})_{\frac{1}{2}} \to 
\IndL(\Omega_{\X})_{\frac{1}{2}}.
\end{align}
The result of Lemma~\ref{lem:ssupport} immediately implies the following: 
\begin{cor}\label{cor:nilp}
The subcategories 
\begin{align*}
    \LL_{\mathcal{N}}(\Omega_{\X})_{\frac{1}{2}} \subset \LL(\Omega_{\X})_{\frac{1}{2}}, \ 
      \IndL_{\mathcal{N}}(\Omega_{\X})_{\frac{1}{2}} \subset \IndL(\Omega_{\X})_{\frac{1}{2}}
\end{align*}
are exactly those having supports $0\in \mathfrak{g}\ssslash G$ with respect
to the monoidal action (\ref{action:L}), (\ref{action:L2}) respectively. 
\end{cor}

Since the action (\ref{action:L}) is regarded as a classical limit analogue of the
action (\ref{act:perfD}) for D-modules, and compact D-modules
are supported over $0\in \fg\ssslash G$ by Lemma~\ref{lem:supp:dmod},
we expect that $\LL_{\mathcal{N}}(\Omega_{\X})_{\frac{1}{2}}$ may be regarded as the classical 
limit of $\text{D-mod}(\X)^{\mathrm{cp}}$. 
By taking the ind-completion, we expect the following degeneraton of categories 
\begin{align*}
    \text{D-mod}(\X)  \leadsto \IndL_{\mathcal{N}}(\Omega_{\X})_{\frac{1}{2}}.
\end{align*}
More precisely, we 
expect the following analogue of Conjecture~\ref{conj:L}:
\begin{conj}\label{conj:limitLN}
For any compact object 
$\mathcal{E} \in \mathrm{D}\text{-}\mathrm{mod}(\mathcal{X})^{\mathrm{cp}}$, there is 
an object $\widetilde{\mathcal{E}} \in \mathrm{D}\text{-}\mathrm{mod}^{\mathrm{filt}}_{\mathrm{coh}}(\mathcal{X})$
(i.e. good filtration of $\mathcal{E}$) such that 
\begin{align*}
    \mathrm{gr}(\widetilde{\mathcal{E}}) \in \LL_{\mathcal{N}}(\Omega_{\mathcal{X}})_{\frac{1}{2}}. 
\end{align*}
    \end{conj}

\begin{remark}\label{rmk:left}
In this section, we discussed the classical limit of the category of right D-modules. 
If instead we consider left D-modules $\text{D-mod}^l(\X)$, we expect 
that its classical limit is $\IndL_{\mathcal{N}}(\Omega_{\X})_{-\frac{1}{2}}$. 
The equivalence 
\begin{align*}
\otimes \omega_{\X} \colon 
    \IndL_{\mathcal{N}}(\Omega_{\X})_{-\frac{1}{2}}\stackrel{\sim}{\to}\IndL_{\mathcal{N}}(\Omega_{\X})_{\frac{1}{2}}
\end{align*}
may be regarded as a classical limit of the equivalence (\ref{equiv:rl}). 
\end{remark}

\subsection{Limit categories and categories of D-modules for loop stacks}\label{subsection:DLL}
We next mention a strategy for proving Conjectures~\ref{conj:L} and \ref{conj:limitLN} and the claims made in Remark~\ref{rem:singularsupport}. 
More precisely, for $\X$ a smooth QCA stack, 
we claim that there exists an $\mathbb{A}^1/\mathbb{G}_m$-linear category \[\hat{\text{L}}(\mathbb{T}_\X[-1])^{B\mathbb{G}_a\rtimes\mathbb{G}_m}_{\frac{1}{2}}\] whose fiber over $1$ is equivalent to $\text{D-mod}_{\mathrm{coh}}(\X)$ and whose fiber over $0/\mathbb{G}_m$ is a $\mathbb{G}_m$-equivariant version of the limit category $\text{L}_{\mathbb{G}_m}(\Omega_\X)_{\frac{1}{2}}$. 
Conjecture~\ref{conj:L} would then follow from the (expected) surjectivity of the restriction map:
\[\text{res}\colon K(\hat{\text{L}}(\mathbb{T}_\X[-1])^{B\mathbb{G}_a\rtimes\mathbb{G}_m}_{\frac{1}{2}})\to K(\text{D-mod}_{\mathrm{coh}}(\X)).\]
The construction of $\hat{\text{L}}(\mathbb{T}_\X[-1])^{B\mathbb{G}_a\rtimes\mathbb{G}_m}_{\frac{1}{2}}$ involves the following:
\begin{enumerate}
    \item a version of Koszul duality which describes D-modules as loop-equivariant sheaves on the formal completion $\hat{\mathscr{L}}\X$ of the loop stack of $\mathscr{L}\X$ along the constant loops $\X\hookrightarrow \mathscr{L}\X$, see~\cite{BZNa, HChen},
    \item a construction analogous to limit categories for the dg-category $\Coh(\mathscr{L}\X)$ and other related categories (namely loop-equivariant \textit{continuous ind-coherent sheaves} on $\hat{\mathscr{L}}\X$, see~\cite{HChen, ChenDhillon}), and 
    \item a semiorthogonal decomposition for $\Coh(\mathscr{L}\X)$ and related categories analogous to the semiorthogonal decomposition~\eqref{intro:sod:general}.
\end{enumerate}
We postpone discussing (ii) and (iii) for future work, and we briefly explain the strategy for completing these steps in the next two subsections. 
 We expect that, for general smooth stacks $\X$, the semiorthogonal decompositions mentioned above should use the `component lattice of a stack' introduced in~\cite{BHLINK}, and thus the combinatorics involved is more involved than in the case of~\eqref{intro:sod:L} or~\eqref{intro:sod:general}.

\subsection{D-modules via Koszul duality}

Let $\X$ be a smooth QCA stack. We recall a description of the category of D-modules using a category of loop-equivariant coherent sheaves on $\hat{\mathscr{L}}\X$, see~\cite[Theorem A]{HChen}. 
First, consider the loop stack \[\mathscr{L}\X:=\text{Map}(S^1,\X):=\X\times_{\X\times\X}\X.\]
Let $\X\hookrightarrow\mathscr{L}\X$ be the closed substack of constant loops, and consider its normal bundle
\[N_{\X|\mathscr{L}\X}=\mathbb{T}_{\X}[-1].\]
Next, consider the completion of $\mathbb{T}_\X[-1]$ along the zero section $0\colon \X\hookrightarrow \mathbb{T}_\X[-1]$: \[\hat{\mathbb{T}}_\X[-1].\] 
For definition and examples of $S^1$-actions, see~\cite{Preygel},~\cite{HChen}, and~\cite[Appendix A]{BZCHN}.
There is an $S^1$-action on $\mathscr{L}\X$, inducing an $S^1$-action on $\hat{\mathbb{T}}_\X[-1]$, which factors through an action of $B\mathbb{G}_a$ (the affinization of $S^1$). 
Further, both $\hat{\mathbb{T}}_\X[-1]$ and $B\mathbb{G}_a$ admit compatible scaling $\mathbb{G}_m$-actions, and we thus obtain an action  
\[B\mathbb{G}_a\rtimes \mathbb{G}_m \curvearrowright \hat{\mathbb{T}}_\X[-1].\] 
The category \[\text{IndCoh}(\hat{\mathbb{T}}_\X[-1])^{\omega B\mathbb{G}_a\rtimes\mathbb{G}_m}\] is $\mathbb{A}^1/\mathbb{G}_m$-linear, where $\omega$ denotes a certain renormalization of the category~\cite[Subsection 3.1.8]{HChen}. The fiber of $\text{IndCoh}(\hat{\mathbb{T}}_\X[-1])^{\omega B\mathbb{G}_a\rtimes\mathbb{G}_m}$ over $1$ is the graded Tate category $\text{IndCoh}(\hat{\mathbb{T}}_\X[-1])^{\text{grTate}}$ and the fiber over $0$ is $\text{IndCoh}_{\mathbb{G}_m}(\hat{\mathbb{T}}_\X[-1])$.
There are compatible Koszul equivalences, see~\cite[Theorem A]{HChen}:
\begin{equation*}
    \begin{tikzcd}
        \text{IndCoh}(\hat{\mathbb{T}}_\X[-1])^{\text{grTate}} 
        \arrow[d, "\sim" {anchor=south, rotate=90}]& 
        \text{IndCoh}(\hat{\mathbb{T}}_\X[-1])^{\omega B\mathbb{G}_a\rtimes\mathbb{G}_m} 
        \arrow[r]\arrow[l]\arrow[d, "\sim" {anchor=south, rotate=90}]& 
        \text{IndCoh}_{\mathbb{G}_m}(\hat{\mathbb{T}}_\X[-1]) \arrow[d, "\sim" {anchor=south, rotate=90}]\\
        \text{D-mod}(\X)& 
        \text{D-mod}^{\text{filt}}(\X) \arrow[r, "\text{gr}"]\arrow[l]& \text{QCoh}_{\mathbb{G}_m}(\Omega_{\X}).
    \end{tikzcd}
\end{equation*}

We next state an analogous result for the category of coherent D-modules, using the category of \textit{continuous ind-coherent sheaves}. For a stack $\mathcal{Y}$ with a closed substack $\mathcal{Z}\hookrightarrow \mathcal{Y}$, Chen--Dhillon~\cite{ChenDhillon} consider
the category of ind-coherent sheaves \[\hat{\Coh}(\Y_\mathscr{Z})\subset \text{IndCoh}(\hat{\Y}_\mathscr{Z}),\] which  is an enlargment of the category $\Coh(\hat{\Y}_\mathscr{Z})$ of coherent sheaves on the formal completion of $\Y$ along the closed substack $\mathscr{Z}$. In our context, we consider
\[\hat{\Coh}(\mathbb{T}_\X[-1])^{\omega B\mathbb{G}_a\rtimes\mathbb{G}_m}\subset  \text{IndCoh}(\hat{\mathbb{T}}_\X[-1])^{\omega B\mathbb{G}_a\rtimes\mathbb{G}_m}.\]
There are compatible Koszul equivalences, see~\cite[Theorem A]{HChen}:
\begin{equation*}
    \begin{tikzcd}
        \hat{\Coh}(\mathbb{T}_\X[-1])^{\text{grTate}} 
        \arrow[d, "\sim" {anchor=south, rotate=90}]& 
        \hat{\Coh}(\mathbb{T}_\X[-1])^{ B\mathbb{G}_a\rtimes\mathbb{G}_m} 
        \arrow[r]\arrow[l]\arrow[d, "\sim" {anchor=south, rotate=90}]& 
        \hat{\Coh}_{\mathbb{G}_m}(\mathbb{T}_\X[-1]) \arrow[d, "\sim" {anchor=south, rotate=90}]\\
        \text{D-mod}_{\mathrm{coh}}(\X)& 
        \text{D-mod}^{\text{filt}}_{\Coh}(\X) \arrow[r, "\text{gr}"]\arrow[l]& \Coh_{\mathbb{G}_m}(\Omega_{\X}).
    \end{tikzcd}
\end{equation*}

\subsection{Limit categories for loop stacks}

We next discuss the construction of the category 
\begin{equation}\label{def:Lhat}
\hat{\text{L}}(\mathbb{T}_{\X}[-1])^{B\mathbb{G}_a\rtimes\mathbb{G}_m}_{\frac{1}{2}},
\end{equation} which is an $\mathbb{A}^1/\mathbb{G}_m$-linear category whose fiber over $1$ is denoted by $\hat{\text{L}}(\mathbb{T}_{\X}[-1])^{\text{grTate}}$ and whose fiber over $0/\mathbb{G}_m$ is equivalent, under the Koszul equivalence, to $\LL_{\mathbb{G}_m}(\Omega_\X)_{\frac{1}{2}}$.

We first define a category analogous to the limit category for the loop stack
\[\text{L}(\mathscr{L}\X)_\delta\subset \Coh(\mathscr{L}\X),\]
where $\delta\in\text{Pic}(\X)_\mathbb{R}$. As before, when $\delta:=\omega_{\X}^{1/2}$, we denote this category by $\text{L}(\mathscr{L}\X)_{\frac{1}{2}}$. 

The definition is completely analogous to Definition~\ref{def:L}, with the only exception being the replacement of $T_x$ in~\eqref{wt:nu} by $T'_x$, which we now define. Consider a field extension $k'/k$ and a map $\nu\colon (\bgm)_{k'}\to\mathcal{L}\X$ be a map with $\nu(0)=x\in \mathcal{L}\X(k')$. 
We define $T_x'$ to be \begin{align*}T'_x:=\mathcal{H}^{1}(\mathbb{T}_{\mathcal{L}\X}|_x)\oplus \mathcal{H}^{1}(\mathbb{T}_{\mathcal{L}\X}|_x)^\vee.\end{align*}
We have that $\mathfrak{g}_x:=\mathcal{H}^{-1}(\mathbb{T}_{\mathcal{L}\X}|_x)$ is the Lie algebra of the stabilizer group of $x$ in $\mathcal{L}\X$. There is a (non-canonical) splitting 
\[\mathcal{H}^0(\mathbb{T}_{\mathcal{L}\X}|_x)\cong \mathcal{H}^{-1}(\mathbb{T}_{\mathcal{L}\X}|_x)\oplus \mathcal{H}^{1}(\mathbb{T}_{\mathcal{L}\X}|_x).\]

There exists a regularization $\nu^{\text{reg}}\colon \mathfrak{g}_x^\vee[-1]/(\mathbb{G}_m)_{k'}\to \mathcal{L}\X$ of $\nu$ as in Construction-Definition~\ref{def:construction}, and a natural map $\iota\colon \mathfrak{g}_x^{\vee}[-1]/(\mathbb{G}_m)_{k'} \to 
    (\bgm)_{k'}$. 

\begin{defn}

We define the subcategory \[\text{L}(\mathcal{L}\X)_\delta\subset \Coh(\mathcal{L}\X)\] to be consisting of objects $\mathcal{E}$ 
such that for, any field extension $k'/k$ and a map
$\nu \colon (\bgm)_{k'} \to \mathcal{L}\X$ 
with $\nu(0)=x$ and regularization (\ref{nu:reg}),
we have:
\[\wt(\iota_{\ast}\nu^{\mathrm{reg}\ast}(\mathcal{E}))
\subset \left[\frac{1}{2} c_1 (T_x'^{<0}),
\frac{1}{2} c_1 (T_x'^{>0}) \right]+
\frac{1}{2}c_1(\mathfrak{g}_x)+c_1(\nu^*\delta).\]

\end{defn}

There is also a version $\text{L}(\mathbb{T}_\X[-1])_\delta\subset \hat{\Coh}(\mathbb{T}_\X[-1])$ with an action of $B\mathbb{G}_a\rtimes \mathbb{G}_m$. There are forgetful functors \[\text{forg}\colon \hat{\Coh}(\mathbb{T}_\X[-1])^{B\mathbb{G}_a\times\mathbb{G}_m}\to \hat{\Coh}(\mathbb{T}_\X[-1]).\]
Then~\eqref{def:Lhat} is the full subcategory of $\hat{\Coh}(\mathbb{T}_\X[-1])^{B\mathbb{G}_a\times\mathbb{G}_m}$ consisting of objects $\mathcal{E}$ such that \begin{align*}\text{forg}(\mathcal{E})\in \text{L}(\mathbb{T}_\X[-1])\subset \hat{\Coh}(\mathbb{T}_\X[-1]).
\end{align*}
We expect the category~\eqref{def:Lhat} to be linear over $\mathbb{A}^1/\mathbb{G}_m$ and 
with fiber over $0/\mathbb{G}_m$
equivalent, under the Koszul equivalence, to $\LL_{\mathbb{G}_m}(\Omega_\X)_{\frac{1}{2}}$. 
Its fiber over $1$ is $\hat{\text{L}}(\mathbb{T}_\X[-1])^{\text{grTate}}_{\frac{1}{2}}$. 

\begin{conj}
    The inclusion $\hat{\mathrm{L}}(\mathbb{T}_\X[-1])^{B\mathbb{G}_a\times\mathbb{G}_m}_{\frac{1}{2}}\subset \hat{\mathrm{Coh}}(\mathbb{T}_\X[-1])^{B\mathbb{G}_a\times\mathbb{G}_m}$ induces an equivalence 
    \[\hat{\mathrm{L}}(\mathbb{T}_\X[-1])^{\mathrm{grTate}}_{\frac{1}{2}}
    \xrightarrow{\sim} \hat{\mathrm{Coh}}(\mathbb{T}_\X[-1])^{\mathrm{grTate}}\xrightarrow{\sim} 
    \mathrm{D}\text{-}\mathrm{mod}_{\mathrm{coh}}(\X).\]
\end{conj}

As mentioned above, the plan of the proof of the above claim is through a semiorthogonal decomposition of $\hat{\mathrm{Coh}}(\mathbb{T}_\X[-1])^{B\mathbb{G}_a\times\mathbb{G}_m}$, where one summand is the category $\hat{\mathrm{L}}(\mathbb{T}_\X[-1])^{B\mathbb{G}_a\times\mathbb{G}_m}_{\frac{1}{2}}$. 
Indeed, we may attempt to construct a semiorthogonal decomposition \[\Coh(\mathcal{L}\X)=\big\langle \mathrm{L}(\mathcal{L}\mathcal{Z})_\varepsilon, \mathrm{L}(\mathcal{L}\X)_{\frac{1}{2}}\big\rangle\] for certain pairs $(\mathcal{Z}, \varepsilon)$,  where $\mathcal{Z}$ is a connected component of $\text{Map}(\bgm, \mathcal{X})$ and, if $\mathcal{Z}$ is different from $\mathcal{X}$, then $\varepsilon\in \text{Pic}(\mathcal{Z})_\mathbb{R}$ is not in the image of the pullback map $\text{Pic}(\mathcal{X})_\mathbb{R}\to \text{Pic}(\mathcal{Z})_\mathbb{R}$. 
One needs to argue that there are analogous semiorthogonal decompositions for $\hat{\Coh}(\mathbb{T}_\X[-1])$ and for \[\hat{\Coh}(\mathbb{T}_\X[-1])^{B\mathbb{G}_a\rtimes\mathbb{G}_m}=\big\langle \hat{\text{L}}(\mathbb{T}_\mathcal{Z}[-1])_\varepsilon^{B\mathbb{G}_a\rtimes\mathbb{G}_m}, \hat{\mathrm{L}}(\mathbb{T}_\X[-1])^{B\mathbb{G}_a\times\mathbb{G}_m}_{\frac{1}{2}}\big\rangle ,\] for certain pairs $(\mathcal{Z},\varepsilon)$ as above. 
It thus suffices to show that 
\[\hat{\text{L}}(\mathbb{T}_\mathcal{Z}[-1])_\varepsilon^{\text{grTate}}=0\] for $\varepsilon$ as above. Alternatively, it suffices to show that 
\begin{equation}\label{zerovarepsilon}
\hat{\Coh}(\mathbb{T}_\mathcal{Z}[-1])_\varepsilon^{\text{grTate}}=0.
\end{equation}
The stack $\mathcal{Z}$ has an étale cover $\mathcal{Z}\to\mathcal{Z}'\times \bgm$. The claim~\eqref{zerovarepsilon} is expected to be a corollary of the following vanishing for the stack $\bgm$:
\[\hat{\Coh}(\mathbb{T}_{\bgm}[-1])^{\text{grTate}}_a=0\] for $a\neq 0$. 
This is proved in~\cite[Subsection 3.5.6]{HChen}.

\section{Functorial properties of limit categories}\label{sec:functL}
In this section, we prove the functorial properties of limit categories. 
Namely, for a morphism of smooth stacks, we show that the associated lagrangian 
correspondence induces functors between limit categories as stated in Theorem~\ref{intro:thmfunct}. 
The result of this section will be used to construct Hecke operators on limit 
categories in Section~\ref{sec:hecke}.
\subsection{Functors for cotangent stacks}\label{subsec:funccotan}
Below we use the notation in Subsection~\ref{subsec:cotan}. 
Let $f \colon \X \to \Y$ be a morphism of smooth QCA
stacks. We have the following 
commutative diagram 
\begin{align}\label{dia:alphabeta}
    \xymatrix{
    \Omega_{\X} \ar[d]_{\pi_{\X}} & f^{\ast}\Omega_{\Y} 
    \ar[l]_-{\beta} \ar[r]^-{\alpha} \ar[d]\diasquare & \Omega_{\Y} \ar[d]^-{\pi_{\Y}} \\
    \X \ar@{=}[r] & \X \ar[r]^-{f} & \Y.  
    }
\end{align}
Here $f^{\ast}\Omega_{\Y}:=\Omega_{\Y}\times_{\Y}\X$
and the right square is Cartesian. 
\begin{lemma}\label{lem:beta}
We have the following:
\begin{enumerate}
    \item if $f$ is schematic, then $\beta$ is quasi-smooth; 
    \item if $f$ is smooth, then $\beta$ is a closed immersion. 
\end{enumerate}
\end{lemma}
\begin{proof}
The lemma follows since $\mathbb{L}_f$ is perfect of cohomological amplitude $[-1, 0]$ if $f$ is schematic, 
and is perfect of cohomological amplitude $[0, 1]$ if $f$ is smooth. 
\end{proof}

We have the induced functors on 
ind-coherent sheaves 
\begin{align}\label{fomega}
    &f^{\Omega!} :=\beta_{*}\alpha^! \colon 
    \IndCoh(\Omega_{\Y}) \to \IndCoh(\Omega_{\X}).
\end{align}
If $f$ is schematic, we also have the functor 
\begin{align}\label{fomega2}
    &f^{\Omega}_{*} :=\alpha_{*}\beta^* \colon 
    \IndCoh(\Omega_{\X}) \to \IndCoh(\Omega_{\Y}).
\end{align}
Here $\beta^*$ is the pull-back of ind-coherent sheaves for 
a quasi-smooth morphism, see~\cite[Section~11.6]{MR3136100}. 

These functors satisfy the functorial properties. 
The following lemmas are straightforward to prove 
by the base change for ind-coherent sheaves~\cite[Corollary~3.7.14]{MR3037900}. 
\begin{lemma}\label{lem:funct}
Let $f_1 \colon \X_1 \to \X_2$ and $f_2\colon \X_2 \to \X_3$ be 
morphisms of smooth stacks and let $f_3$ be their composition 
\begin{align*}
    f_3=f_2 \circ f_1 : 
    \X_1 \stackrel{f_1}{\to} \X_2 \stackrel{f_2}{\to} \X_3. 
\end{align*}
\begin{enumerate}
    \item 
    We have 
    \begin{align*}
        f_3^{\Omega!}=f_1^{\Omega!} \circ f_2^{\Omega !} \colon \IndCoh(\Omega_{\X_3}) \to 
        \IndCoh(\Omega_{\X_1}). 
    \end{align*}
\item If $f_1$ and $f_2$ are schematic, we have 
\begin{align*}
        f_{3*}^{\Omega}=f_{2*}^{\Omega} \circ f_{1*}^{\Omega} \colon \IndCoh(\Omega_{\X_1}) \to 
        \IndCoh(\Omega_{\X_3}). 
    \end{align*}
\end{enumerate}
\end{lemma}

\begin{lemma}\label{lem:funct2}
Consider the derived Cartesian square of smooth stacks 
\begin{align*}
    \xymatrix{
\X_4 \ar[r]^-{f_4} \ar[d]_-{f_3} \diasquare& \X_1 \ar[d]^-{f_1} \\
\X_3 \ar[r]^-{f_2} & \X_2.     
    }
\end{align*}
Suppose that $f_1, f_3$ are schematic.
Then we have 
\begin{align*}
   f_2^{\Omega!} \circ f_{1\ast}^{\Omega} \cong f_{3\ast}^{\Omega} \circ f_4^{\Omega!}
   \colon \IndCoh(\Omega_{\X_1}) \to \IndCoh(\Omega_{\X_3}).
\end{align*}
\end{lemma}
Recall the functor $\otimes^{\Omega!}$ from Definition~\ref{def:tensor}. We also have the following lemma:
\begin{lemma}\label{lem:funct3}
(i) For a morphism $f\colon \X \to \Y$ of smooth stacks, we have 
\begin{align*}
    f^{\Omega!}(-)\otimes^{\Omega!}f^{\Omega!}(-)\cong f^{\Omega!}(-\otimes^{\Omega!}-)
    \end{align*}
    as functors
    \begin{align*} 
    \IndCoh(\Omega_{\Y})\otimes \IndCoh(\Omega_{\Y}) \to \IndCoh(\Omega_{\X}).
    \end{align*}

(ii) For a schematic morphism $f\colon \X \to \Y$ of smooth stacks, we have 
\begin{align*}
    f^{\Omega}_{*}(-)\otimes^{\Omega!}(-)\cong f^{\Omega}_{*}(-\otimes^{\Omega!}f^{\Omega!}(-))\end{align*}
    as functors \begin{align*} 
    \IndCoh(\Omega_{\Y})\otimes \IndCoh(\Omega_{\X}) \to \IndCoh(\Omega_{\X}).
\end{align*}
\end{lemma}

Suppose that $f\colon \X \to \Y$ is a smooth morphism of smooth stacks, and 
consider the diagram (\ref{dia:alphabeta}). Then 
$\alpha$ is smooth and $\beta$ is a closed immersion. In this case, the 
functor (\ref{fomega}) restricts to the functor 
\begin{align}\label{pback!}
f^{\Omega!} :=\beta_{\ast}\alpha^{!} \colon \Coh(\Omega_{\Y})
\to \Coh(\Omega_{\X})
\end{align}
Here $\alpha^!(-)=\alpha^{\ast}(-) \otimes \omega_{f}[\dim f]$ since $\alpha$ is smooth, 
see~\cite[Equation (3.12)]{MR3037900}.

Suppose that $f$ is proper. 
Then $\beta$ is quasi-smooth and $\alpha$ is proper. 
In this case, the functor (\ref{fomega2}) restricts to the functor 
\begin{align}\label{ppush}
    f^{\Omega}_{\ast}:=\alpha_{\ast}\beta^{\ast} \colon 
    \Coh(\Omega_{\X}) \to \Coh(\Omega_{\Y}). 
\end{align}
If $f$ is smooth and proper, the 
functor $(\ref{pback!})$ is the right adjoint 
of (\ref{ppush}).

\subsection{Smooth pull-back}
Let $f \colon \X \to \Y$ be a smooth morphism of 
smooth QCA stacks. 
Let $\nu$ be a map 
\begin{align*}
\nu \colon (\bgm)_{k'} \to f^{\ast}\Omega_{\Y}, \ 0\mapsto z. 
\end{align*}
In the notation of the diagram (\ref{dia:alphabeta}), we set
\begin{align*}\overline{x}=\beta(z) \in \Omega_{\X}(k'), \
\overline{y}=\alpha(z) \in \Omega_{\Y}(k'), \ 
x=\pi_{\X}(\overline{x}) \in \X(k'), \ y=\pi_{\Y}(\overline{y})\in \Y(k').
\end{align*}
Note that $f(x)=y$, and the above 
situation is depicted in the diagram below:
\begin{align}\label{dia:bgm}
    \xymatrix{
& ((\bgm)_{k'}, 0) \ar[ld]_-{\nu_{\overline{x}}} \ar[d]_-{\nu} \ar[rd]^-{\nu_{\overline{y}}} & \\
(\Omega_{\X}, \overline{x}) \ar[d] & (f^{*}\Omega_{\Y}, z) \linclusion_-{\beta} \ar[r]^-{\alpha} \ar[d] \diasquare& (\Omega_{\Y}, \overline{y}) \ar[d] \\
(\X, x) \ar@{=}[r] & (\X, x) \ar[r] & (\Y, y).  
    }
\end{align}

We have the diagram 
\begin{align*}
    \xymatrix{
\Omega_{\X} & f^{\ast}\Omega_{\Y} \linclusion_-{\beta}
 \ar[r]^-{\alpha} & \Omega_{\Y} \\
 \mathfrak{g}_x^{\vee}[-1]/(\mathbb{G}_m)_{k'} \ar[rd]_-{\iota_x} \ar[u]^-{\overline{\nu}_x^{\mathrm{reg}}}_-{(0\mapsto \overline{x})} & & \fg_y^{\vee}[-1]/(\mathbb{G}_m)_{k'} 
 \ar[ld]^-{\iota_y} \ar[u]^-{\overline{\nu}_y^{\mathrm{reg}}}_-{(0\mapsto \overline{y})} \\
 &(\bgm)_{k'}.  &
    }
\end{align*}
Here 
$\overline{\nu}_x^{\mathrm{reg}}$ and $\overline{\nu}_y^{\mathrm{reg}}$ are defined as in the diagram (\ref{dia:nuxreg}). 
Note that $\alpha$ is smooth and $\beta$ is a closed immersion since $f$ is smooth, see Lemma~\ref{lem:beta}.
\begin{lemma}\label{lem:pull}
For $\mathcal{E} \in \Coh(\Omega_{\Y})$, we have 
\begin{align*}
    \wt(\iota_{y\ast}\overline{\nu}_y^{\mathrm{reg}\ast}(\mathcal{E})) \subset 
     \left[\frac{1}{2}c_1(T_{\Y, y}^{<0}-T_{\Y, y}^{>0}),  \frac{1}{2}c_1(T_{\Y, y}^{>0}-T_{\Y, y}^{<0})  \right]+\frac{1}{2}c_1(\mathfrak{g}_y)+c_1(\nu^*\delta)
\end{align*}
if and only if 
\begin{align*}
    \wt(\iota_{x\ast}\overline{\nu}_x^{\mathrm{reg}\ast}\beta_{\ast}\alpha^{!}(\mathcal{E})) \subset 
     \left[\frac{1}{2}c_1(T_{\X, x}^{<0}-T_{\X, x}^{>0}),  \frac{1}{2}c_1(T_{\X, x}^{>0}-T_{\X, x}^{<0})  \right]+\frac{1}{2}c_1(\mathfrak{g}_x)+c_1(\nu^*\delta')
\end{align*}
where $\delta'=f^{\ast}\delta \otimes \omega_{f}^{1/2}$. 
\end{lemma}
\begin{proof}
By the distinguished triangle 
\begin{align*}
    \mathbb{T}_f \to \mathbb{T}_{\X} \to f^{\ast}\mathbb{T}_{\Y}
\end{align*}
and using that $f$ is smooth, there are exact sequences
of $\mathrm{Aut}(x)$-representations
\begin{align}\notag
    &0 \to A \to \mathfrak{g}_x \to B \to 0, \ 
    0 \to B \to \mathfrak{g}_y \to W \to 0, \\
    \label{exact:AB}&0 \to W \to N \to V \to 0, \ 0 \to V \to 
    T_{\X, x} \to T_{\Y, y} \to 0
\end{align}
where $A=\mathcal{H}^{-1}(\mathbb{T}_f|_{x})$
and $N=\mathcal{H}^0(\mathbb{T}_f|_{x})$
and $\mathcal{H}^i(\mathbb{T}_f|_{x})=0$ for $i\neq -1, 0$. 

On the other hand, the map $f^*\Omega_{\Y}\hookrightarrow \Omega_{\X}$ on fibers 
at $x\in \X(k')$ is given by 
$\mu_y^{-1}(0)\to \mu_x^{-1}(0)$, induced by the commutative 
diagram 
\begin{align*}
    \xymatrix{
T_{\Y, y}^{\vee} \ar[r]^-{\mu_y} \ar[d] & \mathfrak{g}_y^{\vee} \ar[d] \\
T_{\X, x}^{\vee} \ar[r]^-{\mu_x}  & \mathfrak{g}_x^{\vee}.  
    }
\end{align*}
Together with the last exact sequence of (\ref{exact:AB}), 
we have the commutative diagram 
\begin{align*}
    \xymatrix{
\Omega_{\X} \diasquare& f^{\ast}\Omega_{\Y} \linclusion_-{\beta}
 \ar[r]^-{\alpha} & \Omega_{\Y} \\
 \mathfrak{g}_x^{\vee}[-1]/(\mathbb{G}_m)_{k'} \ar[rd]_-{\iota_x} \ar[u]^-{\overline{\nu}_x^{\mathrm{reg}}} & 
 (\fg_y \oplus V)^{\vee}[-1]/(\mathbb{G}_m)_{k'} \ar[r]^-{\eta_2} \ar[l]_-{\eta_1}  \ar[u]^-{\overline{\nu}_z^{\mathrm{reg}}}\ar[d]_-{\iota_{z}}& \fg_y^{\vee}[-1]/(\mathbb{G}_m)_{k'} 
 \ar[ld]^-{\iota_y} \ar[u]_-{\overline{\nu}_y^{\mathrm{reg}}} \\
 &(\bgm)_{k'}.  &
    }
\end{align*}
In the above, the top left square is Cartesian, $\overline{\nu}_z^{\mathrm{reg}}$ is defined by the 
above commutative diagram, 
and $\eta_2$ is the projection. 

Let $\nu_x$ be the composition 
$(\bgm)_{k'} \stackrel{\nu_{\overline{x}}}{\to} \Omega_{\X} \to \X$.
For $\mathcal{E} \in \Coh(\Omega_{\Y})$, we have 
the isomorphisms 
\begin{align*}    \iota_{x\ast}\overline{\nu}_x^{\mathrm{reg}\ast}\beta_{\ast}\alpha^{!}\mathcal{E} &\cong \iota_{x\ast}\eta_{1\ast}\overline{\nu}_z^{\mathrm{reg}\ast}\alpha^{!}\mathcal{E} \\
&\cong \iota_{z\ast}\overline{\nu}_z^{\mathrm{reg}\ast}\alpha^{!}\mathcal{E} \\
& \cong \iota_{z\ast}\eta_2^{\ast}\overline{\nu}_y^{\mathrm{reg}\ast}\mathcal{E}\otimes \nu_x^*(\omega_f[\dim f]) \\
&\cong \iota_{y\ast}\eta_{2\ast}\eta_2^{\ast}\overline{\nu}_y^{\mathrm{reg}\ast}\mathcal{E}\otimes \nu_x^*(\omega_f[\dim f]) \\
&\cong \iota_{y\ast}\overline{\nu}_y^{\mathrm{reg}\ast}\mathcal{E}\otimes \nu_x^*(\omega_f[\dim f])\otimes 
\wedge^{\bullet}V. 
\end{align*}
  It follows that 
the weights
$\wt(\iota_{y\ast}\overline{\nu}_y^{\mathrm{reg}\ast}(\mathcal{E}))$ lie in $[a, b]$ if and only if 
\begin{align*}
    \wt(\iota_{x\ast}\overline{\nu}_x^{\mathrm{reg}\ast}(\mathcal{E}))
    \subset [a+c_1(V^{<0}), b+c_1(V^{>0})]+c_1(\omega_f|_{x}). 
\end{align*}
We have the following equality in $K((\bgm)_{k'})$
\begin{align*}
    &\frac{1}{2}T_{\Y, y}^{>0}-\frac{1}{2}T_{\Y, y}^{<0}
    +\frac{1}{2}\mathfrak{g}_y+V^{>0}+\omega_f|_{x}  \\
    &=\frac{1}{2}(T_{\X, x}-V)^{>0}-\frac{1}{2}(T_{\X, x}-V)^{<0}+\frac{1}{2}(\mathfrak{g}_x+W-A)+V^{>0} +\omega_f|_{x} \\
    &=\frac{1}{2}T_{\X, x}^{>0}-\frac{1}{2}T_{\X, x}^{<0}
    +\frac{1}{2}\mathfrak{g}_x+\frac{1}{2}(N-A)+\omega_f|_{x}. 
\end{align*}

Similarly, we have 
\begin{align*}
\frac{1}{2}T_{\Y, y}^{<0}-\frac{1}{2}T_{\Y, y}^{>0}
    +\frac{1}{2}\mathfrak{g}_y+V^{<0} 
    =\frac{1}{2}T_{\X, x}^{<0}-\frac{1}{2}T_{\X, x}^{>0}
    +\frac{1}{2}\mathfrak{g}_x+\frac{1}{2}(N-A)+\omega_f|_{x}. 
    \end{align*}
    Note that $\det (N-A)=\omega_{f}^{-1}|_{x}$. 
    Therefore the lemma holds.

\end{proof}

In particular, we have the following proposition: 
\begin{prop}\label{prop:pback}
For a smooth morphism $f \colon \X \to \Y$, the 
functor (\ref{pback!})
restricts to a functor
\begin{align}\label{funct:pback}f^{\Omega!} \colon 
    \LL(\Omega_{\Y})_{\delta\otimes \omega_{\Y}^{1/2}} \to \LL(\Omega_{\X})_{f^{\ast}\delta \otimes \omega_{\X}^{1/2}}  
\end{align}    
\end{prop}
\begin{proof}
    The proposition is immediate from Lemma~\ref{lem:wtOmega} and Lemma~\ref{lem:pull}. 
\end{proof}
\subsection{Push-forward under closed immersion}\label{subsec:push}
Let $f \colon \X \to \Y$ be a proper morphism of smooth QCA stacks. 
Let $\nu$ be a map 
\begin{align*}
    \nu \colon (\bgm)_{k'} \to f^{\ast}\Omega_{\Y}, \ 0\mapsto z,  
\end{align*}
and use the notation in the diagram (\ref{dia:bgm}). 
First, assume that $f$ is a closed immersion. 
We have the commutative diagram 
\begin{align*}
    \xymatrix{
\Omega_{\X} & f^{\ast}\Omega_{\Y} \ar[l]_-{\beta}
 \inclusion^-{\alpha} & \Omega_{\Y} \\
 \mathfrak{g}_x^{\vee}[-1]/(\mathbb{G}_m)_{k'} \ar[d]_-{\iota_x} \ar[u]^-{\overline{\nu}_x^{\mathrm{reg}}} & & \ar[ll]_-{\eta}^-{\cong} \fg_y^{\vee}[-1]/(\mathbb{G}_m)_{k'} 
 \ar[lld]^-{\iota_y} \ar[u]_-{\overline{\nu}_y^{\mathrm{reg}}} \\
 (\bgm)_{k'}. & &
    }
\end{align*}
Note that $\beta$ is smooth, $\alpha$ is a closed immersion, 
and $\eta$ is an isomorphism as $f$ is a closed immersion.

\begin{lemma}\label{lem:pull2}
Suppose that $f$ is a closed immersion. 
For $\mathcal{E} \in \Coh(\Omega_{\X})$, we have 
\begin{align*}
   \wt( \iota_{x\ast}\overline{\nu}_x^{\mathrm{reg}\ast}(\mathcal{E})) \subset 
    & \left[\frac{1}{2}c_1(T_{\X, x}^{<0}-T_{\X, x}^{>0}),  \frac{1}{2}c_1(T_{\X, x}^{>0}-T_{\X, x}^{<0})  \right]+\frac{1}{2}c_1(\mathfrak{g}_x)+c_1(\nu^*f^*\delta)
\end{align*}
if and only if 
\begin{align*}
   \wt( \iota_{y\ast}\overline{\nu}_y^{\mathrm{reg}\ast}\alpha_{\ast}\beta^{\ast}(\mathcal{E})) \subset 
    & \left[\frac{1}{2}c_1(T_{\Y, y}^{<0}-T_{\Y, y}^{>0}),  \frac{1}{2}c_1(T_{\Y, y}^{>0}-T_{\Y, y}^{<0})  \right] +\frac{1}{2}c_1(\mathfrak{g}_x)+c_1(\nu^*\delta')
\end{align*}
where $\delta'=\delta \otimes \omega_{f}^{-1/2}$. 
\end{lemma}
\begin{proof}
    Since $\beta$ is smooth and $\alpha$ is a closed immersion, the proof is analogous to Lemma~\ref{lem:pull} and follows by a similar weight computation.
\end{proof}

Consequently, we have the following: 
\begin{prop}\label{prop:push}
Suppose that $f$ is a closed immersion. Then the functor (\ref{ppush}) 
restricts to the functor 
    \begin{align}\notag
        f^{\Omega}_{*} \colon 
    \LL(\Omega_{\X})_{f^{\ast}\delta\otimes \omega_{\X}^{1/2}} \to \LL(\Omega_{\Y})_{\delta \otimes \omega_{\Y}^{1/2}}. 
    \end{align}
\end{prop}

\subsection{Projective push-forward}\label{subsec:push2}
We show that the conclusion of Proposition~\ref{prop:push} 
holds for any projective morphism. 
\begin{thm}\label{thm:proj}
Let $f \colon \X \to \Y$ be a projective morphism of smooth QCA stacks. 
Then the functor (\ref{ppush}) 
restricts to the functor 
    \begin{align}\label{funct:Lpush}
        f^{\Omega}_{*} \colon 
    \LL(\Omega_{\X})_{f^{\ast}\delta\otimes \omega_{\X}^{1/2}} \to \LL(\Omega_{\Y})_{\delta \otimes \omega_{\Y}^{1/2}}. 
    \end{align}
    \end{thm}
\begin{proof}
    For $\mathcal{E} \in \LL(\Omega_{\X})_{f^{\ast}\delta \otimes \omega_{\X}^{1/2}}$, 
    we need to show that $f^{\Omega}_{\ast}\mathcal{E} \in \Coh(\Omega_{\Y})$ 
    satisfies the condition in Lemma~\ref{lem:wtOmega}.
    Let $\nu \colon (\bgm)_{k'} \to \Omega_{\Y}$ be a map 
    with composition
    \begin{align*}
      \nu_y \colon (\bgm)_{k'} \to \Omega_{\Y} \stackrel{\pi_{\Y}}{\to} \Y, \ 
      0 \mapsto y. 
    \end{align*}
    We may assume that the induced cocharacter $(\mathbb{G}_m)_{k'} \to \mathrm{Aut}(y)$ is injective. As in Remark~\ref{rmk:k'}, we may assume that $k'$ is 
    algebraically closed. Then by replacing $\Y$ with $\Y_{k'}$, we may further assume that $k'=k$. 
    
    Then by~\cite[Theorem~1.1]{AHRLuna}, there is an affine $k$-scheme $Y=\Spec A$ with 
    $(\mathbb{G}_m)$-action, 
    a $\mathbb{G}_m$-fixed point $0 \in Y$ and a smooth morphism 
    $g \colon Y/\mathbb{G}_m \to \Y$ such that $g|_{0/\mathbb{G}_m}=\nu_y$. 
We have the induced functor 
\begin{align*}
    g^{\Omega!} \colon \Coh(\Y) \to \Coh(Y/\mathbb{G}_m). 
\end{align*}
    By Lemma~\ref{lem:pull}, we have 
    $f^{\Omega}_{\ast}(\mathcal{E}) \in \LL(\Omega_{\Y})_{\delta \otimes \omega_{\Y}^{1/2}}$ if 
    and only if 
    \begin{align}\label{gf1}
    g^{\Omega!}f^{\Omega}_{\ast}(\mathcal{E})
    \in \LL(Y/\mathbb{G}_m)_{g^{\ast}\delta \otimes g^{\ast}\omega_{\Y}^{1/2} \otimes \omega_g^{1/2}}
    = \LL(Y/\mathbb{G}_m)_{g^*\delta\otimes\omega_{Y/\mathbb{G}_m}^{1/2}}
    \end{align}
    for any map $\nu$ as above. We have the Cartesian square 
    \begin{align*}
        \xymatrix{
X/\mathbb{G}_m \ar[r]^-{g'} \ar[d]_-{f'} \diasquare& \X \ar[d]^-{f} \\
Y/\mathbb{G}_m \ar[r]^-{g} & \Y        
        }
    \end{align*}
    where $X=\X\times_{\Y}Y$. Since $f$ is projective, 
    we have that $X$ is a smooth $k$-scheme which is projective over $Y$. 
    By Lemma~\ref{lem:funct2}, we have the isomorphism
    \begin{align}\label{gf2}
       g^{\Omega!}f^{\Omega}_{\ast}(\mathcal{E})
       \cong (f')^{\Omega}_{\ast} (g')^{\Omega!}(\mathcal{E}). 
    \end{align}
    By Proposition~\ref{prop:pback}, we have 
    \begin{align}\label{gf3}
        (g')^{\Omega!}(\mathcal{E}) \in \LL(\Omega_{X/\mathbb{G}_m})_{g^{'\ast}\omega_{\X}^{1/2}\otimes g^{'\ast}f^*\delta\otimes\omega_{g'}^{1/2}}=\LL(\Omega_{X/\mathbb{G}_m})_{f^{'\ast}g^*\delta\otimes\omega_{X/\mathbb{G}_m}^{1/2}}.
    \end{align}
By (\ref{gf2}) and (\ref{gf3}), it is enough to show that 
the object (\ref{gf3}) applied to $(f')^{\Omega}_{*}$ satisfies (\ref{gf1}). 
    Therefore, we can replace 
$(\X, \Y)$ with $(X/\mathbb{G}_m, Y/\mathbb{G}_m)$. 

By~\cite[Theorem~5.2.1]{Brion0}, there 
is an $f'$-ample $\mathbb{G}_m$-linearized 
line bundle on $X$. By shrinking $Y$ if necessary, 
we obtain the commutative diagram 
\begin{align*}
\xymatrix{
X/\mathbb{G}_m \inclusion^-{i} \ar[d]_-{f'}& \mathbb{P}_Y^n /\mathbb{G}_m\ar[ld]^-{f''} 
\\
Y/\mathbb{G}_m &
}
\end{align*}
where $i$ is a closed immersion, 
$\mathbb{P}_Y^n=\mathbb{P}^n \times Y$,
and $\mathbb{G}_m$ acts on $\mathbb{P}^n$ linearly. 
By Lemma~\ref{lem:funct}, 
we have $(f'')^{\Omega}_{\ast}i^{\Omega}_{\ast}=(f')^{\Omega}_{\ast}$. 
Also by Proposition~\ref{prop:push} and (\ref{gf3}), we have 
\begin{align*}
    i^{\Omega}_{\ast}(g')^{\Omega!}(\mathcal{E}) \in     \LL(\Omega_{\mathbb{P}_Y^n/\mathbb{G}_m})_{ f^{''\ast}g^*\delta \otimes \omega_{\mathbb{P}_Y^n/\mathbb{G}_m}^{1/2}}. 
\end{align*}
Therefore we can replace $(X, f')$ with $(\mathbb{P}_Y^n, f'')$, and it is enough 
to prove the proposition for the object 
$\mathcal{E}'=i^{\Omega}_{\ast}(g')^{\Omega!}(\mathcal{E})$
with respect to $f''$. 

We have the following commutative diagram 
\begin{align*}
    \xymatrix{
\Omega_{\mathbb{P}_Y^n/\mathbb{G}_m} \dinclusion_-{j_1}\diasquare
& \linclusion_-{\beta_1}
f^{''\ast}\Omega_{Y/\mathbb{G}_m} \ar[r]^-{\alpha_1} \dinclusion_-{j_2}\diasquare
&
\Omega_{Y/\mathbb{G}_m} \dinclusion_-{j_3} \\
\Omega_{\mathbb{P}_Y^n}/\mathbb{G}_m &(\mathbb{P}^n \times \Omega_Y)/\mathbb{G}_m \ar[r]^-{\alpha_2} \linclusion_-{\beta_2} & \Omega_Y/\mathbb{G}_m  
    }
\end{align*}
where the vertical arrows are closed immersions, 
each square is Cartesian, $\beta_2$ is the product with the zero section 
$\mathbb{P}^n \hookrightarrow \Omega_{\mathbb{P}^n}$ and $\alpha_2$ is the projection. 
By the base change, we have 
\begin{align*}
    j_{3\ast}\alpha_{1\ast}\beta_1^{\ast}\mathcal{E}'
    \cong \alpha_{2\ast}j_{2\ast}\beta_1^{\ast}\mathcal{E}' 
    \cong \alpha_{2\ast}\beta_2^{\ast}j_{1\ast}\mathcal{E}'. 
\end{align*}
By Lemma~\ref{lem:char2}, we need to show 
the following: for any $\mathbb{G}_m$-fixed
point $\overline{q} \in \Omega_{Y}$ with $\pi_{Y}(\overline{q})=q$, 
where $\pi_Y \colon \Omega_Y \to Y$ is the projection, 
we have 
\begin{align}\label{wtcond:m}    \wt((\alpha_{2\ast}\beta_2^{\ast}j_{1\ast}\mathcal{E}')|_{\overline{q}}) \subset [m_1, m_2]+c_1((g^*\delta)|_{q}) 
\end{align}
where $m_i$ are given by 
\begin{align}\label{def:m}
m_1=c_1 (\Omega_Y|_{q}^{<0}), \ m_2=c_1 (\Omega_Y|_{q}^{>0}). 
\end{align}

Note that 
\begin{align*}    (\alpha_{2\ast}\beta_2^{\ast}j_{1\ast}\mathcal{E}')|_{\overline{q}}\cong     \Gamma(\mathbb{P}^n, (\beta_2^{\ast}j_{1\ast}\mathcal{E}')|_{\mathbb{P}^n\times \overline{q}}).
\end{align*}
By Lemma~\ref{lem:char2}, the object 
$\beta_2^{\ast}j_{1\ast}\mathcal{E}'$ satisfies 
the following: 
for any $\mathbb{G}_m$-fixed point 
$(p, \overline{q}) \in \mathbb{P}^n \times \Omega_Y$
with $\pi_{\Y}(\overline{q})=q$, 
we have 
\begin{align}\label{wtcond:m2}
    \wt((\beta_2^{\ast}j_{1\ast}\mathcal{E})|_{(p, \overline{q})}) \subset 
     \left[c_1(\Omega_{\mathbb{P}^n}|_{p}^{<0}), c_1(\Omega_{\mathbb{P}^n}|_{p}^{>0}) \right]
     +[m_1, m_2]
     +c_1((g^*\delta)|_{q})
\end{align}
where $m_i$ are given by (\ref{def:m}).
Then the condition (\ref{wtcond:m}) follows from (\ref{wtcond:m2}) and Lemma~\ref{lem:wtbound} below. 
\end{proof}

In the proof above, we used Lemma~\ref{lem:wtbound}. The following lemma will be used 
in the proof of Lemma~\ref{lem:wtbound}. 
\begin{lemma}\label{lem:Pn}
Let $\mathbb{G}_m$ act on $\mathbb{P}^n$
linearly.
 Let $f \colon \mathbb{P}^n/\mathbb{G}_m \to \bgm$ be the structure morphism. 
Let $\mathcal{E} \in \Coh(\mathbb{P}^n/\mathbb{G}_m)$. 
Suppose that, for any $p \in (\mathbb{P}^n)^{\mathbb{G}_m}$, 
we have 
\begin{align}\label{wt:condE}
\wt^{\mathrm{max}}(\mathcal{E}|_{p}) \leq
 c_1(\Omega_{\mathbb{P}^n}|_p^{>0}).    
\end{align}
Then, for any $a>0$, we have 
\begin{align}\notag
\Hom(f^*\mathcal{O}_{\bgm}(a), \mathcal{E})
=0. 
\end{align}
\end{lemma}
\begin{proof}
    We prove the lemma by induction on $n$. 
    The case of $n=0$ is obvious. 

Let $[x_0 : \cdots : x_n]$ be a homogeneous coordinate 
of $\mathbb{P}^n$. We may assume that there is a 
partition 
\begin{align*}
\{0, 1, \ldots, n\}=\coprod_{l=1}^k I_l
    \end{align*}
    and $\lambda_1>\cdots>\lambda_k$ such that
    the $\mathbb{G}_m$-action on $\mathbb{P}^n$ is given by 
    \begin{align*}
        t[x_0: \cdots : x_n]=[t^{a_0}x_0: \cdots : t^{a_n}x_n]
    \end{align*}
    where $a_i=\lambda_l$ for $i\in I_l$. 
    Let 
    $Z \subset \mathbb{P}^n$ be the closed subscheme 
    given by $x_l=0$ for all $l\in I_k$, and let
    \begin{align*}
j \colon U:= \mathbb{P}^n \setminus Z
\subset \mathbb{P}^n        
    \end{align*}
    be its complement. 
We have the distinguished 
    triangle in $\QCoh(\mathbb{P}^n)$
    \begin{align}\notag
\Gamma_Z(\mathcal{E}) \to \mathcal{E} \to j_{\ast}j^{\ast}\mathcal{E}. 
            \end{align}
            It is enough to show that 
            \begin{align}\label{show:vanish}
                \Hom(f^{\ast}\mathcal{O}_{\bgm}(a), j_{\ast}j^{\ast}\mathcal{E})=\Hom(f^{\ast}\mathcal{O}_{\bgm}(a), \Gamma_Z(\mathcal{E}))=0. 
            \end{align}

By setting $N=\lvert I_k \rvert -1$, 
we have the map 
\begin{align}\label{map:piU}
\pi \colon U \to \mathbb{P}^{N}    
\end{align}
by taking the projection onto homogeneous 
coordinates $x_j$ with $j\in I_k$. 
The morphism $\pi$ is a $\mathbb{G}_m$-equivariant affine space fibration 
with zero section $0$, where $\mathbb{G}_m$ acts on 
$\mathbb{P}^{N}$ trivially. 
The weights of $\mathbb{G}_m$-action on 
fibers of $\pi$ are strictly positive, hence 
we have the semiorthogonal decomposition, see~\cite[Amplification~3.18]{halp}
\begin{align}\label{sod:pi}
&\Coh(U/\mathbb{G}_m)= \\
\notag &\langle 
\ldots, \pi^{\ast}\Coh(\mathbb{P}^N/\mathbb{G}_m)_{-1}, 
\pi^{\ast}\Coh(\mathbb{P}^N/\mathbb{G}_m)_0, 
\pi^{\ast}\Coh(\mathbb{P}^N/\mathbb{G}_m)_1
\ldots \rangle.  
\end{align}
By the assumption (\ref{wt:condE}) and 
the non-negativity of $\mathbb{G}_m$-weights on $U$, 
we have 
\begin{align*}
\wt^{\mathrm{max}}(0^{\ast}j^{\ast}\mathcal{E}) \leq 
c_1(\Omega_{\mathbb{P}^n}|_{\Im 0}^{>0})=0. 
\end{align*} Here, $\Im 0 \cong \mathbb{P}^N$ is the image of the zero section of (\ref{map:piU}).
    It follows that, from the semiorthogonal 
    decomposition (\ref{sod:pi}) that 
\begin{align}\label{j*E}
j^{\ast}\mathcal{E} \in \langle \pi^{\ast}\Coh(\mathbb{P}^N/\mathbb{G}_m)_w :
w\leq 0\rangle. 
\end{align}
Since $j^{\ast}f^{\ast}\mathcal{O}_{\bgm}(a)$
lies in $\pi^{\ast}\Coh(\mathbb{P}^N)_a$, for $a>0$ 
we have 
\begin{align}\notag
    \Hom(f^{\ast}\mathcal{O}_{\bgm}(a), j_{\ast}j^{\ast}\mathcal{E})
    =\Hom(j^{\ast}f^{\ast}\mathcal{O}_{\bgm}(a), j^{\ast}\mathcal{E})
    =0
\end{align}
by (\ref{j*E}) and the semiorthogonal decomposition (\ref{sod:pi}). 
Therefore we have the first vanishing of (\ref{show:vanish}). 

Below we show the second vanishing of (\ref{show:vanish}). 
We have 
\begin{align*}
    \Gamma_Z(\mathcal{E})=\colim_{Z \subset Z'} \mathcal{H}om_{\mathbb{P}^n}(\mathcal{O}_{Z'}, \mathcal{E})
\end{align*}
    where the colimit is after all closed 
    subschemes $Z'\subset \mathbb{P}^n$ with 
    $\mathrm{Supp}(Z')=\mathrm{Supp}(Z)$. 
Since $f^{\ast}\mathcal{O}_{\bgm}(a)$ is compact in 
$\QCoh(\mathbb{P}^n/\mathbb{G}_m)$, we have 
\begin{align}\label{isom:colim}
    &\Hom(f^{\ast}\mathcal{O}_{\bgm}(a), \Gamma_Z(\mathcal{E})) \\
  \notag  &\cong \colim_{Z\subset Z'} \Hom(f^{\ast}\mathcal{O}_{\bgm}(a), \mathcal{H}om_{\mathbb{P}^n}(\mathcal{O}_{Z'}, \mathcal{E})).
\end{align}
        Let $H_j=(x_j=0) \subset \mathbb{P}^n$ be the hyperplane, 
        and denote by $a_j H_j \subset \mathbb{P}^n$ the closed subscheme 
        given by $(x_j^{a_j}=0)$. 
    Any closed subscheme $Z'\subset \mathbb{P}^n$ with $\mathrm{Supp}(Z)=\mathrm{Supp}(Z')$ is contained in 
    \begin{align*}
        W=\bigcap_{j \in I_k}a_j H_j
    \end{align*}
    for some $a_j \geq 1$. In order to show the vanishing of (\ref{isom:colim}), it is enough to show 
    that 
    \begin{align*}
        \Hom(f^{\ast}\mathcal{O}_{\bgm}(a), \mathcal{H}om(\mathcal{O}_W, \mathcal{E}))=0.
    \end{align*}
    By taking the Koszul resolution of $\mathcal{O}_W$, 
    we have 
    \begin{align}\label{isom:OW}
        \mathcal{H}om_{\mathbb{P}^n}(\mathcal{O}_W, \mathcal{E})
        \cong \mathcal{E}|_{W} \otimes \mathcal{O}_W\left(\sum_{j\in I_k}a_j H_j \right)[-\lvert I_k\rvert].
    \end{align}
    
    Moreover from the exact sequences
    \begin{align*}
        0 \to \mathcal{O}_{H_j}(-mH_j) \to \mathcal{O}_{(m+1)H_j} \to \mathcal{O}_{mH_j} \to 0
    \end{align*}
    the object (\ref{isom:OW}) is filtered by objects 
    of the form 
    \begin{align*}
    \mathcal{F}:=
        \mathcal{E}|_{Z} \otimes \mathcal{O}_Z\left(\sum_{j\in I_k}c_j H_j \right)[-\lvert I_k\rvert]
    \end{align*}
    for some $c_j \geq 1$. 
    We claim that, for each $p\in Z^{\mathbb{G}_m}$, we
    have 
    \begin{align}\label{wt:p}
\wt^{\mathrm{max}}(\mathcal{F}|_{p}) \leq c_1(\Omega_Z|_{p}^{>0}). 
    \end{align}
    
    A point $p \in Z^{\mathbb{G}_m}$ satisfies 
    the following: there is $1\leq l \leq k-1$ such that 
    $x_i=0$ for $i\notin I_l$. 
    Then 
    \begin{align*}
        c_1(\Omega_{\mathbb{P}^n}|_p^{>0})=
        \lvert I_k \rvert (\lambda_l-\lambda_k)+
         \lvert I_{k-1} \rvert (\lambda_l-\lambda_{k-1})+
        \cdots+
        \lvert I_{l+1} \rvert (\lambda_l-\lambda_{l+1}). 
    \end{align*}
    By the assumption (\ref{wt:condE}) and the $\mathbb{G}_m$-weight of $\mathcal{O}_Z(H_j)|_p$ for $j\in I_k$ is 
    $(\lambda_k-\lambda_l)$, 
    we have 
    \begin{align*}
        \wt^{\mathrm{max}}(\mathcal{F}|_p)
        &\leq \lvert I_k\rvert(\lambda_l-\lambda_k)
        +\cdots+\lvert I_{l+1}\rvert (\lambda_l-\lambda_{l+1})
        +\left(\sum_{j\in I_k}c_j\right)(\lambda_k-\lambda_l) \\
        &\leq \lvert I_{k-1}\rvert(\lambda_l-\lambda_{k-1})
        +\cdots+\lvert I_{l+1}\rvert (\lambda_l-\lambda_{l+1}) =c_1(\Omega_{Z}|^{>0}_{p}).        
    \end{align*}
    Therefore (\ref{wt:p}) holds. 
    
    Since $Z$ is a projective space with 
    $\dim Z<n$, by the 
    assumption of the induction we have 
    \begin{align*}\Hom_{Z/\mathbb{G}_m}(f^{\ast}\mathcal{O}_{\bgm}(a)|_{Z}, 
    \mathcal{F})=0
    \end{align*}
    for $a>0$. Hence the second vanishing of (\ref{show:vanish}) holds.
\end{proof}
\begin{lemma}\label{lem:wtbound}
Let $\mathbb{G}_m$ act on $\mathbb{P}^n$ linearly. 
For $\mathcal{E} \in \Coh(\mathbb{P}^n/\mathbb{G}_m)$, 
suppose that for any $\mathbb{G}_m$-fixed point 
$p \in \mathbb{P}^n$ we have 
\begin{align*}
    \wt(\mathcal{E}|_{p}) \subset \left[c_1 (\Omega_{\mathbb{P}^n}|_{p}^{<0}), 
    c_1 (\Omega_{\mathbb{P}^n}|_{p}^{>0})
    \right]+[m_1, m_2]
\end{align*}
for $m_1 \leq m_2$. 
Then we have 
\begin{align*}
    \wt(\Gamma(\mathbb{P}^n, \mathcal{E}))
    \subset [m_1, m_2]. 
\end{align*}
\end{lemma}
\begin{proof}
    Let $f \colon \mathbb{P}^n/\mathbb{G}_m \to \bgm$. We take $k>m_2$. 
    Then by Lemma~\ref{lem:Pn}, we have
\begin{align*}
    \Hom(f^{\ast}\mathcal{O}_{\bgm}(k), 
    \mathcal{E}) =\Hom(\mathcal{O}_{\bgm}(k), f_{\ast}\mathcal{E})=0. 
\end{align*}
    Therefore we have 
    \begin{align}\label{f*E}
        f_{\ast}\mathcal{E} \in \Coh(\bgm)_{\leq m_2}. 
    \end{align}
    
    We also have the condition 
  \begin{align*}
    \wt(\mathcal{E}^{\vee} \otimes \omega_{\mathbb{P}^N}|_{p}) \subset \left[c_1 (\Omega_{\mathbb{P}^n}|_{p}^{<0}), 
    c_1 (\Omega_{\mathbb{P}^n}|_{p}^{>0})
    \right]+[-m_2, -m_1].
\end{align*}  
By applying (\ref{f*E}) for $\mathcal{E}^{\vee}\otimes \omega_{\mathbb{P}^n}[n]$, we
have that 
\begin{align*}
    f_{\ast}(\mathcal{E}^{\vee}\otimes \omega_{\mathbb{P}^n}[n])
    \cong f_{\ast}(\mathcal{E})^{\vee} \in 
    \Coh(\bgm)_{\leq -m_1}
\end{align*}
where the first isomorphism is Serre duality. 
Therefore we have 
\begin{align*}
    f_{\ast}\mathcal{E} \in \Coh(\bgm)_{\geq m_1}
\end{align*}
and the lemma holds. 
\end{proof}

As a corollary of Proposition~\ref{prop:pback} and Theorem~\ref{thm:proj}, we have 
the following: 
\begin{cor}\label{cor:adjoint}
Suppose that $f \colon \X \to \Y$ is a
smooth projective morphism of smooth QCA stacks. We have 
the adjoint pair of functors 
\begin{align*}
    &f^{\Omega!} \colon \LL(\Omega_{\Y})_{\delta\otimes \omega_{\Y}^{1/2}} \to \LL(\Omega_{\X})_{f^{\ast}\delta \otimes \omega_{\X}^{1/2}}, \\
    &f^{\Omega}_{*} \colon \LL(\Omega_{\X})_{f^{\ast}\delta\otimes \omega_{\X}^{1/2}} \to \LL(\Omega_{\Y})_{\delta \otimes \omega_{\Y}^{1/2}}
\end{align*}
where $f^{\Omega!}$ is the right adjoint functor of 
$f^{\Omega}_{*}$. 
    \end{cor}

\subsection{Functors between limit categories with nilpotent singular supports}
We show that the functors in Proposition~\ref{prop:pback} and Theorem~\ref{thm:proj} restrict to the 
subcategories with nilpotent singular supports. 
These may be interpreted as classical limits of the functors in (\ref{dmod:funct}). 

For a morphism $f\colon \X \to \Y$ of smooth stacks, 
we consider the morphism 
\begin{align*}
    (\alpha, \beta) \colon f^{*}\Omega_{\Y} \to \Omega_{\Y}\times \Omega_{\X}
\end{align*}
where $\alpha$, $\beta$ are morphisms in the diagram (\ref{dia:alphabeta}). 
We consider its $(-2)$-shifted conormal stack 
\begin{align}\label{def:-2}
    \Omega_{(\alpha, \beta)}[-2]=\Spec \mathrm{Sym}(\mathbb{T}_{(\alpha, \beta)}[2]) \to f^*\Omega_{\Y}.
\end{align}
From the distinguished triangle 
\begin{align*}
    (\alpha, \beta)^*(\mathbb{L}_{\Omega_{\Y}\times \Omega_{\X}}) \to  \mathbb{L}_{f^{*}\Omega_{\Y}} \to \mathbb{L}_{(\alpha, \beta)}
\end{align*}
we have the induced diagram 
\begin{align*}
    \Omega_{\Omega_{\X}}[-1] \stackrel{v}{\leftarrow} \Omega_{(\alpha, \beta)}[-2] \stackrel{u}{\to} 
    \Omega_{\Omega_{\Y}}[-1].
\end{align*}
The following proposition is proved in~\cite{T2} (and relies on results in~\cite{AG}):
\begin{prop}\emph{(\cite[Proposition~2.4]{T2})}\label{prop:singss}

(i) Assume that $\alpha$ is quasi-smooth and $\beta$ is proper. 
Then for $\mathcal{E} \in \Coh(\Omega_{\Y})$ we have 
\begin{align*}
    \mathrm{Supp}^{\mathrm{AG}}(\beta_{*}\alpha^! \mathcal{E}) \subset vu^{-1}(\mathrm{Supp}^{\mathrm{AG}}(\mathcal{E}))
    \subset \Omega_{\Omega_{\X}}[-1].
\end{align*}

(ii) Assume that $\beta$ is quasi-smooth and $\alpha$ is proper. 
Then for $\mathcal{E} \in \Coh(\Omega_{\X})$ we have 
\begin{align*}
    \mathrm{Supp}^{\mathrm{AG}}(\alpha_{*}\beta^* \mathcal{E}) \subset uv^{-1}(\mathrm{Supp}^{\mathrm{AG}}(\mathcal{E}))
    \subset \Omega_{\Omega_{\Y}}[-1].
\end{align*}    
\end{prop}

For $z \in f^{*}\Omega_{\Y}(k')$, let $\overline{x} \in \Omega_{\X}(k')$, $\overline{y} \in \Omega_{\Y}(k')$, 
$x\in \X(k')$, and $y\in \Y(k')$ be as in the diagram (\ref{dia:bgm}).
We have the following lemma: 
\begin{lemma}\label{fib:F}
The classical truncation $F_z$ of the fiber of the projection (\ref{def:-2}) at $z$ is given by the kernel 
of a map $\fg_{x} \to T_{\Y, y}^{\vee}$. 
\end{lemma}
\begin{proof}
Let $\pi \colon f^* \Omega_{\Y} \to \X$ be the projection. 
Let $\gamma \colon f^* \mathbb{L}_{\Y} \to \mathbb{L}_{\X}$ be induced by $f$. 
We have the following commutative diagram 
\begin{align*}
    \xymatrix{
\pi^* \mathbb{L}_{\X} \oplus \pi^{*}f^* \mathbb{L}_{\Y} \ar[r] \ar[d]_-{\id+\pi^* \gamma} & 
(\alpha, \beta)^* \mathbb{L}_{\Omega_{\X}\times \Omega_{\Y}} \ar[r] \ar[d] & \pi^* \mathbb{T}_{\X}\oplus \pi^* f^* \mathbb{T}_{\Y} \ar[d]^-{\id+\pi^* \gamma^{\vee}} \\
\pi^* \mathbb{L}_{\X} \ar[r] & \mathbb{L}_{f^{*}\Omega_{\Y}} \ar[r] & \pi^* f^* \mathbb{T}_{\Y}.
    }
\end{align*}
    Here each horizontal sequence is a distinguished triangle. Considering the cones of vertical maps, we obtain 
    the distinguished triangle 
    \begin{align*}
        \pi^* f^* \mathbb{L}_{\Y} \to \mathbb{L}_{(\alpha, \beta)}[-1] \to \pi^* \mathbb{T}_{\X}. 
    \end{align*}
    Therefore we obtain the exact sequence 
    \begin{align*}
        0\to \mathcal{H}^{-2}(\mathbb{L}_{(\alpha, \beta)}|_{z}) \to \fg_x \to T_{\Y, y}^{\vee}
    \end{align*}
    which proves the lemma. 
\end{proof}

In order to give a class of morphisms whose pull-backs preserve nilpotent singular 
supports, we introduce the following definition: 
\begin{defn}\label{def:nileff}
We say that a morphism $f\colon \X \to \Y$ of stacks is \textit{nil-effective} if 
for any $x\in \X(k')$ with $y=f(x)\in \Y(k')$, the induced morphism 
$\phi \colon \fg_x \to \fg_y$ satisfies 
\begin{align*}
    \mathrm{Nilp}(\fg_x)=\phi^{-1}(\mathrm{Nilp}(\fg_y)).
\end{align*}
\end{defn}

\begin{example}
    (i) A schematic morphism is nil-effective because $\fg_x \to \fg_y$ is injective. 

    (ii) Let $G$ be a reductive group with an action of an affine scheme $Y$. 
    For a cocharacter $\lambda\colon \mathbb{G}_m \to G$, the morphism 
    $Y^{\lambda \geq 0}/G^{\lambda \geq 0} \to Y^{\lambda}/G^{\lambda}$ is nil-effective. 
\end{example}

As in (\ref{nilp:N}), let 
$\mathcal{N}_{\X} \subset \Omega_{\Omega_{\X}}[-1]$
be the closed substack consisting of nilpotent elements in the fiber of $\Omega_{\Omega_{\X}}[-1]$. We have the following lemma: 
\begin{lemma}\label{lem:nilpotent}
(i) In Proposition~\ref{prop:singss} (i), assume in addition that $f$ (hence $\alpha$) is nil-effective. 
Then we have 
\begin{align*}
    vu^{-1}(\mathcal{N}_{\Y}) \subset \mathcal{N}_{\X} \subset \Omega_{\Omega_{\X}}[-1].
\end{align*}
  (ii) In Proposition~\ref{prop:singss} (ii), we have 
\begin{align*}
    uv^{-1}(\mathcal{N}_{\X}) \subset \mathcal{N}_{\Y} \subset \Omega_{\Omega_{\X}}[-1].
\end{align*}  
\end{lemma}
\begin{proof}
For $z\in f^*\Omega_{\Y}(k')$, let $F_z$ be as in Lemma~\ref{fib:F}. 
    The map 
    \begin{align*}
        F_z \stackrel{v|_{F_z}}{\to} \mathfrak{g}_{\overline{x}} \subset \fg_x
    \end{align*}
    is given by the natural inclusion $F_z \subset \fg_x$ by Lemma~\ref{fib:F}. 
    In particular, it is injective. 
    
    The map 
     \begin{align*}
        F_z \stackrel{u|_{F_z}}{\to} \mathfrak{g}_{\overline{y}} \subset \fg_y
    \end{align*}
    is given by the map $\fg_x \to \fg_y$ induced by $f$ restricted to $F_z$.  Therefore, if $f$ is nil-effective, we have 
    \begin{align*}
        v|_{F_z}u|_{F_z}^{-1}(\mathrm{Nilp}(\fg_{\overline{y}})) \subset \mathrm{Nilp}(\fg_{\overline{x}}).
    \end{align*}
    Therefore (i) holds. The proof of (ii) is similar. 
\end{proof}

By combining Proposition~\ref{prop:singss} and Lemma~\ref{lem:nilpotent},
we obtain the following corollary: 
\begin{cor}\label{cor:nilp2}
(i) In Proposition~\ref{prop:pback}, assume in addition that $f$ (hence $\alpha$) is nil-effective. 
Then the functor (\ref{funct:pback}) restricts to a functor  
    \begin{align}\notag f^{\Omega!} \colon 
    \LL_{\mathcal{N}}(\Omega_{\Y})_{\delta\otimes \omega_{\Y}^{1/2}} \to \LL_{\mathcal{N}}(\Omega_{\X})_{f^{\ast}\delta \otimes \omega_{\X}^{1/2}}  
\end{align} 

(ii) In Theorem~\ref{thm:proj}, the functor (\ref{funct:Lpush}) restricts to a functor  
     \begin{align}\notag
        f^{\Omega}_{*} \colon 
    \LL_{\mathcal{N}}(\Omega_{\X})_{f^{\ast}\delta\otimes \omega_{\X}^{1/2}} \to \LL_{\mathcal{N}}(\Omega_{\Y})_{\delta \otimes \omega_{\Y}^{1/2}}. 
    \end{align}
\end{cor}

\section{Semiorthogonal decomposition of the magic category}\label{sec:magic}
In this section, we prove a result similar to Theorem~\ref{conj:intro} for magic categories for smooth stacks, namely 
construct their semiorthogonal decompositions associated with $\Theta$-stratifications. 
The magic category is a toy version of limit categories, and generalizes non-commutative 
crepant resolutions~\cite{SvdB} and magic window subcategories~\cite{hls}.
The semiorthogonal decomposition constructed in this section are used to prove a local model 
of the semiorthogonal decomposition for the limit category for moduli stacks of Higgs bundles, which is one of the steps in the proof of Theorem~\ref{thm:mainLG} (discussed in Subsection~\ref{subsec:proofofthm}).

\subsection{\texorpdfstring{Windows for a $\Theta$-stratum}{Windows for a Theta-stratum}}\label{subsec:thetast}
Let $\mathcal{Y}$ be a smooth QCA stack.  
Let \[\Theta:=\mathbb{A}^1/\mathbb{G}_m,\] where 
$\mathbb{G}_m$ acts on $\mathbb{A}^1$ by weight one. 
Following~\cite{halpinstab}, a \textit{$\Theta$-stratum} is defined to be 
an open and 
closed substack 
$\mathcal{S} \subset \mathrm{Map}(\Theta, \mathcal{Y})$
such that the composition 
\begin{align*}
    \mathcal{S} \subset \mathrm{Map}(\Theta, \mathcal{Y}) 
   \stackrel{\mathrm{ev_1}}{\to} 
    \mathcal{Y}
\end{align*}
is a closed immersion.  
Here, $\mathrm{ev}_1$ is the evaluation morphism at 
$1 \in \mathbb{A}^1/\mathbb{G}_m$. 
For a $\Theta$-stratum $\mathcal{S}$, we use the 
same notation $\mathcal{S} \subset \mathcal{Y}$ 
for the image $\mathrm{ev}_1(\mathcal{S})$. 

The \textit{center} of a $\Theta$-stratum $\mathcal{S}$ is 
the unique open and closed substack $\mathcal{Z}\subset \mathrm{Map}(\bgm, \mathcal{Y})$
such that the maps 
\begin{align*}\bgm \hookrightarrow \mathbb{A}^1/\mathbb{G}_m \twoheadrightarrow \bgm
\end{align*}
induces maps $\mathcal{Z}\to\mathcal{S}\to\mathcal{Z}$. 
We have the following diagram 
\begin{align}\label{dia:SYZ}
	\xymatrix{\mathcal{S}\ar[d]_-{q} \inclusion^-{p} & \mathcal{Y} \\
	\mathcal{Z}\ar[ur]_-{r} \ar@/^15pt/[u] &	
}
	\end{align}
where $p$ is the evaluation at $1\in \mathbb{A}/\mathbb{G}_m$ and $q$ is the evaluation at $0\in \mathbb{A}^1/\mathbb{G}_m$. 
The stacks $\mathcal{S}$ and $\mathcal{Z}$ are smooth, see Lemma~\ref{em:sym} below. 

There is an orthogonal weight decomposition
	\begin{align}\label{decom:Z}
		\Coh(\mathcal{Z})=\bigoplus_{i\in \mathbb{Z}}\Coh(\mathcal{Z})_i
		\end{align}
	with respect to the canonical $\bgm$-action on $\mathcal{Z}$. 
We denote by $\Upsilon$ the functor 
	\begin{align}\label{funct:up}
		\Upsilon :=p_{\ast}q^{\ast} \colon \Coh(\mathcal{Z}) \to \Coh(\mathcal{Y}). 
		\end{align}
For $A \in \Coh(\mathcal{Z})$, we denote by $\wt(A)\subset \mathbb{Z}$ the set of non-zero $\mathbb{G}_m$-weights under the decomposition (\ref{decom:Z}).

For each $k\in \mathbb{R}$, we define 
\begin{align*}
\mathcal{W}(k) \subset \Coh(\Y)
\end{align*}
to be the full subcategory of objects $\mathcal{E}$ such that the set of weights 
of $\mathcal{E}|_{\mathcal{Z}}$ satisfies 
\begin{align*}
    \wt(\mathcal{E}|_{\mathcal{Z}}) \subset 
    [k, k+n), \mbox{ where } \ n:=\wt \det N_{\mathcal{S}/\mathcal{Y}}^{\vee}|_{\mathcal{Z}}.
\end{align*}
We will use the following version of window theorem~\cite{MR3895631, halp, HalpK32}
(see~\cite[Theorem~3.3.1]{HalpK32} for example):
\begin{thm}\emph{(\cite{MR3895631, halp, HalpK32})}\label{thm:window}
Let $k_0=\lfloor k \rfloor \in \mathbb{Z}$ be the round-down of $k$. 
The functor (\ref{funct:up}) is fully-faithful on each $\Coh(\mathcal{Z})_i$, 
and there is a semiorthogonal decomposition 
    \begin{align*}
    &\Coh(\Y)=\\
    &\langle\ldots, \Upsilon \Coh(\mathcal{Z})_{k_0-2}, \Upsilon \Coh(\mathcal{Z})_{k_0-1}, 
    \mathcal{W}(k), \Upsilon \Coh(\mathcal{Z})_{k_0}, \Upsilon \Coh(\mathcal{Z})_{k_0+1}, \ldots\rangle.
    \end{align*}
    Moreover, the following composition functor is an equivalence:
    \begin{align*}
        \mathcal{W}(k) \subset \Coh(\Y) \to \Coh(\Y^{\circ}).
    \end{align*}
\end{thm}

\subsection{Conjectural semiorthogonal decomposition of the magic category}
In what follows, we assume that $\Y$ is a smooth QCA symmetric stack, see Definition~\ref{def:sstack}. For $\delta \in \mathrm{Pic}(\Y)_{\mathbb{R}}$, 
recall the magic category in Definition~\ref{defn:QBY}:
\begin{align*}
    \MM(\Y)_{\delta} \subset \Coh(\Y).
\end{align*}
In this subsection, we propose conjectural semiorthogonal decompositions
of the magic category along a $\Theta$-stratum. Let $\mathcal{S}$ and $\mathcal{Z}$ be 
the stacks as in the diagram (\ref{dia:SYZ}). We first prepare some lemmas: 

\begin{lemma}\label{em:sym}
	The stacks $\mathcal{S}$ and $\mathcal{Z}$ are smooth, and 
	$\mathcal{Z}$ is symmetric. Moreover, we have the identity
	\begin{align}\label{sym:N}
		[N_{\mathcal{S}/\mathcal{Y}}]=[\mathbb{L}_{\mathcal{S}/\mathcal{Z}}] \in K(\mathcal{S}). 
		\end{align}
        Here $N_{\mathcal{S}/\mathcal{Y}}$ is the normal bundle of $\mathcal{S}$ in $\mathcal{Y}$. 
	\end{lemma}
\begin{proof}
	There is a semiorthogonal decomposition 
	\begin{align}\label{sod:S}
		\Coh(\mathcal{S})=\langle \cdots, \Coh(\mathcal{S})_{i-1}, \Coh(\mathcal{S})_{i}, \Coh(\mathcal{S})_{i+1}, \cdots \rangle
		\end{align}
	such that we have the equivalence 
    $q^{\ast} \colon \Coh(\mathcal{Z})_i \stackrel{\sim}{\to} \Coh(\mathcal{S})_i$, see~\cite[Proposition~1.1.2, Lemma~1.5.6]{HalpK32}. 
	In particular for any $E \in \Coh(\mathcal{S})$ and $i\in \mathbb{Z}$ there is a unique 
	distinguished triangle 
	\begin{align}\label{triangle:beta}
		\beta_{>i}(E) \to E \to \beta_{\leq i}(E)
		\end{align}
	such that we have 
    \begin{align*}\beta_{>i}(E) \in \langle \Coh(\mathcal{S})_j : j>i\rangle, \ 
	\beta_{\leq i}(E) \in \langle \Coh(\mathcal{S})_j : j\leq i\rangle.
    \end{align*}
    
		By~\cite[Lemma~1.3.2, Lemma~1.5.5]{HalpK32}, the distinguished triangles 
	\begin{align*}
		\beta_{>0}(\mathbb{L}_{\mathcal{Y}}|_{\mathcal{S}}) \to 
		\mathbb{L}_{\mathcal{Y}}|_{\mathcal{S}} \to \beta_{\leq 0}(\mathbb{L}_{\mathcal{Y}}|_{\mathcal{S}}), \ 
		\beta_0(\mathbb{L}_{\mathcal{Y}}|_{\mathcal{S}}) \to 
		\beta_{\leq 0}(\mathbb{L}_{\mathcal{Y}}|_{\mathcal{S}}) \to \beta_{<0}(\mathbb{L}_{\mathcal{Y}}|_{\mathcal{S}})
		\end{align*}
	are identified with 
	\begin{align}\label{tri:NL}
		N_{\mathcal{S}/\mathcal{Y}}^{\vee} \to \mathbb{L}_{\mathcal{Y}}|_{\mathcal{S}} \to \mathbb{L}_{\mathcal{S}}, \ 
		q^{\ast}\mathbb{L}_{\mathcal{Z}} \to \mathbb{L}_{\mathcal{S}} \to \mathbb{L}_{\mathcal{S}/\mathcal{Z}}
		\end{align}
	respectively. 
    In particular, the cotangent complexes $\mathbb{L}_{\mathcal{S}}$ and $\mathbb{L}_{\mathcal{Z}}$ 
    are perfect with cohomological amplitude $[0, 1]$, hence $\mathcal{S}$ and $\mathcal{Z}$ 
    are smooth. 
    
    By the symmetric condition of $\mathcal{Y}$ 
    and the uniqueness of 
	the triangle (\ref{triangle:beta}), we have 
	\begin{align*}
		[\mathbb{L}_{\mathcal{S}/\mathcal{Z}}]=[\beta_{<0}(\mathbb{L}_{\mathcal{Y}}|_{\mathcal{S}})]=[\beta_{>0}(\mathbb{L}_{\mathcal{Y}}|_{\mathcal{S}})]^{\vee}=
		[N_{\mathcal{S}/\mathcal{Y}}]
		\end{align*}
	in $K(\mathcal{S})$. Similarly $\beta_0(\mathbb{L}_{\mathcal{Y}}|_{\mathcal{S}})$ is self-dual in $K(\mathcal{S})$, which implies that 
	$[\mathbb{L}_{\mathcal{Z}}]=[\mathbb{L}_{\mathcal{Z}}^{\vee}]$ in $K(\mathcal{Z})$, hence $\mathcal{Z}$ is symmetric. 
	\end{proof}

\begin{prop}\label{prop:functUp}
Let $\delta' \in \mathrm{Pic}(\zZ)_{\mathbb{R}}$ be given by 
\begin{align*}
		\delta'=\left(\delta \otimes (\det N_{\mathcal{S}/\mathcal{Y}})^{1/2}\right)|_{\mathcal{Z}}.
		\end{align*}
Then an object $A \in \Coh(\zZ)$ satisfies $\Upsilon(A) \in \MM(\Y)_{\delta}$ if and only if 
  $A \in \MM(\zZ)_{\delta'}$. 
  In particular, the functor (\ref{funct:up})
	restricts to the fully-faithful functor 
	\begin{align}\label{restrict:Up}
		\Upsilon \colon \MM(\mathcal{Z})_{\delta'} \hookrightarrow  \MM(\mathcal{Y})_{\delta}. 
		\end{align}
		\end{prop}
\begin{proof}
	We first show that for $A \in \MM(\mathcal{Z})_{\delta'}$
 we have $\Upsilon(A) \in \MM(\Y)_{\delta}$. 
 We need to show that 
	$\Upsilon(A)$ satisfies the required weight condition for each 
	$\nu \colon \bgm \to \mathcal{Y}$. Since $\Upsilon(A)$ is supported 
	on $\mathcal{S}$, we may assume that the image of $\nu$ lies in $\mathcal{S}$. 
    Let $\widetilde{\nu}$ be given by 
    \begin{align*}
        \widetilde{\nu} \colon \bgm  \stackrel{\nu}{\to} \mathcal{S} 
        \stackrel{q}{\to} \mathcal{Z} \stackrel{r}{\to} \mathcal{Y}. 
    \end{align*}
    Here the arrows are given in the diagram (\ref{dia:SYZ}).
    Note that $n_{\mathcal{Y}, \nu}=n_{\mathcal{Y}, \widetilde{\nu}}$, 
    since the class $[\mathbb{L}_{\Y}|_{\mathcal{S}}]\in K(\mathcal{S})$
    is of the form $[q^* F]$ for some $F \in K(\mathcal{Z})$ by the semiorthogonal decomposition (\ref{sod:S}). Together with the upper semi-continuity, we may replace 
    $\nu$ with $\widetilde{\nu}$ 
    so that 
	$\nu$ factors through $\nu \colon \bgm \to \mathcal{Z} \stackrel{r}{\to} \mathcal{Y}$. 
	
	For such a map $\nu$, since $A \in \MM(\mathcal{Z})_{\delta'}$, we have 
	\begin{align}\label{weight:1}
		\mathrm{wt}(\nu^{\ast}A) \subset \left[-\frac{1}{2}n_{\mathcal{Z}, \nu}, 
		\frac{1}{2}n_{\mathcal{Z}, \nu}  \right]+\langle \nu, \delta\rangle+\frac{1}{2}\langle \nu, N_{\mathcal{S}/\mathcal{Y}}\rangle.		\end{align}
	Since $p \colon \mathcal{S} \hookrightarrow \mathcal{Y}$ is a closed immersion, 
 the object 
 \begin{align*}p^{\ast}\Upsilon(A)=p^{\ast}p_{\ast}q^{\ast}A \in \Coh(\mathcal{S})
 \end{align*}
 admits a filtration whose associated graded is 
	the direct sum of $q^{\ast}A \otimes \wedge^i N_{\mathcal{S}/\mathcal{Y}}^{\vee}[i]$
 for $i\in \mathbb{Z}$. It follows that 
	the $\mathbb{G}_m$-weights of $\nu^{\ast}\Upsilon(A)$ are contained in 
	\begin{align}\label{weight:2}
		\left[-\frac{1}{2}n_{\mathcal{Z}, \nu}+\langle \nu, (N_{\mathcal{S}/\mathcal{Y}}^{\vee})^{\nu<0} \rangle, 
		\frac{1}{2}n_{\mathcal{Z}, \nu}+\langle \nu, (N_{\mathcal{S}/\mathcal{Y}}^{\vee})^{\nu>0}\rangle
		 \right]+\langle \nu, \delta\rangle+\frac{1}{2}\langle \nu, N_{\mathcal{S}/\mathcal{Y}}\rangle.
		\end{align}
        
	Using (\ref{sym:N}) and (\ref{tri:NL}), we have the equality in $K(\mathcal{S})$
	\begin{align*}
		[\mathbb{L}_{\mathcal{Y}}|_{\mathcal{S}}]
		&=[q^{\ast}\mathbb{L}_{\mathcal{Z}}]+[N_{\mathcal{S}/\mathcal{Y}}^{\vee}]+
		[\mathbb{L}_{\mathcal{S}/\mathcal{Z}}] \\
        &=[q^{\ast}\mathbb{L}_{\mathcal{Z}}]+[N_{\mathcal{S}/\mathcal{Y}}^{\vee}]+
		[N_{\mathcal{S}/\mathcal{Y}}]. 
		\end{align*}
	By the above identity, we see that 
	\begin{align*}
		\pm\frac{1}{2}n_{\mathcal{Z}, \nu}+\langle \nu, (N_{\mathcal{S}/\mathcal{Y}}^{\vee})^{\pm \nu>0} \rangle
		+\frac{1}{2}\langle \nu, N_{\mathcal{S}/\mathcal{Y}}\rangle
		=\pm \frac{1}{2}n_{\mathcal{Y}, \nu}.  
		\end{align*}
	Therefore $\Upsilon(A) \in \MM(\mathcal{Y})_{\delta}$ holds. 

 Conversely suppose that $A \in \Coh(\zZ)$ satisfies $\Upsilon(A) \in \MM(\Y)_{\delta}$. 
 Assume that $A \notin \MM(\zZ)_{\delta'}$. Then there is $\nu \colon \bgm \to \zZ$
 such that $\nu^{\ast}A$ does not satisfy the condition (\ref{weight:1}). As in the argument 
 above, it implies that the set of weights of $\nu^{\ast}\Upsilon(A)$ is not contained in the interval 
 (\ref{weight:2}), which contradicts to $\Upsilon(A) \in \MM(\Y)_{\delta}$. Therefore $A \in \MM(\zZ)_{\delta'}$. 

The functor $\Upsilon$ restricted to each $\Coh(\mathcal{Z})_{i}$ is fully-faithful 
	by~\cite[Theorem~3.3.1]{HalpK32}.
 Therefore the functor (\ref{funct:up}) restricts to the fully-faithful 
 functor (\ref{restrict:Up}). 
	\end{proof}
    
Let $j$ be the open immersion 
\begin{align*}
    j \colon \mathcal{Y}^{\circ}:=\mathcal{Y} \setminus \mathcal{S} \hookrightarrow 
    \mathcal{Y}. 
\end{align*}
By the definition of $\MM(\mathcal{Y})_{\delta}$, the pull-back 
$j^{\ast} \colon \Coh(\mathcal{Y}) \to \Coh(\mathcal{Y}^{\circ})$ restricts to 
the functor 
\begin{align}\label{funct:j}
    j^{\ast} \colon \MM(\mathcal{Y})_{\delta} \to \MM(\mathcal{Y}^{\circ})_{\delta}.
\end{align}
We expect that the following conjecture holds. We prove it 
for quotient stacks in Theorem~\ref{thm:qstack}. 
\begin{conj}\label{conj:adjoint}
(i) 
The functor (\ref{funct:j}) admits a fully-faithful 
right adjoint
\begin{align*}
    j_{\ast} \colon \MM(\mathcal{Y}^{\circ})_{\delta} \hookrightarrow \MM(\mathcal{Y})_{\delta}
\end{align*}
such that we have the semiorthogonal decompositions 
\begin{align*}
    \MM(\mathcal{Y})_{\delta}
    =\langle j_{\ast}(\MM(\mathcal{Y}^{\circ})_{\delta}), \Upsilon(\MM(\mathcal{Z})_{\delta'})\rangle. 
\end{align*}

(ii) 
The functor (\ref{funct:j}) admits a fully-faithful 
left adjoint
\begin{align*}
    j_{!} \colon \MM(\mathcal{Y}^{\circ})_{\delta} \hookrightarrow \MM(\mathcal{Y})_{\delta}
\end{align*}
such that we have the semiorthogonal decompositions 
\begin{align*}
    \MM(\mathcal{Y})_{\delta}
    =\langle \Upsilon(\MM(\mathcal{Z})_{\delta'}), j_{!}(\MM(\mathcal{Y}^{\circ})_{\delta})
    \rangle. 
\end{align*}

	\end{conj}

\subsection{\texorpdfstring{Semiorthogonal decompositions via $\Theta$-stratifications}{Semiorthogonal decompositions via Theta-stratifications}}
The rest of this section is devoted to prove Conjecture~\ref{conj:adjoint}
for some quotient stacks. 

Let $G$ be a reductive group and $Y$ be a self-dual $G$-representation, i.e. 
$Y \cong Y^{\vee}$ as $G$-representation. 
As in Example~\ref{exam:quiver}, we 
consider the following quotient stack, which is a symmetric stack:
\begin{align*}
\mathcal{Y}=Y/G.
\end{align*}
Below
we fix maximal torus and Borel subgroup $T\subset B\subset G$, and use the 
notation in Subsection~\ref{subsec:notation:reductive}.
 Note that $\mathrm{Pic}(BG)=M(T)^W$, and we often regard an 
 element of $\mathrm{Pic}(BG)$ as an element of $M(T)^W$. 
 We take 
$\delta \in \mathrm{Pic}(BG)_{\mathbb{R}}$, which is naturally 
regarded as a $\mathbb{R}$-line bundle by pull-back via the projection 
$\mathcal{Y} \to BG$. 
In this case, 
 we have the following alternative description of $\MM(\mathcal{Y})_{\delta}$:
\begin{prop}\emph{(\cite[Corollary~3.35]{PTquiver})}\label{prop:QB:weight}
The category $\MM(\mathcal{Y})_{\delta}$ is generated by $V_G(\chi) \otimes \mathcal{O}_{\mathcal{Y}}$ 
where $\chi \in M(T)$ is a dominant weight satisfying 
\begin{align}\label{poly:W}
\chi+\rho-\delta \in \mathbf{W}:=\frac{1}{2}[0, \beta] \subset M(T)_{\mathbb{R}}.
\end{align}
Here, $V_G(\chi)$ is the irreducible $G$-representation 
with highest weight $\chi$, $[0, \beta]$ is the Minkowski sum of $T$-weights of $Y$,
and $\rho$ is half the sum of positive roots. 
\end{prop}

For $\ell \in \mathrm{Pic}(BG)_{\mathbb{Q}}$, let $Y^{\ell\text{-ss}} \subset Y$ 
be the open subset corresponding to the $\ell$-semistable locus. 
It consists of points $y\in Y$ satisfying the following: 
for any 
one-parameter subgroup $\lambda \colon \mathbb{G}_m \to G$
such that $\lim_{t\to 0}\lambda(t)(y)$ exists, we have 
$\langle \ell, \lambda \rangle \geq 0$. 
By fixing a Weyl-invariant norm on 
$N(T)_{\mathbb{R}}$, there is an associated 
Kempf-Ness stratification, see~\cite[Section~2.1]{halp}
\begin{align*}
    Y=S_1 \sqcup \cdots \sqcup S_k \sqcup Y^{\ell\text{-ss}}
\end{align*}
with associated one-parameter subgroup $\lambda_i \colon \mathbb{G}_m \to T$
for each $1\leq i \leq k$. 
By taking the quotients by $G$, we have the $\Theta$-stratification~\cite{halpinstab}
\begin{align*}
    \mathcal{Y}=\mathcal{S}_1 \sqcup \cdots \sqcup \mathcal{S}_k \sqcup \mathcal{Y}^{\ell\text{-ss}}. 
\end{align*}
We define $\Y_i$ to be 
\begin{align*}
    \mathcal{Y}_i :=\mathcal{Y} \setminus (\mathcal{S}_1 \sqcup \cdots \sqcup \mathcal{S}_{i-1}), 
\end{align*}
so that $\mathcal{S}_i \subset \mathcal{Y}_i$ is a $\Theta$-stratum of $\mathcal{Y}_i$. 
Below we prove Conjecture~\ref{conj:adjoint} for $\mathcal{S}_i \subset \mathcal{Y}_i$
and $\delta \in \mathrm{Pic}(BG)_{\mathbb{R}}$. 

Let $\mathcal{Z}_i \subset \mathrm{Map}(\bgm, \Y_i)$ be the center of $\mathcal{S}_i$. 
As in the diagram (\ref{dia:SYZ}), we 
have the diagram 
\begin{align*}
    \mathcal{Z}_i \stackrel{q_i}{\leftarrow} \mathcal{S}_i \stackrel{p_i}{\hookrightarrow} \mathcal{Y}_i. 
\end{align*}
Then, as in (\ref{restrict:Up}), we have 
the following functor
\begin{align*}
 \Upsilon_i=p_{i\ast}q_i^{\ast} \colon \Coh(\mathcal{Z}_i) \to \Coh(\mathcal{Y}_i).
 \end{align*}
 Let $\overline{\mathcal{Z}}_i=Y^{\lambda_i}/G^{\lambda_i}$, 
 where $Y^{\lambda_i}\subset Y$ is the $\lambda_i$-fixed locus 
 and $G^{\lambda_i}$ the Levi subgroup associated with $\lambda_i$. 
We use the following fact: 
\begin{prop}\emph{(\cite[Proposition~2.18]{Hoskins})}\label{prop:Z}
For each $i$, there is $\ell_{\zZ_i}\in \mathrm{Pic}(\overline{\zZ}_i)$ such that $\zZ_i=\overline{\zZ}_i^{\ell_{\zZ_i}\text{-ss}}$.
\end{prop}
The following is an important input from~\cite[Proposition~3.22]{PThiggs}.
\begin{prop}\label{prop:ess}
The pull-back functors 
\begin{align*}
j^{\ast} \colon \MM(\mathcal{Y})_{\delta} \to \MM(\mathcal{Y}^{\mathrm{ss}})_{\delta}, \ 
  \MM(\overline{\mathcal{Z}}_i)_{\delta_i} \to \MM(\mathcal{Z}_i)_{\delta_i}
\end{align*}
are essentially surjective. 
    \end{prop}
    \begin{proof}
    The proposition is proved in the setting of symmetric quivers 
    in~\cite[Proposition~3.22]{PTquiver}, and the proof is essentially same. 
In Subsection~\ref{subsec:essurj}, we explain how to modify the argument in our setting. 
    \end{proof}
We define the subcategory 
\begin{align*}
\mathcal{W}_i \subset \Coh(\mathcal{Y})
\end{align*}
to be
consisting of objects 
$\mathcal{E}$ such that for each $1\leq j\leq i-1$, the $\lambda_j$-weights of 
$\mathcal{E}|_{\mathcal{Z}_j}$ 
satisfy that
\begin{align}\label{wt:window}
    \mathrm{wt}_{\lambda_j}(\mathcal{E}|_{\mathcal{Z}_j}) \subset 
    \left(-\frac{1}{2}n_{\lambda_j}, \frac{1}{2}n_{\lambda_j}  \right]+\langle \delta, \lambda_j\rangle.
\end{align}
Here for a cocharacter $\lambda$ of $T$, we write 
\begin{align*}
    n_{\lambda}:=\langle \lambda, (Y^{\vee})^{\lambda>0} -(\mathfrak{g}^{\vee})^{\lambda>0}\rangle 
    \in \mathbb{Z}. 
\end{align*}
We have $n_{\lambda}=n_{\mathcal{Y}, \nu}$, where $\nu \colon \bgm \to \mathcal{Y}$ 
is the map which corresponds to $\lambda \colon \mathbb{G}_m \to G=\mathrm{Aut}(0)$
to the origin $0\in Y$. We also have that 
\begin{align*}
    n_{\lambda_j}=\wt \det N_{\mathcal{S}_j/\Y_j}^{\vee}|_{\mathcal{Z}_j}.
\end{align*}
Below, for simplicity, we set
\begin{align}\label{def:mi}
m_j:=-\frac{1}{2}n_{\lambda_j}+\langle \delta, \lambda_j\rangle.     
\end{align}
Then the right-hand side of (\ref{wt:window}) is $(m_j, m_j+n_{\lambda_j}]$.

We have the chain of subcategories
\begin{align*}
\Coh(\mathcal{Y})=\mathcal{W}_1 \supset \mathcal{W}_2 \supset \cdots \supset \mathcal{W}_k
\end{align*}
such that, applying the window theorem (Theorem~\ref{thm:window}), the 
following composition is an equivalence
\begin{align}\label{eq:Phii}
\Phi_i \colon \mathcal{W}_i \subset \Coh(\mathcal{Y}) \twoheadrightarrow \Coh(\mathcal{Y}_i). 
\end{align}
Note that we have the following commutative diagram 
\begin{align}\notag
    \xymatrix{
\mathcal{W}_{i+1} \inclusion \ar[d]_-{\sim}^-{\Phi_{i+1}} & \mathcal{W}_i \ar[d]_-{\sim}^-{\Phi_i} \\
\Coh(\Y_{i+1}) & \ar@{->>}[l] \Coh(\Y_i)    
    }
\end{align}
where the bottom arrow is the restriction functor for the 
open immersion $\Y_{i+1} \subset \Y_i$. 
We then define 
\begin{align}\label{up:dag}
\Upsilon_i^{\dag}:= \Phi_i^{-1} \circ \Upsilon_i \colon 
\Coh(\mathcal{Z}_i) \to \Coh(\mathcal{Y}_i) \stackrel{\Phi_i^{-1}}{\to} \mathcal{W}_i. 
\end{align}

Let $\delta_i \in \mathrm{Pic}(\mathcal{Z}_i)_{\mathbb{R}}$ be defined by 
\begin{align*}
    \delta_i:=(\delta \otimes (\det N_{\mathcal{S}_i/\mathcal{Y}_i})^{1/2})|_{\mathcal{Z}_i}. 
\end{align*}
It is induced by the following $G^{\lambda_i}$-character
\begin{align*}
    \delta_i=\delta-\frac{1}{2}(Y^{\lambda_i>0} - \mathfrak{g}^{\lambda_i>0}) \in M(T)_{\mathbb{R}}^{W_i}. 
\end{align*}
Here $W_i$ is the Weyl-group of $G^{\lambda_i}$. 
Note that $m_i=\langle \lambda_i, \delta_i\rangle$, where 
$m_i$ is given by (\ref{def:mi}). 

\begin{prop}\label{prop:sod:qb}
There is a semiorthogonal decomposition 
\begin{align}\label{sod:qb}
\MM(\mathcal{Y})_{\delta} \cap \mathcal{W}_i=
\langle \Upsilon_i^{\dag}(\MM(\mathcal{Z}_i)_{\delta_i}), \MM(\mathcal{Y})_{\delta} \cap \mathcal{W}_{i+1} \rangle. 
\end{align}
    \end{prop}
    \begin{proof}
We prove the proposition by the induction on $i$. 
The base case of $i=0$ is obvious. We
assume that the semiorthogonal decomposition
\begin{align}\label{sod:qb:ind}
\MM(\mathcal{Y})_{\delta} \cap \mathcal{W}_j=
\langle \Upsilon_j^{\dag}(\MM(\mathcal{Z}_j)_{\delta_j}), \MM(\mathcal{Y})_{\delta} \cap \mathcal{W}_{j+1} \rangle
\end{align}
holds for all $j\leq i-1$, and show that (\ref{sod:qb}) holds. 
    
        By Theorem~\ref{thm:window} for GIT quotient stacks~\cite{MR3895631, halp}, 
        we have the semiorthogonal decomposition 
        \begin{align}\label{window}
        \mathcal{W}_i=\langle 
            \Upsilon_i^{\dag}\Coh(\mathcal{Z}_i)_{\leq m_i}, \mathcal{W}_{i+1}, \Upsilon_i^{\dag}\Coh(\mathcal{Z}_i)_{>m_i} \rangle. 
        \end{align}
        Here for an interval $I\subset \mathbb{R}$, we have denoted by 
        $\Upsilon_i^{\dag}\Coh(\mathcal{Z}_i)_I$ the subcategory generated by 
        $\Upsilon_i^{\dag}\Coh(\mathcal{Z}_i)_a$ for $a \in I$. 
        Since we have 
        \begin{align*}
            \MM(\mathcal{Z}_i)_{\delta_i} \subset \Coh(\mathcal{Z}_i)_{m_i}
        \end{align*}
        the right-hand side of (\ref{sod:qb}) is semiorthogonal. 

        We next show that the right-hand side of (\ref{sod:qb}) is 
        contained in the left-hand side. 
        It is enough to show that the functor (\ref{up:dag}) restricts 
        to the functor 
        \begin{align}\label{rest:Updag}
            \Upsilon_i^{\dag} \colon \MM(\mathcal{Z}_i)_{\delta_i} \to \MM(\mathcal{Y})_{\delta} \cap \mathcal{W}_i. 
        \end{align}     
   We set 
\begin{align*}
    \overline{\mathcal{S}}_i:=Y^{\lambda_i\geq 0}/G^{\lambda_i \geq 0}, \ 
    \overline{\mathcal{Z}}_i:=Y^{\lambda_i}/G^{\lambda_i}. 
\end{align*}
Note that $\mathcal{S}_i$ and $\mathcal{Z}_i$ are open substacks of 
$\overline{\mathcal{S}}_i$, $\overline{\mathcal{Z}}_i$, respectively. 
We have the following diagram 
\begin{align}\notag
\xymatrix{
\mathcal{Z}_i \ar@<-0.3ex>@{^{(}->}[d] \ar@/^10pt/[r]&
\mathcal{S}_i \ar@<-0.3ex>@{^{(}->}[d] \ar[l]^-{q_i} \inclusion_-{p_i} &
\mathcal{Y}_i \ar@<-0.3ex>@{^{(}->}[d] \\
\overline{\mathcal{Z}}_i  \ar@/_10pt/[r]& \overline{\mathcal{S}}_i \ar[l]_-{\overline{q}_i} \ar[r]^-{\overline{p}_i} & \mathcal{Y}.
}
\end{align}
Here each vertical arrows are open immersions. 
Since $\overline{p}_i$ is proper, we 
have the following functor
\begin{align}\label{ov:up:dag}
        \overline{\Upsilon}_i=\overline{p}_{i\ast}\overline{q}_{i}^{\ast} \colon 
    \Coh(\overline{\mathcal{Z}}_i) \to \Coh(\mathcal{Y}). 
\end{align}
By Lemma~\ref{lem:QB:rest} below, the functor (\ref{ov:up:dag})
restricts to the functor 
\begin{align}\label{restrict:QBdag}
\overline{\Upsilon}_i \colon \MM(\overline{\mathcal{Z}}_i)_{\delta_i} \to \MM(\mathcal{Y})_{\delta}.     
\end{align}

For $\mathcal{E} \in \MM(\mathcal{Z}_i)_{\delta_i}$, let $\overline{\mathcal{E}} \in \MM(\overline{\mathcal{Z}}_i)_{\delta_i}$ be its lift 
which exists by Proposition~\ref{prop:ess}.
We consider the object 
\begin{align}\notag
\overline{\Upsilon}_i(\overline{\mathcal{E}}) \in \MM(\mathcal{Y})_{\delta}.
\end{align}
By the induction hypothesis (\ref{sod:qb:ind}), we have the semiorthogonal decomposition 
\begin{align}\label{ind:tr}
    \MM(\mathcal{Y})_{\delta}=\langle \Upsilon_1^{\dag}\MM(\mathcal{Z}_1)_{\delta_1}, \ldots, 
    \Upsilon_{i-1}^{\dag} \MM(\mathcal{Z}_{i-1})_{\delta_{i-1}}, \MM(\mathcal{Y})_{\delta} \cap \mathcal{W}_i
    \rangle. 
\end{align}
Let $\mathrm{pr} \colon \MM(\mathcal{Y})_{\delta} \to \MM(\mathcal{Y})_{\delta} \cap \mathcal{W}_i$
be the projection with respect to the above semiorthogonal decomposition, and set 
\begin{align*}
    \mathcal{F}=\mathrm{pr}\overline{\Upsilon}_i(\overline{\mathcal{E}}) \in 
    \MM(\mathcal{Y})_{\delta} \cap \mathcal{W}_i. 
\end{align*}
Let $\Phi_i$ be the equivalence (\ref{eq:Phii}). 
Then 
\begin{align*}
\Phi_i(\mathcal{F})=\mathcal{F}|_{\Y_i} \cong\overline{\Upsilon}_i(\overline{\mathcal{E}})|_{\mathcal{Y}_i}
\end{align*}
and it is isomorphic to 
$\Upsilon_i(\mathcal{E})$ by Lemma~\ref{lem:extend} below. 
By the equivalence \[\Phi_i \colon \mathcal{W}_i \stackrel{\sim}{\to} \Coh(\mathcal{Y}_i)\] in 
(\ref{eq:Phii}),
it follows that 
\begin{align*}
    \Upsilon_i^{\dag}(\mathcal{E})=\Phi_i^{-1}\Upsilon_i(\mathcal{E})=\mathcal{F}
    \in \MM(\mathcal{Y})_{\delta} \cap \mathcal{W}_i. 
\end{align*}
Therefore the functor (\ref{up:dag}) restricts to the functor (\ref{rest:Updag}). 

Finally we prove that the left-hand side in (\ref{sod:qb}) is contained in the right 
hand side. In Lemma~\ref{lem:delta} below, we will construct a functor
\begin{align*}\Delta_i\colon \MM(\mathcal{Y})_{\delta}\to 
\MM(\overline{\mathcal{Z}}_i)_{\delta_i}
\end{align*}
with a natural transform $\id \to \overline{\Upsilon}_i \Delta_i$. 
For $\E \in \MM(\Y)_{\delta} \cap \mathcal{W}_i$, let 
\begin{align}\label{dist:E}
\E \to \ol{\Upsilon}_i \Delta_i(\E)
\end{align}
be the natural morphism. 
Since $\ol{\Upsilon}_i \Delta_i(\E) \in \MM(\Y)_{\delta}$, 
by the induction hypothesis (\ref{ind:tr})
there is a distinguished triangle 
\begin{align*}
    \mathrm{pr}\ol{\Upsilon}_i\Delta_i(\E) \to \ol{\Upsilon}_i \Delta_i(\E) \to A
\end{align*}
such that 
\begin{align*}
    \mathrm{pr}\ol{\Upsilon}_i\Delta_i(\E) \in \MM(\Y)_{\delta} \cap \mathcal{W}_i, \ 
    A \in \langle \Upsilon_i^{\dag}\MM(\zZ_j)_{\delta_j} : 1\leq j\leq i-1  \rangle. 
\end{align*}
Since $\mathcal{E} \in \MM(\mathcal{Y})_{\delta} \cap \mathcal{W}_i$, 
the 
composition 
\begin{align*}
    \E \to  \ol{\Upsilon}_i \Delta_i(\E) \to A
    \end{align*}
is a zero morphism by the semiorthogonal decomposition (\ref{ind:tr}). 

Therefore the morphism (\ref{dist:E}) factors through 
a morphism $\E \to \mathrm{pr}\ol{\Upsilon}_i \Delta_i(\E)$. 
We take the distinguished triangle 
\begin{align}\label{tr:E'}
    \E' \to \E \to \mathrm{pr}\ol{\Upsilon}_i \Delta_i(\E). 
\end{align}
It is enough to show that 
\begin{align}\label{enough:show}
\E' \in \MM(\Y)_{\delta} \cap \mathcal{W}_{i+1}, \ \mathrm{pr}\ol{\Upsilon}_i \Delta_i(\E)
\in \Upsilon_i^{\dag}\MM(\zZ_i)_{\delta_i}. 
\end{align}
We have 
\begin{align}\label{isom:Upsilon}
(\mathrm{pr}\ol{\Upsilon}_i \Delta_i(\E))|_{\Y_i} \cong 
(\ol{\Upsilon}_i \Delta_i(\E))|_{\Y_i} \cong \Upsilon_i(\Delta_i(\E)|_{\zZ_i})
\end{align}
where the last isomorphism follows from Lemma~\ref{lem:extend}. 
Then from the equivalence (\ref{eq:Phii}), we have 
\begin{align*}
   \mathrm{pr}\ol{\Upsilon}_i \Delta_i(\E)\cong\Upsilon_i^{\dag}(\Delta_i(\E)|_{\zZ_i})
\in \Upsilon_i^{\dag}\MM(\zZ_i)_{\delta_i}.
\end{align*}
Therefore the second condition of (\ref{enough:show}) holds. 

As for the first condition of (\ref{enough:show}), we take the restriction of (\ref{tr:E'})
to $\sS_i$ 
\begin{align*}
    \E'|_{\sS_i} \to \E|_{\sS_i} \to (\mathrm{pr}\ol{\Upsilon}_i \Delta_i(\E))|_{\sS_i}
    \cong \Upsilon_i(\Delta_i(\E)|_{\zZ_i})|_{\sS_i}. 
\end{align*}
Here the last isomorphism is due to (\ref{isom:Upsilon}). 
By unraveling the construction of $\Delta_i$ in Lemma~\ref{lem:delta}, the second arrow in the 
above triangle is identified with the canonical map 
\begin{align}\label{map:betam}
    \E|_{\sS_i} \to p_i^{\ast}p_{i\ast}\beta_{m_i}(\E|_{\sS_i})
\end{align}
obtained by taking the adjoint of (\ref{wt:trun}) and restricting to $\mathcal{S}_i$. 
As in the proof of Proposition~\ref{prop:functUp}, the 
objects in (\ref{map:betam}) have $\lambda_i$-weights in 
$[m_i, m_i+n_{\lambda_i}]$ on $\zZ_i$, such that the weight $m_i$-part is an 
isomorphism. 
Therefore $\mathcal{E}'|_{\zZ_i}$ has $\lambda_i$-weights contained in 
$(m_i, m_i+n_{\lambda_i}]$, which implies 
that $\E' \in \MM(\Y)_{\delta} \cap \mathcal{W}_{i+1}$. 
\end{proof}

We have postponed proofs of some lemmas in the above proposition, 
which are given below.

\begin{lemma}\label{lem:QB:rest}
    The functor (\ref{ov:up:dag}) restricts to the functor (\ref{restrict:QBdag}). 
\end{lemma}
\begin{proof}
    Using the descriptions of $\MM(\mathcal{Y})_{\delta}$ and $\MM(\overline{\mathcal{Z}}_i)_{\delta_i}$ 
    with respect to the highest weights of the generating vector bundles
    in Proposition~\ref{prop:QB:weight}, 
    the lemma follows from the Borel-Weil-Bott theorem. For more details, see the argument of~\cite[Proposition~8.3]{PTtop}. 
\end{proof}

\begin{lemma}\label{lem:extend}
    For any $\overline{\mathcal{E}} \in \Coh(\overline{\mathcal{Z}_i})$,
    the object $\Upsilon_i(\overline{\mathcal{E}}|_{\mathcal{Z}_i})$ is isomorphic 
    to $\overline{\Upsilon}_i(\overline{\mathcal{E}})|_{\mathcal{Y}_i}$. 
\end{lemma}
\begin{proof}
    We have the natural morphism 
    \begin{align}\label{nat:mor}
    \mathcal{S}_i \to \overline{\mathcal{S}}_i\times_{\mathcal{Y}}\mathcal{Y}_i.
    \end{align}
By the base change, it is enough to show that (\ref{nat:mor}) is an isomorphism.  
By the definition of $\Theta$-stratum, the morphism (\ref{nat:mor})    
is an isomorphism onto the open substack of the right-hand side, given by the preimage of 
$\mathcal{Z}_i \subset \overline{\mathcal{Z}}_i$ 
under the composition \begin{align*}\overline{\mathcal{S}}_i \times_{\mathcal{Y}} \mathcal{Y}_i \to 
\overline{\mathcal{S}}_i \to \overline{\mathcal{Z}}_i.\end{align*}

On the other hand, the morphism (\ref{nat:mor}) is
a closed morphism as both sides are proper over $\mathcal{Y}_i$. 
Therefore (\ref{nat:mor}) is an open immersion whose image is closed. 
Since $Y^{\lambda_i \geq 0}$ is an affine space, $\overline{\mathcal{S}_i}$ is connected, 
hence its Zariski open substack $\overline{\mathcal{S}}_i \times_{\mathcal{Y}}\mathcal{Y}_i$ is also connected. Therefore (\ref{nat:mor}) is 
an isomorphism. 
\end{proof}

\begin{lemma}\label{lem:delta}
There is a natural functor 
\begin{align*}
\Delta_i \colon \MM(\mathcal{Y})_{\delta} \to \MM(\overline{\mathcal{Z}}_i)_{\delta_i}
\end{align*}
with a natural transform $\id \to \ol{\Upsilon}_i \Delta_i$.
\end{lemma}
\begin{proof}
The functor $\Delta_i$ is constructed as a coproduct-like functor in~\cite[Lemma~5.1]{PT1}. 
For $\E \in \MM(\Y)_{\delta}$, 
by the definition of $\MM(\Y)_{\delta}$ 
the object $\ol{p}_i^{\ast}\E \in \Coh(\ol{\sS}_i)$
satisfies that
\begin{align}\label{wt:EZ}
\wt_{\lambda_i}(\mathcal{E}|_{\ol{\zZ}_i}) \subset [m_i, m_i+n_{\lambda_i}]. 
\end{align}
As in the proof of Lemma~\ref{em:sym} (see~\cite[Proposition~1.1.2, Lemma~1.5.6]{HalpK32}),
we have the semiorthogonal decomposition 
\begin{align*}
    \Coh(\overline{S}_i)=\langle \overline{q}_i^{\ast}\Coh(\overline{\mathcal{Z}_i})_j :
    j\in \mathbb{Z}\rangle
\end{align*}
and the condition (\ref{wt:EZ}) implies that
\begin{align*}
\overline{p}_i^{\ast}\mathcal{E}\in 
\langle \overline{q}_{i}^{\ast}\Coh(\overline{\mathcal{Z}_i})_{m_i}, \ 
\overline{q}_{i}^{\ast}\Coh(\overline{\mathcal{Z}_j})_{m_i+1}, \ldots, 
\overline{q}_{i}^{\ast}\Coh(\overline{\mathcal{Z}_j})_{m_i+n_{\lambda_i}}
\rangle.
\end{align*}

We have the weight truncation 
\begin{align}\label{wt:trun}
\ol{p}_i^{\ast} \mathcal{E} \to \beta_{m_i}(\ol{p}_i^{\ast}\E), \ 
\beta_{m_i}(\ol{p}_i^{\ast}\E)\in \overline{q}_{i}^{\ast}\Coh(\overline{\mathcal{Z}_j})_{m_i}
\end{align}
determined by the above semiorthogonal decomposition. 
The object \begin{align}\notag\Delta_i(\mathcal{E}) \in \Coh(\ol{\zZ}_i)_{m_i}\end{align} is defined to be the unique 
object such that 
$\beta_{m_i}(\ol{p}_i^{\ast}\E) \cong \ol{q}_i^{\ast}\Delta_i(\mathcal{E})$. 
Then using the description in terms of highest weight in Proposition~\ref{prop:QB:weight}, 
one can show that 
\begin{align}\label{nablai}
\Delta_i(\mathcal{E}) \in \MM(\ol{\zZ}_i)_{\delta_i} \subset \Coh(\ol{\zZ}_i)_{m_i}.
\end{align}

The proof of (\ref{nablai}) 
is the same as in the proof of~\cite[Lemma~5.1]{PT1}, 
where it is proven in a special example of the representation space of a quiver. 
In Subsection~\ref{subsec:nablai}, we explain how to generalize the proof in our setting. 

The morphism (\ref{wt:trun}) is $\overline{p}_i^{\ast}(\E) \to \ol{q}_i^{\ast}\Delta_i(\E)$. 
By adjunction, we obtain \[\E \to \ol{\Upsilon}_i\Delta_i(\E),\] giving 
the natural transform $\id \to \ol{\Upsilon}_i \Delta_i$. 
\end{proof}

We define $\MM(\mathcal{W}_i)_{\delta} \subset \mathcal{W}_i$ 
to be the subcategory such that the equivalence (\ref{eq:Phii}) restricts 
to the equivalence 
\begin{align*}
    \Phi_i \colon \MM(\mathcal{W}_i)_{\delta} \stackrel{\sim}{\to} \MM(\Y_i)_{\delta}. 
\end{align*}
We have the following proposition: 
\begin{prop}\label{prop:QBW}
    For each $i$, we have 
    \begin{align}\notag
    \MM(\Y)_{\delta} \cap \mathcal{W}_i=\MM(\mathcal{W}_i)_{\delta}.         
    \end{align}
\end{prop}
\begin{proof}
    Note that it is obvious that 
    \begin{align}\label{inc:obv}
\MM(\Y)_{\delta} \cap \mathcal{W}_i \subset \MM(\mathcal{W}_i)_{\delta}. 
    \end{align}
    Below we show the converse direction by the induction on 
    $i$. The base case is $i=k+1$, and in this case $\mathcal{Y}_{k+1}=\mathcal{Y}^{\mathrm{ss}}$. 
    We need to show that the restriction functor 
    \begin{align}\label{ess:0}
        \MM(\Y)_{\delta} \cap \mathcal{W}_{k+1} \to \MM(\Y^{\ell\text{-ss}})_{\delta}
    \end{align}
    is essentially surjective. 
    
    We perturb $\delta$ and take 
    $\delta'=\delta \otimes \ell^{-\varepsilon}$ for $0<\varepsilon \ll 1$. 
    By Proposition~\ref{prop:ess}, the restriction functor 
    \begin{align}\notag
\MM(\Y)_{\delta'} \to \MM(\Y^{\ell\text{-ss}})_{\delta'}=\MM(\Y^{\ell\text{-ss}})_{\delta}
    \end{align}
    is essentially surjective. Here the 
    last identity holds since for any $\nu \colon \bgm \to \Y^{\ell\text{-ss}}$ 
the weight of $\nu^{\ast}\ell$ is zero by the definition of $\ell$-semistability. 
On the other hand, since $\lambda_i$-weight of $\ell$ is strictly negative for each $i$, 
we have 
\begin{align*}
    \MM(\Y)_{\delta'} \subset \MM(\Y)_{\delta} \cap \mathcal{W}_{k+1}. 
\end{align*}
Therefore the functor (\ref{ess:0}) is essentially surjective. 

We take $i\leq k$, 
and by induction assume that
\begin{align}\label{ind:j}
    \MM(\Y)_{\delta} \cap \mathcal{W}_{j}=\MM(\mathcal{W}_j)_{\delta}
\end{align}
for all $i<j\leq k+1$. 
In what follows, we take 
$\E \in \MM(\mathcal{W}_i)_{\delta}$
and show that $\E\in \MM(\Y)_{\delta}\cap \mathcal{W}_i$. 
Note that we have the semiorthogonal decomposition, see (\ref{window})
\begin{align}\label{sod:W2}
    \mathcal{W}_i=\langle \Upsilon_i^{\dag} \Coh(\zZ_i)_{\leq m_i}, \mathcal{W}_{i+1}, 
    \Upsilon_i^{\dag}\Coh(\zZ_i)_{>m_i} \rangle. 
\end{align}
Let $\mathrm{pr} \colon \mathcal{W}_i \twoheadrightarrow \mathcal{W}_{i+1}$ be 
the projection with respect to the above semiorthogonal decomposition. 
We have the commutative diagram 
\begin{align}\notag
    \xymatrix{
\mathcal{W}_i \ar@{->>}[r]^-{\mathrm{pr}} \ar[d]_-{\Phi_i}^-{\sim} & \mathcal{W}_{i+1} 
\ar[d]^-{\Phi_{i+1}}_-{\sim} \\
\Coh(\mathcal{Y}_i) \ar@{->>}[r] & \Coh(\Y_{i+1})    
    }
\end{align}
where the bottom functor is the restriction functor. 
Since the latter functor sends $\MM(\Y_i)_{\delta}$ to $\MM(\Y_{i+1})_{\delta}$, 
the projection $\mathrm{pr}$ sends $\MM(\mathcal{W}_i)_{\delta}$ to $\MM(\mathcal{W}_{i+1})_{\delta}$. 

Moreover $\E|_{\zZ_i}$ has $\lambda_i$-weights in $[m_i, m_i+n_{\lambda_i}]$ 
since $\mathcal{E} \in \MM(\mathcal{W}_i)_{\delta}$. It follows that 
the semiorthogonal decomposition (\ref{sod:W2}) yields the distinguished 
triangle 
\begin{align}\label{dist:E''}
  \E' \to \E \to \E'', \ \E'=\mathrm{pr}(\mathcal{E}) \in \MM(\mathcal{W}_{i+1})_{\delta}, \ \E'' \in \Upsilon_i^{\dag}\Coh(\mathcal{Z}_i)_{m_i}.   
\end{align}
By the induction hypothesis (\ref{ind:j}), we have 
\begin{align}\notag
   \E' \in \MM(\Y)_{\delta} \cap \mathcal{W}_{i+1}. 
\end{align}
In particular we have $\E' \in \MM(\mathcal{W}_i)_{\delta}$ because 
\begin{align*}
   \MM(\Y)_{\delta} \cap \mathcal{W}_{i+1} \subset  \MM(\Y)_{\delta} \cap \mathcal{W}_{i}
   \subset \MM(\mathcal{W}_i)_{\delta}. 
\end{align*}
The last inclusion is due to (\ref{inc:obv}). 
By the distinguished triangle (\ref{dist:E''}), it follows that 
$\E'' \in \MM(\mathcal{W}_i)_{\delta}$. 

This implies that 
\begin{align*}
    \Phi_i(\E'') \in \Upsilon_i \Coh(\zZ_i)_{m_i} \cap \MM(\Y_i)_{\delta}=
    \Upsilon_i \MM(\zZ_i)_{\delta_i}
\end{align*}
where $\Phi_i$ is the equivalence (\ref{eq:Phii}) and the last identity is due to Proposition~\ref{prop:functUp}. 
Therefore $\E'' \in \Upsilon_i^{\dag}\MM(\zZ_i)_{\delta_i}$. 
By the distinguished triangle (\ref{dist:E''}),
we obtain 
\begin{align*}
    \E \in \langle \Upsilon_i^{\dag}\MM(\zZ_i)_{\delta_i}, \MM(\Y)_{\delta} \cap \mathcal{W}_{i+1}
    \rangle =\MM(\Y)_{\delta} \cap \mathcal{W}_i
\end{align*}
where the last identity is proved in Proposition~\ref{prop:sod:qb}. 
Therefore (\ref{ind:j}) also holds for $j=i$, and by induction 
we conclude the assertion. 
\end{proof}

\begin{thm}\label{thm:qstack}
Conjecture~\ref{conj:adjoint} holds for $\sS_i \subset \Y_i$
and $\delta \in \mathrm{Pic}(BG)_{\mathbb{R}}$.  
\end{thm}
\begin{proof}
We only prove Conjecture~\ref{conj:adjoint} (ii) 
for $\sS_i \subset \Y_i$, as the proof for (i) is similar. 

	By Proposition~\ref{prop:sod:qb} and Proposition~\ref{prop:QBW}, we have the 
	semiorthogonal decomposition 
    \begin{align*}
        \MM(\mathcal{W}_i)_{\delta}=\langle \Upsilon_i^{\dag}\MM(\mathcal{Z}_i)_{\delta_i},  \MM(\mathcal{W}_{i+1})_{\delta} \rangle. 
    \end{align*}
    By applying the functor $\Phi_i$ in (\ref{eq:Phii}), we obtain 
	\begin{align*}
		\MM(\Y_{i})_{\delta}=\langle \Upsilon_i \MM(\zZ_i)_{\delta_i}, \Phi_i \MM(\mathcal{W}_{i+1})_{\delta}\rangle. 
		\end{align*}
        
Let $j \colon \Y_{i+1} \hookrightarrow \Y_i$ be the open immersion.   
Then by the definition of $\MM(\mathcal{W}_{i+1})_{\delta}$, 
the pull-back $j^{\ast} \colon \Coh(\Y_i) \to \Coh(\Y_{i+1})$ restricts to the 
equivalence 
\begin{align*}
	j^{\ast} \colon \Phi_i \MM(\mathcal{W}_{i+1})_{\delta} \stackrel{\sim}{\to} 
	\MM(\Y_{i+1})_{\delta}. 
	\end{align*}  
The functor 
\begin{align*}j_{!}=(j^{\ast})^{-1} \colon 
	\MM(\Y_{i+1})_{\delta} \stackrel{\sim}{\to} \Phi_i \MM(\mathcal{W}_{i+1})_{\delta}
	\subset \MM(\Y_{i})_{\delta}
	\end{align*}
gives the fully-faithful left adjoint of $j^{\ast} \colon \MM(\Y_{i})_{\delta} \to \MM(\Y_{i+1})_{\delta}$ with semiorthogonal decomposition. 
\begin{align*}
		\MM(\Y_{i})_{\delta}=\langle \Upsilon_i \MM(\zZ_i)_{\delta_i}, j_{!}\MM(\Y_{i+1})_{\delta}\rangle. 
	\end{align*}
    Therefore Conjecture~\ref{conj:adjoint} (ii) holds for $\mathcal{S}_i \subset \mathcal{Y}_i$. 
\end{proof}

\subsection{The case of symmetric quivers}\label{subsec:sQ}
    Let $Q=(I, E)$ be a symmetric quiver, 
    where $I$ is the set of vertices and $E$ is the set of edges. 
    Here a quiver $Q$ is called \textit{symmetric} if for any $a, b \in I$, 
    the number of edges from $a$ to $b$ is the same as that from $b$ to $a$. 
    Let 
    $Y=R(d)$ be the representation space of $Q$ for $d=(d^a)_{a\in I} \in \mathbb{Z}^I$;
    \begin{align*}
        R(d)=\bigoplus_{e \in E} \Hom(V^{s(e)}, V^{t(e)})
    \end{align*}
    where $s(e) \in I$ is the source of $e$ and $t(e) \in I$ is the target
    of $e$, and $V^a$ is a $k$-vector space of dimension $d^a$. 
    
    Let $G(d)$ be given by 
    \begin{align*}
        G(d)=\prod_{a\in I} \mathrm{GL}(V^a). 
    \end{align*}
    It acts on $R(d)$ by the conjugation, and $R(d)$ is a self-dual $G(d)$-representation. 
The stack 
\begin{align*}\mathcal{Y}(d):=R(d)/G(d)
\end{align*}
is the moduli stack of $Q$-representations of dimension $d$. 

 For an element 
$(\mu_a)_{a\in I} \in \mathbb{R}^I$, and 
a non-zero dimension vector $d=(d^a)_{a \in I}$, we set 
\begin{align*}
    \mu(d)=\frac{\sum_{a\in I}\mu_a d^a}{\underline{d}} \in \mathbb{R}
\end{align*}
where $\underline{d}=\sum_{a\in I}d^a$. 
The above slope function determines $\mu$-stability on the category of 
representations of $Q$. 
We have the $\mu$-semistable locus 
\begin{align*}
    \Y(d)^{\mu\text{-ss}}=R(d)^{\mu\text{-ss}}/G(d) \subset \Y(d). 
\end{align*}
Let $\ell \in \mathrm{Pic}(BG(d))_{\mathbb{R}}$ be given by 
\begin{align*}
    \ell :=\bigotimes_{a \in I} (\det V^a)^{\mu_a-\mu(d)}. 
\end{align*}
By~\cite{Kin}, the $\mu$-semistable locus 
$\Y(d)^{\mu\text{-ss}}$ equals to the $\ell$-semistable locus 
in the GIT sense. 

We have the $\Theta$-stratification 
\begin{align}\label{theta:Y(d)}
    \mathcal{Y}(d)=\mathcal{S}_1 \sqcup \cdots \sqcup \mathcal{S}_k \sqcup \mathcal{Y}(d)^{\mu\text{-ss}}
\end{align}
with associated cocharacters $\lambda_i \colon \mathbb{G}_m \to G(d)$ for each $1\leq i \leq k$. Let $d=d_1+\cdots+d_m$ be a partition in $\mathbb{Z}^I$ corresponding to $\lambda_i$. 
    Then 
    \begin{align*}\Y(d)^{\lambda_i}=R(d)^{\lambda_i}/G(d)^{\lambda_i}
    =\times_{j=1}^m \Y(d_j)
    \end{align*}
    and the $\Theta$-stratification (\ref{theta:Y(d)}) coincides with the Harder-Narasimhan 
    stratification, see~\cite[Section~5]{P0}. So we have 
    \begin{align*}
        \mathcal{Z}_i=\times_{j=1}^m \Y(d_j)^{\mu\text{-ss}}.
    \end{align*}

Let $M(d)^W=M(T(d))^W$ be the Weyl-invariant part of the character lattice 
of the maximal torus $T(d) \subset G(d)$ which is identified with $\mathrm{Pic}(BG(d))$, and $\mathfrak{g}(d)$ the Lie algebra 
of $G(d)$. 
By Theorem~\ref{thm:qstack}, we obtain the following: 
\begin{thm}\label{thm:quiver}
For $\delta \in M(d)_{\mathbb{R}}^W$, there is a semiorthogonal decomposition
\begin{align*}
    \MM(\Y(d))_{\delta}=\left\langle  \bigotimes_{i=1}^k \MM(\Y(d_i)^{\mu\text{-ss}})_{\delta_i} :
    \mu(d_1)>\cdots> \mu(d_k) \right\rangle. 
\end{align*}
Here the right-hand side is after all the decompositions $d=d_1+\cdots+d_k$, and 
$\delta_i\in M(d_i)^{W_i}_{\mathbb{R}}$ is determined by 
\begin{align*}
    \sum_{i=1}^k \delta_i=\delta+\frac{1}{2}(R(d)^{\lambda>0}-\mathfrak{g}(d)^{\lambda>0})
\end{align*}
where $\lambda$ is an anti-dominant cocharacter with associated partition 
$(d_i)_{i=1}^k$.     
\end{thm}

As in~\cite{PTquiver}, we set 
\begin{align*}\tau_d=\frac{1}{\underline{d}}\left(\sum_{a\in I} \beta^a \right)
\in M(d)_{\mathbb{R}}^W.
\end{align*}
where $\beta^a$ are the weights of fundamental representation 
of $G(d)$, 
and we write $\MM(\Y(d))_{w} :=\MM(\Y(d))_{w\tau_d}$. 
We have the following corollary: 
\begin{cor}\label{cor:quiver}
Suppose that the number of edges between two different vertices is even and 
the number of loops at each vertex is odd. 
Then there is a semiorthogonal decomposition 
\begin{align*}
    \MM(\Y(d))_w=\left\langle  \bigotimes_{i=1}^k \MM(\Y(d_i)^{\mu\text{-ss}})_{w_i} :
    \mu(d_1)>\cdots> \mu(d_k), \frac{w_1}{\underline{d_1}}=\cdots=\frac{w_k}{\underline{d_k}} \right\rangle. 
\end{align*}
Here, the right-hand side is after all partitions $(d, w)=(d_1, w_1)+\cdots+(d_k, w_k)$. 
\end{cor}
\begin{proof}
Let $\delta=w\tau_d$ and take an anti-dominant cocharacter $\lambda$ with associated 
partition $d=d_1+\cdots+d_k$. 
    Under the assumption of $Q$, we have 
    \begin{align*}
        \frac{1}{2}(R(d)^{\lambda>0}-\mathfrak{g}(d)^{\lambda>0}) \in \bigoplus_{i=1}^k M(d_i)^{W_{d_i}}.
    \end{align*}
    The above element is regarded as a line bundle $\mathcal{L}$ on 
$\times_{i=1}^k \Y(d_i)$. 

Let $w_i \in \mathbb{Q}$ be determined by $w_i/\underline{d_i}=w/\underline{d}$, so that 
we have $\sum_{i=1}^k w_i \tau_{d_i}=w\tau_d$. 
By taking the tensor product with $\mathcal{L}$, there is an equivalence
\begin{align*}
    \bigotimes_{i=1}^k \MM(\Y(d_i)^{\mu \text{-ss}})_{\delta_i} \simeq 
     \bigotimes_{i=1}^k \MM(\Y(d_i)^{\mu \text{-ss}})_{w_i}.
\end{align*}
Therefore the corollary follows from Theorem~\ref{thm:quiver}. 
\end{proof}

\begin{remark}\label{rmk:quiver}
Under the assumption of Corollary~\ref{cor:quiver}, 
we should also have the semiorthogonal decomposition 
    \begin{align*}
    \Coh(\Y(d)^{\mu\text{-ss}})_w=\left\langle  \bigotimes_{i=1}^k \MM(\Y(d_i)^{\mu\text{-ss}})_{w_i} :
    \mu(d_1)=\cdots= \mu(d_k), \frac{w_1}{\underline{d_1}}<\cdots<\frac{w_k}{\underline{d_k}} \right\rangle. 
\end{align*}
Here the right-hand side is after all partitions $(d, w)=(d_1, w_1)+\cdots+(d_k, w_k)$. 
   The above semiorthogonal decomposition is proved in~\cite{PTquiver} in the 
   case that $\mu_i=0$. For generic $\mu$, we expect that 
   the argument as in~\cite{PTK3, PThiggs} 
   via reduction to local model of good moduli spaces applies. The above 
   semiorthogonal decomposition
   may be regarded as a `dual version' of the semiorthogonal 
   decomposition in Corollary~\ref{cor:quiver}. 
\end{remark}

\section{Dolbeault geometric Langlands conjecture}
In the previous section, we constructed a semiorthogonal decomposition of the magic category. We now provide an analogue of it for limit categories of moduli stacks of Higgs bundles, and use this to formulate the precise version of the Dolbeault geometric Langlands conjecture in terms of limit categories.
We state our main theorem on semiorthogonal decompositions of 
limit categories, and discuss its counterpart on the spectral side of the Dolbeault Langlands duality.

\subsection{Moduli stacks of Higgs bundles}\label{subsec:limHiggs}
Let $C$ be a smooth projective curve, $G$ an algebraic group, and $\fg$ its Lie algebra. 
We apply the construction in Subsection~\ref{subsec:limcot} to the stack
\begin{align*}\X=\Bun_G,\end{align*}
the moduli stack of $G$-bundles on $C$. It is a smooth stack, whose  
components are non-quasi-compact in general. 

By definition, a \textit{$G$-Higgs bundle} is a pair 
\begin{align}\label{def:higgs}
(E_G, \phi_G), \ E_G\to C, \ \phi_G \in H^0(C, \mathrm{Ad}_{\mathfrak{g}}(E_G)\otimes \Omega_C)
\end{align}
where $E_G\to C$ is a principal $G$-bundle
and $\mathrm{Ad}_{\mathfrak{g}}(E_G)=E_G\times^G \mathfrak{g}$ for the adjoint 
representation of $G$ on $\mathfrak{g}$. 
We denote by $\Hig_G$ the derived moduli stack of $G$-Higgs bundles on $C$.
In the case that $G$ is reductive, we have 
\begin{align*}
    \Hig_G=\Omega_{\Bun_G}.
\end{align*}

Below we assume that $G$ is reductive, and $^{L}G$ its Langlands dual. 
As in~\cite{DoPa}, we have $\pi_0(\Hig_G)=\pi_1(G)$, 
and $\Hig_G$ is a $Z(G)$-gerbe, where $Z(G) \subset G$ is the 
center. 
We have the corresponding decomposition 
into connected components 
\begin{align*}
    \Hig_G=\coprod_{\chi \in \pi_1(G)}\Hig_G(\chi). 
\end{align*}
Consider the Hitchin fibration 
\begin{align*}
    h \colon \Hig_G \to \mathrm{B}:=\Gamma(C, (\mathfrak{t}\otimes \Omega_C)^W)
\end{align*}
where $\mathfrak{t}$ is the Lie algebra of the maximal torus $T$ and $W$ is the Weyl group, 
see~\cite[Section~3]{DPlecture}. The Hitchin base for $^{L}G$ is also identified with $B$.
Consider the families over $B$:
 \begin{align}\label{dia:Hitchin} \xymatrix{ \Hig_{^{L}G}\ar[rd] & & \Hig_{G} \ar[ld] \\ & \mathrm{B} & } \end{align} which, outside of the discriminant $\Delta \subset \mathrm{B}$, 
 are dual abelian fibrations~\cite[Theorem~A]{DoPa}.

\subsection{Limit category for Higgs bundles}

We consider the category of ind-coherent sheaves on $\Hig_G$. 
There is an orthogonal decomposition into connected components and 
$Z(G)$-weights
\begin{align*}
    \IndCoh(\Hig_G)=\bigoplus_{\chi \in \pi_1(G), w\in Z(G)^{\vee}}
    \IndCoh(\Hig_G(\chi))_{w}.
\end{align*}
Here, $(-)_w$ is the subcategory of objects with $Z(G)$-weight $w$. 

We denote by 
\begin{align}\label{univ:E}
\mathscr{E}_G \to C\times \Bun_G
\end{align}
the universal $G$-bundle. 
For $c\in C$, there is a $G$-bundle 
\begin{align}\label{univ:Pc}
    \mathscr{E}_{G, c}:=\mathscr{E}|_{c\times \Bun_G} \to \Bun_G. 
\end{align}
For $w\in Z(G)^{\vee}$, let $\mu_G(w)\in G^{\vee}_{\mathbb{Q}}$ be the associated element 
corresponding to the isomorphism 
\begin{align*}G^{\vee}_{\mathbb{Q}} :=\text{Hom}(G,\mathbb{G}_m)\otimes_\mathbb{Z}\mathbb{Q}  \stackrel{\cong}{\to} Z(G)_{\mathbb{Q}}^{\vee}:=\text{Hom}(Z(G),\mathbb{G}_m)\otimes_\mathbb{Z}\mathbb{Q}.
\end{align*}
The element $\mu_G(w)$ applied to (\ref{univ:Pc})
determines a $\mathbb{Q}$-line bundle 
\begin{align}\label{deltaw}
    \delta_w \in \mathrm{Pic}(\Bun_G)_{\mathbb{Q}} 
\end{align}
which has $Z(G)$-weight $w$ up to torsion. 
We have the finite decomposition 
\begin{align}\label{dsum:w'}
    \IndL(\Omega_{\Bun_G})_{\delta=\delta_w}=
    \bigoplus_{w'\in Z(G)^{\vee}, \mu_G(w')=\mu_G(w)}
    \IndL(\Omega_{\Bun_G})_{\delta_w, w'}. 
\end{align}
\begin{defn}\label{def:indlim:Higgs}
The \textit{ind-limit category} for $G$-Higgs bundles is defined to be 
\begin{align}\label{def:higL}
    \IndL(\Hig_G(\chi))_w:=\IndL(\Omega_{\Bun_G})_{\delta_w, w}. 
\end{align}
The \textit{limit category} is defined to be its subcategory of compact objects
\begin{align}\notag
    \LL(\Hig_G(\chi))_w:= \IndL(\Hig_G(\chi))_w^{\mathrm{cp}}.
\end{align}
\end{defn}
\begin{remark}\label{rmk:ind:c}
The (ind-)limit categories in Definition~\ref{def:higL} are independent of $c\in C$ since 
$c_1(\delta_w)$ is independent of $c$. 
\end{remark}
\begin{remark}
    If $Z(G)^{\vee}$ is torsion free (e.g. $G=\GL_r$), then there is only one summand in (\ref{dsum:w'}),
    namely $w'=w$. In this case, one can simply define 
    \begin{align*}
          \IndL(\Hig_G(\chi))_w:=\IndL(\Omega_{\Bun_G})_{\delta=\delta_w}. 
    \end{align*}
\end{remark}

\subsection{The Dolbeault geometric Langlands conjecture}\label{subsec:DGLC}
On the spectral side, we consider the quasi-compact open substack of 
semistable $G$-Higgs bundles, see Subsection~\ref{subsec:ssH}
\begin{align*}
\Hig_G(\chi)^{\mathrm{ss}} \subset \Hig_G(\chi). 
\end{align*}

The following is our first formulation of Dolbeault geometric Langlands conjecture. 

\begin{conj}\label{conj:DL0}
For each element 
\begin{align*}
(\chi, w) \in \pi_1(G)\times Z(G)^{\vee}=Z(^{L}G)^{\vee}\times \pi_1(^{L}G)
\end{align*}
there is a $\mathrm{B}$-linear equivalence 
\begin{align}\label{conj:DLequiv}
    \IndCoh(\Hig_{^{L}G}(w)^{\mathrm{ss}})_{-\chi} \simeq \IndL(\Hig_{G}(\chi))_w.
\end{align}
Moreover the above equivalence restricts to an equivalence of subcategories with nilpotent singular supports
\begin{align}\label{conj:DLequiv:nilp}
    \IndCoh_{\mathcal{N}}(\Hig_{^{L}G}(w)^{\mathrm{ss}})_{-\chi} \simeq \IndL_{\mathcal{N}}(\Hig_{G}(\chi))_w.
\end{align}
\end{conj}
\begin{remark}
By considering the subcategories of compact objects, the equivalence in (\ref{conj:DLequiv})
also yields an equivalence 
\begin{align}\label{conj:DL2}
\Coh(\Hig_{^{L}G}(w)^{\mathrm{ss}})_{-\chi} \simeq \LL(\Hig_{G}(\chi))_w.
\end{align}
\end{remark}
\begin{remark}\label{rmk:dual}
    Outside of the discriminant $\Delta \subset \mathrm{B}$, the two fibrations (\ref{dia:Hitchin}) are dual 
    abelian fibrations~\cite{DoPa}, so they are Fourier-Mukai equivalent~\cite{Mu1}. In particular, 
    Conjecture~\ref{conj:DL0} is true over $\mathrm{B}^{\mathrm{sm}}:=\mathrm{B}\setminus \Delta$. In the case of $G=\GL_r$, it can be 
    further extended to a locus in $\mathrm{B}^{\mathrm{ell}}\subset \mathrm{B}$ where the spectral curves are irreducible by~\cite{Ardual}. For general reductive groups, see~\cite{ArFe}.
\end{remark}

\begin{remark}\label{rmk:torus}
    As a special case of Remark~\ref{rmk:dual}, if $G=T$ is a torus then $\Delta=\emptyset$ and Conjecture~\ref{conj:DL0}
    holds in this case. Indeed, in this case, we have  
\begin{align*}
    \Hig_T(\chi)^{\mathrm{ss}}=\Hig_T(\chi)=\times_{i=1}^r\Omega_{\Bun_{\mathbb{G}_m}(\chi_i)}
\end{align*}
    and then 
    \begin{align*}
        \Coh(\Hig_T(\chi)^{\mathrm{ss}})=\LL( \Hig_T(\chi))=\bigotimes_{i=1}^r
        \Coh(\Omega_{\Bun_{\mathbb{G}_m}(\chi_i)}).
    \end{align*}
    Therefore, limit categories are not new examples of categories for $G=T$. 
\end{remark}

\begin{remark}\label{rmk:limit2}
An equivalence (\ref{conj:DLequiv}) is regarded as a classical limit of $w$-twisted
renormalized version of GLC, see~\cite[Remark~1.6.9]{GLC1}
for the renormalized version. A genuine classical limit of the equivalence (\ref{intro:GLC})
is expected to be an equivalence
\begin{align}\notag
    \IndCoh_{\mathcal{N}}(\Hig_{^{L}G}(0)^{\mathrm{ss}})_{-\chi} \simeq \IndL_{\mathcal{N}}(\Hig_{G}(\chi))_0.
\end{align}
\end{remark}

\begin{remark}\label{squareroot}
More precisely, the automorphic side of GLC is 
the $1/2$-twisted dg-category of D-modules $\text{D-mod}(\Bun_G)_{\frac{1}{2}}$, 
see~\cite{GLC1}. However, because of the $1/2$-twist in Conjecture~\ref{conj:limitLN}, 
its natural classical limit should be 
\begin{align*}
\text{D-mod}(\Bun_G(\chi))_{\frac{1}{2}} \leadsto \IndL(\Hig_G(\chi))_{0}.
\end{align*}

One can non-canonically identify the $1/2$-twisted categories with untwisted ones; 
given $\omega_C^{1/2} \in \mathrm{Pic}(C)$, there is an associated $\omega_{\Bun_G}^{1/2}\in \mathrm{Pic}(\Bun_G)$ by~\cite[Section~4]{BD0}. Therefore we have an equivalence 
\begin{align}\label{equiv:1/2lim}
\otimes \omega_{\Bun_G}^{1/2} \colon 
\IndL(\Hig_{G}(\chi))_{\delta} \stackrel{\sim}{\to} \IndL(\Hig_G(\chi))_{\delta\otimes \omega_{\Bun_G}^{1/2}}.
\end{align}
In what follows, we fix $\omega_{C}^{1/2}$ and identify the $1/2$-twisted categories 
with untwisted ones by the equivalence (\ref{equiv:1/2lim}). 
\end{remark}

\begin{remark}\label{rmk:nilpeq}
    As discussed in Subsection~\ref{subsec:ssupport}, both sides of (\ref{conj:DLequiv}) are 
    modules over $\mathrm{QCoh}(\fg\ssslash G)$.
    If an equivalence (\ref{conj:DLequiv}) is linear over $\mathrm{QCoh}(\fg\ssslash G)$, 
    then the nilpotent singular support version (\ref{conj:DLequiv:nilp}) also follows by applying 
    $\otimes\mathrm{QCoh}_0(\fg\ssslash G)$ over $\mathrm{QCoh}(\fg\ssslash G)$
    where 
    \begin{align*}\mathrm{QCoh}_0(\fg\ssslash G) \subset \mathrm{QCoh}(\fg\ssslash G)
    \end{align*} is the 
    subcategory of objects supported at $0\in \fg\ssslash G$. 
\end{remark}

\begin{remark}\label{rmk:Hitchin}
It is expected that, under a conjectural equivalence in Conjecture~\ref{conj:DL0}, 
the structure sheaf of $\Hig_{^{L}G}(0)^{\mathrm{ss}}$ corresponds to some Hitchin section, see~\cite[Section~3.4]{HauICM}. 
It requires some care since the image of the Hitchin section is not closed inside 
$\Hig_G$. 

In fact, we can justify this using limit categories; for $G=\GL_r$ and $g\geq 2$, 
we have a natural Hitchin section 
\begin{align*}
B\stackrel{s}{\hookrightarrow} \Hig_{\GL_r}(0)^{\mathrm{ss}} \stackrel{j}{\hookrightarrow} \Hig_{\GL_r}(0).
\end{align*}
Here $s$ is a closed immersion and $j$ is an open immersion. 
Then we have the following object 
\begin{align}\label{Hsection}
    j_{!} s_{*}\mathcal{O}_B \in \LL(\Hig_{\GL_r}(0))_0
\end{align}
which is expected to correspond to the structure sheaf of $\Hig_{\GL_r}(0)^{\mathrm{ss}}$. 
The existence of $j_!$ will be proved in Proposition~\ref{prop:sod:KZ}, which is one of the key 
properties of limit categories. 

Indeed, we can interpret the object (\ref{Hsection}) as a classical limit of the 
vacuum Poincar\'{e} sheaf in~\cite{GLC1}, which is an object in $\text{D-mod}(\Bun_G)$
corresponding to the structure sheaf of $\mathcal{O}_{\mathrm{Locsys}_{^{L}G}}$ under GLC.
The details will be discussed elsewhere. 
\end{remark}
\subsection{Quasi-BPS categories for semistable Higgs bundles}\label{subsec:qBPS:ss}
Here we define quasi-BPS categories for semistable Higgs bundles as special 
cases of limit categories: 
\begin{defn}\label{def:qbpsG}
For $(\chi, w) \in \pi_1(G)\times Z(G)^{\vee}$, the \textit{quasi-BPS category} is defined 
to be 
\begin{align*}
    \mathbb{T}_G(\chi)_w:=\LL(\Hig_G(\chi)^{\mathrm{ss}})_{\delta_w, w} \subset \Coh(\Hig_G(\chi)^{\mathrm{ss}})_w.
\end{align*}
The \textit{quasi-BPS category with nilpotent singular support} is defined to be 
\begin{align*}
    \mathbb{T}_{\mathcal{N}, G}(\chi)_w:=\LL_{\mathcal{N}}(\Hig_G(\chi)^{\mathrm{ss}})_{\delta_w, w} \subset \Coh_{\mathcal{N}}(\Hig_G(\chi)^{\mathrm{ss}})_w.
\end{align*}
\end{defn}

\begin{remark}\label{rmk:GLr}
In the case of $G=\mathrm{GL}_r$, the dg-category $\mathbb{T}_G(\chi)_w$ coincides with quasi-BPS category 
    considered in~\cite{PThiggs}. 
    There is a good moduli space~\cite{AHPS} 
    \begin{align*}
        \Hig_{\GL_r}(\chi)^{\mathrm{ss, cl}}\to \mathrm{H}_{\GL_r}(\chi)^{\mathrm{ss}},
    \end{align*}
    where $\mathrm{H}_{\GL_r}(\chi)^{\mathrm{ss}}$ is a quasi-projective variety over $k$. 
    When $(r, \chi)$ are coprime, then we have 
    \begin{align*}
        \mathbb{T}_{\GL_r}(\chi)_w \simeq \Coh(\text{H}_{\GL_r}(\chi)^{\text{ss}}\times k[-1]).
    \end{align*}
    In this case, $\mathrm{H}_{\GL_r}(\chi)^{\rm{ss}}$ is a holomorphic 
    symplectic manifold with an integrable 
system, the Hitchin system. 

If $(r, \chi)$ are not necessarily coprime, but $(r, \chi, w)$ is primitive, then, after removing a trivial derived structure $k[-1]$, the dg-category $\mathbb{T}_{\GL_r}(\chi)_w$ is smooth over $k$, proper and Calabi–Yau over the Hitchin base, see~\cite{PThiggs, PThiggs2} for details. It gives a twisted categorical crepant resolution of $\mathrm{H}_{\GL_r}(\chi)^{\mathrm{ss}}$, 
    therefore it is regarded as a `non-commutative Hitchin system'.  Furthermore, in this case, the dg-category 
    $\mathbb{T}_{\GL_r}(\chi)_w$ recovers the BPS invariant
    of the non-compact Calabi–Yau 3-fold: 
    \begin{align*}X=\mathrm{Tot}_C(\mathcal{O}_C\oplus \Omega_C),
    \end{align*} hence we call it \textit{a BPS category}.
\end{remark}
\subsection{Semiorthogonal decomposition (spectral side)}\label{subsec:SOD0}
Here we recall the main result in our previous work~\cite{PThiggs}. It was proved 
only for $G \in \{\GL_r, \mathrm{PGL}_r, \mathrm{SL}_r\}$, but we formulate it 
which makes sense for any reducutive $G$. 

We first prepare some notations. 
Let $T \subset B\subset G$ be a fixed maximal torus and a Borel subgroup, and 
we use the notation in Subsection~\ref{subsec:notation:reductive}.
Let $R\subset M(T)$ be the root lattice and $R^{\mathrm{co}}\subset N(T)$ be the coroot lattice. 
We have 
\begin{align*}
    \pi_1(G)=N(T)/R^{\mathrm{co}}, \ Z(G)^{\vee}=M(T)/R. 
\end{align*}
In particular there are natural maps 
\begin{align}\notag
    N(T)^W \to \pi_1(G), \ M(T)^W \to Z(G)^{\vee}
\end{align}
which are isomorphisms over $\mathbb{Q}$. Here $(-)^W$ is the Weyl-invariant part. 
We define the following \textit{slope maps}
\begin{align}\label{def:slopemaps}
   & \mu_G \colon \pi_1(G) \to \pi_1(G)_{\mathbb{Q}}
  \notag  \cong N(T)_{\mathbb{Q}}^W, \\
   & \mu_G \colon Z(G)^{\vee} \to (Z(G)^{\vee})_{\mathbb{Q}} \cong M(T)_{\mathbb{Q}}^W.
\end{align}

A parabolic subgroup $P \subset G$ is called \textit{standard} if it contains the given Borel 
subgroup $B$. Let $P\subset G$ be a standard parabolic and $P\to M$ its Levi quotient
with canonical splitting $M\subset P$ (determined by a choice of $T$ and $B$). 
The inclusion $M\subset G$ induces 
the map 
\begin{align}\label{def:aM}
    a_{M} \colon \pi_1(M)\times Z(M)^{\vee} \to 
    \pi_1(G) \times Z(G)^{\vee}. 
\end{align}

Let $W_{M}$ be the Weyl-group of $M$. 
The slope map for $M$ gives a map 
\begin{align*}
    \mu_M \colon \pi_1(M)\times Z(M)^{\vee} \to N(T)_{\mathbb{Q}}^{W_M}\times M(T)^{W_M}_{\mathbb{Q}}.
\end{align*}
We say that $\chi \in M(T)^{W_{M}}$ is \textit{strictly $W$-dominant} if 
$\chi \in M(T)_+$ and $\mathrm{Stab}_{\chi}(W)=W_{M}$, where $\mathrm{Stab}_{\chi}(W) \subset W$ is the stabilizer subgroup of $\chi$. A strictly $W$-dominant condition for $N(T)^{W_{M}}$ is 
similarly defined. 
We denote by 
\begin{align*}M(T)_{\mathbb{Q}+}^{W_{M}} \subset M(T)^{W_{M}}_{\mathbb{Q}}, \ N(T)_{\mathbb{Q}+}^{W_{M}}\subset N(T)_{\mathbb{Q}}^{W_{M}}
\end{align*}
the subsets of strictly $W$-dominant elements. 
\begin{remark}\label{rmk:dominant}
One can check that the slope map $\mu_M \colon \pi_1(M) \to N(T)_{\mathbb{Q}}^{W_M}$ and the strictly $W$-dominant 
condition is the same as the slope map $\phi_P$ and dominant $P$-regular condition in~\cite[Section~2.1]{SchHN}.
\end{remark}

Using the above notation, we have the following semiorthogonal decomposition: 
\begin{thm}\emph{(\cite{PThiggs})}\label{thm:PThiggs}
If $G \in \{\GL_r, \mathrm{PGL}_r, \mathrm{SL}_r\}$, there is a semiorthogonal decomposition 
\begin{align*}
    \Coh(\Hig_G(\chi)^{\mathrm{ss}})_w=
    \left\langle \mathbb{T}_{M}(\chi_{M})_{w_{M}} : \begin{array}{ll}\mu_M(\chi_M)=\mu_G(\chi) \\
     \mu_M(w_{M}) \in 
   -M(T)_{\mathbb{Q}+}^{W_{M}} \end{array}
    \right\rangle. 
\end{align*}
    Here the right-hand side is after all standard parabolics $P\subset G$ with 
    Levi quotient $P\twoheadrightarrow M$, satisfying
    \begin{align*}(\chi_{M}, w_{M}) \in \pi_1(M) \times Z(M)^{\vee}, \mathbb{}a_{M}(\chi_{M}, w_{M})=(\chi, w).
    \end{align*}
\end{thm}
\begin{remark}
It is expected that the semiorthogonal decomposition in Theorem~\ref{thm:PThiggs} holds 
for any reductive $G$. 
\end{remark}

One can also show that the above semiorthogonal decompostion restricts to the subcategory with nilpotent 
singular supports (see Subsection~\ref{subsec:pfnil}): 
\begin{thm}\label{thm:nilp0}
The semiorthogonal decomposition in Theorem~\ref{thm:PThiggs} induces 
semiorthogonal decomposition with nilpotent singular supports: 
\begin{align*}
    &\Coh_{\mathcal{N}}(\Hig_G(\chi)^{\mathrm{ss}})_w=
    \left\langle \mathbb{T}_{\mathcal{N}, M}(\chi_{M})_{w_{M}} : \begin{array}{ll}\mu_M(\chi_M)=\mu_G(\chi) \\
     \mu_M(w_{M}) \in 
   -M(T)_{\mathbb{Q}+}^{W_{M}} \end{array}
    \right\rangle.
\end{align*}
\end{thm}

\subsection{Semiorthogonal decomposition (automorphic side)}\label{subsec:SOD}
Here we state one of our main results on the compact generation of the limit category (\ref{def:higL}) 
and its semiorthogonal decomposition into quasi-BPS categories. 
The following theorem will be proved in Subsection~\ref{subsec:proofofthm}:

\begin{thm}\label{thm:mainLG}
The dg-category (\ref{def:higL}) is compactly generated. Its subcategory 
of compact objects 
$\LL(\Hig_G(\chi))_w$ admits a semiorthogonal decomposition 
\begin{align*}
    \LL(\Hig_G(\chi))_w=\left\langle \mathbb{T}_{M}(\chi_{M})_{w_{M}} : \begin{array}{ll}
     \mu_M(\chi_{M})\in N(T)_{\mathbb{Q}+}^{W_M} \\ \mu_M(w_M)=\mu_G(w) \end{array}\right\rangle.
\end{align*}
   Here the right-hand side is after all standard parabolics $P\subset G$ with 
    Levi quotient $P\twoheadrightarrow M$, satisfying
    \begin{align*}(\chi_{M}, w_{M}) \in \pi_1(M) \times Z(M)^{\vee}, \mathbb{}a_{M}(\chi_{M}, w_{M})=(\chi, w).
    \end{align*}
\end{thm}
In Subsection~\ref{subsec:pfnil}, 
we also prove that the semiorthogonal decompositions in the above theorems preserve
subcategories with nilpotent singular supports: 
\begin{thm}\label{thm:nilp}
The semiorthogonal decomposition in Theorem~\ref{thm:mainLG} induces the 
semiorthogonal decomposition with nilpotent singular supports: 
\begin{align*}
    &\LL_{\mathcal{N}}(\Hig_G(\chi))_w=\left\langle \mathbb{T}_{\mathcal{N}, M}(\chi_{M})_{w_{M}} : \begin{array}{ll}
     \mu_M(\chi_{M})\in N(T)_{\mathbb{Q}+}^{W_M} \\ \mu_M(w_M)=\mu_G(w) \end{array}\right\rangle. 
\end{align*}
\end{thm}

We expect that the semiorthogonal decompositions in Theorem~\ref{thm:PThiggs} and Theorem~\ref{thm:mainLG}
are compatible under the equivalence~\eqref{conj:DLequiv}:
\begin{conj}\label{conj:compati:sod}
Under the equivalence Conjecture~\ref{conj:DL0}, the semiorthogonal 
decompositions in Theorem~\ref{thm:PThiggs} and Theorem~\ref{thm:mainLG}
are compatible.
\end{conj}

\subsection{Mirror symmetry for Higgs bundles}
As a consequence of Conjecture~\ref{conj:compati:sod}, we expect the following to hold, which generalizes the main conjecture in~\cite{PThiggs}:
\begin{conj}\label{conj:G}
There is an equivalence of quasi-BPS categories
    \begin{align}\label{conj:PThiggs}
        \mathbb{T}_{^{L}G}(w)_{-\chi} \simeq \mathbb{T}_{G}(\chi)_w,
    \end{align}
which restrict to the equivalence of quasi-BPS categories with nilpotent singular supports.
\end{conj}

As mentioned in Subsection~\ref{subsec:intro:catDT}, we regard the above conjecture as a categorical version of the Hausel--Thaddeus topological mirror symmetry for general reductive groups $G$. In particular, we expect the above conjecture to be related to a 2-periodic version of the topological mirror symmetry in~\cite[Conjecture 10.3.18]{BDIKP}, formulated in terms of BPS cohomology of loops stacks of semistable Higgs bundles.

\smallskip

By Theorem~\ref{thm:PThiggs} and (\ref{thm:mainLG}),
if the equivalence (\ref{conj:PThiggs}) holds
then both sides of (\ref{conj:DL2}) admit semiorthogonal 
decompositions such that each semiorthogonal summands are 
equivalent. In particular, we have the following consequence of K-theories
(we refer to~\cite{Blanc} for topological K-theory of dg-categories)
\begin{cor}\label{cor:Kth}
Suppose that 
$G \in \{\GL_r, \mathrm{PGL}_r, \mathrm{SL}_r\}$
and there exist equivalences (\ref{conj:PThiggs}).
    Then there is an isomorphism of (algebraic/topological) K-groups
    for $\star \in \{\mathrm{alg}, \mathrm{top}\}$
    \begin{align}\label{isom:K}
       K^{\star}(\Coh(\Hig_{^{L}G}(w)^{\mathrm{ss}})_{-\chi}) \cong K^{\star}(\LL(\Hig_{G}(\chi))_{w}). 
    \end{align}
\end{cor}

\subsection{Failure of compact generation}
To justify the use of limit categories rather than $\IndCoh$ or $\QCoh$, we discuss the failure of compact generation in the latter categories.
 Note that, by Zariski 
descent, the dg-category of ind-coherent sheaves on $\Hig_G$ is given by 
\begin{align}\label{def:IndL}
    \IndCoh(\Hig_G(\chi))
   =
    \lim_{\mathcal{U}\subset \Hig_G(\chi)}
    \IndCoh(\mathcal{U})
\end{align}
where the  limit is taken with respect 
to quasi-compact open substacks $\mathcal{U} \subset \Hig_G(\chi)$. 
We show that the dg-category 
(\ref{def:IndL}) is not compactly generated when $G=\mathrm{GL}_r$ with $r \geq 2$.

Let $\mathcal{N}_{\mathrm{glob}}$ be the global nilpotent 
cone
\begin{align*}
    \mathcal{N}_{\mathrm{glob}}:=\{(F, \theta) : \theta \mbox{ is nilpotent}\}
    \subset \Hig_G(\chi). 
\end{align*}

In the following proposition, we show that the category (\ref{def:IndL}) is not, in general, compactly generated, and thus we believe it is likely that it does not feature in a Dolbeault geometric Langlands equivalence. At the very least, the study of such a category requires tools different from those used in the de Rham geometric Langlands equivalence.
The result of Proposition~\ref{prop:cgen} will not be used in the later sections, so we 
give its proof in Subsection~\ref{pf:prop}. 
\begin{prop}\label{prop:cgen}
Suppose that $\mathcal{E} \in \IndCoh(\Hig_{\mathrm{GL}_r}(\chi))$ 
is a compact object. Then we have that
$\mathcal{E} \in \Coh(\Hig_{\mathrm{GL}_r}(\chi))$. Further, 
if $r\geq 2$, we have 
\begin{align}\label{Esupport}
    \mathrm{Supp}(\mathcal{E}) \cap \mathcal{N}_{\mathrm{glob}}=\emptyset. 
\end{align}
In particular, 
$\IndCoh(\Hig_{\mathrm{GL}_r}(\chi))$ is not compactly generated for $r\geq 2$. 
\end{prop}

\begin{remark}\label{rmk:compact}
A similar proof also shows that the categories 
$\mathrm{QCoh}(\Hig_{\GL_r}(\chi))$, 
$\IndCoh_{\mathcal{N}}(\Hig_{\GL_r}(\chi))$
and $\mathrm{QCoh}(\Bun_{\GL_r}(\chi))$ are not compactly generated when 
$r\geq 2$. 
\end{remark}

\subsection{Examples of semiorthogonal decompositions}\label{subsec:SODLang}

In the following examples we describe the semiorthogonal decompositions in Theorem~\ref{thm:PThiggs}
and Theorem~\ref{thm:mainLG} and give a proof of Corollary~\ref{intro:corK} for $G=\GL_r$. The cases of 
$G=\mathrm{PGL}_r, \mathrm{SL}_r$ are similar as the relevant results on topological K-theories in 
these cases
are already proved in~\cite{PThiggs, PThiggs2}. 
\subsubsection{Example: $G=\GL_r$}\label{subsubsec:GL}
For $G=\GL_r$, we have $\pi_1(G)=Z(G)^{\vee}=\mathbb{Z}$, so 
$(\chi, w)\in \mathbb{Z}^2$. We have $N(T)=M(T)=\mathbb{Z}^r$, and a Levi subgroup of $G$ is
of the form 
\begin{align*}
    M=\times_{i=1}^k \GL_{r_i}, \ r_1+\cdots+r_k=r.
\end{align*}
Both of the slope maps  
\begin{align*}
    \mu_M \colon \pi_1(M)=\mathbb{Z}^k \to \mathbb{Q}^r, \ Z(M)^{\vee}=\mathbb{Z}^r \to \mathbb{Q}^r
\end{align*}
are given by 
\begin{align*}
(x_1, \ldots, x_k) \mapsto 
    \bigl(\overbrace{\frac{x_1}{r_1}, \ldots, \frac{x_1}{r_1}}^{r_1}, \ldots, \overbrace{\frac{x_k}{r_k}, \ldots, \frac{x_k}{r_k}}^{r_k} \bigr).
    \end{align*}
It is strictly $W$-dominant if and only if $x_1/r_1<\cdots<x_k/r_k$. 
Therefore, the semiorthogonal decomposition in Theorem~\ref{thm:mainLG} is 
\begin{align}\label{sod:GLr}
   \LL(\Hig_{\mathrm{GL}_r}(\chi))_w=\left\langle \bigotimes_{i=1}^k \mathbb{T}_{\mathrm{GL}_{r_i}}(\chi_i)_{w_i} : \frac{\chi_1}{r_1}>\cdots>\frac{\chi_k}{r_k}, \frac{w_i}{r_i}=\frac{w}{r}  \right\rangle,  
\end{align}
where the right-hand side is after all partitions 
\begin{align}\notag
    (r, \chi, w)=(r_1, \chi_1, w_1)+\cdots+(r_k, \chi_k, w_k).
\end{align}

In (\ref{sod:GLr}), we have also used an equivalence 
\begin{align}\label{equiv:Ttensor}
\bigotimes_{i=1}^k \mathbb{T}_{\GL_{r_i}}(\chi_i)_{w_i} \stackrel{\sim}{\to}\mathbb{T}_G(\chi)_w
\end{align}
for the group $G=\GL_{r_1}\times \cdots \times \GL_{r_k}$, and for
$\chi=\chi_1+\cdots+\chi_k$ and $w=w_1+\cdots+w_k$, see Subsection~\ref{subsec:Ttensor}. 

For $G=\mathbb{G}_m$, we have 
\begin{align*}
    \Coh(\Hig_{\mathbb{G}_m}(\chi)^{\mathrm{ss}})_w \simeq 
    \LL(\Hig_{\mathbb{G}_m}(\chi))_w \simeq \mathbb{T}_{\mathbb{G}_m}(\chi)_w=\Coh(\Omega_{\mathrm{Pic}^0(C)}), 
\end{align*}
where $\Omega_{\mathrm{Pic}^0(C)}=\mathrm{Pic}^0(C) \times \mathbb{A}^g$. 
In this case, 
an equivalence (\ref{conj:DL2}) holds by taking the Fourier-Mukai transform~\cite{Mu1}. 

For $G=\GL_2$, the semiorthogonal decomposition in Theorem~\ref{thm:PThiggs} is 
\begin{align*}
    \Coh(\Hig_{\GL_2}(w)^{\mathrm{ss}})_{-\chi}=
    \left\langle \mathbb{T}_{\mathbb{G}_m}(w_1)_{-\chi_1}\otimes \mathbb{T}_{\mathbb{G}_m}(w_2)_{-\chi_2}, 
    \mathbb{T}_{\GL_2}(w)_{-\chi}\right\rangle
\end{align*}
where $(\chi, w)=(\chi_1, w_1)+(\chi_2, w_2)$, $w_1=w_2$ and $\chi_1>\chi_2$. 
Similarly, the semiorthogonal decomposition in Theorem~\ref{thm:mainLG} is 
\begin{align*}
    \LL(\Hig_{\GL_2}(\chi))_{w}=
    \left\langle \mathbb{T}_{\mathbb{G}_m}(\chi_1)_{w_1}\otimes \mathbb{T}_{\mathbb{G}_m}(\chi_2)_{w_2}, 
    \mathbb{T}_{\GL_2}(\chi)_{w}\right\rangle
\end{align*}
where $(\chi, w)=(\chi_1, w_1)+(\chi_2, w_2)$, $w_1=w_2$ and $\chi_1>\chi_2$. 
By~\cite[Lemma~3.8]{PThiggs}, there is an equivalence when both $d$ and $w$ are even:
\begin{align*}
 \mathbb{T}_{\GL_2}(\chi)_w\simeq\mathbb{T}_{\GL_2}(0)_0.
 \end{align*}
 Then, by Corollary~\ref{cor:Kth} and the isomorphism (\ref{isom:topK}) below, 
we see that the isomorphism (\ref{isom:K}) holds for rational topological K-theories. 

In the case that 
 $(r, w)$ are coprime, from Theorem~\ref{thm:PThiggs} and Theorem~\ref{thm:mainLG}
we have 
\begin{align*}
    \Coh(\Hig_{\GL_r}(w)^{\mathrm{ss}})_{\chi} &=\mathbb{T}_{\GL_r}(w)_{-\chi}, \\ 
    \LL(\Hig_{\GL_r}(\chi))_w&=\mathbb{T}_{\mathrm{GL}_r}(\chi)_w. 
\end{align*}
Therefore an equivalence~\eqref{conj:DL2} is equivalent to (\ref{conj:PThiggs}) in this case. 
In~\cite{PThiggs2}, we proved that 
\begin{align}\label{isom:topK}
    K^{\mathrm{top}}(\mathbb{T}_{\GL_r}(w)_{-\chi})_{\mathbb{Q}} \cong K^{\mathrm{top}}(\mathbb{T}_{\mathrm{GL}_r}(\chi)_w)_{\mathbb{Q}}
\end{align}
when $(r, \chi, w)$ is primitive. 
In particular, the isomorphism (\ref{isom:K}) holds for rational topological K-theory
if $(r, w)$ are coprime.

\subsubsection{Example: $g=0$}\label{subsubsec:g=0}
In the case of $g=0$, i.e. for $C=\mathbb{P}^1$, 
we have 
\begin{align*}
    \Hig_{\GL_r}(\chi)^{\mathrm{ss}} \cong \begin{cases}
\mathfrak{gl}_r^{\vee}[-1]/\mathrm{GL}_r, & r\mid \chi \\
\emptyset, & r \nmid \chi. 
    \end{cases}
\end{align*}
By Koszul duality in Theorem~\ref{thm:Kduality}, we have an equivalence 
\begin{align*}
\Psi \colon 
    \Coh(\mathfrak{gl}_r^{\vee}[-1]/\mathrm{GL}_r) \stackrel{\sim}{\to}
    \mathrm{MF}^{\mathrm{gr}}(\mathfrak{gl}_r/\mathrm{GL}_r, 0)=\Coh^{\shear}(\mathfrak{gl}_r/\mathrm{GL}_r). 
\end{align*}

Suppose that $r\mid \chi$. 
From the computation in~\cite[Lemma~3.3]{Toquot2}, we have 
\begin{align}\label{qbps:g=0}
    \mathbb{T}_{\mathrm{GL}_r}(\chi)_w \simeq \begin{cases}
        \Coh^{\shear}(\mathfrak{h}_r\ssslash W_r),& r\mid w, \\
         0, & r\nmid w.
    \end{cases}
\end{align}
Here, $\mathfrak{h}_r$ is the Cartan subalgebra of $\mathfrak{gl}_r$, and $W_r=S_r$ is the Weyl group of $\GL_r$.
For $w=rk$ with $k\in \mathbb{Z}$, the above equivalence is given by 
the functor
\begin{align*}
\otimes \det(-)^k \circ \pi^{\ast} \colon
    \Coh^{\shear}(\mathfrak{h}_r\ssslash W_r) \stackrel{\sim}{\to} \Psi(\mathbb{T}_{\GL_r}(\chi)_w) \subset \Coh^{\shear}(\mathfrak{gl}_r/\mathrm{GL}_r)
\end{align*}
where $\pi \colon \mathfrak{gl}_r/\mathrm{GL}_r \to \mathfrak{gl}_r \ssslash \mathrm{GL}_r=\mathfrak{h}_r\ssslash W_r$ is the good moduli space map. From Theorem~\ref{thm:PThiggs}, we have the semiorthogonal decomposition 
\begin{align*}
    \Coh(\Hig_{\GL_r}(\chi)^{\mathrm{ss}})_w
    =\left\langle \bigotimes_{i=1}^k \Coh^{\shear}(\mathfrak{h}_{r_i} \ssslash W_{r_i}) : r_1+\cdots+r_k=r\right\rangle.
\end{align*}
Here, for each partition $r_1+\cdots+r_k=r$, the semiorthogonal summands are indexed by 
$(\chi_1', w_1'), \ldots, (\chi_k', w_k') \in \mathbb{Z}^2$ such that 
\begin{align*}
    \sum_{i=1}^k (r_i \chi_i', r_i w_i')=(\chi, w), \ 
    \chi_1'=\cdots=\chi_k', \ w_1'<\cdots<w_k'. 
\end{align*}

Similarly, by Theorem~\ref{thm:mainLG}, 
we have the semiorthogonal decomposition 
\begin{align*}
    \LL(\Hig_{\GL_r}(\chi))_w
    =\left\langle \bigotimes_{i=1}^k \Coh^{\shear}(\mathfrak{h}_{r_i} \ssslash W_{r_i}) : r_1+\cdots+r_k=r\right\rangle.
\end{align*}
Here, for each partition $r_1+\cdots+r_k=r$, the semiorthogonal summands are indexed by 
$(\chi_1', w_1'), \ldots, (\chi_k', w_k') \in \mathbb{Z}^2$ such that 
\begin{align*}
    \sum_{i=1}^k (r_i \chi_i', r_i w_i')=(\chi, w), \ 
    \chi_1'>\cdots>\chi_k', \ w_1'=\cdots=w_k'. 
\end{align*}
In this case, an equivalence (\ref{conj:PThiggs}) holds from (\ref{qbps:g=0}), 
and in particular we have an isomorphism (\ref{isom:K}).

\begin{remark}\label{rem:g0-trivial-degeneration}
When $g=0$, there is an isomorphisms of derived stacks
\[
  \mathrm{Locsys}_{\GL_r} \simeq \Hig_{\GL_r}(0)^{\mathrm ss}
  \simeq \mathfrak{gl}_r[-1]/\GL_r .
\]
In particular, the degeneration (1.15) is trivial in this case.
On the automorphic side, we expect a parallel statement for $C=\PP^1$.

\begin{conj}
    There is an equivalence:
    \[
  \mathrm{D\text{-}mod}\left(\Bun_{\GL_r}(0)\right)
  \simeq \IndL_{\mathcal{N}}\left(\Hig_{\GL_r}(0)\right)_0,
\]
where the right-hand side denotes the neutral component of the limit category.
\end{conj}

\end{remark}

\subsubsection*{Example: $g=1$}
Let $C$ be an elliptic curve. By the Fourier–Mukai transform, there is an equivalence
\[
  \Hig_{\GL_r}(\chi)^{\mathrm ss} \simeq \mathcal{S}(m),
  \quad m=\gcd(r,\chi),
\]
where $\mathcal{S}(m)$ is the derived moduli stack of zero–dimensional sheaves of length $m$
on $\mathcal{S}:=C\times\mathbb{A}^1$, see Subsection~\ref{subsec:qbps:zero}. Consequently,
\[
  \mathbb{T}_{\GL_r}(\chi)_w \simeq \mathbb{T}(m)_w,
\]
where the right-hand side is the quasi–BPS category for zero–dimensional sheaves on $S$.
It is expected that
\begin{align}\label{qbps:g=1}
\mathbb{T}(m)_w \simeq \Coh^{\shear}(X^{\times d}/S_d), \ d=\gcd(m, w),
\end{align}
where $X$ is the Calabi–Yau 3-fold
\begin{align*}X=\mathrm{Tot}(\mathcal{O}_C\oplus \Omega_C)
=C\times \mathbb{A}^2
\end{align*}
and the shearing 
is with respect to the weight two $\mathbb{G}_m$-action on the last $\mathbb{A}^1$-factor. 
\begin{remark}\label{rmk:CPT}
An equivalence (\ref{qbps:g=1}) for $X=\mathbb{C}^3$ was conjectured in~\cite{PT0}, and its 
proof will appear in~\cite{CauPT}. The global version (\ref{qbps:g=1}) is expected to follow from 
the result of~\cite{CauPT}, see the discussion in~\cite[Subsection 1.3]{PT2}, and it will be discussed in detail elsewhere. 
\end{remark}

Assuming (\ref{qbps:g=1}), by Theorem~\ref{thm:PThiggs} we have the semiorthogonal 
decomposition 
\begin{align*}
    \Coh(\Hig_{\GL_r}(\chi)^{\mathrm{ss}})_w = 
    \left\langle \bigotimes_{i=1}^k \Coh^{\shear}(X^{\times d_i}/S_{d_i})  \right\rangle. 
\end{align*}
Here for each $(d_1, \ldots, d_k)$, a semiorthogonal summand is indexed by 
$(r_i', \chi_i', w_i')\in \mathbb{Z}^3$ for $1\leq i\leq k$ such that 
\begin{align*}
    \sum_{i=1}^k d_i(r_1', \chi_i', w_i')=(r, \chi, w), \ 
    \frac{\chi_1'}{r_1'}=\cdots= \frac{\chi_k'}{r_k'}, \ 
    \frac{w_1'}{r_1'}<\cdots<\frac{w_k'}{r_k'}. 
\end{align*}

Similarly, by Theorem~\ref{thm:mainLG}, there is a semiorthogonal decomposition 
\begin{align*}
    \LL(\Hig_{\GL_r}(\chi))_w = 
    \left\langle \bigotimes_{i=1}^k \Coh^{\shear}(X^{\times d_i}/S_{d_i})  \right\rangle. 
\end{align*}
Here, for each $(d_1, \ldots, d_k)$, a semiorthogonal summand is indexed by 
$(r_i', \chi_i', w_i')\in \mathbb{Z}^3$ for $1\leq i\leq k$ such that 
\begin{align*}
    \sum_{i=1}^k d_i(r_1', \chi_i', w_i')=(r, \chi, w), \ 
    \frac{\chi_1'}{r_1'}>\cdots> \frac{\chi_k'}{r_k'}, \ 
    \frac{w_1'}{r_1'}=\cdots=\frac{w_k'}{r_k'}. 
\end{align*}
In this case an equivalence (\ref{qbps:g=1}) implies an equivalence (\ref{conj:PThiggs}), 
which also implies an isomorphism (\ref{isom:K}).

 \subsubsection{Example: other reductive groups}
We describe how the semiorthogonal decompositions in Theorem~\ref{thm:mainLG}
look like for reductive groups $G$ other than $\mathrm{GL}_r$. 
Here we describe the case of $(G, ^{L}G)=(\mathrm{SO}_{2n+1}, \mathrm{Sp}_{2n})$. 

For $G=\mathrm{SO}_{2n+1}$, we have 
$\pi_1(G)=\mathbb{Z}/2$ and $Z(G)^{\vee}=0$. 
Therefore, we can omit the index $w\in Z(G)^{\vee}$, and the quasi-BPS categories are 
\begin{align*}
    \mathbb{T}_{\mathrm{SO}_{2n+1}}(\chi) \subset \Coh(\Hig_{\mathrm{SO}_{2n+1}}(\chi)^{\mathrm{ss}}), \ \mathbb{\chi}\in \mathbb{Z}/2.
\end{align*}
A Levi subgroup is of the form 
\begin{align*}
\mathrm{GL}_{r_1}\times \cdots \times \mathrm{GL}_{r_k} \times \mathrm{SO}_{2m+1}, \ 
r_1+\cdots+r_k+m=n.
\end{align*}
Then the semiorthogonal decomposition in Theorem~\ref{thm:mainLG} is 
 \begin{align*}
 \LL(\Hig_{\mathrm{SO}_{2n+1}}(\chi))=\left\langle \bigotimes_{i=1}^k \mathbb{T}_{\mathrm{GL}_{r_i}}(\chi_i)_{0} \otimes \mathbb{T}_{\mathrm{SO}_{2m+1}}(\chi_0) : 
 \frac{\chi_1}{r_1}>\cdots>\frac{\chi_k}{r_k}>0
 \right\rangle. 
\end{align*}
 Here, the right-hand side is after 
 \begin{align*}(r_1, \overline{\chi}_1)+\cdots+(r_k, \overline{\chi}_k)+(m, \chi_0)
 =(n, \chi) \in \mathbb{Z}\oplus \mathbb{Z}/2
 \end{align*}
 where, for $\chi_i \in \mathbb{Z}$, we denote its reductive modulo two by $\overline{\chi}_i\in \mathbb{Z}/2$. 

For $G=\mathrm{Sp}_{2n}$, we have $\pi_1(G)=0$ and $Z(G)^{\vee}=\mathbb{Z}/2$. 
The quasi-BPS categories are 
\begin{align*}
    \mathbb{T}_{\mathrm{Sp}_{2n}, w}\subset \Coh(\Hig^{\mathrm{ss}}_{\mathrm{Sp}_{2n}})_{w}, \ w\in \mathbb{Z}/2.
\end{align*}
A Levi subgroup is of the form 
\begin{align*}
\mathrm{GL}_{r_1}\times \cdots \times \mathrm{GL}_{r_k} \times \mathrm{Sp}_{2m}, \ 
r_1+\cdots+r_k+m=n.
\end{align*}
The semiorthogonal decomposition in Theorem~\ref{thm:mainLG} is 
 \begin{align*}
 \LL(\Hig_{\mathrm{Sp}_{2n}})_w=\left\langle \bigotimes_{i=1}^k \mathbb{T}_{\mathrm{GL}_{r_i}}(\chi_i)_{0} \otimes \mathbb{T}_{\mathrm{Sp}_{2m}, w} : 
 \frac{\chi_1}{r_1}>\cdots>\frac{\chi_k}{r_k}>0
 \right\rangle. 
\end{align*}
 Here the right-hand side is after $r_1+\cdots+r_k+m=n$ and 
 all $\chi_1, \ldots, \chi_k \in \mathbb{Z}$.

\section{Semiorthogonal decomposition of the limit category}\label{sec:sodlim}
In this section, we give a proof of Theorem~\ref{thm:mainLG}. 
Our strategy is to 
consider closed immersions of quasi-compact open substacks of $\Hig_G$ into
moduli stacks of $L$-Higgs bundles, where $L$ is a line bundle with $\deg L\gg 0$\footnote{For type A groups and for semistable Higgs bundles, this construction was used by Maulik--Shen~\cite{MSendscopic} to prove the Hausel--Thaddeus topological mirror symmetry conjecture~\cite{HauTha} (proved before in~\cite{GrWy}), and a similar idea was used by Kinjo--Koseki~\cite{KinjoKoseki} to prove the cohomological $\chi$-independence for Higgs bundles.}.
We then use the Koszul equivalence to reduce the problem to the existence of semiorthogonal 
decomposition of some dg-categories of matrix factorizations. 
The latter is proved using 
the result of Section~\ref{sec:magic} by reducing the claim to the case of smooth quotient stacks. 
A flow chart of the proof is depicted in Figure~\ref{fig:flow}. 

\begin{figure}[p!]
\centering
\newdimen\MainW
\MainW=.95\linewidth
\begin{tikzpicture}[node distance=7mm and 6mm,
  scale=0.96, every node/.style={transform shape}]

\tikzset{
  box/.style   ={rounded corners, draw, align=left, inner sep=4pt,
                 text width=\MainW, minimum height=8mm},
  lemma/.style ={box}
}


\node[box] (setup) {Setup/Notation (Subsections~\ref{subsec:sHiggs}, \ref{subsec:ssH}, \ref{subsec:HNhiggs}, \ref{subsec:embL}):
\begin{itemize}\setlength\itemsep{0pt}
\item Consider quasi-compact open substacks $\mathcal{U}_{\circ}\subset\mathcal{U}\subset \Hig_G$ such that $\mathcal{S}=\mathcal{U}\setminus \mathcal{U}_{\circ}$ is a 
$\Theta$-stratum and $\mathcal{Z}$ is its center.
\item Consider a quasi-compact open substack $\mathcal{U}^L \subset \Hig_G^{L}$ for $\deg L \gg 0$
and a closed immersion $\mathcal{U} \hookrightarrow \mathcal{U}^L$.
\item There is a vector bundle $\mathcal{V}^L \to \mathcal{U}^L$ with a section $s^L$ such that 
$\mathcal{U}=(s^L)^{-1}(0)$.
\end{itemize}};

\node[box, below=of setup] (lemone) {$\Theta$-stratum of $(\mathcal{V}^L)^{\vee}$ (Subsection~\ref{subsec:ThetaH}): there is a $\Theta$-stratum $\mathcal{S}_{\mathcal{V}}\subset (\mathcal{V}^L)^{\vee}$ 
whose restriction to $\Omega_{\mathcal{U}}[-1]$ is a pull-back of $\mathcal{S}$.
We denote by $\mathcal{Z}_{\mathcal{V}}$ the center of $\mathcal{S}_{\mathcal{V}}$ and 
$(\mathcal{V}^L)_{\circ}^{\vee}$ the complement of $\mathcal{S}_{\mathcal{V}}$.
};

\node[box, below=of lemone] (lemtwo) {SOD of magic categories (Proposition~\ref{prop:sod:V}): construct a semiorthogonal decomposition 
\begin{align}\label{SOD:flow0}
\MM_{\mathbb{G}_m}((\V^L)^{\vee})_{\delta}=\langle\MM_{\mathbb{G}_m}(\zZ_{\V})_{\delta'}, \MM_{\mathbb{G}_m}((\V^L)^{\vee}_{\circ})_{\delta}\rangle. 
\end{align}
This is proved by reduction to a local model and using the result of Section~\ref{sec:magic}.
};

\node[box, below=of lemtwo] (lemthree) {SOD of limit categories (Proposition~\ref{prop:sod:KZ}): using the SOD for magic 
categories, construct a semiorthogonal decomposition 
\begin{align}\label{SOD:flow}
\LL(\mathcal{U})_{\delta}=\langle \LL(\mathcal{Z})_{\delta'}, j_{!}\LL(\mathcal{U}_{\circ})_{\delta}\rangle.
\end{align}
Here, $j_{!}$ is a fully-faithful left adjoint of the restriction functor 
$j^* \colon \LL(\mathcal{U})_{\delta}\to \LL(\mathcal{U}_{\circ})$.
};

\node[box, below=of lemthree] (lem4) {Proof of (\ref{SOD:flow}): using the Koszul equivalence, 
there is an equivalence 
\begin{align*}
    \Coh(\mathcal{U})\simeq \mathrm{MF}^{\mathrm{gr}}((\mathcal{V}^L)^{\vee}, f).
\end{align*}
The above equivalence restricts to an equivalence 
\begin{align*}
    \LL(\mathcal{U})_{\delta}\simeq \mathrm{MF}^{\mathrm{gr}}(\MM((\mathcal{V}^L)^{\vee})_{\delta\otimes (\omega_{\mathcal{U}^L})^{1/2}}, f).
\end{align*}
Then use (\ref{SOD:flow0}) to construct (\ref{SOD:flow}).};

\node[box, below=of lem4] (ded) {Proof of Theorem~\ref{thm:mainLG} (Subsection~\ref{subsec:proofofthm}): 
apply (\ref{SOD:flow}) to show that
\begin{align*}
	\IndL(\Hig_G(\chi))_w \simeq \colim_{\mu \in N(T)_{\mathbb{Q}+}} \IndL(\Hig_G(\chi)_{\preceq \mu})_w.
\end{align*}
It gives a compact generation of $\IndL(\Hig_G(\chi))_w$ and the subcategory of 
compact object is 
\begin{align*}
    \LL(\Hig_G(\chi))_w \simeq \colim_{\mu \in N(T)_{\mathbb{Q}+}} \LL(\Hig_G(\chi)_{\preceq \mu})_w.
\end{align*}
By applying (\ref{SOD:flow}) again, we obtain Theorem~\ref{thm:mainLG}.
};

\draw[->] (setup) -- (lemone);
\draw[->] (lemone) -- (lemtwo);
\draw[->] (lemtwo) -- (lemthree);
\draw[->] (lemthree) -- (lem4);
\draw[->] (lem4) -- (ded);

\end{tikzpicture}
\caption{Flowchart of the proof of Theorem~\ref{thm:mainLG}}\label{fig:flow}
\end{figure}

\subsection{\texorpdfstring{$L$-twisted $G$-Higgs bundles}{L-twisted G-Higgs bundles}}
\label{subsec:sHiggs}
Let $C$ be a smooth projective curve and let $L\to C$ be a line bundle
of the form $L=\Omega_C(D)$ for an effective divisor $D$ on $C$. 
By definition, an \textit{$L$-twisted $G$-Higgs bundle} consists of data
\begin{align*}
    (E_G, \phi_G), \ E_G\to C, \ \phi_G \in H^0(C, \mathrm{Ad}_{\mathfrak{g}}(E_G)\otimes L)
\end{align*}
where $E_G$ is a $G$-bundle over $C$. 

We denote by 
\begin{align}\label{Higgs:pi}
\pi^L \colon 
    \Hig_G^{L}=\coprod_{\chi\in \pi_1(G)}\Hig_G^{L}(\chi) \to \Bun_G
\end{align}
the derived moduli stack of $L$-twisted $G$-Higgs bundles, the 
decomposition into connected components corresponding to $\chi \in \pi_1(G)$
and the projection to $\Bun_G$. 
In the case of $L=\Omega_C$, 
we omit $L$ and write $\Hig_G=\Hig_G^{L=\Omega_C}$ as before. 
We denote by 
\begin{align}\label{univ:G}
    (\mathscr{F}_G^{L}, \varphi_G^{L}), \ \mathscr{F}_G^{L} \to C\times \Hig_G^{L}, \ 
    \ \varphi_G^{L} \in H^0(C\times \Hig_G^{L}, \Ad_{\mathfrak{g}}(\mathscr{F}_G^{L})\boxtimes L)
\end{align}
the universal $L$-twisted $G$-Higgs bundle. 

Let $\mathscr{E}_G \to C\times \Bun_G$ be the universal $G$-bundle. 
The stack $\Hig_G^{L}$ is written as 
\begin{align}\notag
    \Hig_G^{L}=\Spec_{\Bun_G}\mathrm{Sym}((p_{\mathrm{B}\ast}(\mathrm{Ad}_{\mathfrak{g}}(\mathscr{E}_G)\boxtimes L))^{\vee})
\end{align}
where $p_{\mathrm{B}}^{}\colon C\times \Bun_G \to \Bun_G$ is the projection. 
From the above description and noting that 
\begin{align*}\mathbb{T}_{\Bun_G}|_{E_G}=\Gamma(C, \mathrm{Ad}_{\mathfrak{g}}(E_G))
\end{align*}
the following lemma is immediate: 
\begin{lemma}\label{lem:cotangent}
    There is a distinguished triangle 
    \begin{align*}
        \Gamma(C, \mathrm{Ad}_{\mathfrak{g}}(E_G)) \stackrel{[\phi_G^{L}, -]}{\to}
        \Gamma(C, \mathrm{Ad}_{\mathfrak{g}}(E_G)\otimes L)
        \to \mathbb{T}_{\Hig_G^{L}}|_{(E_G, \phi_G^{L})}
    \end{align*}
\end{lemma}

For an $L$-twisted $G$-Higgs bundle $(E_G, \phi_G^{L})$, we simply write 
\begin{align}\label{def:Ti}
    \mathbb{T}_{(E_G, \phi_G^{L})}^i :=\mathcal{H}^i(\mathbb{T}_{\Hig_G^{L}}|_{(E_G, \phi_G^{L})}), \ i\in \mathbb{Z}.
\end{align}
From Lemma~\ref{lem:cotangent}, we have the exact sequence
\begin{align}\label{exact}
    0&\to \mathbb{T}_{(E_G, \phi_G)}^{-1} \to H^0(\mathrm{Ad}_{\mathfrak{g}}(E_G))
    \to H^0(\mathrm{Ad}_{\mathfrak{g}}(E_G)\otimes L)  \\
 \notag   &\to \mathbb{T}_{(E_G, \phi_G)}^{0}\to H^1(\mathrm{Ad}_{\mathfrak{g}}(E_G))\to H^1(\mathrm{Ad}_{\mathfrak{g}}(E_G)\otimes L)
    \to \mathbb{T}_{(E_G, \phi_G)}^{1}\to 0.
\end{align}

For an algebraic group homomorphism $G_1 \to G_2$, there is a natural 
morphism 
\begin{align*}
    (-)\times^{G_1}G_2 \colon \Hig_{G_1}^L\to \Hig_{G_2}^L
    \end{align*}
    given by 
    \begin{align*}
    (E_{G_1}, \phi_{G_1}^L) \mapsto (E_{G_1}\times^{G_1}G_2, \phi_{G_1}\times^{G_1}G_2).
\end{align*}
Here $\phi_{G_1}\times^{G_2}G_2$
is given by the natural map of vector bundles
\begin{align*}
    (-)\times^{G_1} G_2 \colon \mathrm{Ad}_{\mathfrak{g}_1}(E_{G_1}) \to \mathrm{Ad}_{\mathfrak{g}_2}(E_{G_2}).
\end{align*}
Here $\fg_i$ is the Lie algebra of $G_i$.

\subsection{Semistable Higgs bundles}\label{subsec:ssH}
Suppose that $G$ is reductive with fixed maximal torus and Borel subgroups $T\subset B\subset G$. Let $P\subset G$ be a standard parabolic with Levi quotient $P\twoheadrightarrow M$. 
There is an associated diagram 
\begin{align*}
    \xymatrix{
\Hig_P^L \ar[r]^-{(-)\times^P G} \ar[d]_-{(-)\times^P M} & \Hig_G^{L} \\
\Hig_M^L.  &
}
\end{align*}
For an $L$-twisted $G$-Higgs bundle $(E_G, \phi_G^{L})$, a \textit{$P$-reduction} is an $L$-twisted $P$-Higgs bundle 
$(E_P, \phi_P^L)$ such that $(E_P, \phi_P^L)\times^P G \cong(E_G, \phi_G^{L})$. 

Let 
\begin{align}\label{pcones}
    N(T)_{\mathrm{pos}} \subset N(T), \ N(T)_{\mathbb{Q},\mathrm{pos}}\subset N(T)_{\mathbb{Q}}
\end{align}
be the monoid generated by positive simple coroots in $N(T)$, 
and the cone generated by them, respectively. 
There is a partial order $\leq_G$ on $N(T)_{\mathbb{Q}}$ by setting 
\begin{align*}\mu' \leq_G \mu \ \stackrel{\mathrm{def}}{\leftrightarrow} \ \mu-\mu'\in N(T)_{\mathbb{Q}, \mathrm{pos}}.
\end{align*}
The semistable $L$-twisted $G$-Higgs bundles are defined similarly to that of principal $G$-bundles 
in~\cite{SchHN}: 
\begin{defn}\label{def:Hsemi}
An $L$-twisted $G$-Higgs bundle $(E_G, \phi_G^{L}) \in \Hig_G^{L}(\chi)$ is called \textit{semistable} if, for any 
standard parabolic $P\subset G$ and a $P$-reduction 
$(E_P, \phi_P^L) \in \Hig_P^L$ with
$(E_P, \phi_P^L)\times^P M \in \Hig_M^L(\chi_M)$ for $\chi_M \in 
\pi_1(M)$, we have 
\begin{align*}
    \mu_M(\chi_M) \leq_G \mu_G(\chi). 
\end{align*}
    
\end{defn}

We denote by 
\begin{align*}
    \Hig_G^{L}(\chi)^{\mathrm{ss}} \subset \Hig_G^{L}(\chi)
\end{align*}
the open substack of semistable $L$-twisted $G$-Higgs bundles. 

\subsection{Harder-Narasimhan filtrations}\label{subsec:HNhiggs}
Similarly to the case of Higgs bundles for $G=\mathrm{GL}_r$, there is also the 
notion of Harder-Narasimhan filtrations for any reductive $G$. 
\begin{prop}\emph{(\cite{ArPa})}
For an $L$-twisted $G$-Higgs bundle $(E_G, \phi_G^{L})$, there is 
a standard parabolic $P\subset G$ with Levi quotient $M$, 
and $P$-reduction $(E_P, \phi_P^L)$
such that 
\begin{itemize}
    \item the $L$-twisted $M$-Higgs bundle  
    $(E_M, \phi_M^L)$ is semistable;
    \item we have $\mu_M(\chi_M) \in N(T)^{W_M}_{\mathbb{Q}+}$, i.e. it is strictly $W$-dominant.
\end{itemize}
\end{prop}

The pair of $P$ and the $P$-reduction $(E_P, \phi_P^L)$ is uniquely determined 
up to isomorphism. The $L$-twisted $P$-Higgs bundle $(E_P, \phi_P^L)$ is called 
\textit{the Harder-Narasimhan filtration} of $(E_G, \phi_G^{L})$,
and the element $\mu_M(\chi_M)\in N(T)^{W_M}_{\mathbb{Q}+}$
is called the \textit{HN type} of $(E_G, \phi_G^{L})$. 

We denote by 
\begin{align*}
    \Hig_G^{L}(\chi)_{(\mu)} \subset \Hig_G^{L}(\chi)
\end{align*}
the locally closed substack consisting of $L$-twisted $G$-Higgs bundles with 
HN type $\mu$. We have the stratification 
\begin{align*}
    \Hig_G^{L}(\chi)=\coprod_{\mu \in N(T)_{\mathbb{Q}+}}\Hig_G^{L}(\chi)_{(\mu)}
\end{align*}

The above stratification is obtained as a $\Theta$-stratification~\cite{HalpTheta, Halphiggs, halpinstab} (see~\cite[Section~1.0.6]{Halphiggs}).
In particular, there is a total order $\preceq$ on $N(T)_{\mathbb{Q}+}$ such that 
\begin{align}\label{higgs:strata}
\Hig_G^{L}(\chi)_{\preceq \mu}:=\coprod_{\mu'\preceq \mu}\Hig_G^{L}(\chi)_{(\mu')}
\end{align}
is a quasi-compact open substack of $\Hig_G^{L}(\chi)$ (in particular, it is a finite union).
Moreover, the substack 
\begin{align*}
    \Hig_G^{L}(\chi)_{(\mu)} \subset 
\Hig_G^{L}(\chi)_{\preceq \mu}
\end{align*}
is a closed $\Theta$-stratum. 

As we have seen above, the stack $\Hig^L_G(\chi)$ \textit{without stability} is not
quasi-compact. 
However, it is covered by quasi-compact open substacks 
\begin{align}\label{cov:open}
    \Hig^L_G(\chi)=\bigcup_{\mu\in N(T)_{\mathbb{Q}+}}
    \Hig^L_G(\chi)_{\preceq \mu}.
\end{align}

\begin{remark}\label{rmk:G=GLr}
    In the case of $G=\GL_r$, an $L$-twisted $G$-Higgs bundle is the same as a pair 
    \begin{align*}
        E=(F, \phi), \ \phi \colon F \to F\otimes L
    \end{align*}
    where $F$ is a rank $r$ vector bundle on $C$ and $\phi$ is a morphism. 
    It is semistable if and only if, for any Higgs subbundle 
    $(F', \phi') \subset (F, \phi)$, we have 
    \begin{align*}\mu(F') \leq \mu(F), \ 
    \mu(F):=\frac{\deg(F)}{\mathrm{rank}(F)}.
    \end{align*}
    Similarly, a Harder-Narasimhan filtration is 
    \begin{align*}
        0=F_0 \subset F_1 \subset \cdots \subset F_k=F
    \end{align*}
    such that $\phi(F_i) \subset F_i \otimes \Omega_C$ and $(F_i/F_{i-1}, \phi|_{F_i})$ are
    semistable Higgs bundles such that
    \begin{align*}\mu^{\mathrm{max}}(F, \theta):=\mu(F_1)>\mu(F_2/F_1)>\cdots>\mu(F_k/F_{k-1}).
    \end{align*}
\end{remark}

\begin{remark}\label{rmk:HN}
For each $c\geq 0$, let 
\begin{align}\label{higgsc}
    \Hig_G^{L}(\chi)_{\leq c} \subset \Hig_G^{L}(\chi)
\end{align}
be the substack of $L$-twisted $G$-Higgs bundles $(E_G, \phi_G^{L})$ such that 
for any its $P$-reduction $(E_P, \phi_P^L)$ we have $\deg (E_P\times^P \Ad(\fp)) \leq c$.
One can show that (\ref{higgsc}) is an open substack such that 
\begin{align*}
    \Hig_G^{L}(\chi)_{\leq c} \setminus \Hig_G^{L}(\chi)_{<c} 
\end{align*}
is a finite disjoint union of open and closed HN strata. 

The above fact is discussed in~\cite[page 73]{Behthesis} for moduli stacks of $G$-bundles and it is straightforward 
to generalize the argument in loc. cit. to $L$-twisted $G$-Higgs bundles. 
Moreover, we have the open covering
\begin{align*}
    \Hig_G^{L}(\chi)=\bigcup_{c\geq 0}\Hig_{G}^L(\chi)_{\leq c}.
\end{align*}
The existence of a total order $\preceq$ on $N(T)_{\mathbb{Q}+}$ in (\ref{higgs:strata}) follows from the above observation. 
\end{remark}
\begin{remark}\label{rmk:order:Bun}
In~\cite{SchHN}, an explicit description of an order of the HN stratification of 
$\Bun_G$ is obtained. Indeed, it is the same as
the partial order $\leq_G$. We expect an analogue of it for $\Hig_G^{L}$, but in loc. cit. 
the author uses the Drinfeld compactification $\overline{\Bun}_P \to \Bun_G$
which is not available for $L$-twisted $G$-Higgs bundles. 
\end{remark}

\subsection{Embedding into smooth stacks}\label{subsec:embL}
Let $L=\Omega_C(D)$ for an effective divisor $D$ on $C$. 
The inclusion $\Omega_C \hookrightarrow L$ induces a closed immersion 
\begin{align}\label{iota:D}
    \iota_D \colon \Hig_G(\chi) \hookrightarrow \Hig^L_G(\chi)
\end{align}
by sending $(E_G, \phi_G)$ to $(E_G, \phi_G^{L})$, where $\phi_G^{L}$ is
the image of $\phi_G$ under the map 
\begin{align*}
\mathrm{Ad}_{\mathfrak{g}}(E_G)\otimes \Omega_C \hookrightarrow 
\mathrm{Ad}_{\mathfrak{g}}(E_G)\otimes L.
\end{align*}
\begin{lemma}\label{lem:cart}
For any $\mu\in N(T)_{\mathbb{Q}+}$, 
the morphism (\ref{iota:D}) restricts to the closed immersion 
\begin{align}\label{iota:DC}
\iota_D \colon \Hig_G(\chi)_{\preceq \mu} \hookrightarrow 
\Hig_G^{L}(\chi)_{\preceq \mu}
\end{align}
such that 
there is a Cartesian square: 
\begin{align}\notag
\xymatrix{
\Hig_G(\chi)_{\preceq\mu} \inclusion^-{\iota_D} \ar@<-0.3ex>@{^{(}->}[d] \diasquare& 
\Hig^L_G(\chi)_{\preceq \mu} \ar@<-0.3ex>@{^{(}->}[d] \\
\Hig_G(\chi) \inclusion^-{\iota_D} & 
\Hig^L_G(\chi).
}
\end{align}
\end{lemma}
\begin{proof}
For $\mu \in N(T)_{\mathbb{Q}+}$, it 
is enough to show that the following diagram is Cartesian 
\begin{align}\label{dia:iota2}
\xymatrix{
\Hig_G(\chi)_{(\mu)} \inclusion^-{\iota_D} \ar@<-0.3ex>@{^{(}->}[d]\diasquare & 
\Hig^L_G(\chi)_{(\mu)} \ar@<-0.3ex>@{^{(}->}[d] \\
\Hig_G(\chi) \inclusion^-{\iota_D} & 
\Hig^L_G(\chi).
}
\end{align}

Note that, in general, an $L$-twisted $G$-Higgs bundle $(E_G, \phi_G^{L})$ comes 
    from the image of $\iota_D$ if and only if $\phi_G^{L}|_{D}=0$. 
Suppose that $(E_G, \phi_G^{L})$ admits a $P$-reduction $(E_P, \phi_P^L)$. 
Then, from $\phi_G^{L}|_{D}=0$, we have $\phi_P^L|_{D}=0$, therefore 
$(E_P, \phi_P^L)$ comes from the image of $\iota_D$. The above fact easily implies that if $(E_G, \phi_G^{L})=\iota_D(E_G, \phi_G)$, then
the HN filtration of $(E_G, \phi_G)$ is given by $\iota_D(E_P, \phi_P)$
for the HN filtration $(E_P, \phi_P)$ of $(E_G, \phi_G)$. 
Therefore, the diagram (\ref{dia:iota2}) is Cartesian. 
\end{proof}

\begin{lemma}\label{lem:smooth}
For a fixed $\mu \in N(T)_{\mathbb{Q}+}$, there is $l(\mu)\in \mathbb{Z}$ such that, 
if
$\deg D>l(\mu)$, then   
the stack $\Hig^L_G(\chi)_{\preceq \mu}$ is smooth. 
\end{lemma}
\begin{proof}
Let $(E_G, \phi_G^{L})$ be an $L$-twisted $G$-Higgs bundle. 
It is enough to show that 
$\mathbb{T}^1_{(E_G, \phi_G^{L})}=0$. 
We have 
\begin{align*}
H^1(C, \Ad_{\mathfrak{g}}(E_G)\otimes L) =H^0(C, \Ad_{\mathfrak{g}}(E_G)\otimes\mathcal{O}_C(-D))^{\vee}
\end{align*}
which vanishes for $\deg D\gg 0$ by the boundedness of 
objects in $\Hig_G^{L}(\chi)_{\preceq \mu}$. 
Therefore, from the exact sequence (\ref{exact}),
we have $\mathbb{T}^1_{(E_G, \phi_G^{L})}=0$ for $\deg D \gg 0$. 
\end{proof}

The following lemma is well-known at least in the case of $G=\GL_r$. We will 
give its proof for any reductive $G$ in Subsection~\ref{subsec:lem:smooth2}:
\begin{lemma}\label{lem:smooth2}
Suppose that $l>l(\mu)$ as in Lemma~\ref{lem:smooth} so that 
$\Hig^L_G(\chi)_{\preceq \mu}$ is smooth. Then we can write 
\begin{align}\notag
\Hig^L_G(\chi)_{\preceq \mu}=U_{\preceq \mu}^L/G'
\end{align}
where $U_{\preceq \mu}^L$ is a smooth quasi-projective scheme and $G'$ is a reductive 
algebraic group which acts on $U_{\preceq \mu}^L$. 
\end{lemma}

Let $(\mathscr{F}_G^{L}, \varphi^L_G)$ be the universal 
Higgs bundle (\ref{univ:G}). We define
\begin{align}\label{def:V}
\mathcal{V}^L:=p_{\mathrm{H}\ast}
\Ad_{\mathfrak{g}}(\mathscr{F}_{G}^L|_{D\times \Hig_G^{L}})
\to \Hig^L_G
\end{align}
which is a vector bundle on the smooth stack 
$\Hig^L_G(\chi)_{\preceq \mu}$ in Lemma~\ref{lem:smooth}. 
Here $p_{\mathrm{H}}^{} \colon D \times \Hig^L_G \to \Hig^L_G$ 
is the projection. 
There is a section $s^L$ over $\Hig^L_G(\chi)_{\preceq \mu}$
\begin{align}\label{section:sL}
    \xymatrix{
\mathcal{V}^L \ar[r] & \ar@/_20pt/[l]^-{s^L}\Hig^L_G(\chi)_{\preceq \mu}    
    }
\end{align}
given by $(E_G, \phi_G^{L}) \mapsto \phi_G^{L}|_{D}$.

\begin{lemma}\label{lem:dstack}
We have the equivalence of derived stacks 
\begin{align}\notag
   \Hig_G(\chi)_{\preceq \mu} \stackrel{\sim}{\to} (s^L)^{-1}(0).   
\end{align}
Here the right-hand side is the derived zero locus of $s^L$
on $\Hig_G^{L}(\chi)_{\preceq \mu}$. 
\end{lemma}
\begin{proof}
   From the construction of $s^L$, there is a natural 
   morphism of derived stacks 
   \begin{align}\label{nat:higgs}
  \Hig_G(\chi)_{\preceq \mu} \to (s^L)^{-1}(0).
   \end{align}
   Since an $L$-twisted $G$-Higgs bundle $(E_G, \phi_G^{L})$ comes from $\iota_D$ in (\ref{iota:D})
   if and only if $\phi_G^{L}|_{D}=0$, the above 
   morphism induces a bijection on isomorphism classes on 
   $k$-valued points. 
   In order to show that it is an equivalence of derived stacks, 
   it is enough to show that the morphism of tangent 
   complexes induced by (\ref{nat:higgs}) is a quasi-isomorphism.  
   This follows from Lemma~\ref{lem:dtr} below, as the 
   tangent complex of $(s^L)^{-1}(0)$ is given by 
   the fiber of the last arrow in (\ref{dtr:E}). 
\end{proof}

For a $G$-Higgs bundle $(E_G, \phi_G)$, let $(E_G, \phi_G^{L})$ be the 
corresponding $L$-twisted $G$-Higgs bundle via the map (\ref{iota:D}).
We have used the following lemma to identify tangent complexes in Lemma~\ref{lem:dstack}:
\begin{lemma}\label{lem:dtr}
There is a distinguished triangle 
\begin{align}\label{dtr:E}
\mathbb{T}_{\Hig_G}|_{(E_G, \phi_G)}\to 
\mathbb{T}_{\Hig_G^{L}}|_{(E_G, \phi_G^{L})} \to \Gamma(C, \Ad_{\mathfrak{g}}(E_G)|_{D}\otimes L).
\end{align}
\end{lemma}
\begin{proof}
    The lemma follows by taking the cones of the following diagram 
    \begin{align*}
\xymatrix{
\Gamma(\Ad_{\mathfrak{g}}(E_G))
\ar[r] \ar[d]_-{\id} & \Gamma(\Ad_{\mathfrak{g}}(E_G)\otimes \Omega_C) \ar[r] \ar[d] & \mathbb{T}_{\Hig_G}|_{(E_G, \phi_G)}\ar[d] \\
\Gamma(\Ad_{\mathfrak{g}}(E_G)) \ar[r] &   \Gamma(\Ad_{\mathfrak{g}}(E_G)\otimes L)\ar[r] & \mathbb{T}_{\Hig_G^{L}}|_{(E_G, \phi_G^{L})}. 
}        
    \end{align*}
    Here the middle vertical arrow is induced by $\Omega_C \hookrightarrow L=\Omega_C(D)$, whose cone is 
    $\Gamma(C, \Ad_{\mathfrak{g}}(E_G)|_{D}\otimes L)$. 
\end{proof}

For $\mu \in N(T)_{\mathbb{Q}+}$, consider the quasi-compact open substack 
\begin{align}\label{stack:U}
\mathcal{U} =\Hig_G(\chi)_{\preceq \mu}
\subset \Hig_G(\chi)
\end{align}
and 
an effective divisor $D$ on $C$ such that 
$\deg (D)>l(\mu)$ as in Lemma~\ref{lem:smooth}.
There is a closed immersion 
\begin{align}\label{emb:jU}
i \colon \mathcal{U} \stackrel{\iota_D}{\hookrightarrow} 
    \mathcal{U}^L:=\Hig^L_G(\chi)_{\preceq \mu}
\end{align}
for $L=\Omega_C(D)$, where $\iota_D$ is defined as in (\ref{iota:DC}).
As in (\ref{deltaw}), we take 
\begin{align*}
    \delta=\delta_w \in \mathrm{Pic}(\Bun_G(\chi))_{\mathbb{Q}}.
\end{align*}
We often omit $w$ and just write it $\delta$. 
We also use the same symbol $\delta=\delta_w$ for its pull-back to 
$\Hig^L_G(\chi)$. 
By Lemma~\ref{lem:emb}, we have the following description of $\LL(\mathcal{U})_{w}$:

\begin{lemma}\label{defn:QB}
The subcategory
\begin{align}\label{def:QBH}
\LL(\mathcal{U})_{\delta=\delta_w} \subset \Coh(\mathcal{U})
\end{align}
consists of objects $\mathcal{E}$ such that, for all $\nu \colon \bgm \to \mathcal{U}$,
we have 
\begin{align}\notag
\wt(\nu^{\ast}i^{\ast}i_{\ast}\mathcal{E}) \subset 
\left[\frac{1}{2}c_1(\nu^{\ast}(\mathbb{L}_{\V^L}|_{\U})^{<0}), 
\frac{1}{2}c_1(\nu^{\ast}(\mathbb{L}_{\V^L}|_{\U})^{>0}) \right]
    +\wt(\nu^{\ast}\delta_w). 
\end{align}
\end{lemma}

\subsection{\texorpdfstring{A $\Theta$-stratum on moduli stacks of Higgs bundles}{A Theta-stratum on moduli stacks of Higgs bundles}}\label{subsec:ThetaH}
Let $\mathcal{U}=\Hig_G(\chi)_{\preceq \mu}$ be as in (\ref{stack:U}). 
Consider the  $\Theta$-stratum
\begin{align*}
    \mathcal{S}=\Hig_G(\chi)^{(\mu)} \subset \mathrm{Map}(\Theta, \mathcal{U})
\end{align*}    
which fits into the following natural diagram, where $\mathcal{Z} \subset \mathrm{Map}(\bgm, \mathcal{U})$ is the center of $\mathcal{S}$:
\begin{align}\label{derived:theta}
\xymatrix{
\mathcal{S} \inclusion^-{p} \ar[d]^-{q} & \mathcal{U}. \\
\mathcal{Z} \ar[ur] \ar@/^10pt/[u] &
}
\end{align}
Here, the map $\zZ \to \mathcal{S}$ is induced by the projection $\Theta=\mathbb{A}^1/\mathbb{G}_m\to \bgm$.
The center $\mathcal{Z}$ is given by 
     \begin{align}\notag
        \zZ=\coprod_{a_M(\chi_M)=\chi, \mu_M(\chi_M)=\mu}\Hig_M(\chi_M)^{\mathrm{ss}}.
    \end{align}  
    The right side is after standard parabolic $P\subset G$ with Levi quotient $M$, 
    which is a finite disjoint union. 
    In the argument below, without loss of generality it is enough to assume that there is only one
    component in the right-hand side, and we write $\mathcal{Z}=\Hig_M(\chi_M)^{\mathrm{ss}}$. 
    From the definition of HN filtrations, we have the Cartesian square 
    \begin{align}\label{center:Z}
        \xymatrix{
        \mathcal{S} \inclusion \ar[d] \diasquare& \Hig_P \ar[d] \\
        \mathcal{Z}=\Hig_M(\chi_M)^{\mathrm{ss}} \inclusion & \Hig_M.
        }
    \end{align}
    
We take the embedding $\mathcal{U} \hookrightarrow \mathcal{U}^L$ as in (\ref{emb:jU}), 
which satisfies that \[\U^L \cap \Hig_G(\chi)=\U.\] 
From Lemma~\ref{lem:cart} and its proof, 
there is a $\Theta$-stratum $\sS^L \subset \U^L$ with center 
$\zZ^L$, complement $\U^L_{\circ}$ and the commutative diagram 
\begin{align}\label{dia:strata}
\xymatrix{
\U_{\circ} \ar@<-0.3ex>@{^{(}->}[d] \inclusion\diasquare & \U \ar@<-0.3ex>@{^{(}->}[d]\diasquare
& \sS \ar@<0.3ex>@{_{(}->}[l]_-{p} \ar@<-0.3ex>@{^{(}->}[d] \ar[r]^-{q} & \zZ \ar@<-0.3ex>@{^{(}->}[d] \\
\U_{\circ}^L \inclusion &  \U^L  & \sS^L \ar@<0.3ex>@{_{(}->}[l]_-{p^L} \ar[r]^-{q^L} &\zZ^L
}
\end{align}
where the left and middle squares are Cartesian. 
The stack $\sS^L$ is a Harder-Narasimhan stratum of $L$-twisted $G$-Higgs bundles with type $\mu$
as in (\ref{center:Z}). 

\begin{remark}\label{rmk:shrink}
Let $\mathcal{Z} \subset \mathcal{Z}^{L'} \subset \mathcal{Z}^L$ be an open neighborhood of 
$\mathcal{Z}$, and set $\mathcal{S}^{L'}=(q^L)^{-1}(\mathcal{Z}^{L'})$
and $\mathcal{U}^{L'}=\mathcal{U}_{\circ}^{L} \cup \mathcal{S}^{L'}$. 
Then \[\mathcal{U}^{L'} \subset \Hig_G^{L}(\chi)_{\preceq \mu}\] is open with 
$\mathcal{U}^{L'} \cap \Hig_G(\chi)_{\preceq \mu}=\mathcal{U}$, 
$\mathcal{S}^{L'} \subset \mathcal{U}^{L'}$ is also $\Theta$-stratum, and 
we have also the diagram (\ref{dia:strata}) after replacing 
$(\mathcal{U}^L, \mathcal{S}^L, \mathcal{Z}^L)$ with $(\mathcal{U}^{L'}, \mathcal{S}^{L'}, \mathcal{Z}^{L'})$. In what follows, we sometimes shrink the open 
neighborhood $\mathcal{U} \subset \mathcal{U}^L$ by a replacement as above. 
\end{remark}

\subsection{\texorpdfstring{A $\Theta$-stratum of $(-1)$-shifted cotangent}{A Theta-stratum of (-1)-shifted cotangent}}\label{subsec:thetashift}
We take the $(-1)$-shifted cotangent of $\Hig_G(\chi)$
\begin{align}\label{M:shift}
    \Omega_{\Hig_G(\chi)}[-1]. 
\end{align}
It is a $(-1)$-shifted symplectic derived stack~\cite{PTVV}, which is described as 
a global critical locus locally on $\Hig_G(\chi)$ as follows. 
Let $\mathcal{U}$ be a quasi-compact open substack as in (\ref{stack:U}), 
and take an embedding $\U \hookrightarrow \U^L$ as in (\ref{emb:jU}). 
We use the same notation 
$\V^L \to \U^L$ for the restriction of the vector bundle (\ref{def:V}) to 
$\U^L$. 

Let $f$ be the function 
\begin{align}\label{funct:f}
    f \colon (\V^L)^{\vee} \to \mathbb{A}^1, \ f(x, v)=\langle s^L(x), v \rangle
\end{align}
where $s^L$ is the section (\ref{section:sL}), $x\in \mathcal{U}^L$ and 
$v\in (\mathcal{V}^L|_{x})^{\vee}$. 
We have the commutative diagram 
\begin{align}\label{dia:crit}
\xymatrix{
\mathrm{Crit}(f) \ar[d]_-{\pi_{\mathrm{crit}}} \inclusion & (\V^L)^{\vee} \ar[d]^-{\pi^L} \ar[r]^-{f} & \mathbb{C} \\
\U \inclusion & \U^L. &  
}
\end{align}
The critical locus $\mathrm{Crit}(f)$ gives a critical locus 
description of the $(-1)$-shifted cotangent 
$\Omega_{\Hig_G(\chi)}[-1]$ over $\mathcal{U}$, see~\cite[Section~2.1.1]{T}
\begin{align}\label{MXcrit}
\Omega_{\Hig_G(\chi)}[-1]^{\mathrm{cl}}\times_{\Hig_G(\chi)}\U \cong \mathrm{Crit}(f). 
\end{align}

Below we give a $\Theta$-stratum of $\mathrm{Crit}(f)$. 
Note that the stack $(\mathcal{V}^L)^{\vee}$ 
consists of tuples
\begin{align}\label{tuple}
(E_G, \phi_G^{L}, \psi_G^{L}), \ \phi_G^{L} \in H^0(C, \Ad_{\mathfrak{g}}(E_G)\otimes L)
\end{align}
where $(E_G, \phi_G^{L})$ is an $L$-twisted $G$-Higgs bundle in $\U^L$, and 
$\psi_G^{L}$ is an element 
\begin{align*}
    \psi_G^{L} \in \Gamma(\Ad_{\mathfrak{g}}(E_G)|_{D} \otimes L)^{\vee} \cong \Gamma(\Ad_{\mathfrak{g}}(E_G)|_{D})
\end{align*}
which corresponds to a point in the fiber of $(\mathcal{V}^L)^{\vee} \to \mathcal{U}^L$ at $(E_G, \phi_G^{L})$. 
The above isomorphism follows from Serre duality for $D$, noting that 
$L|_{D}=\Omega_C(D)|_{D}\cong \omega_D$. 

The closed substack 
\begin{align*}(\pi^L)^{-1}(\sS^L) \subset (\V^L)^{\vee}
\end{align*}
corresponds to tuples (\ref{tuple})
such that $(E_G, \phi_G^{L})$ admits a (unique) $P$-reduction $(E_P, \phi_P^L)$ 
and the 
associated $L$-twisted $M$-Higgs bundle satisfies
\begin{align*}
(E_M, \phi_M^L)=(E_P, \phi_P^L)\times^P M \in \Hig_M^L(\chi_M)
\end{align*}
where $\chi_M$ is as in (\ref{center:Z})).
The pair $(E_M, \phi_M^L)$ is automatically a semistable $L$-twisted $M$-Higgs bundle 
due to the condition $(E_G, \phi_G^{L}) \in \U^L$, 
as otherwise it lies in a deeper Harder-Narasimhan stratum.

We define the closed substack 
\begin{align}\label{SV}
\sS_{\V} \subset (\pi^L)^{-1}(\sS^L)
\end{align}
corresponding to (\ref{tuple}), which furthermore 
satisfy the condition
$\psi_G^{L} \in \Gamma(\Ad_{\mathfrak{p}}(E_P)|_{D})$.

\begin{lemma}\label{lem:theta:crit}
The closed substack (\ref{SV}) is a $\Theta$-stratum of $(\V^L)^{\vee}$. 
\end{lemma}
\begin{proof}
Recall that $\Theta=\mathbb{A}^1/\mathbb{G}_m$ where $\mathbb{G}_m$ acts on 
$\mathbb{A}^1$ by weight one. 
As in~\cite[Lemma~6.3.1]{halpinstab}, the 
stack $\mathrm{Map}(\Theta, (\V^L)^{\vee})$ classifies
standard parabolic $P'\subset G$ together with $L$-twisted $P'$-Higgs bundle 
$(E_{P'}, \phi_{P'}^L)$ and $\psi_{P'}^L \in \Gamma(\Ad_{\mathfrak{p}'}(E_{P'})|_{D})$
such that, by setting $(E_{M'}, \phi_{M'}^G)$ to be the associated $M'$-bundle for the 
Levi quotient $M'$ we have 
\begin{align}\label{cond:theta}
    (E_{P'}, \phi_{P'}^L) \times^{P'} G \in \mathcal{U}^L, \ (E_{M'}, \phi_{M'}^L)\times^{M'} G \in \mathcal{U}^L.
\end{align}

Let 
\begin{align*}\mathcal{S}_{\V}' \subset \mathrm{Map}(\Theta, (\V^L)^{\vee})
\end{align*}
be an open and closed substack corresponding to $P'=P$, where $P\subset G$ is the parabolic subgroup for the $\Theta$-stratum 
$\Hig_G(\chi)^{(\mu)}$ as in (\ref{center:Z}), and 
$P$-reduction $(E_P, \phi_P^L)$ such that 
$(E_M, \phi_M^L)$ has type $\chi_M$ in (\ref{center:Z}). 
Then the evaluation map at $1 \in \mathbb{A}^1/\mathbb{G}_m$ gives a 
morphism 
\begin{align}\label{map:SV}\sS_{\V}' \to \sS_{\V}.
\end{align}

The map (\ref{map:SV}) 
is an isomorphism since for any point 
$(E_G, \phi_G^{L}, \psi_G^{L})$
in $\sS_{\V}$, the $L$-twisted $G$-Higgs bundle $(E_G, \phi_G^{L})$ admits a unique $P$-reduction
$(E_P, \phi_P^L)$ 
satisfying (\ref{cond:theta})
because of the uniqueness of 
the Harder-Narasimhan filtration of $L$-twisted $G$-Higgs bundles and the definition of $\sS_{\V}$.
Therefore $\sS_{\V}$ is a $\Theta$-stratum of $(\V^L)^{\vee}$. 
\end{proof}

A $\Theta$-stratum of $\mathrm{Crit}(f)$ is given in the following proposition. 
\begin{prop}\label{lem:theta:crit2}
We have the identities 
\begin{align*}
    \sS_{\mathrm{crit}}:=\pi_{\mathrm{crit}}^{-1}(\sS)=
    \mathrm{Crit}(f)\cap (\pi^L)^{-1}(\sS^L)=\mathrm{Crit}(f) \cap \sS_{\V}. 
\end{align*}
Here $\pi_{\mathrm{crit}}$ is given in the diagram (\ref{dia:crit}). 
Moreover $\sS_{\mathrm{crit}}$ is a 
$\Theta$-stratum of $\mathrm{Crit}(f)$. 
\end{prop}
\begin{proof}
The first identity follows from the diagrams (\ref{dia:strata}), (\ref{dia:crit}). 
Below we prove the last identity. It is enough to show that 
\begin{align}\label{incl:crit}
\pi_{\mathrm{crit}}^{-1}(\sS) \subset \mathrm{Crit}(f) \cap \sS_{\V}. 
\end{align}

    For a point $p \in \sS \subset \U$ corresponding 
    to a $G$-Higgs bundle $(E_G, \phi_G)$, the fiber of 
    $\pi_{\mathrm{crit}} \colon \mathrm{Crit}(f) \to \U$ at $p$ is 
    \begin{align*}
(\pi_{\mathrm{crit}})^{-1}(p)=(\mathbb{T}_{(E_G, \phi_G)}^1)^{\vee}=\mathbb{T}_{(E_G, \phi_G)}^{-1}
    \end{align*}
    by (\ref{M:shift}) and (\ref{MXcrit}). 
Note that we have 
\begin{align*}
    \mathbb{T}_{(E_G, \phi_G)}^{-1}=\Ker(H^0(\Ad_{\fg}(E_G)) \stackrel{[\phi_G, -]}{\to}
    H^0(\Ad_{\fg}(E_G)\otimes \Omega_C)).
\end{align*}

The fiber of $\pi^L \colon (\V^L)^{\vee} \to \U^L$ 
at $p \in \U \subset \U^L$ is 
\begin{align*}
(\pi^L)^{-1}(p)=\Gamma(\Ad_{\fg}(E_G)|_{D}\otimes L)^{\vee}=
\Gamma(\Ad_{\fg}(E_G)|_{D}).
\end{align*}
   The embedding $\mathrm{Crit}(f) \subset (\V^L)^{\vee}$ 
   at the fibers over $p$ is given by 
   the canonical map 
   \begin{align}\label{compose:Ad}
       \mathbb{T}_{(E_G, \phi_G)}^{-1} \hookrightarrow H^0(\Ad_{\fg}(E_G))
       \to \Gamma(\Ad_{\fg}(E_G)|_{D}).
   \end{align}
   
   As $p\in \sS$, the $G$-Higgs bundle $(E_G, \phi_G)$ admits a 
   HN filtration $(E_P, \phi_P)$. 
   Since $\mathcal{S}$ is a $\Theta$-stratum of $\mathcal{U}$, 
   we have 
   \begin{align*}
         \mathbb{T}_{(E_P, \phi_P)}^{-1} \stackrel{\cong}{\to} 
         \mathbb{T}_{(E_G, \phi_G)}^{-1}.
   \end{align*}
   Since the left-hand side is a subspace of $H^0(\Ad_{\fp}(E_P))$, 
   the above isomorphism implies that the image of the composition (\ref{compose:Ad})
   factors through $\Gamma(\Ad_{\fp}(E_P)|_{D})\subset \Gamma(\Ad_{\fg}(E_G)|_{D})$.
   Therefore from the definition of $\mathcal{S}_{\mathcal{V}}$, the inclusion (\ref{incl:crit}) holds.  
   \end{proof}

We denote by 
\begin{align*}
    \mathcal{Z}_{\mathcal{V}}\subset \mathrm{Map}(\bgm, (\mathcal{V}^L)^{\vee})
\end{align*}
the center of $\mathcal{S}_{\mathcal{V}}$. 
From the definition of $\sS_{\V}$, the center $\mathcal{Z}_{\mathcal{V}}$ is described 
as follows: let $\V^L_M$ be the vector bundle 
\begin{align}\label{vect:i}
	\V^L_M \to \Hig_M^L(\chi_M)^{\mathrm{ss}}
	\end{align} 
whose fiber over $(E_M, \phi_M^L)$ is $\Gamma(\Ad_{\fm}(E_M)|_{D}\otimes L)$,
where $\fm$ is the Lie algebra of $M$. 
We thus have  
\begin{align}\notag
	\zZ_{\V}=(\V_M^L)^{\vee} \to \zZ^L=\Hig_M^L(\chi_M)^{\mathrm{ss}}.
	\end{align}
As in (\ref{section:sL}) and (\ref{funct:f}), there is a section $s_M^L$ 
of (\ref{vect:i}) such that 
\begin{align}\label{Z:zero}
	\zZ =\Hig_M(\chi_M)^{\mathrm{ss}}
    =(s_M^L)^{-1}(0)
\subset \Hig_M^{L}(\chi_M)^{\mathrm{ss}}.
	\end{align}
    
    Furthermore, we have the function 
    \begin{align*}
        f_{\mathcal{Z}} \colon \mathcal{Z}_{\mathcal{V}} \to \mathbb{A}^1, \ 
        (z, v)\mapsto \langle s_M^L(z), v\rangle
    \end{align*}
where $z\in \mathcal{Z}^L$ and 
$v\in (\mathcal{V}^L_M)^{\vee}|_{z}$. 
As a summary, we have the following commutative diagram 
\begin{align}\label{dia:SV}
	\xymatrix{\mathcal{S}_{\V}\ar[d]_-{q} \inclusion^-{p} & (\mathcal{V}^L)^{\vee} \ar[d]^-{f} & (\V^L)_{\circ}^{\vee}:=(\V^L)^{\vee} \setminus \sS_{\V} \ar@<0.3ex>@{_{(}->}[l] \\
	\mathcal{Z}_{\V}\ar[ur] \ar@/^15pt/[u] \ar[r]^-{f_{\mathcal{Z}}} &	\mathbb{A}^1. &
}
	\end{align}

  \subsection{Window subcategory for the smooth ambient space} 
In this subsection, we investigate a window subcategory of 
 the smooth stack $(\V^L)^{\vee}$. 
We first note the following: 
\begin{lemma}\label{lem:sym:V}
The stack $(\V^L)^{\vee}$ is symmetric in the sense of Definition~\ref{def:sstack}.     
\end{lemma}
\begin{proof}
Since $(\V^L)^{\vee} \to \U^L$ is a vector bundle, 
it is enough to show that
\begin{align}\label{id:dual}
    [\mathbb{T}_{(\V^L)^{\vee}}|_{\U^L}]= [\mathbb{L}_{(\V^L)^{\vee}}|_{\U^L}]
\end{align}
in $K(\U^L)$. Here $\U^L \hookrightarrow (\V^L)^{\vee}$ is the zero section. 

For an $L$-twisted $G$-Higgs bundle $(E_G, \phi_G^{L})$ in $\U^L$, 
we have the following identity in $K(B\mathrm{Aut}(E_G, \phi_G^{L}))$
\begin{align*}
\mathbb{T}_{(\V^L)^{\vee}}|_{(E_G, \phi_G^{L})} 
= \mathbb{T}_{\Hig_G^{L}}|_{(E_G, \phi_G^{L})}+\Gamma(\Ad_{\fg}(E_G)|_{D}\otimes L)^{\vee}.
\end{align*}
The right hand-side is equal to 
\begin{align*}
&\Gamma(\Ad_{\fg}(E_G))[1]+\Gamma(\Ad_{\fg}(E_G)\otimes L)+\Gamma(\Ad_{\fg}(E_G)|_{D}\otimes L)^{\vee} \\
&=\Gamma(\Ad_{\fg}(E_G)\otimes \Omega_C)^{\vee}+\Gamma(\Ad_{\fg}(E_G)\otimes \Omega_C) \\
&\qquad \qquad \qquad \qquad  +\Gamma(\Ad_{\fg}(E_G)|_{D}\otimes L)+
\Gamma(\Ad_{\fg}(E_G)|_{D}\otimes L)^{\vee}.
\end{align*}

The last formula is self-dual. 
By the above equality, the identity (\ref{id:dual}) holds point-wise on $\mathcal{U}^L$. 
It is straightforward 
to extend the above computation for families of sheaves over $\U^L$, which 
shows the identity (\ref{id:dual}). 
\end{proof}
Let $\mathbb{G}_m$ act on the fibers of 
$(\V^L)^{\vee} \to \U^L$ by weight two. We define the following window 
subcategory: 
\begin{defn}\label{def:window:V}
We define 
\begin{align*}
    \mathcal{W}_{\V, \delta} \subset \Coh_{\mathbb{G}_m}((\V^L)^{\vee})
\end{align*}
as the full subcategory consisting of objects $\mathcal{E}$ such that $\mathcal{E}|_{\zZ_{\V}}$ satisfying 
\begin{align*}
    \wt(\E|_{\zZ_{\V}}) \subset 
    \left(-\frac{1}{2}c_1 (N_{\sS_{\V}/(\V^L)^{\vee}}^{\vee}|_{\zZ_{\V}}), 
    \frac{1}{2}c_1 (N_{\sS_{\V}/(\V^L)^{\vee}}^{\vee}|_{\zZ_{\V}})
    \right]+c_1(\delta|_{\zZ_{\V}}). 
\end{align*}
Here, $\wt(-)$ is the set of weights with respect to the decomposition (\ref{decom:Z}). 
\end{defn}
\begin{remark}\label{rmk:nsv}
At each point in $z\in \zZ_{\V}$, there is a canonical map $\nu \colon \bgm \to \zZ_{\V}$
with $\nu(0)=z$ 
by the definition of the center of a $\Theta$-stratum. 
Then we have 
\begin{align*}
    c_1 (N_{\sS_{\V}/(\V^L)^{\vee}}^{\vee}|_{\zZ_{\V}})=n_{(\V^L)^{\vee}, \nu} 
\end{align*}
where $n_{(\V^L)^{\vee}, \nu}$ is given by (\ref{def:nV}). 
    \end{remark}
Let $(\V^L)_{\circ}^{\vee}$ be the complement of the $\Theta$-stratum 
$\mathcal{S}_{\mathcal{V}}$
as in the diagram (\ref{dia:SV}). 
By applying the window theorem (Theorem~\ref{thm:window}) to the $\Theta$-stratum $\mathcal{S}_{\mathcal{V}}\subset (\mathcal{V}^L)^{\vee}$, 
the composition 
\begin{align}\label{comp:func}
\Phi \colon 
    \mathcal{W}_{\V, \delta} \subset \Coh_{\mathbb{G}_m}((\V^L)^{\vee}) \twoheadrightarrow 
    \Coh_{\mathbb{G}_m}((\V^L)^{\vee}_{\circ})
\end{align}
is an equivalence.

Recall from Definition~\ref{defn:QBY} the ($\mathbb{G}_m$-equivariant version of) the magic category 
\begin{align*}
    \MM_{\mathbb{G}_m}((\V^L)^{\vee})_{\delta} \subset \Coh_{\mathbb{G}_m}((\V^L)^{\vee})
\end{align*}
see Subsection~\ref{subsec:Koszuleq} for the notation. 
As in Proposition~\ref{prop:functUp}, we 
also have the functor from the diagram (\ref{dia:SV})
\begin{align}\label{funct:UpM}
    \Upsilon_{\V}=p_{\ast}q^{\ast} \colon \MM_{\mathbb{G}_m}(\zZ_{\V})_{\delta'} \to \MM_{\mathbb{G}_m}((\V^L)^{\vee})_{\delta}
\end{align}
where $\delta'$ is 
\begin{align*}
    \delta'=\left(\delta\otimes (\det N_{\sS_{\V}/(\V^L)^{\vee}})^{1/2})\right)|_{\zZ_{\mathcal{V}}} \in \mathrm{Pic}(\zZ_{\V})_{\mathbb{R}}. 
\end{align*}

The following proposition, which proves Conjecture~\ref{conj:adjoint} in this case, is 
the key result discussed in this section: 

\begin{prop}\label{prop:sod:V}
By replacing $\U^L$ with an open neighborhood of $\U \subset \U^L$ as in Remark~\ref{rmk:shrink} if necessary, 
there is a semiorthogonal decomposition 
\begin{align}\label{sod:V}
\MM_{\mathbb{G}_m}((\V^L)^{\vee})_{\delta}=\langle \Upsilon_{\V} (\MM_{\mathbb{G}_m}(\zZ_{\V})_{\delta'}), \MM_{\mathbb{G}_m}((\V^L)^{\vee})_{\delta} \cap 
\mathcal{W}_{\V, \delta}\rangle. 
\end{align}
Moreover the functor (\ref{funct:UpM}) is fully-faithful and 
(\ref{comp:func}) restricts to an equivalence 
\begin{align}\label{eq:QBV}
\Phi \colon 
\MM_{\mathbb{G}_m}((\V^L)^{\vee})_{\delta} \cap \mathcal{W}_{\V, \delta} \stackrel{\sim}{\to}
\MM_{\mathbb{G}_m}((\V^L)^{\vee}_{\circ})_{\delta}.     
\end{align}
    \end{prop}
\begin{proof}
We prove the proposition by reducing the problem to
the local case, which is proved in Theorem~\ref{thm:qstack}. 
A subtlety of this reduction is that the local model is not necessarily  a 
quotient of symmetric representations of a reductive group. Instead we modify 
the local model to make it a quotient of a symmetric representation, and apply Theorem~\ref{thm:qstack}.
Since the proof is long, we divide into some steps, and postpone some technical lemmas in 
Subsection~\ref{subsec:postpone}.

\begin{step}
An alternative statement for the semiorthogonal decomposition (\ref{sod:V}).
\end{step}
We first give an alternative statement for the semiorthogonal decomposition (\ref{sod:V}). 
By Theorem~\ref{thm:window}, 
there is a semiorthogonal decomposition 
\begin{align}\label{sod:VL}
    \Coh_{\mathbb{G}_m}((\V^L)^{\vee})=\langle \Upsilon_{\V} \Coh_{\mathbb{G}_m}(\zZ_{\V})_{\leq m}, \mathcal{W}_{\V, \delta}, \Upsilon_{\V} \Coh_{\mathbb{G}_m}(\zZ_{\V})_{>m}\rangle
\end{align}
where $m$ is given by 
\begin{align*}
    m:= -\frac{1}{2} c_1 (N_{\sS_{\V}/(\V^L)^{\vee}}^{\vee}|_{\zZ_{\V}})+
    c_1(\delta|_{\zZ_{\V}}). 
\end{align*}
Moreover, the functor $\Upsilon_{\mathcal{V}}$ is fully-faithful on each $\Coh_{\mathbb{G}_m}(\mathcal{Z}_{\mathcal{V}})_i$.
For an object $\E \in \MM_{\mathbb{G}_m}((\V^L)^{\vee})_{\delta}$, 
the semiorthogonal decomposition (\ref{sod:VL}) yields 
a distinguished triangle 
\begin{align}\label{show:tri}
    \E'' \to \E \to \E', \ 
    \E' \in \Upsilon_{\V} \Coh_{\mathbb{G}_m}(\zZ_{\V})_{m}, \ \E'' \in \mathcal{W}_{\V, \delta}. 
\end{align}

Therefore, it is enough to show that 
\begin{align}\label{ets:E}
\E' \in \Upsilon_{\V} \MM_{\mathbb{G}_m}(\zZ_{\V})_{\delta'}, \ \E'' \in \mathcal{W}_{\V, \delta} \cap \MM_{\mathbb{G}_m}((\V^L)^{\vee})_{\delta}.     
\end{align}
It is also enough to show the second condition of (\ref{ets:E}), since then 
we have $\E' \in \MM_{\mathbb{G}_m}((\V^L)^{\vee})_{\delta}$, 
and the first condition follows from Proposition~\ref{prop:functUp}. 
\begin{step}
    Reduction of (\ref{ets:E}) to the local case. 
\end{step}
Note that each point of $\sS$ is obtained as a small deformation of a point $z \in \sS$, 
given by the image of a closed point $z'\in\mathcal{Z}$ under the natural map $\zZ \to \sS$
in the diagram (\ref{derived:theta}). 
It 
corresponds to a $G$-Higgs bundle of the form
\begin{align}\label{deg:E}
(E_G, \phi_G)=(E_M, \phi_M)\times^M G
\end{align}
where $(E_M, \phi_M)$ is a polystable $M$-Higgs bundle, i.e. a
closed point of the stack $\Hig_M(\chi_M)^{\mathrm{ss}}$ where $\chi_M\in \pi_1(M)$ 
is as in (\ref{center:Z}). The point $z$ corresponds to $(E_G, \phi_G)$ and $z'$ 
corresponds to $(E_M, \phi_M)$.
Let 
\begin{align*}
(E_G, \phi_G^{L})=\iota_D(E_G, \phi_G)
\end{align*}
be the corresponding $L$-twisted $G$-Higgs bundle by the map (\ref{iota:D}). 

We set 
    $\mathbf{E}:=\mathbb{T}^0_{(E_G, \phi_G^{L})}$ and 
\begin{align*}
    \mathbf{V}:= \mathbf{E}\oplus \Gamma(\Ad_{\fg}(E_G)|_{D}\otimes L)^{\vee}=
    \mathbf{E}\oplus \Gamma(\Ad_{\fg}(E_G)|_{D}).
\end{align*}
We regard $\mathbf{V}$ as a vector bundle over $\mathbf{E}$ by the natural 
projection. 

Let $G_E$ be the reductive subgroup of $\mathrm{Aut}(E_G, \phi_G)$ given by 
\begin{align*}
G_E :=\Aut(E_M, \phi_M) \subset \mathrm{Aut}(E_G, \phi_G)=\mathrm{Aut}(E_G, \phi_G^{L}).
\end{align*}
Let $G_E$ acts on $\mathbf{E}$ by conjugation. 
There is a smooth map (which follows from~\cite[Theorem~1.1]{AHRLuna})
\begin{align*}
    \widehat{\mathbf{E}}/G_E \to \mathcal{U}^L, \ 0 \mapsto \iota_D(z). 
\end{align*}
Here $\widehat{\mathbf{E}}$ is the pull-back of $\mathbf{E}$ under the 
formal completion 
\begin{align}\label{formal:comp}
    \Spec \widehat{\mathcal{O}}_{\mathbf{E}\ssslash G_E}
    \to  \Spec \mathcal{O}_{\mathbf{E}\ssslash G_E}.
\end{align}

Below we use the notation $\widehat{(-)}$ to denote the pull-back under 
the above morphism. 
Then we have 
\begin{align*}
    \widehat{\mathbf{V}}/G_E=(\V^L)^{\vee}\times_{\U^L}(\widehat{\mathbf{E}}/G_E). 
\end{align*}
\begin{step}
    Modification to the self-dual representation. 
\end{step}
The $G_E$-representation $\mathbf{V}$ is not necessarily self-dual, 
so we modify it to make it self-dual. 
By the distinguished triangle (\ref{dtr:E}), we have the 
morphism 
\begin{align*}
    \Gamma(\mathrm{Ad}_{\fg}(E_G)|_{D}\otimes L) \twoheadrightarrow \mathbb{T}^{1}_{(E_G, \phi_G)}
\end{align*}
which is surjective by $\mathbb{T}^1_{(E_G, \phi_G^{L})}=0$. 
By taking its dual, we obtain the following diagram of $G_E$-representations
\begin{align*}
    \xymatrix{
    \Gamma(\Ad_{\fg}(E_G)|_{D}) \ar@{.>}[rd]^-{\eta}& \\
    \mathbb{T}^{-1}_{(E_G, \phi_G)} \ar@{->>}[r]\ar@<-0.3ex>@{^{(}->}[u] & \mathbb{T}^{-1}_{(E_G, \phi_G)}/\mathbb{T}^{-1}_{(E_M, \phi_M)}.
    }
\end{align*}
Since $G_E$ is reductive, 
there is a 
dotted arrow $\eta$ which is a surjective morphism of $G_E$-representations. 
Let $\Gamma(\Ad_{\fg}(E_G)|_{D})_0$ be the kernel of $\eta$ 
\begin{align*}
\Gamma(\Ad_{\fg}(E_G)|_{D})_0 :=\Ker(\eta\colon  \Gamma(\Ad_{\fg}(E_G)|_{D}) \twoheadrightarrow
\mathbb{T}^{-1}_{(E_G, \phi_G)}/\mathbb{T}^{-1}_{(E_M, \phi_M)}
).
\end{align*}
We set 
\begin{align*}
\mathbf{V}_0 :=\mathbf{E} \oplus \Gamma(\Ad_{\fg}(E_G)|_{D})_0. 
\end{align*}
It is a sub $G_E$-representation of $\mathbf{V}$ and self-dual by Lemma~\ref{lem:symW} below.

Then we obtain the following commutative diagram 
\begin{align}\label{dia:VE}
\xymatrix{
\widehat{\mathbf{V}}_0/G_E \inclusion  \ar@/^15pt/[rr]^-{g}\ar[d] & \widehat{\mathbf{V}}/G_E \ar[d] \ar[r] & (\V^L)^{\vee} \ar[d] \\
\widehat{\mathbf{E}}/G_E \ar@{=}[r] & \widehat{\mathbf{E}}/G_E \ar[r] & \U^L. 
}
\end{align}
Recall that $z\in \mathcal{S}$ corresponds to $(E_G, \phi_G)$ and $z'\in \mathcal{Z}'$ 
corresponds to $(E_M, \phi_M)$ in (\ref{deg:E}). 
The canonical map $\nu \colon \mathbb{G}_m \to \U$ associated with $z'$ is given by 
$\nu(0)=z$ and 
a one-parameter subgroup 
\begin{align}\label{lambda:E}
\lambda \colon \mathbb{G}_m \to G_E \subset 
\mathrm{Aut}(E_G, \phi_G), \ G_E=\Aut(E_G, \phi_G)^{\lambda}.
\end{align}

Then we have the commutative diagram 
\begin{align}\label{dia:WG}
\xymatrix{
\widehat{\mathbf{V}}_0^{\lambda}/G_E \inclusion^-{\cong} & \widehat{\mathbf{V}}^{\lambda}/G_E 
\ar[r]^-{g^{\lambda}} & \zZ_{\V} \\
\widehat{\mathbf{V}}_0^{\lambda \geq 0}/G_E \ar[u]^-{q_0}  \ar@<-0.3ex>@{^{(}->}[d]_-{p_0}  \inclusion  \diasquare & \widehat{\mathbf{V}}^{\lambda \geq 0}/G_E \ar[r] \ar[u]\ar@<-0.3ex>@{^{(}->}[d]  \diasquare& \sS_{\V} \ar@<-0.3ex>@{^{(}->}[d] \ar[u] \\
\widehat{\mathbf{V}}_0/G_E \inclusion \ar@/_15pt/[rr]_-{g}& \widehat{\mathbf{V}}/G_E \ar[r] & (\V^L)^{\vee}. }    
\end{align}
Here we note that $G_E=G_E^{\lambda \geq 0}=G_E^{\lambda}$. As we will prove in Lemma~\ref{lem:cart0}, the bottom 
squares are derived Cartesians and the top left arrow is an isomorphism. 
\begin{step}
Proof of (\ref{sod:V}) via application of the results in Section~\ref{sec:magic}.
\end{step}
By taking the pull-back of the triangle (\ref{show:tri}) by the map $g$ in (\ref{dia:VE}), 
we have the distinguished triangle in $\Coh(\widehat{\mathbf{V}}_0/G_E)$
\begin{align}\label{tr:gE}
    g^{\ast}\mathcal{E}'' \to g^{\ast}\mathcal{E} \to g^{\ast}\mathcal{E}'. 
\end{align}
By Lemma~\ref{lem:neq} and the assumption $\E \in \MM_{\mathbb{G}_m}((\V^L)^{\vee})_{\delta}$, 
we conclude
\begin{align*}
    g^{\ast}\E \in \MM(\widehat{\mathbf{V}}_0/G_E)_{g^{\ast}\delta}.
\end{align*}
The distinguished triangle (\ref{tr:gE}) is the one 
induced by a $\Theta$-stratum
\[\widehat{\mathbf{V}}_0^{\lambda \geq 0}/G_E \hookrightarrow \widehat{\mathbf{V}}_0/G_E\]
as in (\ref{show:tri}). 
Namely, let  
\begin{align*}
    \mathcal{W}_{\widehat{\mathbf{V}}_0, g^{\ast}\delta} \subset \Coh(\widehat{\mathbf{V}}_0/G_E)
\end{align*}
be the window subcategory for the $\Theta$-stratum 
$\widehat{\mathbf{V}}_0^{\lambda\geq 0}/G_E \hookrightarrow \widehat{\mathbf{V}}_0/G_E$, 
consisting of objects whose $\lambda$-weights on 
$\widehat{\mathbf{V}}_0^{\lambda}/G_E$ are contained in 
\begin{align*}
\left(-\frac{1}{2}n_{\mathbf{V}_0/G, \lambda}, 
\frac{1}{2}n_{\mathbf{V}_0/G, \lambda}\right]+\langle \delta, \lambda\rangle. 
\end{align*}
Here we have identified $\lambda$ with a map 
$\bgm \to \mathbf{V}_0/G_E$ 
to the origin given by 
$\lambda \colon \mathbb{G}_m\to G_E=\mathrm{Aut}(0)$. 
Then we have 
\begin{align}\label{cond:g*E}
    g^{\ast}\E' \in \Upsilon_0 \Coh(\widehat{\mathbf{V}}_0^{\lambda}/G_E)_m, \ 
    g^{\ast}\E'' \in \mathcal{W}_{\widehat{\mathbf{V}}_0, g^{\ast}\delta}
\end{align}
where $\Upsilon_0=p_{0\ast}q_0^{\ast}$ in the diagram (\ref{dia:WG}). 
The first condition follows from Lemma~\ref{lem:cart0} together with the base change, 
and the second condition follows from the diagram (\ref{dia:WG}) together with 
Lemma~\ref{lem:neq}. 

We apply Proposition~\ref{prop:QBW} to the stack $\widehat{\mathbf{V}}_0/G_E$, 
which is possible since $\mathbf{V}_0$ is a self-dual representation of the reductive group $G_E$
\footnote{The setting here is slightly different from Proposition~\ref{prop:QBW} as we 
are taking base-change of the formal completion (\ref{formal:comp}). It is straightforward 
to check that the same argument works in this setting.}. 
Then we conclude that 
\begin{align}\label{ets:E2}
    g^{\ast}\mathcal{E}' \in \Upsilon_0 \MM(\widehat{\mathbf{V}}_0^{\lambda}/G_E)_{g^{\lambda\ast}\delta'}, \ 
    g^{\ast}\mathcal{E}'' \in \MM(\widehat{\mathbf{V}}_0/G_E)_{g^{\ast}\delta}. 
\end{align}
Indeed let $(\widehat{\mathbf{V}}_0)_{\circ}=\widehat{\mathbf{V}}_0 \setminus \widehat{\mathbf{V}}_0^{\lambda \geq 0}$. Then we have 
\begin{align*}
    g^{*}\mathcal{E}''|_{(\widehat{\mathbf{V}}_0)_{\circ}/G_E} =g^{*}\mathcal{E}|_{(\widehat{\mathbf{V}}_0)_{\circ}/G_E} \in \MM((\widehat{\mathbf{V}}_0)_{\circ}/G_E)_{g^{\ast}\delta}.
\end{align*}

Therefore by the second condition in (\ref{cond:g*E}) and Proposition~\ref{prop:QBW}, the object $g^* \mathcal{E}''$ satisfies the second condition in (\ref{ets:E2}). 
Then by Proposition~\ref{prop:functUp}, we have 
\begin{align*}
     g^{\ast}\mathcal{E}'\in &\Upsilon_0 \Coh(\widehat{\mathbf{V}}_0^{\lambda}/G_E)_m \cap 
     \MM(\widehat{\mathbf{V}}_0/G_E)_{g^{\ast}\delta} \\
     &=\Upsilon_0 \MM(\widehat{\mathbf{V}}_0^{\lambda}/G_E)_{g^{\lambda\ast}\delta'}.
\end{align*}
Therefore (\ref{ets:E2}) holds. 
By the distinguished triangle (\ref{show:tri}), we also have 
\begin{align*}
\E''|_{(\V^L)^{\vee}_{\circ}} \cong \E|_{(\V^L)^{\vee}_{\circ}} 
\in \MM_{\mathbb{G}_m}((\V^L)_{\circ}^{\vee})_{\delta}.     
\end{align*}
By applying Lemma~\ref{lem:pullback} below,  
the second condition of (\ref{ets:E2}) implies the second 
condition of (\ref{ets:E}), therefore we obtain the semiorthogonal 
decomposition (\ref{sod:V}). 

\begin{step}
Proof of the equivalence (\ref{eq:QBV}).
\end{step}
Finally we show the equivalence (\ref{eq:QBV}).
Let $\Phi$ be the equivalence (\ref{comp:func}).
Obviously we have the inclusion 
\begin{align}\label{incl:phi}
\MM_{\mathbb{G}_m}((\V^L)^{\vee})_{\delta} \cap \mathcal{W}_{\V, \delta}
\subset \Phi^{-1} \MM_{\mathbb{G}_m}((\V^L)^{\vee}_{\circ})_{\delta}.   
\end{align}
It is enough to show that the above inclusion is an equivalence. 
For an object $A$ in the right-hand side of (\ref{incl:phi}), 
by the diagram (\ref{dia:WG}) and Proposition~\ref{prop:QBW}, 
we have 
\begin{align*}
    g^{\ast}A \in
    \MM(\mathcal{W}_{\widehat{\mathbf{V}}_0, g^{\ast}\delta})=
    \MM(\widehat{\mathbf{V}}_0/G_E)_{g^{\ast}\delta} \cap 
    \mathcal{W}_{\widehat{\mathbf{V}}_0, g^{*}\delta}.
\end{align*}
Therefore the object $A$ lies in the left-hand 
side of (\ref{incl:phi}) by Lemma~\ref{lem:pullback}. 
\end{proof}

We will postpone the proofs of the following lemmas in Subsection~\ref{subsec:postpone}:

\begin{lemma}\label{lem:symW}
The $G_E$-representation $\mathbf{V}_0$ is self-dual. 
\end{lemma}
\begin{proof}
    See Subsection~\ref{subsub:1}.
\end{proof}
\begin{lemma}\label{lem:neq}
For any map $\nu \colon (\bgm)_{k'} \to \widehat{\mathbf{V}}_0/G_E$, 
we have 
$n_{\widehat{\mathbf{V}}_0/G_E, \nu}=n_{(\V^L)^{\vee}, \nu}$. 
\end{lemma}
\begin{proof}
    See Subsection~\ref{subsub:2}.
\end{proof}

\begin{lemma}\label{lem:pullback}
We take an object $\E'' \in D_{\mathbb{G}_m}^b((\V^L)^{\vee})$
satisfying 
\begin{align*}\E''|_{(\V^L)^{\vee}_{\circ}} \in \MM_{\mathbb{G}_m}((\V^L)^{\vee}_{\circ})_{\delta}, \
g^{\ast}\E'' \in \MM(\widehat{\mathbf{V}}_0/G_E)_{g^{\ast}\delta}.
\end{align*}
Then by shrinking $\U^L$ as in Remark~\ref{rmk:shrink} if necessary, we have  
$\E'' \in \MM_{\mathbb{G}_m}((\V^L)^{\vee})_{\delta}$. 
\end{lemma}
\begin{proof}
    See Subsection~\ref{subsub:3}.
\end{proof}

\subsection{Semiorthogonal decompositions via the Koszul equivalence}\label{sod:kequiv}
Recall that for the $\Theta$-stratum $\mathcal{S} \subset \mathcal{U}$ 
we have the following diagram 
\begin{align}\notag
	\xymatrix{\mathcal{S}\ar[d]_-{q} \inclusion^-{p} & \U & \linclusion \mathcal{U}_{\circ}:=\mathcal{U}\setminus \mathcal{S} \\
		\mathcal{Z}.\ar[ur] \ar@/^15pt/[u] &	& &
	}
\end{align}
Since $q$ is quasi-smooth and $p$ is proper, we 
have the functor, see~\cite{PoSa}
\begin{align*}
	\Upsilon_{\mathcal{U}}=p_{\ast}q^{\ast} \colon \Coh(\mathcal{Z}) \stackrel{q^*}{\to}\Coh(\mathcal{S}) \stackrel{p_{*}}{\to}\Coh(\mathcal{U}). 
	\end{align*}
Note that $\mathcal{Z}$ is a quasi-smooth and QCA stack, 
so that the limit category 
\begin{align*}
    \LL(\mathcal{Z})_{\delta|_{\mathcal{Z}}}\subset \Coh(\zZ)
\end{align*}
is defined 
as in Definition~\ref{def:Lcat}.
\begin{prop}\label{prop:sod:KZ}
For the open immersion $j \colon \mathcal{U}_{\circ} \hookrightarrow \U$, 
its pull-back 
\[j^{\ast} \colon \LL(\U)_{\delta} \to \LL(\U_{\circ})_{\delta}\]
admits a fully-faithful left adjoint 
\begin{align}\label{j!}
	j_{!} \colon \LL(\U_{\circ})_{\delta} \hookrightarrow \LL(\U)_{\delta}	
\end{align}
and the semiorthogonal decomposition 
	\begin{align}\label{sod:qbud}
		\LL(\mathcal{U})_{\delta}=
		\langle \Upsilon_{\mathcal{U}} \LL(\mathcal{Z})_{\delta'}, j_{!}\LL(\mathcal{U}_{\circ})_{\delta} \rangle. 
		\end{align}
	Here $\delta'=\delta \otimes (\det \mathbb{L}_{p})^{1/2}|_{\mathcal{Z}}$, and 
    the functor $\Upsilon_{\mathcal{U}}$ is fully-faithful on $\LL(\mathcal{Z})_{\delta'}$.
	\end{prop} 
\begin{proof}
By the Koszul equivalence in Theorem~\ref{thm:Kduality} 
associated with the diagram (\ref{dia:crit}), we have an equivalence
\begin{align*}
    \Psi \colon \Coh(\U) \stackrel{\sim}{\to} \mathrm{MF}^{\mathrm{gr}}((\V^L)^{\vee}, f). 
\end{align*}
By Proposition~\ref{prop:comparemagic}, it restricts to an equivalence
\begin{align*}
\Psi \colon \LL(\mathcal{U})_{\delta}\stackrel{\sim}{\to}
\mathrm{MF}^{\mathrm{gr}}(\MM((\mathcal{V}^L)^{\vee})_{\delta\otimes (\omega_{\mathcal{U}^L})^{1/2}}, f). 
    \end{align*}
 For the notation of the right-hand side, see Subsection~\ref{subsec:Koszuleq}.
 
Similarly, we also have the equivalence 
\begin{align*}
	\Psi_{\circ} \colon \Coh(\U_{\circ}) \stackrel{\sim}{\to}
	\mathrm{MF}^{\mathrm{gr}}((\V^L)^{\vee} \setminus (\pi^L)^{-1}(\sS^L), f). 
	\end{align*}
By Proposition~\ref{lem:theta:crit2}, we have 
\begin{align*}
		\mathrm{MF}^{\mathrm{gr}}((\V^L)^{\vee} \setminus (\pi^L)^{-1}(\sS^L))
	=
	\mathrm{MF}^{\mathrm{gr}}((\V^L)^{\vee}_{\circ}, f)
	\end{align*}
since any matrix factorization is supported on the critical locus $\mathrm{Crit}(f)$, 
see~\cite[Corollary~3.18]{MR3112502}. 
Therefore we have the equivalence 
\begin{align*}
	\Psi_{\circ} \colon \Coh(\mathcal{U}_{\circ}) \stackrel{\sim}{\to}
	\mathrm{MF}^{\mathrm{gr}}((\V^L)_{\circ}, f)	
	\end{align*}
which restricts to an equivalence 
\begin{align*}
	\Psi_{\circ} \colon \LL(\U_{\circ})_{\delta} \stackrel{\sim}{\to}
	\mathrm{MF}^{\mathrm{gr}}(\MM(\V^L)^{\vee}_{\circ})_{\delta\otimes (\omega_{\U^L})^{1/2}}, f).
	\end{align*}
    
We have the commutative diagram 
\begin{align}\label{dia:QB:U}
	\xymatrix{
\LL(\U)_{\delta} \ar[r]^-{\Psi}_-{\sim} \ar[d]_-{j^{\ast}} & 	\mathrm{MF}^{\mathrm{gr}}(\MM(\V^L)^{\vee})_{\delta\otimes (\omega_{\U^L})^{1/2}}, f) \ar[d]^-{j_{\V}^{\ast}} \\
\LL(\U_{\circ})_{\delta}  \ar[r]^-{\Psi_{\circ}}_-{\sim} & \mathrm{MF}^{\mathrm{gr}}(\MM(\V^L)^{\vee}_{\circ})_{\delta\otimes (\omega_{\U^L})^{1/2}}, f).
}
	\end{align}
Here $j_{\V} \colon (\V^L)^{\vee}_{\circ} \hookrightarrow (\V^L)^{\vee}$ is the open 
immersion. By Proposition~\ref{prop:sod:V} and applying it to matrix factorizations for the functions in (\ref{dia:SV}), the right vertical arrow 
admits a fully-faithful left adjoint 
\begin{align*}
	j_{\V!} \colon 
	 \mathrm{MF}^{\mathrm{gr}}(\MM(\V^L)^{\vee}_{\circ})_{\delta\otimes (\omega_{\U^L})^{1/2}}, f) \hookrightarrow
	  \mathrm{MF}^{\mathrm{gr}}(\MM(\V^L)^{\vee})_{\delta\otimes (\omega_{\U^L})^{1/2}}, f)
	\end{align*}
such that we have the semiorthogonal decomposition 
\begin{align*}
&\mathrm{MF}^{\mathrm{gr}}(\MM(\V^L)^{\vee})_{\delta\otimes (\omega_{\U^L})^{1/2}}, f)
=\\ \notag
&\left\langle \Upsilon_{\V} \mathrm{MF}^{\mathrm{gr}}(\MM(\zZ_{\V})_{\delta'\otimes (\omega_{\U^L})^{1/2}}), f_{\mathcal{Z}}), 
j_{\V!}  \mathrm{MF}^{\mathrm{gr}}(\MM(\V^L)^{\vee}_{\circ})_{\delta\otimes (\omega_{\U^L})^{1/2}}, f) \right\rangle. 	
	\end{align*}
In particular, the left vertical arrow in (\ref{dia:QB:U}) admits a 
fully-faithful left adjoint functor $j_!$ as in (\ref{j!}). 

We also have the following compatibility of the functor $\Upsilon_{\mathcal{U}}$ with Koszul equivalence, see~\cite[Section~2.4]{T},~\cite[Proposition~3.1]{P2}
\begin{align}\label{dia:koszul}
	\xymatrix{
\Coh(\zZ) \ar[r]^-{\Psi_{\zZ}'}_-{\sim} \ar[d]_-{\Upsilon_{\mathcal{U}}} & \mathrm{MF}^{\mathrm{gr}}(\zZ_{\V}, f_{\mathcal{Z}}) \ar[d]^-{\Upsilon_{\V}}	\\
\Coh(\U) \ar[r]^-{\Psi}_-{\sim} & \mathrm{MF}^{\mathrm{gr}}((\V^L)^{\vee}, f). 
}
	\end{align}
Here $\Psi_{\mathcal{Z}}'$ is given by 
\begin{align*}
	\Psi_{\zZ}'(-)=\Psi_{\zZ}(-\otimes 
	\det (\V^L|_{\zZ}^{>0})^{\vee}[\rank(\V^L|_{\zZ}^{>0})]),
	\end{align*}
where $\Psi_{\zZ}$ is the Koszul equivalence 
from the derived zero locus description (\ref{Z:zero})
\begin{align*}
	\Psi_{\zZ} \colon 
	\Coh(\zZ) \stackrel{\sim}{\to} \mathrm{MF}^{\mathrm{gr}}(\zZ_{\V}, f_{\mathcal{Z}}). 
	\end{align*}
Let $\delta'' \in \mathrm{Pic}(\zZ)_{\mathbb{Q}}$ be defined
by 
\begin{align*}
	\delta''=\delta \otimes (\omega_{\U^L})^{1/2}|_{\zZ} \otimes \det (N_{\sS_{\V}/(\V^L)^{\vee}})^{1/2}|_{\zZ} \otimes (\omega_{\zZ^L})^{-1/2}|_{\zZ}
	\otimes \det (\V^L|_{\zZ}^{>0}).  
	\end{align*}
    
Then from the diagrams (\ref{dia:QB:U}), (\ref{dia:koszul}), 
we obtain the semiorthogonal decomposition 
	\begin{align*}
	\LL(\mathcal{U})_{\delta}=
	\langle \Upsilon_{\mathcal{U}} \LL(\mathcal{Z})_{\delta''}, j_{!}\LL(\mathcal{U}_{\circ})_{\delta} \rangle
\end{align*}
and the functor $\Upsilon_{\mathcal{U}}$ is fully-faithful on $\LL(\mathcal{Z})_{\delta'}$.
From Lemma~\ref{lem:delta2}, 
we have $\LL(\zZ)_{\delta''}=\LL(\zZ)_{\delta'}$, therefore 
we obtain the semiorthogonal decomposition (\ref{sod:qbud}). 
\end{proof}

We have used the following lemma. It follows from a direct computation, and its proof 
will be given in Subsection~\ref{subsec:pfdelta}.
\begin{lemma}\label{lem:delta2}
	For any map $\nu \colon \bgm \to \mathcal{Z}$, we have 
	$c_1(\nu^{\ast}\delta')=c_1(\nu^{\ast}\delta'')$. 
	\end{lemma}

As a corollary of Proposition~\ref{prop:sod:KZ}, we have the following:  
\begin{cor}\label{cor:higgsc}
	There is a semiorthogonal decomposition 
\begin{align*}
    \LL(\Hig_G(\chi)_{\preceq \mu})_w=\left\langle \mathbb{T}_{M}(\chi_{M})_{w_{M}} : 
     \begin{array}{ll}\mu_M(\chi_{M})\in N(T)_{\mathbb{Q}+}^{W_M}, \ \mu_M(\chi_M)\preceq \mu \\ \mu_M(w_M)=\mu_G(w)
     \end{array}
    \right\rangle. 
\end{align*}
    Here the right-hand side is after all standard parabolics $P\subset G$ with Levi quotient 
    $M$ and $(\chi_{M}, w_{M}) \in \pi_1(M) \times Z(M)^{\vee}$ satisfying $a_{M}(\chi_{M}, w_{M})=(\chi, w)$.
	\end{cor}
\begin{proof}
	By Lemma~\ref{lem:square} below, we have 
	\begin{align*}
		\LL(\zZ)_{\delta_w\otimes \det(\mathbb{L}_p)^{1/2}|_{\zZ}}
		\simeq \LL(\zZ)_{\delta_w|_{\zZ}} =\bigoplus_{w_M \in Z(M)^{\vee}}\mathbb{T}_M(\chi_M)_{w_M}
		\end{align*}
        where the first equivalence is given by $\otimes \mathcal{L}^{-1}$ in Lemma~\ref{lem:square} and 
        the direct sum in the right-hand side is after $w_M \in Z(M)^{\vee}$ with $a_M(w_M)=w$ and $\mu_M(w_M)=\mu_G(w) \in M(T)_{\mathbb{Q}}^W$. 
	Therefore by Proposition~\ref{prop:sod:KZ}, 
    we obtain the semiorthogonal decomposition 
	\begin{align*}
		\LL(\U)_{w}=\left\langle 
        \mathbb{T}_M(\chi_M)_{w_M},
		j_{!}\LL(\U_{\circ})_w  : a_M(w_M)=w, \mu_M(w_M)=\mu_G(w)\right\rangle. 
		\end{align*}
	The corollary follows by applying the above semiorthogonal decomposition 
	for the $\Theta$-stratification (\ref{higgs:strata}) of 
	$\Hig_G(\chi)_{\preceq \mu}$. 	
	\end{proof}

    We have used the following lemma in the proof of Corollary~\ref{cor:higgsc}, whose
    proof will be given in Subsection~\ref{pf:lemsquare}.
\begin{lemma}\label{lem:square}
There exists $\mathcal{L} \in \mathrm{Pic}(\zZ)$ such that, 
for any map $\nu \colon \bgm \to \zZ$, we have 
$c_1(\nu^*\mathbb{L}_p)=2c_1(\nu^*\mathcal{L})$. 
	\end{lemma}

\subsection{Proof of Theorem~\ref{thm:mainLG}}\label{subsec:proofofthm}
In this subsection, we prove Theorem~\ref{thm:mainLG}. 
\begin{proof}
For $\mu_2 \preceq \mu_1$, let $j_{12}$ be the open immersion 
\begin{align}\label{map:j12}
    j_{12} \colon \Hig_G(\chi)_{\preceq \mu_2} \hookrightarrow
    \Hig_G(\chi)_{\preceq \mu_1}. 
\end{align}
We have the pull-back functor
\begin{align}\label{pback2.0}
    j_{12}^{\ast} \colon \LL(\Hig_G(\chi)_{\preceq \mu_1})_{w}
    \to \LL(\Hig_G(\chi)_{\preceq \mu_2})_{w}
\end{align}
which induces the continuous functor by taking its ind-completion
(since $\Hig_G(\chi)_{\preceq \mu}$ is quasi-compact)
\begin{align}\label{pback2}
    j_{12}^{\ast} \colon \IndL(\Hig_G(\chi)_{\preceq \mu_1})_{w}
    \to \IndL(\Hig_G(\chi)_{\preceq \mu_2})_{w}. 
\end{align}
By the open covering (\ref{cov:open})
and the definition of 
$\IndL(\Hig_G(\chi))_{w}$, 
we have an equivalence 
\begin{align*}
\IndL(\Hig_G(\chi))_w \simeq \lim_{\mu\in N(T)_{\mathbb{Q}+}}\IndL(\Hig_G(\chi)_{\preceq \mu})_w
\end{align*}
where the limit is taken with respect to the pull-backs (\ref{pback2}). 

For the open immersion $j_{12}$ in (\ref{map:j12}), 
by Proposition~\ref{prop:sod:KZ} there is a fully-faithful 
left adjoint of the functor $j_{12}^{\ast}$ in (\ref{pback2.0})
\begin{align}\label{j12!}
	j_{12!} \colon \LL(\Hig_G(\chi)_{\preceq \mu_2})_w \hookrightarrow 
	\LL(\Hig_G(\chi)_{\preceq \mu_1})_w. 
	\end{align}	
By taking its ind-completion, we obtain the continuous fully-faithful functor
\begin{align}\label{ind:j2}
	j_{12!} \colon \IndL(\Hig_G(\chi)_{\preceq \mu_2})_w \hookrightarrow 
	\IndL(\Hig_G(\chi)_{\preceq \mu_1})_w. 
	\end{align}
    Since $j_{12}^*$ admits continuous left adjoint $j_{12!}$ as above, 
by Proposition~\ref{prop:dgcat} (i) the dg-category 
$\IndL(\Hig_G(\chi))_w$ is also equivalent to the colimit 
\begin{align}\label{IndL:0}
	\IndL(\Hig_G(\chi))_w \simeq \colim_{\mu \in N(T)_{\mathbb{Q}+}} \IndL(\Hig_G(\chi)_{\preceq \mu})_w
	\end{align}
where the colimit is taken 
with respect to the functors (\ref{ind:j2}). 

In particular for the open immersion \begin{align*}j_{\mu} \colon \Hig_G(\chi)_{\preceq \mu} \hookrightarrow \Hig_G(\chi),\end{align*} 
the pull-back 
\begin{align*}
	j_{\mu}^{\ast} \colon \IndL(\Hig_G(\chi))_w \to \IndL(\Hig_G(\chi)_{\preceq \mu})_w
	\end{align*}
admits a left adjoint 
\begin{align}\label{jc!}
	j_{\mu!} \colon \IndL(\Hig_G(\chi)_{\preceq \mu})_w \hookrightarrow \IndL(\Hig_G(\chi))_w
	\end{align}
which is fully-faithful since each (\ref{ind:j}) is fully-faithful by Proposition~\ref{prop:dgcat} (iv). 

For $\mu_1 \preceq \mu_2$, by Proposition~\ref{prop:dgcat} (iv) the natural morphism 
$j_{12!} \to j_{\mu_2}^{\ast}j_{\mu_1!}$
is an isomorphism since (\ref{jc!}) is fully-faithful. 
In particular, the functor (\ref{jc!}) restricts to the functor 
\begin{align*}
    j_{\mu!} \colon \LL(\Hig_G(\chi)_{\preceq \mu})_w \hookrightarrow \Coh(\Hig_G(\chi))_w. 
\end{align*}
By Proposition~\ref{prop:dgcat} (iii), an object $\E$ in the left-hand side of (\ref{IndL:0}) is compact if and only 
if 
\begin{align}\label{jcE}\E \cong j_{\mu!}\E_{\mu} \in \Coh(\Hig_G(\chi))_w
	\end{align}
for some $\mu \in N(T)_{\mathbb{Q}+}$ and 
$\E_{\mu} \in \LL(\Hig_G(\chi)_{\preceq \mu})_w$. 

From the above argument together with Proposition~\ref{prop:dgcat}, 
    we conclude that the dg-category $\IndL(\Hig_G(\chi))_w$ is compactly generated 
with the subcategory of compact objects given by
\begin{align}\label{LL:colim}
\LL(\Hig_G(\chi))_w \simeq \colim_{\mu \in N(T)_{\mathbb{Q}+}} \LL(\Hig_G(\chi)_{\preceq \mu})_w. 
\end{align}
Here the colimit is taken with respect to (\ref{j12!}). 
The desired semiorthogonal decomposition follows from (\ref{LL:colim}) and Corollary~\ref{cor:higgsc}. 
\end{proof}

\subsection{Proof of Theorem~\ref{thm:nilp0},~\ref{thm:nilp}}\label{subsec:pfnil}
\begin{proof}
We only give a proof for Theorem~\ref{thm:nilp}, and the proof for Theorem~\ref{thm:nilp0}
is similar. 

 From the arguments in the previous subsections, it is enough to show that
 the semiorthogonal decomposition (\ref{sod:qbud}) restricts to the semiorthogonal decomposition 
 \begin{align*}
     \LL_{\mathcal{N}}(\mathcal{U})_{\delta}=\langle \Upsilon_{\mathcal{U}} \LL_{\mathcal{N}}(\mathcal{Z})_{\delta'}, j_{!}\LL_{\mathcal{N}}(\mathcal{U}_{\circ})_{\delta}\rangle.
 \end{align*}
 
 For $c \in C$, let $\Bun_G \to BG$ be the restriction map to $c$, $\mathcal{E}\mapsto \mathcal{E}|_{c}$. 
 By the arguments in Subsection~\ref{subsec:limsupp}, the above map induces the $\mathrm{Perf}^{\shear}(\fg\ssslash G)$-action 
 on $\IndCoh(\Hig_G)$, which also induce the action on $\LL(\mathcal{U})_{\delta}$. 
 Since the action of the generator $\mathcal{O}_{\fg\ssslash G}\in \mathrm{Perf}^{\shear}(\fg\ssslash G)$ 
 is trivial on isomorphism classes of objects, it restricts to the actions of the semiorthognal summands in (\ref{sod:qbud}). 
 
 Therefore we have 
 \begin{align*}
   \LL_{\mathcal{N}}(\mathcal{U})_{\delta}=\langle  \LL_{\mathcal{N}}(\mathcal{U})_{\delta}\cap\Upsilon_{\mathcal{U}}\LL(\mathcal{Z})_{\delta'},  \LL_{\mathcal{N}}(\mathcal{U})_{\delta}\cap j_{!}\LL(\mathcal{U}_{\circ})_{\delta}\rangle.
   \end{align*}
It is enough to show that 
\begin{align*}
     \LL_{\mathcal{N}}(\mathcal{U})_{\delta}\cap\Upsilon_{\mathcal{U}}\LL(\mathcal{Z})_{\delta'}=\Upsilon_{\mathcal{U}}\LL_{\mathcal{N}}(\mathcal{Z})_{\delta'}, \ \LL_{\mathcal{N}}(\mathcal{U})_{\delta}\cap j_{!}\LL(\mathcal{U}_{\circ})_{\delta}=j_{!}\LL_{\mathcal{N}}(\mathcal{U}_{\circ})_{\delta}.
\end{align*}
 The latter identity is obvious since $j_{!}$ is fully-faithful and 
 $j^*\colon \LL(\mathcal{U})_{\delta} \to \LL(\mathcal{U}_{\circ})_{\delta}$ is linear over  
 $\mathrm{Perf}^{\shear}(\fg\ssslash G)$. 
 
 For the first identity, note that 
 $\mathrm{Perf}^{\shear}(\fm\ssslash M)$ acts on 
 $\IndCoh(\Hig_M)$ which induces the action on 
 $\LL(\mathcal{Z})_{\delta'}$. 
 Let $\phi \colon \fm\ssslash M \to \fg\ssslash G$ be the morphism 
 induced by the inclusion $M\subset G$. The functor 
 \begin{align*}
     \phi^* \colon \mathrm{Perf}^{\shear}(\fg\ssslash G)\to \mathrm{Perf}^{\shear}(\fm\ssslash M)
 \end{align*}
 is a monoidal functor, hence $\mathrm{Perf}^{\shear}(\fg\ssslash G)$ acts on 
 $\IndCoh(\Hig_M)$ and $\LL(\mathcal{Z})_{\delta'}$.
 Since $\phi^{-1}(0)=0 \in \fm\ssslash M$, it
  is enough to show that the functor 
 \begin{align*}
     \Upsilon_{\mathcal{U}} \colon \LL(\mathcal{Z})_{\delta'} \to \LL(\mathcal{U})_{\delta}
 \end{align*}
 is linear over $\mathrm{Perf}^{\shear}(\fg\ssslash G)$. 
 We have the following diagram 
 \begin{align*}\Bun_M \stackrel{q}{\leftarrow} \Bun_P \stackrel{p}{\to} \Bun_G.
 \end{align*}
 It is enough to show that the functor 
 \begin{align*}
     p_{*}^{\Omega}q^{\Omega!} \colon \IndCoh(\Hig_M) \to \IndCoh(\Hig_P)
 \end{align*}
 is linear over $\mathrm{Perf}^{\shear}(\fg\ssslash G)$. The last claim easily follows from Lemma~\ref{lem:funct3} together with $\LL(\Omega_{BP})_{\frac{1}{2}}\simeq \mathrm{Perf}^{\shear}(\fm\ssslash M)$ from the computation in Subsection~\ref{subsub:GA1}.
\end{proof}

\subsection{The relation with limit of limit categories}\label{subsec:Ltilde}
Recall the definition of 
\begin{align*}\widetilde{\LL}(\Hig_{G}(\chi))_w=\lim_{\mathcal{U}\subset \Hig_G(\chi)}
\LL(\mathcal{U})_w.
\end{align*}
in 
(\ref{def:Ltilde0}). 
Here we investigate the relation of $\LL(\Hig_G(\chi))_w$ and $\widetilde{\LL}(\Hig_G(\chi))_w$.
In general they are different, but we have the following: 

\begin{cor}\label{cor:Ltilde}
We have 
\begin{align}\label{LsubLtilde}
\LL(\Hig_{G}(\chi))_w \subset \widetilde{\LL}(\Hig_{G}(\chi))_{w}
\end{align}
and an object $\mathcal{E} \in \widetilde{\LL}(\Hig_{G}(\chi))$
lies in $\LL(\Hig_{G}(\chi))_w$ if and only if 
    there is $\mu \in N(T)_{\mathbb{Q}+}$ such that. for any $A \in \widetilde{\LL}(\Hig_{G}(\chi))_w$
    satisfying 
    \begin{align}\label{Zc}
\mathrm{Supp}(A) \subset Z_{\mu}(\chi):=\Hig_{G}(\chi)\setminus\Hig_G(\chi)_{\preceq \mu},
    \end{align}
    we have $\Hom(\mathcal{E}, A)=0$. 
\end{cor}
\begin{proof}
    The inclusion (\ref{LsubLtilde}) holds 
    by the description of compact objects (\ref{jcE}) and using the fact that the functor (\ref{ind:j2})
    preserves compact objects (by its construction). 

    We prove the second statement. Consider an object $\mathcal{E}$ 
    in $\LL(\Hig_{G}(\chi))_w$. It may be written as $\mathcal{E}=j_{\mu!}\mathcal{E}_{\mu}$
    as in (\ref{jcE}). Then, for the object $A$ satisfying (\ref{Zc}), we have 
    \begin{align*}
        \Hom(j_{\mu!}\mathcal{E}_{\mu}, A)=\Hom(\mathcal{E}_{\mu}, j_{\mu}^{\ast}A)=0. 
    \end{align*}
    Conversely, for $\mathcal{E} \in \widetilde{\LL}(\Hig_{G}(\chi))_w$, 
    we have the distinguished triangle 
    \begin{align}\label{tri:jc}
        j_{\mu!}j_{\mu}^{\ast}\mathcal{E} \to \mathcal{E} \to A
    \end{align}
    where $j_{\mu!}j_{\mu}^{\ast}\mathcal{E} \in \LL(\Hig_{G}(\chi))_w$
    and $A \in \widetilde{\LL}(\Hig_{G}(\chi))_w$. Since $j_{\mu}^*A=0$, the object $A$ 
    satisfies (\ref{Zc}). 
    If $\Hom(\mathcal{E}, A)=0$, then the first arrow in (\ref{tri:jc}) is an isomorphism. 
    Therefore we conclude that $\mathcal{E} \in \LL(\Hig_{G}(\chi))_w$. 
\end{proof}

There are also some cases where $\widetilde{\LL}(\Hig_G(\chi))_w$ is equivalent to 
$\LL(\Hig_G(\chi))_w$: 
\begin{cor}\label{cor:coprime}
Suppose that for any Levi subgroup $M$ and $w_M \in Z(M)^{\vee}$, 
we have $(a_M(w_M), \mu_M(w_M)) \neq (w, \mu_G(w))$,
except when $M=G$ and $w_M=w$. Then there are equivalences
\begin{align*}
\LL(\Hig_{G}(\chi))_w\stackrel{\sim}{\to} \widetilde{\LL}(\Hig_{G}(\chi))_w\stackrel{\sim}{\to}\mathbb{T}_{G}(\chi)_w. 
\end{align*}
\end{cor}
\begin{proof}
Under the assumption, by Corollary~\ref{cor:higgsc},
there is a natural equivalence for $\mu \in N(T)_{\mathbb{Q}+}$:
\begin{align}\label{L:BPS}
    \LL(\Hig_{G}(\chi)_{\preceq \mu})_{w}
    \stackrel{\sim}{\to} \mathbb{T}_{G}(\chi)_w.
\end{align}

On the other hand, by the open covering (\ref{cov:open}) and the 
definition of $\widetilde{\LL}(\Hig_{G}(\chi))_w$ in 
(\ref{def:Ltilde0}), 
there is an equivalence
\begin{align*}
\widetilde{\LL}(\Hig_{G}(\chi))_w \simeq \lim_{\mu\in N(T)_{\mathbb{Q}+}}\LL(\Hig_{G}(\chi)_{\preceq \mu})_w
\end{align*}
where the limit is taken with respect to the pull-backs (\ref{pback2.0}). 
Therefore the corollary follows from (\ref{L:BPS}).
    \end{proof}

\begin{example}\label{exam:coprime}
The assumption of Corollary~\ref{cor:coprime} is satisfied if $G=\GL_r$ and $w\in Z(\GL_r)^{\vee}=\mathbb{Z}$ is 
coprime with $r$. 
\end{example}

\section{Hecke operators on limit categories}\label{sec:hecke}
In this section, we show that the limit categories for the moduli stacks of Higgs bundles
with $G=\GL_r$ admit
Hecke operators. In particular, we establish the existence of Hecke
operators on the direct sum of BPS categories. In general, Hecke operators do not preserve
stability conditions, which has posed a difficulty in constructing Hecke operators for Higgs
bundles. Since the limit category is defined for the stack without imposing the stability
condition, it provides the natural framework for constructing Hecke operators for Higgs bundles.

In the case of general reductive group $G$, we construct Hecke operators associated with minuscule coweights.

\subsection{Hecke actions on D-modules}
For a smooth projective curve $C$, let 
\[\mathcal{C}(m)\] be the moduli stack of 
zero-dimensional sheaves $Q\in \Coh^{\heartsuit}(C)$ such that 
$\chi(Q)=m$. 
The stack $\mathcal{C}(m)$ is a classical smooth 
stack of finite type (in particular quasi-compact) and $\dim \mathcal{C}(m)=0$, with good moduli space 
\begin{align*}
  \pi_m \colon   \mathcal{C}(m)\to \mathrm{Sym}^m(C), \ Q\mapsto \mathrm{Supp}(Q).
\end{align*}

Let $\mathcal{C}(m_1, m_2)$ be the moduli stack
classifying exact sequences in $\Coh^{\heartsuit}(C)$
\begin{align}\label{Q12}
0 \to Q_1 \to Q \to Q_2 \to 0
\end{align}
such that $\dim Q_i=0$ and $\chi(Q_i)=m_i$. 
We have the evaluation morphisms 
\begin{align}\label{dia:Cm12}
    \xymatrix{
\mathcal{C}(m_1, m_2) \ar[d]_-{q_{\mathcal{C}}} \ar[r]^-{p_{\mathcal{C}}} & \mathcal{C}(m_1+m_2) \\
\mathcal{C}(m_1)\times \mathcal{C}(m_2). & 
        }
\end{align}
Here $p_{\mathcal{C}}$ sends (\ref{Q12}) to $Q$ and $q_{\mathcal{C}}$ sends (\ref{Q12}) to $(Q_1, Q_2)$.
The functors 
\begin{align*}
    p_{\mathcal{C}*}q_{\mathcal{C}}^! \colon \text{D-mod}(\mathcal{C}(m_1)) \otimes  \text{D-mod}(\mathcal{C}(m_2))
    \to \text{D-mod}(\mathcal{C}(m_1+m_2))
\end{align*}
give a monoidal structure on 
\begin{align}\label{dmod:hall}
    \bigoplus_{m\in \mathbb{Z}_{\geq 0}}\text{D-mod}(\mathcal{C}(m)).
\end{align}
Since $q_{\mathcal{C}}$ is smooth and $p_{\mathcal{C}}$ is smooth and projective, the above monoidal 
structure restricts to the monoidal structure on the subcategory of coherent D-modules. 

The (classical) Hecke actions are the (right and left) actions of the monoidal category (\ref{dmod:hall}) 
on 
\begin{align}\label{dmod:heckebun}
\bigoplus_{\chi \in \mathbb{Z}}\text{D-mod}(\Bun_{\GL_r}(\chi))
\end{align}
given by the stack of Hecke correspondences, see (\ref{dia:hecke}) and (\ref{dia:hecke2}). 
In what follows, we construct the classical limit of Hecke actions as actions of 
quasi-BPS categories of zero-dimensional sheaves on limit categories. 

\begin{remark}
For a general reductive group $G$, the Hecke operator is given by the 
infinite-dimensional correspondence, as will be given by (\ref{Hecke:Gx}). 
The Hecke action discussed here is rather a classical version used in~\cite[Section~1]{FGV}. 
\end{remark}

\subsection{Quasi-BPS categories of zero-dimensional 
sheaves}\label{subsec:qbps:zero}

We define the derived stack $\mathcal{S}(m)$ to be 
\begin{align*}
\mathcal{S}(m):=\Omega_{\mathcal{C}(m)}.    
\end{align*}
Similarly to the case of Higgs bundles, the derived stack $\mathcal{S}(m)$ is equivalent to the derived moduli stack of zero-dimensional sheaves $Q\in \Coh^{\heartsuit}(S)$ on the local surface 
$S=\mathrm{Tot}_C(\omega_C)$ such that $\chi(Q)=m$. 

Let $\mathcal{Q}$ be the universal sheaf 
\begin{align*}
    \mathcal{Q} \in \Coh^{\heartsuit}(C \times \mathcal{C}(m)). 
\end{align*}
Using the same notation in (\ref{deltaw}), we 
define the following $\mathbb{Q}$-line bundle 
\begin{align}\notag
    \delta_w:=(\det \pi_{\mathcal{C}\ast} \mathcal{Q})^{w/m} \in \mathrm{Pic}(\mathcal{C}(m))_{\mathbb{Q}}. 
\end{align}
Here $\pi_{\mathcal{C}} \colon C\times \mathcal{C}(m) \to \mathcal{C}(m)$ is the projection. 
The following category is introduced and studied in~\cite{P2, PT2}:
\begin{defn}
The \textit{quasi-BPS category for zero-dimensional sheaves} is defined to be
\begin{align*}
    \mathbb{T}(m)_w :=\LL(\mathcal{S}(m))_{\delta=\delta_w} 
    \subset \Coh(\mathcal{S}(m)). 
\end{align*}
\end{defn}

We note the following lemma: 
\begin{lemma}\label{lem:omega12}
For any $\delta \in \mathrm{Pic}(\mathcal{C}(m))_{\mathbb{R}}$, we have 
\begin{align*}
    \LL(\mathcal{S}(m))_{\delta\otimes \omega_{\mathcal{C}(m)}^{1/2}}=\LL(\mathcal{S}(m))_{\delta}.
\end{align*}
\end{lemma}
\begin{proof}
We have the natural morphism 
\begin{align*}
    \pi_m^*\pi_{m*}\omega_{\mathcal{C}(m)}\to \omega_{\mathcal{C}(m)}. 
\end{align*}
The above morphism is an isomorphism. This follows from the fact that the map $\pi_m$ is locally (on $\mathrm{Sym}^m(C)$) isomorphic to 
\[\mathfrak{g}/G \to \fg\ssslash G\] for a reductive group $G$, and the canonical bundle on 
$\fg/G$ is trivial. Therefore, for any map $\nu \colon \bgm \to \mathcal{C}(m)$, we have 
$c_1(\nu^* \omega_{\mathcal{C}(m)})=0$, and the lemma follows. 
\end{proof}

For $(w_1, w_2)\in \mathbb{Z}^2$ such that $w_1/m_1=w_2/m_2$, 
by Proposition~\ref{prop:pback}, Theorem~\ref{thm:proj} and Lemma~\ref{lem:omega12}, 
we have the induced functor 
\begin{align}\label{Tm:prod}
\ast :=p_{\mathcal{C}\ast}^{\Omega}q_{\mathcal{C}}^{\Omega!} \colon 
    \mathbb{T}(m_1)_{w_1}\otimes \mathbb{T}(m_2)_{w_2}
    \to \mathbb{T}(m_1+m_2)_{w_1+w_2}. 
\end{align}
It is also given as follows. 
Let $\mathcal{S}(m_1, m_2)$ be the derived moduli 
stack classifying exact sequences in $\Coh^{\heartsuit}(S)$
\begin{align}\label{exact:Q}
    0 \to Q_1 \to Q \to Q_2 \to 0
\end{align}
such that $\dim Q_i=0$ and $\chi(Q_i)=m_i$. 
Similarly to (\ref{dia:Cm12}), 
we have the evaluation morphisms 
\begin{align*}
    \xymatrix{
\mathcal{S}(m_1, m_2) \ar[d]_-{q_{\mathcal{S}}} \ar[r]^-{p_{\mathcal{S}}} & \mathcal{S}(m_1+m_2) \\
\mathcal{S}(m_1)\times \mathcal{S}(m_2). & 
        }
\end{align*}
The following lemma is straightforward to check, and in fact the proof is the 
same as Lemma~\ref{lem:HeckeS} given later, so we omit details. 
\begin{lemma}\label{lem:easy}
The functor (\ref{Tm:prod}) is isomorphic to the functor 
\begin{align}\label{funct:omega2}
(-)\mapsto 
    p_{\mathcal{S}\ast}(q_{\mathcal{S}}^{\ast}(-) \otimes \omega_{q_{\mathcal{C}}}).  
\end{align}
Here $\omega_{q_{\mathcal{C}}}$ is the relative canonical line bundle of $q_{\mathcal{C}}$
pulled back to $\mathcal{S}(m_1)\times \mathcal{S}(m_2)$.
\end{lemma}

\begin{remark}
Over the point (\ref{exact:Q}) on $\mathcal{S}(m_1, m_2)$, 
the line bundle $\omega_{q_{\mathcal{C}}}$ is given by 
\begin{align}\label{omega:Q}
    \omega_{q_{\mathcal{C}}}|_{(0\to Q_1 \to Q\to Q_2 \to 0)}
    =\det \chi_C(\pi_{*}Q_1, \pi_{*}Q_2).
\end{align}
In particular, it is given by the pull-back of a line bundle on 
$\mathcal{S}(m_1)\times \mathcal{S}(m_2)$, which we also denote by 
$\omega_{q_{\mathcal{C}}}$. Then the functor (\ref{funct:omega2}) can be also written as 
\begin{align*}
(-)\mapsto 
    p_{\mathcal{S}\ast}(q_{\mathcal{S}}^{\ast}(- \otimes \omega_{q_{\mathcal{C}}})).  
\end{align*}

Indeed the line bundle $\omega_{q_{\mathcal{C}}}$ 
further descends to $\mathrm{Sym}^{m_1}(C) \times \mathrm{Sym}^{m_2}(C_2)$, since 
the expression (\ref{omega:Q}) only depends on the supports $\mathrm{Supp}(Q_i)$. 
Hence the functor $\otimes \omega_{q_{\mathcal{C}}}$ preserves $\mathbb{T}(m_1)_{w_1}\otimes \mathbb{T}(m_2)_{w_2}$.
    \end{remark}
The direct sum
\begin{align}\notag
    \mathbb{H}:=\bigoplus_{m\in \mathbb{Z}_{\geq 0}}\mathbb{T}(m)_0
\end{align}
is the categorical Hall algebra of quasi-BPS categories
of zero-dimensional sheaves (of weight zero), 
where the monoidal structure is given by (\ref{Tm:prod}). 
We have $\mathcal{S}(0)=\Spec k$ and 
the monoidal unit is given by 
\begin{align*}
    k \in \mathbb{T}(0)_0=\Coh(\mathrm{pt}).
\end{align*}
By taking the ind-completion, we also have the monoidal 
structure on 
\begin{align*}
\Ind \mathbb{H}:=\bigoplus_{m \in \mathbb{Z}_{\geq 0}}\Ind(\mathbb{T}(m)_0). 
\end{align*}

\subsection{Hecke operator (right action)}\label{subsec:Hecke}
We denote by 
\begin{align*}\mathrm{Hecke}_{\GL_r}(\chi, \chi+m)
\end{align*}
the moduli stack of Hecke correspondences, 
which classifies exact sequences
\begin{align}\label{ex:heart}
    0 \to F \to F' \to Q \to 0
\end{align}
where $F, F'$ are vector bundles on $C$ with 
\begin{align*}(\rank(F), \deg(F))=(r, \chi), \
(\rank(F'), \deg(F'))=(r, \chi+m)
\end{align*}
respectively, 
and $Q \in \Coh^{\heartsuit}(C)$ such that $\dim Q=0$ and 
$\chi(Q)=m$. 

We have the diagram of evaluation morphisms 
\begin{align}\label{dia:hecke}
    \xymatrix{
\mathrm{Hecke}_{\GL_r}(\chi, \chi+m) \ar[r]^-{p_{\mathrm{B}}^{}} \ar[d]_-{q_{\mathrm{B}}^{}=(q_{\mathrm{B}1}^{}, q_{\mathrm{B}2}^{})} & 
    \Bun_{\GL_r}(\chi+m) \\
    \Bun_{\GL_r}(\chi) \times \mathcal{C}(m). & 
    }
\end{align}
The morphism $p_{\mathrm{B}}^{}$ is smooth and projective, 
and the morphism $q_{\mathrm{B}}^{}$ is smooth. Moreover we have $\dim p_{\mathrm{B}}^{}=\dim q_{\mathrm{B}}^{}=rm$. 
 We have the induced functor 
 \begin{align}\label{funct:hecke}
     p^{\Omega}_{\mathrm{B}\ast}q_{\mathrm{B}}^{\Omega!} \colon 
     \Coh(\Hig_{\GL_r}(\chi))\otimes \Coh(\mathcal{S}(m)) \to \Coh(\Hig_{\GL_r}(\chi+m)). 
 \end{align}

Recall the category $\widetilde{\LL}(\Hig_G(\chi))_w$ from Subsection~\ref{subsec:Ltilde}. 
We have the following lemma: 

\begin{lemma}\label{lem:hecke}
The functor (\ref{funct:hecke}) restricts 
to the functor 
\begin{align}\notag
p_{\mathrm{B}\ast}^{\Omega} q_{\mathrm{B}}^{\Omega!} \colon 
\widetilde{\LL}(\Hig_{\GL_r}(\chi))_{\delta_w \otimes \omega_{\Bun}^{1/2}} \otimes \mathbb{T}(m)_0
\to \widetilde{\LL}(\Hig_{\GL_r}(\chi+m))_{\delta_w\otimes \omega_{\Bun}^{1/2}}.
    \end{align}
\end{lemma}
\begin{proof} 
    We have 
    \begin{align*}
        p_{\mathrm{B}}^{\ast}\delta_w =q_{\mathrm{B}}^{\ast}(\delta_w \boxtimes \mathcal{L}_m)
    \end{align*}
    where $\mathcal{L}_m$ is a $\mathbb{Q}$-line bundle on $\mathcal{C}(m)$
    whose fiber at $Q$ is $\det \chi(Q \otimes \mathcal{O}_{c})^{w/r}$ for $c\in C$. 
Indeed we have \begin{align*}p_{\mathrm{B}}^* \delta_w|_{(0 \to F \to F' \to Q \to 0)}&=
\det (F'|_{c})^{w/r}\\ &=\det (F|_{c})^{w/r} \otimes \det \chi(Q \otimes \mathcal{O}_{c})^{w/r}
\\ &=q_{\mathrm{B}}^{\ast}(\delta_w \boxtimes \mathcal{L}_m)|_{(0 \to F \to F' \to Q \to 0)}. \end{align*}
 
   We show that 
    the functor (\ref{funct:hecke}) restricts to the functor 
    \begin{align}\notag
p_{\mathrm{B}\ast}^{\Omega} q_{\mathrm{B}}^{\Omega!} \colon 
\widetilde{\LL}(\Hig_{\GL_r}(\chi))_{\delta_w \otimes \omega_{\Bun}^{1/2}} &\otimes \LL(\mathcal{S}(m))_{\mathcal{L}_m \otimes \omega_{\mathcal{C}(m)}^{1/2}} \\ \label{restrict:tilde}
&\to \widetilde{\LL}(\Hig_{\GL_r}(\chi+m))_{\delta_w\otimes \omega_{\Bun}^{1/2}}.
    \end{align}
Indeed for each quasi-compact open substack $U\subset \Bun_{\GL_r}(\chi+m)$, 
there is a quasi-compact open substack $U'\subset \Bun_{\GL_r}(\chi)$ such that 
the diagram (\ref{funct:hecke}) restricts to the diagram 
\begin{align*}
    U' \times \mathcal{C}(m) \stackrel{q_{\mathrm{B}}^{}}{\leftarrow} p_{\mathrm{B}}^{-1}(U) \stackrel{p_{\mathrm{B}}^{}}{\to} U.
\end{align*}
Therefore for $\mathcal{E}\in \widetilde{\LL}(\Hig_{\GL_r}(\chi))_{\delta_w \otimes \omega_{\Bun}^{1/2}}$ and $T\in \mathbb{T}(m)_0$, we have 
\begin{align*}
    p_{\mathrm{B}\ast}^{\Omega} q_{\mathrm{B}}^{\Omega!}(\mathcal{E}\boxtimes T)|_{\Omega_{U}}=p_{\mathrm{B}\ast}^{\Omega} q_{\mathrm{B}}^{\Omega!}(\mathcal{E}|_{\Omega_{U'}}\boxtimes T) \in 
    \LL(\Omega_{U})_{\delta_w \otimes \omega_{\Bun}^{1/2}}
\end{align*}
since $\mathcal{E}|_{\Omega_{U'}} \in \LL(\Omega_{U'})_{\delta_w \otimes \omega_{\Bun}^{1/2}}$
and using Proposition~\ref{prop:pback} and Theorem~\ref{thm:proj}. 
Therefore the functor (\ref{funct:hecke}) restricts to the functor (\ref{restrict:tilde}). 
    
    The $\mathbb{Q}$-line bundle $\mathcal{L}_m$ descends to $\mathrm{Sym}^m(C)$, 
    so it follows that 
    \begin{align*}
        \LL(\mathcal{S}(m))_{\mathcal{L}_m\otimes \omega_{\mathcal{C}(m)}^{1/2}}=\mathbb{T}(m)_0. 
    \end{align*}
    Therefore the lemma follows. 
\end{proof}

We show that the functor in Lemma~\ref{lem:hecke} induces the action on 
$\LL(\Hig_{\GL_r}(\chi))_w$. 
We denote by
\begin{align*}
    \mathrm{Hecke}_{\GL_r}^{\mathrm{Hig}}(\chi, \chi+m)
\end{align*}
the derived moduli stack which classifies exact sequences 
in $\Coh^{\heartsuit}(S)$
\begin{align}\label{exact:R}
    0 \to E \to E' \to R \to 0
\end{align}
where $E, E'\in \Coh^{\heartsuit}(S)$ are compactly supported pure one-dimensional
coherent sheaves (equivalently Higgs bundles on $C$)
satisfying 
\begin{align*}
(\mathrm{rank}(\pi_{\ast}E), \deg(\pi_{*}E))=(r, \chi), \ 
(\mathrm{rank}(\pi_{\ast}E'), \deg(\pi_{*}E'))=(r, \chi+m)
\end{align*}
and $R\in \Coh^{\heartsuit}(S)$ such that $\dim R=0$ and $\chi(R)=m$. 

We have the 
evaluation morphisms
\begin{align}\label{dia:heckeS}
    \xymatrix{
\mathrm{Hecke}_{\GL_r}^{\mathrm{Hig}}(\chi, \chi+m) \ar[r]^-{p_{\mathrm{H}}^{}} \ar[d]_-{q_{\mathrm{H}}^{}=(q_{\mathrm{H}1}^{}, q_{\mathrm{H}2}^{})} & 
    \Hig_{\GL_r}(\chi+m) \\
    \Hig_{\GL_r}(\chi) \times \mathcal{S}(m). & 
    }
\end{align}
The morphism $p_{\mathrm{H}}^{}$ is quasi-smooth and projective, 
and the morphism $q_{\mathrm{H}}^{}$ is quasi-smooth. 
Note that we have the projection 
\begin{align}\label{hecke:map}
    \mathrm{Hecke}_{\GL_r}^{\mathrm{Hig}}(\chi, \chi+m) \to \mathrm{Hecke}_{\GL_r}(\chi, \chi+m)
\end{align}
given by the push-forward along $\pi \colon S \to C$. 
\begin{lemma}\label{lem:HeckeS}
The functor (\ref{funct:hecke}) is isomorphic to a functor 
\begin{align*}
(-)\mapsto 
    p_{\mathrm{H}*}(q_{\mathrm{H}}^{*}(-)\otimes \omega_{q_{\mathrm{B}}^{}})[\dim q_{\mathrm{B}}^{}].
\end{align*}
Here $\omega_{q_{\mathrm{B}}^{}}$ is the relative canonical line bundle of the map
$q_{\mathrm{B}}^{}$ in the diagram (\ref{dia:hecke}), 
pulled back to $\mathrm{Hecke}_{\GL_r}^{\mathrm{Hig}}(\chi, \chi+m)$ via (\ref{hecke:map}). 
\end{lemma}
\begin{proof}
    From the definitions of $p_{\mathrm{B}\ast}^{\Omega}$ and $q_{\mathrm{B}}^{\Omega!}$, 
    and the base change formula, 
    it is enough to show the existence of the diagram 
\begin{align*}
    \xymatrix{
 & \mathrm{Hecke}_{\GL_r}^{\mathrm{Hig}}(\chi, \chi+m)\ar[ld] \ar[rd] &  \\
 q_{\mathrm{B}}^{*}(\Omega_{\Bun_{\GL_r}(\chi)\times \mathcal{C}(m)}) \ar[d] \ar[rd]\ar@{}[rr]|\square & & 
p_{\mathrm{B}}^{\ast}\Omega_{\Bun_{\GL_r}(\chi+m)} \ar[ld] \ar[d]  \\
\Omega_{\Bun_{\GL_r}(\chi)\times \mathcal{C}(m)} &  
\Omega_{\mathrm{Hecke}_{\GL_r}(\chi, \chi+m)} &  \Omega_{\Bun_{\GL_r}(\chi+m)}
    }
\end{align*}
    such that the middle square is Cartesian, and the compositions of left arrows, 
    right arrows are identified with $q_{\mathrm{H}}^{}$, $p_{\mathrm{H}}^{}$, respectively. 
    
    The above diagram follows from the equivalence 
    \begin{align}\label{equiv:pq}
        \mathrm{Hecke}_{\GL_r}^{\mathrm{Hig}}(\chi, \chi+m) \stackrel{\sim}{\to}\Omega_{(p_{\mathrm{B}}^{}, q_{\mathrm{B}}^{})}[-1]
    \end{align}
   where $(p_{\mathrm{B}}^{}, q_{\mathrm{B}}^{})$ is the morphism 
   \begin{align*}
       (p_{\mathrm{B}}^{}, q_{\mathrm{B}}^{}) \colon \mathrm{Hecke}_{\GL_r}(\chi, \chi+m)\to 
       \Bun_{\GL_r}(\chi+m) \times \Bun_{\GL_r}(\chi)
       \times \mathcal{C}(m)
   \end{align*}
   and $\Omega_{(p_{\mathrm{B}}^{}, q_{\mathrm{B}}^{})}[-1]$ is the conormal stack of $(p_{\mathrm{B}}^{}, q_{\mathrm{B}}^{})$.
   The equivalence (\ref{equiv:pq}) follows from the same argument 
   of~\cite[Proposition~5.1]{Totheta}. 
\end{proof}

\begin{prop}\label{prop:hecke2}
    The functor (\ref{funct:hecke}) restricts to the functor 
\begin{align}\notag
p_{\mathrm{B}\ast}^{\Omega} q_{\mathrm{B}}^{\Omega!} \colon \LL(\Hig_{\GL_r}(\chi))_{\delta_w \otimes \omega_{\Bun}^{1/2}} \otimes \mathbb{T}(m)_0
\to \LL(\Hig_{\GL_r}(\chi+m))_{\delta_w\otimes \omega_{\Bun}^{1/2}}.
    \end{align}
    \end{prop}
    \begin{proof}
We take $\mathcal{E} \in \LL(\Hig_{\GL_r}(\chi))_{\delta_w\otimes \omega_{\Bun}^{1/2}}$ and $T\in \mathbb{T}(m)_0$. 
For $c\geq 0$, let \begin{align*}Z_c(\chi) \subset \Hig_{\GL_r}(\chi)\end{align*}
be the closed substack consisting of Higgs bundles $(F, \theta)$ such that $\mu^{\mathrm{max}}(F, \theta)>c$, see Remark~\ref{rmk:G=GLr} for the notation. 
The substack $Z_c(\chi)$ consists of a disjoint union of Harder-Narasimhan strata. 
By Corollary~\ref{cor:Ltilde}, it is enough to show that 
there is $c \geq 0$ such that 
\begin{align}\label{ETA:0}
    \Hom(p_{\mathrm{B}\ast}^{\Omega}q_{\mathrm{B}}^{\Omega!}(\mathcal{E}\otimes T), A)=0
\end{align}
for any $A \in \widetilde{\LL}(\Hig_{\GL_r}(\chi+m))_{\delta_w\otimes \omega_{\Bun}^{1/2}}$ with
$\mathrm{Supp}(A) \subset Z_c(\chi+m)$. 

By Lemma~\ref{lem:HeckeS} and using adjunction, 
the left-hand side of (\ref{ETA:0}) is isomorphic to 
\begin{align*}
&\Hom(p_{\mathrm{H}\ast}(q_{\mathrm{H}}^*(\mathcal{E}\otimes T)\otimes \omega_{q_{\mathrm{B}}^{}}[\dim q_{\mathrm{B}}^{}]), A) \\
&\cong \Hom(q_{\mathrm{H}}^*(\mathcal{E}\otimes T)\otimes \omega_{q_{\mathrm{B}}^{}}[\dim q_{\mathrm{B}}^{}], p_{\mathrm{H}}^!A) \\
&\cong \Hom(\mathcal{E}, q_{\mathrm{H}1 \ast}^{}((p_{\mathrm{H}}^{}, q_{\mathrm{H}2}^{})^{*}(A\otimes\mathbb{D}'(T))\otimes \omega_{(p_{\mathrm{B}}^{}, q_{\mathrm{B}2}^{})}[\dim (p_{\mathrm{B}}^{}, q_{\mathrm{B}2}^{})])) \\
&\cong \Hom(\mathcal{E}, q_{\mathrm{B}1*}^{\Omega}(p_{\mathrm{B}}^{}, q_{\mathrm{B}2}^{})^{\Omega !}(A\otimes \mathbb{D}'(T))). 
\end{align*}
Here $q_{\mathrm{B}i}^{}, q_{\mathrm{H}i}^{}$ are the components of $q_{\mathrm{B}}^{}, q_{\mathrm{H}}^{}$
in the diagrams (\ref{dia:hecke}), (\ref{dia:heckeS}), 
$\mathbb{D}'(-)=\mathbb{D}(-) \otimes \omega_{\mathcal{C}(m)}[-2rm]$, 
and we have used maps in the following diagrams 
\begin{align}
 \notag&\xymatrix{
\mathrm{Hecke}_{\GL_r}(\chi, \chi+m) \ar[r]^-{q_{\mathrm{B}1}^{}} \ar[d]_-{(p_{\mathrm{B}}^{}, q_{\mathrm{B}2}^{})} & 
    \Bun_{\GL_r}(\chi) \\
    \Bun_{\GL_r}(\chi+m) \times \mathcal{C}(m), & 
    } \\
  \label{dia:HeckeS2}  &\xymatrix{
\mathrm{Hecke}_{\GL_r}^{\mathrm{Hig}}(\chi, \chi+m) \ar[r]^-{q_{\mathrm{H}1}^{}} \ar[d]_-{(p_{\mathrm{H}}^{}, q_{\mathrm{H}2}^{})} & 
    \Hig_{\GL_r}(\chi) \\
    \Hig_{\GL_r}(\chi+m) \times \mathcal{S}(m). & 
    }
\end{align}
The map $q_{\mathrm{B}1}^{}$ is smooth and projective, and $(p_{\mathrm{B}}^{}, q_{\mathrm{B}2}^{})$ is smooth, 
see Lemma~\ref{lem:smooth:hecke}. 
Moreover, we also used an isomorphism 
\begin{align}\label{isom:can}
    \omega_{p_{\mathrm{H}}^{}}\otimes \omega_{q_{\mathrm{B}}^{}}^{-1}\cong \omega_{(p_{\mathrm{B}}^{}, q_{\mathrm{B}2}^{})}\otimes \omega_{\mathcal{C}(m)}
\end{align}
on $\mathrm{Hecke}_{\GL_r}^{\mathrm{Hig}}(\chi, \chi+m)$, 
see Lemma~\ref{lem:isomcan}.

Since $\omega_{\mathcal{C}(m)}$ descends to $\mathcal{C}(m)\to \mathrm{Sym}^m(C)$, we have 
$\mathbb{D}'(T)\in \mathbb{T}(m)_0$. Hence, by an argument similar to Lemma~\ref{lem:hecke}, we have 
\begin{align*}
    q_{\mathrm{B}1*}^{\Omega}(p_{\mathrm{B}}^{}, q_{\mathrm{B}2}^{})^{\Omega !}(A\otimes \mathbb{D}'(T))
    \in \widetilde{\LL}(\Hig_{\GL_r}(\chi))_{\delta_w\otimes \omega_{\Bun}^{1/2}}. 
\end{align*}
The above object is supported on $Z_{c-m/r}(\chi)$ by Lemma~\ref{lem:support} below. 
Therefore by Corollary~\ref{cor:Ltilde}, for $c\gg 0$ we have the vanishing of (\ref{ETA:0}). 
    \end{proof}
    We have used the following lemmas:

\begin{lemma}\label{lem:isomcan}
    There is an isomorphism of line bundles (\ref{isom:can}). 
\end{lemma}
    \begin{proof}
    We have the following 
identifications of fibers
\begin{align*}
    &\omega_{p_{\mathrm{H}}^{}}|_{(0\to E\to E'\to R\to 0)}=\det \chi_S(E, R)^{\vee}, \\
    &\omega_{q_{\mathrm{B}}^{}}|_{(0\to E\to E'\to R\to 0)}=\det \chi_C(\pi_{*}R, \pi_{*}E), \\
     &\omega_{(p_{\mathrm{B}}^{}, q_{\mathrm{B}2}^{}})|_{(0\to E\to E'\to R\to 0)}=\det \chi_C(\pi_{*}E', \pi_{*}R)^{\vee}, \\
     &\omega_{\mathcal{C}(m)}|_{(0\to E\to E'\to R\to 0)}=\det\chi_C(\pi_{*}R, \pi_{*}R).
\end{align*}

Then we have 
\begin{align*}
    \det (\chi_S(E, R)^{\vee}-\chi_C(\pi_{*}R, \pi_{*}E))&=
    \det(\chi_S(R, E)-\chi_C(\pi_{*}R, \pi_{*}E)) \\
    &=\det (-\chi_C(\pi_{*}R\otimes \omega_C^{-1}, \pi_{*}E))\\
    &=\det \chi_C(\pi_{*}E, \pi_{*}R)^{\vee}. 
\end{align*}
A similar computation shows that the fiber of 
the right-hand side of (\ref{isom:can}) is also identified with $\det \chi_C(\pi_{*}E, \pi_{*}R)^{\vee}$, therefore the lemma holds. 
    \end{proof}
    \begin{lemma}\label{lem:support}
In the diagram (\ref{dia:HeckeS2}), we have 
\begin{align}\label{inc:Z}
q_{\mathrm{H}1}^{}(p_{\mathrm{H}}^{}, q_{\mathrm{H}2}^{})^{-1}(Z_c(\chi+m)\times \mathcal{S}(m)) \subset Z_{c-m}(\chi).
\end{align}
    \end{lemma}
\begin{proof}
We consider an exact sequence (\ref{exact:R}) in $\Coh^{\heartsuit}(S)$
such that 
\begin{align*}
    (p_{\mathrm{H}}^{}, q_{\mathrm{H}2}^{})(0 \to E \to E' \to R \to 0) \in Z_c(\chi+m)\times \mathcal{S}(m). 
\end{align*}
The above condition is equivalent to that there is a
subsheaf $F \subset E'$ such that $\mu(F)>c$. 
Let $F'=\mathrm{Ker}(F\hookrightarrow E' \twoheadrightarrow R)$. 
Then $\mu(F') >c-m$, hence 
the inclusion (\ref{inc:Z}) holds. 
\end{proof}

Recall the equivalence (\ref{equiv:1/2lim}) for a fixed choice of $\omega_C^{1/2}\in \mathrm{Pic}(C)$. Then by Lemma~\ref{lem:hecke}, together with Corollary~\ref{cor:coprime} and Example~\ref{exam:coprime}, we obtain the following corollary: 
\begin{cor}\label{cor:1}
For a fixed $(r, w)$, there is a right action of $\mathbb{H}$ on 
\begin{align*}
\bigoplus_{\chi \in \mathbb{Z}}\LL(\Hig_{\GL_r}(\chi))_w \left(=  \bigoplus_{\chi \in \mathbb{Z}}\mathbb{T}_{\mathrm{GL}_r}(\chi)_w \mbox{ if }(r, w) \mbox{ are coprime}\right).
\end{align*}
    \end{cor}

 

    \subsection{Hecke operator (left action)}
We also give a left Hecke action using the torsion pair of $\Coh^{\heartsuit}(C)$, 
following the idea of~\cite{DiaconescuPortaSala2022}. 
    
    Let $\mathcal{T} \subset \Coh^{\heartsuit}(C)$ be
    the subcategory of torsion sheaves and $\mathcal{F} \subset 
    \Coh^{\heartsuit}(C)$ the subcategory of torsion-free sheaves. 
    The pair $(\mathcal{T}, \mathcal{F})$ is a \textit{torsion pair} on 
    $\Coh^{\heartsuit}(C)$ and its tilting gives a heart, see~\cite{HRS}
    \begin{align*}
        \Coh^{\spadesuit}(C)=\langle \mathcal{F}, \mathcal{T}[-1] \rangle_{\mathrm{ex}} \subset \Coh(C). 
    \end{align*}
    Here $\langle -\rangle_{\mathrm{ex}}$ is the extension-closure. 
    We define 
    \begin{align*}
   \mathrm{Hecke}_{\GL_r}^{\spadesuit}(\chi+m, \chi)
    \end{align*}
    to be the moduli stack classifying exact sequences in 
    $\Coh^{\spadesuit}(C)$
    \begin{align}\label{exact:suit}
        0 \to Q[-1] \to F \to F' \to 0
    \end{align}
    where $F, F'$ are vector bundles on $C$ with 
    \begin{align*}(\rank(F), \deg(F))=(r, \chi), \ 
    (\rank(F'), \deg(F'))=(r, \chi+m)
    \end{align*}
    and $Q \in \Coh^{\heartsuit}(C)$ such that $\dim Q=0$ and $\chi(Q)=m$. 
    
    Since giving an exact sequence (\ref{ex:heart}) in $\Coh^{\heartsuit}(C)$ 
    is equivalent to giving 
    an exact sequence (\ref{exact:suit}) in $\Coh^{\spadesuit}(C)$, 
    there is an isomorphism
    \begin{align*}
       \mathrm{Hecke}_{\GL_r}^{\spadesuit}(\chi+m, \chi) \stackrel{\sim}{\to} 
        \mathrm{Hecke}_{\GL_r}(\chi, \chi+m)
    \end{align*}
    given by 
    \begin{align*}
        (0\to Q[-1] \to F\to F'\to 0) \mapsto (0\to F\to F' \to Q \to 0).
    \end{align*}
    
We have the diagram of evaluation morphisms 
\begin{align}\label{dia:hecke2}
    \xymatrix{
\mathrm{Hecke}_{\GL_r}^{\spadesuit}(\chi+m, \chi) \ar[r]^-{p'_{\mathrm{B}}} \ar[d]_-{q'_{\mathrm{B}}} & 
    \Bun_{\GL_r}(\chi) \\
   \mathcal{C}(m) \times \Bun_{\GL_r}(\chi+m). & 
    }
\end{align}
\begin{lemma}\label{lem:smooth:hecke}
The morphism $p'_{\mathrm{B}}$ is smooth and projective, 
and the morphism $q'_{\mathrm{B}}$ is smooth. 
\end{lemma}
 \begin{proof}
For a vector bundle $F$ on $C$, we set 
$\mathbb{D}(F)=F^{\vee}\otimes \omega_C$. 
The statement for $p'_{\mathrm{B}}$ follows from the commutative diagram 
\begin{align*}
\xymatrix{
   \mathrm{Hecke}_{\GL_r}^{\spadesuit}(\chi+m, \chi) \ar[r]^-{\sim} \ar[d]_-{p'_{\mathrm{B}}} &
    \mathrm{Hecke}_{\GL_r}(\chi, \chi+m) \ar[r]_-{\mathbb{D}}^-{\sim} & \mathrm{Hecke}_{\GL_r}(-\chi-m, -\chi) \ar[d]_-{p_{\mathrm{B}}^{}} \\
    \Bun_{\GL_r}(\chi) \ar[rr]_{\mathbb{D}}^-{\sim} & & \Bun_{\GL_r}(-\chi). 
    }
\end{align*}
     Here the right top equivalence is given by 
     \begin{align*}
         (0 \to F \to F' \to Q \to 0) \mapsto (0 \to \mathbb{D}(F') \to \mathbb{D}(F) \to \mathbb{D}(Q) \to 0)
      \end{align*}
      where $\mathbb{D}(Q)=\mathcal{E}xt^1_C(Q, \omega_C)$.
      The statement for $q$ is similarly proved. 
 \end{proof}   
 Similarly to the right action, up to the equivalence (\ref{equiv:1/2lim}) we have the functor 
 \begin{align*}
     p_{\mathrm{B}\ast}'^{\Omega} q_{\mathrm{B}}'^{\Omega!} \colon 
     \mathbb{T}(m)_0 \otimes 
     \LL(\Hig_{\GL_r}(\chi+m))_w  \to 
     \LL(\Hig_{\GL_r}(\chi))_w.
 \end{align*}
Therefore we obtain the following corollary, 
which is a left analogue of Corollary~\ref{cor:1}.
\begin{cor}\label{cor:1prime}
For a fixed $(r, w)$, there is a left action of $\mathbb{H}$ on 
\begin{align*}
\bigoplus_{\chi \in \mathbb{Z}}\LL(\Hig_{\GL_r}(\chi))_w \left(=  \bigoplus_{\chi \in \mathbb{Z}}\mathbb{T}_{\mathrm{GL}_r}(\chi)_w \mbox{ if }(r, w) \mbox{ are coprime}\right).
\end{align*}
    \end{cor}

    
\subsection{Hecke operators for arbitrary reductive groups}\label{subsec:heckeG}
For an arbitrary reductive group $G$, we do not expect an action of $\mathbb{H}$. 
In the context of geometric Langlands, we use the following stack 
\begin{align}\label{Hecke:Gx}
    \mathrm{Hecke}_{G, x}=\{(E, E', \phi): \phi \colon E|_{C\setminus x} \stackrel{\cong}{\to} E'|_{C\setminus x} \}
\end{align}
(or its Ran version, see~\cite[Section~4.4]{Gaitsout}). Since they are infinite-dimensional, 
it is difficult to deal with the limit categories for their cotangent stacks at this moment. 

Instead, following~\cite[Section~2]{DoPa}, for a dominant cocharacter $\mu \in N(T)_{+}$, 
we consider the following closed substack 
\begin{align}\label{Heckemu}
    \mathrm{Hecke}_{G, x}^{\mu} \subset \mathrm{Hecke}_{G, x}
\end{align}
consisting of tuples $(E, E', \phi)$ such that for any dominant weight $\lambda \in M(T)_{+}$, 
the isomorphism $\phi$ induces the injection 
\begin{align*}
    E\times^{G}V(\lambda) \hookrightarrow (E'\times^{G}V(\lambda))\otimes \mathcal{O}_C(\langle \mu, \lambda\rangle)
\end{align*}
where $V(\lambda)$ is the irreducible $G$-representation with highest weight $\lambda$. 
The substack (\ref{Heckemu}) is finite-dimensional, and it is smooth if and only if $\mu$ is minuscule. 

Then we have the following diagram 
\begin{align}\label{Hecke:dia:pq}
    \xymatrix{\mathrm{Hecke}_{G}^{\mu} \ar[r]^-{p_{\mathrm{B}x}^{\mu}} \ar[d]_-{q_{\mathrm{B}x}^{\mu}} & 
    \Bun_{G}(\chi+\overline{\mu}) \\ \Bun_G(\chi). & 
    }
\end{align}
Here $\overline{\mu}\in \pi_1(G)$ is the image of $\mu$ under the natural map 
$N(T)\twoheadrightarrow \pi_1(G)$.
The maps $p_{\mathrm{B}x}^{\mu}$, $q_{\mathrm{B}_x}^{\mu}$ are projective, 
and smooth if $\mu$ is minuscule (indeed fibers are Schubert varieties associated with 
the minuscule cocharacter $\mu$). Therefore under the assumption that $\mu$ is minuscule, 
similarly to Proposition~\ref{prop:hecke2}
we have the associated Hecke functor (with modification by shift)
\begin{align}\label{funct:Hmu}
    H_x^{\mu}:=(q_{\mathrm{B}_x}^{\mu})^{\Omega !}(p_{\mathrm{B}x}^{\mu})^{\Omega}_{*}[-\dim p_{\mathrm{B}_x}^{\mu}] \colon \LL(\Hig_G(\chi))_w \to \LL(\Hig_{G}(\chi+\overline{\mu}))_w. 
\end{align}
\begin{remark}\label{rmk:mini}
If $\mu$ is not minuscule, 
because of the singularities of (\ref{Heckemu}), it is still difficult to construct 
a Hecke functor as above. We expect the construction of such functors via solving the classical limit of geometric Satake correspondence, as formulated in~\cite[Conjecture~1.6, 1.7]{CW}.
\end{remark}

\begin{example}
If $G=\GL_r$, and take $\mu$ to be 
\begin{align*}
    \mu=(\overbrace{1, \ldots, 1}^{m}, 0, \ldots, 0) \in \mathbb{Z}^r
\end{align*}
    for $0\leq m\leq r$. It is minuscule coweight, and the fibers of the maps in (\ref{Hecke:dia:pq}) are Grassmannians $\mathrm{Gr}(r, m)$.
In this case, we can describe the Hecke functor $H_x^{\mu}$ in terms of a right action 
of an object $A_m \in \mathbb{T}(m)_0$, defined in (\ref{Am}) as follows. 

We have the following (non-Cartesian) diagram 
\begin{align*}
\xymatrix{
\mathfrak{gl}_m/\GL_m \inclusion^-{i} \ar[d] & \mathcal{S}(m) \ar[d] \\
B\GL_m \inclusion & \mathcal{C}(m). 
}
    \end{align*}
Here 
the bottom horizontal arrow is the closed immersion corresponding to the closed point 
$F=\mathcal{O}_x^{\oplus m}$, and the top arrow corresponds 
to the Higgs field 
$\Hom(F, F\otimes \omega_C)=\mathfrak{gl}_m$. 
We set 
\begin{align}\label{Am}
    A_m :=i_{\ast}\mathcal{O}_{\mathfrak{gl}_m/\GL_m}[-mr] \in \Coh(\mathcal{S}(m)). 
\end{align}
It is straightforward to check that the functor $H_x^{\mu}$ is given by 
the right action of $A_m$. 
\end{example}
\subsection{Wilson operators}
On the spectral side, let 
\begin{align*}
    \mathscr{E}_{G, x}\to \Bun_G
\end{align*}
be the universal $G$-bundle restricted to $x$ as in (\ref{univ:Pc}).
The coweight $\mu$ is regarded as a weight for the Langlands dual group $^{L}G$, 
and we consider the associated vector bundle 
\begin{align*}
    \mathcal{V}_{x}^{\mu} :=\mathscr{E}_{G, x}\times^{G} V(\mu) \to \Bun_G. 
\end{align*}
The Wilson operator is defined by 
\begin{align*}
W_x^{\mu} :=
    \otimes (\mathcal{V}_{x}^{\mu})^{\vee} \colon \Coh(\Hig_{^{L}G}(w)^{\mathrm{ss}})_{-\chi}
    \to \Coh(\Hig_{^{L}G}(w)^{\mathrm{ss}})_{-\chi-\overline{\mu}}.
\end{align*}

Then we can formulate the following compatibility of Wilson/Hecke operators: 
\begin{conj}\label{conj:Hecke}
A
conjectural equivalence 
\begin{align*}
    \Coh(\Hig_{^{L}G}(w)^{\mathrm{ss}})_{-\chi} \stackrel{\sim}{\to} 
    \LL(\Hig_{G}(\chi))_{w}
\end{align*}
in Conjecture~\ref{conj:Higgs}
intertwines $W_x^{\mu}$ and $H_x^{\mu}$ for minuscule coweight $\mu$. 
\end{conj}

\begin{remark}\label{rmk:adjoint}
Let $\mu^{\vee}$ be the dominant weight such that $V(\mu^{\vee})=V(\mu)^{\vee}$. 
One can check that $H_{x}^{\mu^{\vee}}$ is both of right and left adjoint of 
$H_x^{\mu}$. This is compatible with the fact that $W_{x}^{\mu^{\vee}}$ is both of right and left adjoint of 
$W_x^{\mu}$.  
\end{remark}

\subsection{Interaction of Hecke operators with the semiorthogonal decomposition}
In general, Hecke operators do not preserve the semiorthogonal decomposition (\ref{sod:GLr}). 
However, we can control their amplitudes by the following proposition: 
\begin{prop}\label{prop:Hecke:sod}
The right action of an object in $A\in \mathbb{T}(m)_0$ restricted to a semiorthogonal summand
in (\ref{sod:GLr}) gives a functor
    \begin{align*}
    (-)\ast A \colon 
\bigotimes_{i=1}^k \mathbb{T}_{\GL_{r_i}}(\chi_i)_{w_i} 
\to \left\langle \bigotimes_{i=1}^{k'} \mathbb{T}_{\GL_{r_i'}}(\chi_i')_{w_i'}\right\rangle.
    \end{align*}
    Here the right-hand side is after all partitions 
    $(r_1', \chi_1', w_1')+\cdots+(r_{k'}', \chi_k', w_k')=(r, \chi+m, w)$ such that 
    \begin{align*}
        \frac{\chi_1'}{r_1'}>\cdots>\frac{\chi_{k'}}{r_k'}, \
        \frac{\chi_1'}{r_1'}\geq \frac{\chi_1}{r_1}\geq\frac{\chi_1'-m}{r_1'}, \
        \frac{w_i'}{r_i'}=\frac{w}{r}.
    \end{align*}
\end{prop}
\begin{proof}
    As in the proof of Lemma~\ref{lem:support}, for an exact sequence in $\Coh(S)$
    \begin{align*}
        0\to E\to E'\to R \to 0
    \end{align*}
    where $R$ is zero-dimensional and $E, E'$ pure one-dimensional with compact supports, 
    we have 
    \begin{align*}
        \mu^{\mathrm{max}}(E') \leq \mu^{\mathrm{max}}(E) \leq \mu^{\mathrm{max}}(E')-\frac{m}{r_1'}
    \end{align*}
    Here $r_1'$ is the rank of the first HN factor of $E'$. Then the Proposition holds 
    by the same argument in the proof of Proposition~\ref{prop:hecke2}.
\end{proof}

\begin{example}\label{exam:Hecke}
We describe how Wilson and Hecke operators interact with the semiorthogonal decompositions in Theorem~\ref{thm:PThiggs}, Theorem~\ref{thm:mainLG} in a simple example. 
    We consider the case $g(C)=0$, i.e. $C=\mathbb{P}^1$, and $G=\GL_2$. 
    We take the minuscule coweight $\mu$ to be $\mu=(1, 0)$. 
    
    In this case, 
    we have 
\begin{align*}
    \Coh(\Hig_{\GL_2}(0)^{\mathrm{ss}})_0 &=\Coh^{\shear}(\mathfrak{gl}_2/\GL_2)_0 \\
    &=\langle \mathbb{T}_{\mathbb{G}_m}(0)_{-w}\otimes \mathbb{T}_{\mathbb{G}_m}(0)_w, \mathbb{T}_{\GL_2}(0)_0 : w>0\rangle, \\
    \Coh(\Hig_{\GL_2}(0)^{\mathrm{ss}})_{-1} &=\Coh^{\shear}(\mathfrak{gl}_2/\GL_2)_{-1} \\
    &=\langle \mathbb{T}_{\mathbb{G}_m}(0)_{-w-1}\otimes \mathbb{T}_{\mathbb{G}_m}(0)_w : w\geq 0\rangle.
\end{align*}

Let $V$ be a two-dimensional vector space. Then $\mathbb{T}_{\GL_2}(0)_0$ is generated by 
\begin{align*}E_0:=\mathcal{O}_{\mathfrak{gl}_2/\GL_2}\in \mathbb{T}_{\GL_2}(0)_0.
\end{align*}
Also 
$\mathbb{T}_{\mathbb{G}_m}(0)\otimes \mathbb{T}_{\mathbb{G}_m}(-1)$ is generated by 
\begin{align*}
F_0:=\mathcal{O}_{\mathfrak{gl}_1/\mathbb{G}_m} \otimes\mathcal{O}_{\mathfrak{gl}_1/\mathbb{G}_m}(-1) =
V^{\vee}\otimes \mathcal{O}_{\mathfrak{gl}_2/\GL_2}.
\end{align*}
Moreover 
    $ \mathbb{T}_{\mathbb{G}_m}(0)_{-w}\otimes \mathbb{T}_{\mathbb{G}_m}(0)_w$ for $w>0$ is generated by a two-term complex of the form 
    \begin{align*}
        E_w:=&\left(\mathrm{Sym}^{2w-2}(V) \otimes (\det V)^{1-w} \otimes  \mathcal{O}_{\mathfrak{gl}_2/\GL_2}\right. \\
        &\left.\to \mathrm{Sym}^{2w}(V)\otimes (\det V)^{-w} \otimes  \mathcal{O}_{\mathfrak{gl}_2/\GL_2}\right).
    \end{align*}
    Similarly $\mathbb{T}_{\mathbb{G}_m}(0)_{-w-1}\otimes \mathbb{T}_{\mathbb{G}_m}(0)_w$ for $w>0$ is generated by a two-term complex of the form 
    \begin{align*}
        F_w:=&\left(\mathrm{Sym}^{2w-1}(V)\otimes (\det V)^{-w} \otimes  \mathcal{O}_{\mathfrak{gl}_2/\GL_2} \right. \\
        &\left.\to \mathrm{Sym}^{2w+1}(V)\otimes (\det V)^{-1-w} \otimes  \mathcal{O}_{\mathfrak{gl}_2/\GL_2}\right).
    \end{align*}
    
Therefore we have 
\begin{align*}
    &K(\Coh(\Hig_{\GL_2}(0)^{\mathrm{ss}})_0)=\bigoplus_{w\geq 0}\mathbb{Z}[E_w], \\ 
    & K(\Coh(\Hig_{\GL_2}(0)^{\mathrm{ss}})_{-1})=\bigoplus_{w\geq 0}\mathbb{Z}[F_w].
\end{align*}
By the decomposition 
\begin{align*}
    \mathrm{Sym}(V)\otimes V^{\vee}=(\mathrm{Sym}^{m+1}(V)\otimes \det V^{\vee}) \oplus \mathrm{Sym}^{m-1}(V)
\end{align*}
the operator $W_x^{\mu}$ in K-theory is given by 
\begin{align}\label{wilson:GL2}
    [E_0]\mapsto [F_0], \ [E_w] \mapsto [F_w]+[F_{w+1}], \ (w>0).
\end{align}

In the automorphic side, we have the semiorthogonal decompositions 
\begin{align*}
    \LL(\Hig_{\GL_2}(0))_0 
    &=\langle \mathbb{T}_{\mathbb{G}_m}(\chi)_{0}\otimes \mathbb{T}_{\mathbb{G}_m}(-\chi)_0, \mathbb{T}_{\GL_2}(0)_0 : \chi>0\rangle, \\
    \LL(\Hig_{\GL_2}(1))_{0}
    &=\langle \mathbb{T}_{\mathbb{G}_m}(\chi+1)_{0}\otimes \mathbb{T}_{\mathbb{G}_m}(-\chi)_w : \chi\geq 0\rangle.
\end{align*}
By Proposition~\ref{prop:Hecke:sod}, the Hecke operator $H_x^{\mu}$ for $\mu=(1, 0)$ restricts to the functor
    \begin{align*}
        H_x^{\mu} \colon \mathbb{T}_{\GL_2}(0)_0 \to \mathbb{T}_{\mathbb{G}_m}(1)_0\otimes \mathbb{T}_{\mathbb{G}_m}(0)
        \end{align*}
        and for $\chi>0$, 
        \begin{align*}
            H_x^{\mu} \colon &
       \mathbb{T}_{\mathbb{G}_m}(\chi)_{0}\otimes \mathbb{T}_{\mathbb{G}_m}(-\chi)_0 \\
       &\to \langle \mathbb{T}_{\mathbb{G}_m}(\chi+1)_{0}\otimes \mathbb{T}_{\mathbb{G}_m}(-\chi)_0, 
       \mathbb{T}_{\mathbb{G}_m}(\chi)_{0}\otimes \mathbb{T}_{\mathbb{G}_m}(1-\chi)_0\rangle. 
    \end{align*}
    A simple geometric reason is that, for a vector bundle 
    $\mathcal{O}_{\mathbb{P}^1}(\chi)\oplus \mathcal{O}_{\mathbb{P}^1}(-\chi)$ on $\mathbb{P}^1$ and 
    an exact sequence 
    \begin{align*}
        0 \to \mathcal{O}_{\mathbb{P}^1}(\chi)\oplus \mathcal{O}_{\mathbb{P}^1}(-\chi) \to F \to \mathcal{O}_x \to 0
    \end{align*}
    for a vector bundle $F$ on $\mathbb{P}^1$, we have 
    \begin{align*}
    F\cong \begin{cases}
        \mathcal{O}_{\mathbb{P}^1}(1)\oplus \mathcal{O}_{\mathbb{P}^1}, & (\chi=0), \\
         \mathcal{O}_{\mathbb{P}^1}(\chi+1)\oplus \mathcal{O}_{\mathbb{P}^1}(-\chi), \mbox{ or }
         \mathcal{O}_{\mathbb{P}^1}(\chi)\oplus \mathcal{O}_{\mathbb{P}^1}(1-\chi), & (\chi>0).
        \end{cases}
    \end{align*}
    Therefore it is compatible with the computation of the Wilson operator (\ref{wilson:GL2}). 
\end{example}

\subsection{Hecke actions on (ind-)limit categories with nilpotent singular supports}
Let 
\begin{align*}
    \mathbb{H}_{\mathcal{N}} \subset \mathbb{H}, \ \Ind \mathbb{H}_{\mathcal{N}} \subset \Ind \mathbb{H}
\end{align*}
be the subcategories with nilpotent singular supports. It is easy to see that the above subcategories 
are monoidal subcategories with respect to the product (\ref{Tm:prod}). 
In view of Subsection~\ref{subsec:limsupp}, we regard them as 
a classical limit of (\ref{dmod:hall})
\begin{align*}
    \bigoplus_{m\in \mathbb{Z}_{\geq 0}} \text{D-mod}(\mathcal{C}(m)) \leadsto \Ind\mathbb{H}_{\mathcal{N}}.
\end{align*}
Using Corollary~\ref{cor:nilp2} and the arguments in this section, 
we have the following result, which gives a classical limit of the 
Hecke actions of (\ref{dmod:hall}) on (\ref{dmod:heckebun}): 
\begin{cor}\label{cor:genuinehecke}
The monoidal category $\Ind_{\mathcal{N}}\mathbb{H}$ acts on 
\begin{align*}
    \bigoplus_{\chi \in \mathbb{Z}}\IndL_{\mathcal{N}}(\Hig_{\GL_r}(\chi))_w
\end{align*}
from right and left. These actions restrict to actions of $\mathbb{H}_{\mathcal{N}}$ on 
\begin{align*}
    \bigoplus_{\chi \in \mathbb{Z}}\LL_{\mathcal{N}}(\Hig_{\GL_r}(\chi))_w \left(=\bigoplus_{\chi\in \mathbb{Z}}\mathbb{T}_{\mathcal{N}, \GL_r}(\chi)_w \mbox{ if }(r,w) \mbox{ are coprime} \right).
\end{align*}
\end{cor}

\section{Some complementary results}
In this section, we give proofs of some complementary results. 

\subsection{Proof of Proposition~\ref{prop:ess}}\label{subsec:essurj}
The proof is the same as in~\cite[Proposition~3.22]{PTquiver} with some 
minor modifications. We only explain how to modify the arguments. 

\begin{proof}
It is enough to show the essential surjectivity of 
\begin{align*}
    j^* \colon \MM(\Y)_{\delta} \to \MM(\Y^{\ell\text{-ss}})_{\delta}
\end{align*}
as the essential surjectivity of the pullback \[\MM(\overline{\mathcal{Z}}_i)_{\delta_i} \to 
\MM(\mathcal{Z}_i)_{\delta}\] also follows from the above. 
Let $T \subset G$ be a maximal torus and $T' \subset T$ the 
subgroup of $t\in T$ which acts on $Y$ trivially. We have the 
decomposition 
\begin{align*}
    \Coh(\Y)=\bigoplus_{w \in (T')^{\vee}}\Coh(\Y)_{w}
\end{align*}
where $\Coh(\Y)_{w}$ is the subcategory of $T'$-weights $w$. 
Let $M_{0, \mathbb{R}}$ be the kernel 
\begin{align*}
    M_{0, \mathbb{R}}:=\Ker(M(T)_{\mathbb{R}} \twoheadrightarrow M(T')_{\mathbb{R}}).
\end{align*}

Note that we may assume that $\ell \in M_{0, \mathbb{R}}$ as otherwise we have $\Y^{\ell\text{-ss}}=\emptyset$.
Then we have $\mathbf{W} \subset M_{0, \mathbb{R}}$ where 
$\mathbf{W}$ is the polytope in (\ref{poly:W}). 
Let $\lambda_1, \ldots, \lambda_m\in N(T)$ be anti-dominant cocharacters 
such that $\lambda_i^{\perp}$ give faces of $\mathbf{W}$. 
Let $W$ be the Weyl group of $G$.
Then we have the following generalization of~\cite[Lemma~3.16]{PTquiver}: 
\begin{lemma}\label{lem:magicgen}
    For $\ell\in M_{0, \mathbb{R}}$ and $\delta \in M(T)_{\mathbb{R}}^W$ with 
    $T'$-weight $w$, the category $\Coh(\Y^{\ell\text{-ss}})_w$ is generated by 
    $j^*(\MM(\Y)_{\delta})$ and objects of the form $m_{\lambda}(P)$
    where $P \in \Coh(\Y^{\lambda})_w$ is generated by $V_{G^{\lambda}}(\chi'')\otimes \mathcal{O}_{\Y^{\lambda}}$ such that 
    \begin{align}\label{ineq:chi''}
        \langle \lambda, \chi'' \rangle <-\frac{1}{2}\langle \lambda, \mathbb{L}_{\Y}^{\lambda>0}\rangle+\langle \lambda, \delta\rangle
    \end{align}
    and $\lambda \in \{\lambda_1, \ldots, \lambda_m\}$ such that $\langle \lambda, \ell \rangle=0$. Here $V_{G^{\lambda}}(\chi'')$ is the irreducible $G^{\lambda}$-representation with highest weight $\chi''$. The functor $m_{\lambda}$ is 
    the categorical Hall product $m_{\lambda}=p_{*}q^{*}$ for the diagram 
    \begin{align*}
        \Y^{\lambda} \stackrel{q}{\leftarrow} \Y^{\lambda \geq 0} \stackrel{p}{\to} \Y.
    \end{align*}
\end{lemma}
\begin{proof}
    The proof of~\cite[Lemma~3.16]{PTquiver} applies verbatim, 
    using the induction of the invariants $(r_{\chi}, p_{\chi})$ as in~\cite{SvdB}. 
    The relevant reference~\cite[Proposition~3.4]{hls} in loc. cit. is discussed 
    in the setting of the quotient stack $Y/G$. Consequently, it is applicable in our situation. 
\end{proof}

By Lemma~\ref{lem:magicgen}, it is enough to show the vanishing 
\begin{align}\label{vanish:Mres}
    \Hom(\MM(\Y^{\ell\text{-ss}})_{\delta}, j^*(m_{\lambda}(P)))=0,
\end{align}
where $m_{\lambda}(P)$ is an object as in Lemma~\ref{lem:magicgen}, see the first part 
of the proof of~\cite[Proposition~3.22]{PTquiver}. 
By Lemma~\ref{lem:cart2} below, we have the following diagram 
\begin{align}\label{dia:semistable}
    \xymatrix{
(\Y^{\lambda})^{\ell\text{-ss}} \dinclusion\diasquare & \widetilde{\Y} \ar[l]_-{q'} \ar[r]^-{p'} \dinclusion \diasquare & \Y^{\ell\text{-ss}} \dinclusion \\
\Y^{\lambda} & \ar[l]_-{q} \Y^{\lambda \geq 0} \ar[r]^-{p} & \Y
    }
\end{align}
where both squares are Cartesian. 
Then to show the vanishing (\ref{vanish:Mres}), it is enough to show the vanishing 
\begin{align}\label{vanish:Mres2}
    \Hom(p'^{*}E, q'^{*}P|_{(\Y^{\lambda})^{\ell\text{-ss}}})=0, \ 
    E\in \MM(\Y^{\ell\text{-ss}})_{\delta}.
\end{align}

Let $\Y^{\ell\text{-ss}}\to Y^{\ell\text{-ss}}$ be the good moduli space.
As in the proof of~\cite[Proposition~3.22]{PTquiver}, it is enough to show the 
vanishing (\ref{vanish:Mres2}) formally locally on $Y^{\ell\text{-ss}}$.
Let $p \in Y^{\ell\text{-ss}}$, and use the same symbol $p$ to denote the unique 
close point in the fiber of $\Y^{\ell\text{-ss}}\to Y^{\ell\text{-ss}}$ at $p$. 
Formally locally at $p\in Y^{\ell\text{-ss}}$, the top horizontal diagram of (\ref{dia:semistable}) is of the form 
\begin{align*}
    A^{\lambda}/G_p^{\lambda} \leftarrow A^{\lambda \geq 0}/G_p^{\lambda \geq 0} \to 
    A/G_p
\end{align*}
where $A$ is a smooth affine scheme (of finite type over a complete local ring) with
a $G_p=\mathrm{Aut}(p)$-action, see~\cite[Lemma~3.29]{PTquiver}.
Then the vanishing (\ref{vanish:Mres}) on $A/G_p$ follows by comparing the $\lambda$-weights
of $p'^{*}E$ and $q'^{*}P|_{(\Y^{\lambda})^{\ell\text{-ss}}}$
on $A^{\lambda \geq 0}/G_p^{\lambda\geq 0}$ using the inequality (\ref{ineq:chi''}), 
see the last part of the proof of~\cite[Proposition~3.22]{PTquiver}.

\end{proof}
We have used the following lemma, which is proved in~\cite[Proposition~3.22]{PTquiver}
for representations of quivers using~\cite[Lemma~3.4]{PTquiver}. 
\begin{lemma}\label{lem:cart2}
    We have $q^{-1}((\Y^{\lambda})^{\ell\text{-ss}})=p^{-1}(\Y^{\ell\text{-ss}})$.
\end{lemma}
\begin{proof}
We first prove that 
\begin{align}\label{incl1}q^{-1}((\Y^{\lambda})^{\ell\text{-ss}})\subset 
p^{-1}(\Y^{\ell\text{-ss}}).
\end{align}
We use the same symbol $q$ to denote the map \begin{align*}q\colon Y^{\lambda \geq 0} \to Y^{\lambda}, \
y\mapsto \lim_{t\to 0}\lambda(t)(y).\end{align*}
Let $y\in Y^{\lambda \geq 0}$ be such that 
$q(y) \in Y^{\lambda}$ is $\ell$-semistable. Let $\lambda' \colon \mathbb{G}_m \to G$
be a one-parameter subgroup such that $\lim_{t\to 0}\lambda'(t)(y) \in Y$ exists. 
The above $\lambda'$ determines a map 
$f\colon \Theta \to \Y$ with $f(1)\simeq y$. 
Since $p \colon \Y^{\lambda \geq 0} \to \Y$ is proper, there is 
$\widetilde{f} \colon \Theta \to \Y^{\lambda \geq 0}$ such that $\widetilde{f}(1)\simeq y$
in $\Y^{\lambda\geq 0}$. 
Then we have that 
\begin{align*}
\langle \lambda', \ell\rangle=
    c_1(f(0)^* \ell)=c_1(\widetilde{f}(0)^* \ell)=c_1((q\circ \widetilde{f})(0)^* \ell|_{\Y^{\lambda}}) \geq 0
\end{align*}
where the last inequality follows from that $q(y) \in Y^{\lambda}$ is 
$\ell$-semistable, hence (\ref{incl1}) holds. 

We next prove that 
\begin{align}\label{incl2}q^{-1}((\Y^{\lambda})^{\ell\text{-ss}})\supset 
p^{-1}(\Y^{\ell\text{-ss}}).
\end{align}
    Let $y\in Y^{\lambda \geq 0}$ be a point which is $\ell$-semistable in $Y$. 
    We need to show that $q(y) \in Y^{\lambda}$ is $\ell$-semistable. 
    
    Let $\lambda' \colon \mathbb{G}_m \to G^{\lambda}$ be a one-parameter subgroup 
    such that the limit
    \begin{align*}\lim_{t\to 0} \lambda'(q(y)) \in Y^{\lambda}
    \end{align*}exists. 
    Since the fibers of $Y^{\lambda\geq 0} \to Y^{\lambda}$ have positive $\lambda$-weights, 
    for the one-parameter subgroup $\lambda''=\lambda^a \cdot \lambda'$
for $a\gg 0$ the limit $\lim_{t\to 0}\lambda''(t)(y)$ exists in $Y$. 
Then from the $\ell$-semistability of $y$ and the condition $\langle \lambda, \ell \rangle =0$, 
we have 
$\langle \lambda', \ell\rangle=\langle \lambda'', \ell \rangle \geq 0$. 
Therefore $q(y)$ is $\ell$-semistable and the inclusion (\ref{incl2}) holds. 
\end{proof}

\subsection{Proof of (\ref{nablai}) in the proof of Lemma~\ref{lem:delta}}\label{subsec:nablai}
\begin{proof}
    We may assume that $\lambda_i$ is anti-dominant. 
    Let $\mathbf{W} \subset M(T)_{\mathbb{R}}$ be the polytope defined in (\ref{poly:W}). 
    Let $\chi \in M(T)$ be a dominant weight such that $\chi+\rho-\delta \in \mathbf{W}$.
    By Proposition~\ref{prop:QB:weight}, it is enough to show that 
    \begin{align*}
        \beta_{m_i}\overline{p}_i^{*}(V_G(\chi)\otimes \mathcal{O}_{\Y})\in 
        \MM(\overline{\mathcal{Z}_i})_{\delta_i}.
    \end{align*}
    Let $F(\lambda_i) \subset \mathbf{W}$ be a face determined by $\lambda_i$, given by the equation \[\langle \lambda_i, -\rangle=-\frac{1}{2}\langle \lambda_i, Y^{\lambda_i>0}\rangle.\] 
    If $\chi+\rho-\delta \notin F(\lambda_i)$, then we have 
    \begin{align*}
        \beta_{m_i}\overline{p}_i^{*}(V_G(\chi)\otimes \mathcal{O}_{\Y})=0.
    \end{align*}
    If $\chi+\rho-\delta_i \in F(\lambda_i)$, then we can write 
    \begin{align*}
        \chi+\rho-\delta=-\frac{1}{2}Y^{\lambda_i>0}+\psi+\rho^{\lambda_i}
    \end{align*}
    where 
 $\psi \in \mathbf{W}^{\lambda_i} \subset M(T)_{\mathbb{R}}$ is half the Minkowski sum of 
 $T$-weights of $Y^{\lambda_i}$, 
    $\rho^{\lambda_i}$ is half the sum of positive roots of $G^{\lambda_i}$, 
    see the argument of~\cite[Corollary~3.4]{P}. 
    
    The above equality can be written as 
    \begin{align*}
        \chi=-\frac{1}{2}(Y^{\lambda_i>0}-\fg^{\lambda_i>0})+\psi+\delta=\psi+\delta_i.
    \end{align*}
    Then we have 
    \begin{align*}
        \beta_{m_i}\overline{p}_i^{*}(V_G(\chi)\otimes \mathcal{O}_{\Y})
        =V_{G^{\lambda_i}}(\chi)\otimes \mathcal{O}_{\overline{\mathcal{Z}}_i}
        \in 
        \MM(\overline{\mathcal{Z}}_i)_{\delta_i}
\end{align*}
\end{proof}

\subsection{Proof of Proposition~\ref{prop:cgen}}\label{pf:prop}
\begin{proof}
A compact object \begin{align*}\mathcal{E}\in\IndCoh(\Hig_{\GL_r}(\chi))_w^{\mathrm{cp}}\end{align*} is an object in $\Coh(\Hig_{\GL_r}(\chi))_w$
by~\cite[Proposition~3.4.2]{MR3037900}. Below we show that (\ref{Esupport}) holds 
if $r\geq 2$. 

For $m\in \mathbb{Z}_{>0}$, we set 
$\mathcal{U}_m=\Hig_{\GL_r}(\chi)_{\leq m}$ as in Remark~\ref{rmk:HN}.
We have a chain of open immersions 
\begin{align*}
    \mathcal{U}_1 \subset \mathcal{U}_2 \subset \cdots \subset 
    \mathcal{U}_m \subset \cdots \subset \Hig_{\GL_r}(\chi). 
\end{align*}
Fix $m\in \mathbb{Z}_{>0}$.
    We set $\mathcal{E}_m:=\mathcal{E}|_{\mathcal{U}_m}
    \in \Coh(\mathcal{U}_m)$ and $Z(M):=\mathcal{U}_m \setminus \mathcal{U}_{m-1}$. 
We now show that one of the following 
conditions hold: 
\begin{enumerate}
    \item There is $F_m \in \IndCoh(\mathcal{U}_m)$ such that $F_m|_{\mathcal{U}_{m-1}}=0$
    and there is a non-zero morphism $\mathcal{E}_m \to F_m$; 
    \item $\mathrm{Supp}(\mathcal{E}_m) \subset \mathcal{U}_{m-1}$. 
\end{enumerate}

Indeed, if (ii) does not hold, i.e. if $\mathrm{Supp}(\mathcal{E}_m) \cap Z(M) \neq \emptyset$, 
then there is a non-zero morphism 
$\mathcal{E}_m \to i_{*}i^{*}\mathcal{E}_m$
in $\mathrm{QCoh}(\mathcal{U}_m)$
where $i \colon Z(M) \hookrightarrow \mathcal{U}_m$ is the closed 
immersion. 
By regarding it as a morphism in $\IndCoh(\mathcal{U}_m)$
by the natural fully-faithful functor $\mathrm{QCoh}(\mathcal{U}_m)\hookrightarrow 
\IndCoh(\mathcal{U}_m)$, and setting $F_m=i_{*}i^{*}\mathcal{E}_m$, 
we obtain (i). 

Let $M\subset \mathbb{Z}$ be the set of $m\in \mathbb{Z}$ such that 
(i) holds. 
First assume that $\lvert M \rvert=\infty$. 
Let $j_m\colon \mathcal{U}_m \hookrightarrow \Hig_{\GL_r}(\chi)$
be the open immersion. The object 
\begin{align}\label{direct:F}
F=
    \bigoplus_{m\in M}j_{m\ast}F_m \in \IndCoh(\Hig_{\GL_r}(\chi))
\end{align}
is locally finite, i.e. for each quasi-compact open substack 
$\mathcal{U} \subset \Hig_{\GL_r}(\chi)$ 
the set of $m\in M$ such that $(j_{m\ast}F_m)|_{\mathcal{U}} \neq 0$ is finite. 
A non-zero morphism $\mathcal{E}_m \to F_m$ in (i) induces 
the non-zero morphism $\alpha_m \colon \mathcal{E} \to j_{m\ast}F_m$
in $\IndCoh(\Hig_{\GL_r}(\chi))$. 
Since we have 
\begin{align*}
    \Hom(\mathcal{E}, F)=
    \lim_{m' \in M} \Hom(\mathcal{E}|_{\mathcal{U}_{m'}}, F|_{\mathcal{U}_{m'}})
\end{align*}
and the direct sum (\ref{direct:F}) is locally finite, 
the collection of morphisms $\{\alpha_m\}_{m\in M}$ determines 
an element of $\Hom(\mathcal{E}, F)$. On the other hand, 
as $\mathcal{E}$ is compact, we have 
\begin{align*}
\{\alpha_m\}_{m\in M}\in 
    \Hom(\mathcal{E}, F)=\bigoplus_{m\in M} \Hom(\mathcal{E}, j_{m*}F_m). 
\end{align*}
Since $\lvert M \rvert=\infty$ and $\alpha_m \neq 0$ for all $m\in M$, 
this is a contradiction. 

Therefore we have $\lvert M \rvert<\infty$. 
Then for $m\gg 0$, we have 
$\mathrm{Supp}(\mathcal{E}_m) \subset \mathcal{U}_{m-1}$, 
or equivalently $\mathrm{Supp}(\mathcal{E})\cap \mathcal{U}_{m}
=\mathrm{Supp}(\mathcal{E})\cap \mathcal{U}_{m-1}$. 
It follows that there is $m_0 \gg 0$ such that 
$\mathrm{Supp}(\mathcal{E}) \subset \mathcal{U}_{m_0}$. 
In particular, we have 
\begin{align}\label{supp:N}
    \mathrm{Supp}(\mathcal{E}) \cap \mathcal{N}_{\mathrm{glob}}\subset 
    \mathcal{U}_{m_0} \cap \mathcal{N}_{\mathrm{glob}}. 
\end{align}

Note that the left-hand side in (\ref{supp:N}) is a closed substack
of $\Hig_{\GL_r}(\chi)$. 
Assume that there is a $k$-valued point in $\mathrm{Supp}(\mathcal{E}) \cap \mathcal{N}_{\mathrm{glob}}$, which corresponds to $(F, \theta)$ where $F$ is a rank $r$ vector bundle and $\theta$ is a Higgs field. 
Since $\theta$ is nilpotent, it admits an isotrivial degeneration to $(F, 0)$, 
i.e. there is a flat family $(F_t, \theta_t)$ of Higgs bundles over 
$t\in \mathbb{A}^1$ such that
all $(F_t, \theta_t)$ are isomorphic to $(F, \theta)$ for $t\neq 0$ 
and $(F_0, \theta_0)$ is isomorphic to $(F, 0)$. 
Since $\mathrm{rank}(F) \geq 2$, 
for a line bundle $L$ on $C$ with $\deg L\gg 0$
there is an exact sequence 
\begin{align*}
    0 \to F' \to F \to L \to 0. 
\end{align*}
Then $F$ admits an iso-trivial degeneration to $F'\oplus L$. 
Since the left-hand side in (\ref{supp:N}) is closed in $\Hig_{\GL_r}(\chi)$, 
the Higgs bundle $(F'\oplus L, 0)$ lies in the 
left-hand side in (\ref{supp:N}). However for $\deg L \gg 0$, 
it does not lie in $\mathcal{U}_{m_0}$, which contradicts (\ref{supp:N}). 
Therefore the proposition holds. 
\end{proof}

\subsection{Proof of the equivalence (\ref{equiv:Ttensor})}\label{subsec:Ttensor}
\begin{proof}
It is enough to consider the case $k=2$. 
Note that we have 
\begin{align*}
    \Hig_{\GL_{r_1}}(\chi_1)^{\mathrm{ss}}\times \Hig_{\GL_{r_2}}(\chi_2)^{\mathrm{ss}}
    =\Hig_{G}(\chi)^{\mathrm{ss}}.
\end{align*}
Since 
$\Hig_{\GL_r}(\chi)^{\mathrm{ss}}$ is QCA, using (\ref{indtensor}) we have an equivalence
\begin{align}\label{tensor:Higgs}
\Coh(\Hig_{\GL_{r_1}}(\chi_1)^{\mathrm{ss}})_{w_1}\otimes 
\Coh(\Hig_{\GL_{r_2}}(\chi_2)^{\mathrm{ss}})_{w_2} \stackrel{\sim}{\to}
\Coh(\Hig_{G}(\chi)^{\mathrm{ss}})_w.
\end{align}
By Theorem~\ref{thm:PThiggs},
we have semiorthogonal decompositions 
\begin{align*}
    \Coh(\Hig_{\GL_{r_i}}(\chi_i)^{\mathrm{ss}})_{w_i}
    =\left\langle \bigotimes_{j=1}^{k}\mathbb{T}_{\GL_{r_{ij}}}(\chi_{ij})_{w_{ij}} \right\rangle
\end{align*}
where the right-hand side is after partitions of $(r_i, \chi_i, w_i)$
with $\chi_{ij}/r_{ij}=\chi_i/r_i$ and $w_{i1}/r_{i1}<\cdots<w_{ik}/r_{ik}$.
It follows that the left-hand side of (\ref{tensor:Higgs})
has a semiorthogonal decomposition 
\begin{align}\label{sod:left}
\left\langle 
\bigotimes_{j=1}^{k_1}\mathbb{T}_{\GL_{r_{1j}}}(\chi_{1j})_{w_{1j}} \otimes
\bigotimes_{j'=1}^{k_2}\mathbb{T}_{\GL_{r_{2j'}}}(\chi_{2j'})_{w_{2j'}}
\right\rangle
\end{align}
where $(r_{1j}, \chi_{1j}, w_{1j})$, $(r_{2j'}, \chi_{2j'}, w_{2j'})$ are as above. 

On the other hand, an analogue of the proof of Theorem~\ref{thm:PThiggs}
for $\GL_r$ easily implies it for $\GL_{r_1}\times \GL_{r_2}$. 
Therefore the right-hand side of (\ref{tensor:Higgs})
has a semiorthogonal decomposition 
\begin{align}\label{sod:right}
    \left\langle \bigotimes_{1\leq j\leq k_1, 1\leq j'\leq k_2}
    \mathbb{T}_{\GL_{r_{1j}}\times \GL_{r_{2j'}}}(\chi_{1j}+\chi_{2j'})_{w_{1j}+w_{2j'}}\right\rangle
\end{align}
where $(r_{1j}, \chi_{1j}, w_{1j})$, $(r_{2j'}, \chi_{2j'}, w_{2j'})$ are as above. 
The equivalence (\ref{tensor:Higgs}) is compatible with Hall product, so together with the 
definition of limit categories it restricts to the fully-faithful 
functor 
\begin{align}\notag
    \bigotimes_{j=1}^{k_1}\mathbb{T}_{\GL_{r_{1j}}}(\chi_{1j})_{w_{1j}} &\otimes
\bigotimes_{j'=1}^{k_2}\mathbb{T}_{\GL_{r_{2j'}}}(\chi_{2j'})_{w_{2j'}} \\
\label{FF:T}&\hookrightarrow \bigotimes_{1\leq j\leq k_1, 1\leq j'\leq k_2}\mathbb{T}_{\GL_{r_{1j}}\times \GL_{r_{2j'}}}(\chi_{1j}+\chi_{2j'})_{w_{1j}+w_{2j'}}.
\end{align}
By comparing the semiorthogonal decompositions (\ref{sod:left}), (\ref{sod:right}) under the 
equivalence (\ref{tensor:Higgs}), we conclude that the functor (\ref{FF:T}) is an equivalence. In particular, 
we have an equivalence (\ref{equiv:Ttensor}). 

\end{proof}

\subsection{Proof of Lemma~\ref{lem:smooth2}}\label{subsec:lem:smooth2}
\begin{proof}
Let $\pi^L$ be the projection $(\ref{Higgs:pi})$. Then there is a quasi-compact 
open substack $\mathcal{U} \subset \Bun_G$ such that we have 
the open immersion
\begin{align*}\Hig_G^{L}(\chi)_{\preceq \mu} \subset (\pi^L)^{-1}(\mathcal{U}).
\end{align*}
Since $\pi^L$ is affine and $\Hig_G^{L}(\chi)_{\preceq \mu}$ is smooth, it is enough to show that $\mathcal{U}$ can be 
written as $U/G'$ for a smooth quasi-projective scheme $U$ with an action of 
a reductive algebraic group $G'$. 
This is well-known in the case of $G=\mathrm{GL}_r$, where $U$ can be taken to be an 
open subscheme of a Quot scheme~\cite{MR1450870}. 

In general, 
we take a closed embedding $G \subset \mathrm{GL}_r$. Then we have the map 
\begin{align*}\pi_{G \to \mathrm{GL}_r} \colon 
\Bun_{G} \to \Bun_{\mathrm{GL}_r}, \ E_{G} \mapsto E_G\times^G \mathrm{GL}_r. 
\end{align*}
The fiber of the above morphism over $E_{\GL_r}$ is the space of sections of
$E_{\GL_r}/G \to C$, which is a quasi-projective scheme as it is an open subscheme 
of Hilbert scheme of the quasi-projective variety $(E_{\GL_r}/G) \times C$.
There is a quasi-compact open substack $\mathcal{U}' \subset \Bun_{\mathrm{GL}_r}$ such that 
$\mathcal{U} \subset \pi_{G\to \mathrm{GL}_r}^{-1}(\mathcal{U}')$.
Then the result holds since $\pi_{G\to \mathrm{GL}_r}$ is a schematic, quasi-projective 
morphism. 
    \end{proof}

    \subsection{Some postponed lemmas in the proof of Proposition~\ref{prop:sod:V}}\label{subsec:postpone}
    We postponed proofs of several lemmas in the proof of Proposition~\ref{prop:sod:V}, which will be given below: 

\subsubsection{Proof of Lemma~\ref{lem:symW}}\label{subsub:1}
\begin{proof}
By Lemma~\ref{lem:dtr}, there is an exact sequence of $G_E$-representations 
\begin{align}\label{exact:ad}
    0 \to \mathbb{T}^0_{(E_G, \phi_G)} \to \mathbb{T}^0_{(E_G, \phi_G^{L})} \to \Gamma(\Ad_{\fg}(E_G)|_{D}\otimes L)
    \to \mathbb{T}^1_{(E_G, \phi_G)} \to 0. 
\end{align}
Let $W$ be the cokernel of $\mathbb{T}^0_{(E_G, \phi_G)} \to \mathbb{T}^0_{(E_G, \phi_G^{L})}$. 
By noting that \[(\mathbb{T}^{1}_{(E_G, \phi_G)})^{\vee}\cong \mathbb{T}^{-1}_{(E_G, \phi_G)},\]
we have an isomorphism of $G_E$-representations
\begin{align}\label{rep:V}
    \mathbf{V}\cong \mathbb{T}^0_{(E_G, \phi_G)} \oplus W \oplus W^{\vee} \oplus \mathbb{T}^{-1}_{(E_G, \phi_G)}. 
\end{align}

It follows that 
\begin{align}\label{rep:V2}
      \mathbf{V}_0 \cong \mathbb{T}^0_{(E_G, \phi_G)} \oplus W \oplus W^{\vee} \oplus
      \mathbb{T}^{-1}_{(E_M, \phi_M)}.   
\end{align}
The $G_E$-representation $\mathbb{T}^0_{(E_G, \phi_G)}$ is self-dual by the Serre duality, 
and the last factor $\mathbb{T}^{-1}_{(E_M, \phi_M)}$ is also self-dual since it is the Lie algebra 
of the reductive group $G_E$. Therefore $\mathbf{V}_0$ is self-dual. 
\end{proof}
\begin{lemma}\label{lem:cart0}
In the diagram (\ref{dia:WG}), all the bottom squares are derived 
Cartesians and the top left arrow is an isomorphism. 
\end{lemma}
\begin{proof}
As for the first statement, it is enough to show that the left bottom square is derived Cartesian. 
Let $(E_P, \phi_P)=(E_M, \phi_M)\times^M P$.
Since we have 
\begin{align}\label{T:quot}\mathbb{T}^{-1}_{(E_G, \phi_G)}/\mathbb{T}^{-1}_{(E_M, \phi_M)}
=\mathbb{T}^{-1}_{(E_P, \phi_P)}/\mathbb{T}^{-1}_{(E_M, \phi_M)}
\end{align}
it has positive $\lambda$-weights. Therefore we have 
\begin{align*}
    (\mathbb{T}^{-1}_{(E_G, \phi_G)})^{\lambda \geq 0}=
     (\mathbb{T}^{-1}_{(E_M, \phi_M)})^{\lambda \geq 0}
     \oplus \left(\mathbb{T}^{-1}_{(E_G, \phi_G)}/\mathbb{T}^{-1}_{(E_M, \phi_M)}\right).
\end{align*}

It follows that 
\begin{align*}
    &\mathbf{V}=\mathbf{V}_0 \oplus \left(\mathbb{T}^{-1}_{(E_P, \phi_P)}/\mathbb{T}^{-1}_{(E_M, \phi_M)}\right), \\ \notag
        &\mathbf{V}^{\lambda \geq 0}=\mathbf{V}_0^{\lambda \geq 0} \oplus \left(\mathbb{T}^{-1}_{(E_P, \phi_P)}/\mathbb{T}^{-1}_{(E_M, \phi_M)}\right).
\end{align*}
By (\ref{rep:V}) and (\ref{rep:V2}), the left bottom square in (\ref{dia:WG}) is 
derived Cartesian. 
The top left arrow in (\ref{lem:dtr}) is an isomorphism since (\ref{T:quot}) has positive $\lambda$-weights. 
\end{proof}

\subsubsection{Proof of Lemma~\ref{lem:neq}}\label{subsub:2}
\begin{proof}
It is enough to show the equality  
\begin{align}\label{L=L}
[g^{\ast}\mathbb{L}_{(\V^L)^{\vee}}]=[\mathbb{L}_{\widehat{\mathbf{V}}_0/G_E}]
\end{align}
in $K(\widehat{\mathbf{V}}_0/G_E)$. 
Note that $\mathbb{T}^{-1}_{(E_G, \phi_G)}=\mathbb{T}^{-1}_{(E_G, \phi_G^{L})}$ by Lemma~\ref{lem:dtr}. 
Using the exact sequence (\ref{exact:ad}), we have 
the following identities in $K(\widehat{\mathbf{V}}_0/G_E)$
\begin{align*}
    g^{\ast}\mathbb{L}_{(\V^L)^{\vee}} &=
    (\mathbb{T}_{\Hig_G^{L}}|_{(E_G, \phi_G^{L})})^{\vee}+\Gamma(\Ad_{\fg}(E_G)|_{D}\otimes L) \\
    &=(\mathbb{T}^0_{(E_G, \phi_G^{L})})^{\vee}-(\mathbb{T}^{-1}_{(E_G, \phi_G^{L})})^{\vee}+\Gamma(\Ad_{\fg}(E_G)|_{D}\otimes L) \\
    &=(\mathbb{T}^0_{(E_G, \phi_G^{L})})^{\vee}-(\mathbb{T}^{-1}_{(E_G, \phi_G^{L})})^{\vee}
+\mathbb{T}_{(E_G, \phi_G)}^1+\mathbb{T}_{(E_G, \phi_G^{L})}^0-\mathbb{T}_{(E_G, \phi_G)}^0 \\
    &=\mathbf{V}^{\vee}_0-\mathfrak{g}_E^{\vee}=\mathbb{L}_{\widehat{\mathbf{V}}_0/G_E}. 
\end{align*}
Here $\mathfrak{g}_E=\mathbb{T}_{(E_M, \phi_M)}^{-1}$ is the Lie algebra of $G_E$, and in the above computations 
elements in $K(\widehat{\mathbf{V}}_0/G_E)$ are pulled back from $K(BG_E)$ 
by the projection $\widehat{\mathbf{V}}_0/G_E \to BG_E$. 
Therefore the identity (\ref{L=L}) holds. 
\end{proof}

\subsubsection{Proof of Lemma~\ref{lem:pullback}}\label{subsub:3}
\begin{proof}
For a map $\nu \colon \bgm \to (\V^L)^{\vee}$,  
we need to check the condition of the weights of $\nu^{\ast}\E''$ in Definition~\ref{defn:QBY},
i.e. 
\begin{align}\label{wt:E''}
\wt(\nu^{\ast}\E'') \subset \left[-\frac{1}{2}n_{(\V^L)^{\vee}, \nu}, 
\frac{1}{2}n_{(\V^L)^{\vee}, \nu}
\right] +c_1(\nu^{\ast}\delta). 
\end{align}
By the $\mathbb{G}_m$-equivariance of $\mathcal{E}''$, we may assume that 
$\nu$ factors through the zero section
$\U^L \hookrightarrow (\V^L)^{\vee}$. By shrinking $\U^L$ as in Remark~\ref{rmk:shrink} if necessary, 
we may also assume that $\nu$ lies in $\U \hookrightarrow \U^L$. 
By the assumption $\E''|_{(\V^L)^{\vee}_{\circ}} \in \MM_{\mathbb{G}_m}((\V^L)^{\vee}_{\circ})_{\delta}$, 
we may also assume that $\nu$ lies in $\mathcal{S}$ because 
\begin{align*}
    \mathcal{S}_{\mathcal{V}}\cap \mathcal{U} \subset (\pi^L)^{-1}(\mathcal{S}^L)\cap \mathcal{U} =\mathcal{S}. 
\end{align*}

Let $z' \in \sS$ be the image of $\nu$, and define $z \in \mathcal{S}$ to be 
the image of $z'$ under the map 
\begin{align*}
\mathcal{S}(k) \stackrel{(-)\times^P M}{\to} \zZ(k) \to Z(k) \to \zZ(k) \stackrel{(-)\times^M G}{\to} \mathcal{S}(k).
\end{align*}
Here $\mathcal{Z} \to Z$ is the good moduli space, and 
the map $Z(k) \to \zZ(k)$ is given by sending a point in $Z(k)$ to the unique 
closed point in the fiber of $\zZ\to Z$. Note that 
$z$ corresponds to an object $(E_G, \phi_G)$ of the form (\ref{deg:E}), 
and 
$z'$ is obtained 
as a small deformation of
$z \in \sS$. 
Since 
\begin{align*}\mathrm{Aut}(z') \subset \mathrm{Aut}(z)=\mathrm{Aut}(E_G, \phi_G),
\end{align*}
we have the map $\widetilde{\nu} \colon \bgm \to \mathcal{S}$ with $\widetilde{\nu}(0)=z$ and 
the map $\mathbb{G}_m \to \Aut(z)$ is the same as $\mathbb{G}_m \to \Aut(z')$ given by $\nu$. 
Then we have $n_{(\mathcal{V}^L)^{\vee}, \nu}=n_{(\mathcal{V}^L)^{\vee}, \widetilde{\nu}}$, 
and $\wt(\nu^*\mathcal{E}) \subset \wt(\widetilde{\nu}^* \mathcal{E})$ by the upper semi-continuity. 
Therefore 
we may replace $\nu$ with $\widetilde{\nu}$ and assume that the image of $\nu$ equals to $z$. 
Then the map $\nu$ corresponds to a one-parameter subgroup $\mathbb{G}_m \to \mathrm{Aut}(E_G, \phi_G)$. 
By applying the conjugation, we may assume that its image lies in $G_E \subset \mathrm{Aut}(E_G, \phi_G)$.

Therefore 
we may assume that $\nu$ factors through 
$\bgm \to \widehat{\mathbf{V}}_0/G_E$ whose image lies in $0 \in \mathbf{V}_0$. 
Then the condition $g^{\ast}\E'' \in \MM(\widehat{\mathbf{V}}_0/G_E)_{\delta}$
together with Lemma~\ref{lem:neq} implies the condition (\ref{wt:E''}). 
\end{proof}

\subsection{Proof of Lemma~\ref{lem:delta2}}\label{subsec:pfdelta}
\begin{proof}
Let $\nu \colon \bgm \to \zZ$ be a map 
such that $\nu(0)\in \zZ$ corresponds to 
a polystable $M$-Higgs bundle $(E_M, \phi_M)$ as in (\ref{deg:E}). 
We set
\begin{align*}(E_G, \phi_G)=(E_M, \phi_M)\times^M G, \ (E_G, \phi_G^{L})=\iota_D(E_G, \phi_G).
\end{align*}
Let $\lambda \colon \mathbb{G}_m \to \mathrm{Aut}(E_M, \phi_M)$ 
be the one-parameter subgroup as in (\ref{lambda:E}). 
We have the identities in $K(\bgm)$
\begin{align*}
    &\nu^* \mathbb{L}_{\mathcal{U}^L}=(\mathbb{T}_{\Hig_G^{L}}|_{(E_G, \phi_G^{L})})^{\vee}, \\
    &\nu^* N_{\mathcal{S}_{\V}/(\V^L)^{\vee}}=(\mathbb{T}^0_{(E_G, \phi_G^{L})} -\mathbb{T}^{-1}_{(E_G, \phi_G^{L})}
    +\Gamma(\Ad_{\fg}(E_G)|_{D}\otimes L)^{\vee})^{\lambda<0}, \\
    &\nu^* \mathbb{L}_{\zZ^L}=(\mathbb{T}_{\Hig_G^{L}}|_{(E_G, \phi_G^{L})})^{\lambda\vee}, \\
    &\nu^* \V^L|_{\zZ}^{>0}=\Gamma(\Ad_{\fg}(E_G)|_{D}\otimes L)^{\lambda>0}.
\end{align*}

It follows that $2c_1(\nu^*(\delta'' \otimes \delta^{-1}))$ is calculated as 
\begin{align*}
    &c_1\left((\mathbb{T}_{\Hig_G^{L}}|_{(E_G, \phi_G^{L})})^{\vee}
    +(\mathbb{T}^0_{(E_G, \phi_G^{L})})^{\lambda<0} -(\mathbb{T}^{-1}_{(E_G, \phi_G^{L})})^{\lambda<0} \right. \\
    &\quad \quad\left.+(\Gamma(\Ad_{\fg}(E_G)|_{D}\otimes L)^{\vee})^{\lambda<0} -(\mathbb{T}_{\Hig_G^{L}}|_{(E_G, \phi_G^{L})})^{\lambda\vee}+2\Gamma(\Ad_{\fg}(E_G)|_{D}\otimes L)^{\lambda>0} \right) \\
    &=c_1\left(-\mathbb{T}^0_{(E_G, \phi_G^{L})}+\mathbb{T}^{-1}_{(E_G, \phi_G^{L})}+(\mathbb{T}^0_{(E_G, \phi_G^{L})})^{\lambda<0} 
    -(\mathbb{T}^{-1}_{(E_G, \phi_G^{L})})^{\lambda<0}\right.\\
    &\quad \qquad\left.+(\mathbb{T}^0_{(E_G, \phi_G^{L})})^{\lambda} -(\mathbb{T}^{-1}_{(E_G, \phi_G^{L})})^{\lambda}
    +\Gamma(\Ad_{\fg}(E_G)|_{D}\otimes L)^{\lambda>0} \right)\\
    &=c_1\left(-(\mathbb{T}^0_{(E_G, \phi_G^{L})})^{\lambda>0}+\Gamma(\Ad_{\fg}(E_G)|_{D}\otimes L)^{\lambda>0}+(\mathbb{T}^{-1}_{(E_G, \phi_G^{L})})^{\lambda>0}\right)
\end{align*}
which equals to $c_1(-\mathbb{T}_{\Hig_G}|_{(E_G, \phi_G)}^{\lambda>0})$ by Lemma~\ref{lem:dtr}. 
It follows that 
\begin{align*}
    2c_1(\nu^*(\delta''\otimes \delta^{-1}))=c_1((\mathbb{T}_{\Hig_G}|_{(E_G, \phi_G)})^{\lambda<0})^{\vee})=c_1(\nu^*\mathbb{L}_{p}).
\end{align*}
Therefore the lemma holds. 
	\end{proof}

    \subsection{Proof of Lemma~\ref{lem:square}}\label{pf:lemsquare}
    \begin{proof}
For simplicity, we assume that $g\geq 2$. 
Let $\nu \colon \bgm \to \zZ$ be a map such that $\nu(0)$ corresponds to a 
$M$-Higgs bundle $(E_M, \phi_M)$. We set 
\begin{align*}(E_P, \phi_P)=(E_M, \phi_M)\times^M P, \ (E_G, \phi_G)=(E_M, \phi_M)\times^M G.
\end{align*}

Let 
$\lambda \colon \mathbb{G}_m \to \mathrm{Aut}(E_M, \phi_M)$ 
be the canonical one-parameter subgroup as in (\ref{lambda:E}).
Then we have isomorphisms
\begin{align}\notag
    \det \mathbb{L}_p|_{(E_P, \phi_P)}&=\det(-\mathbb{T}_{\Hig_G}|_{(E_G, \phi_G)}^{\lambda>0})\\
 \label{isom:detR}   &\cong \det\left( \Gamma(\Ad_{\fg}(E_G))^{\lambda>0}-\Gamma(\Ad_{\fg}(E_G)\otimes \Omega_C)^{\lambda>0}\right).
\end{align}
Here, we have used Lemma~\ref{lem:cotangent}
for the last isomorphism. 

Let $D_1, D_2$ be effective divisors on $C$ such that $\deg D_i=g-1$ and 
$D_1 \cap D_2=\emptyset$. We have $c_1(\Omega_C)=D_1+D_2$. 
Let $\mathcal{L}_a \to \zZ$ for $a=1, 2$ be the line bundle whose fiber at $(E_M, \phi_M)$ is given by 
\begin{align*}
    \mathcal{L}_a|_{(E_M, \phi_M)}=\det\left(
    \Gamma(\Ad_{\fg}(E_G)|_{D_a}(D_a))^{\lambda>0}\right). 
\end{align*}

Note that $c_1(\mathcal{L}_1)=c_1(\mathcal{L}_2)$ since 
$\deg D_1=\deg D_2$, and 
we have \[c_1(\nu^*\mathbb{L}_p)=c_1(\nu^*\mathcal{L}_1 \otimes \nu^*\mathcal{L}_2)\] 
by (\ref{isom:detR}). By setting $\mathcal{L}=\mathcal{L}_1$, we obtain 
the lemma. 
	\end{proof}

\newcommand{\etalchar}[1]{$^{#1}$}
\providecommand{\bysame}{\leavevmode\hbox to3em{\hrulefill}\thinspace}
\providecommand{\MR}{\relax\ifhmode\unskip\space\fi MR }
\providecommand{\MRhref}[2]{%
  \href{http://www.ams.org/mathscinet-getitem?mr=#1}{#2}
}
\providecommand{\href}[2]{#2}


\begin{thebibliography}{BBBBJ15}

\bibitem[ABC{\etalchar{+}}a]{GLC2}
D.~Arinkin, D.~Beraldo, J.~Campbell, L.~Chen, J.~Faergeman, D.~Gaitsgory, K.~Lin, S.~Raskin, and N.~Rozenblyum, \emph{Proof of the geometric {L}anglands conjecture {II}: {K}ac-{M}oody localization and the {FLE}}, arXiv:2405.03648.

\bibitem[ABC{\etalchar{+}}b]{GLC4}
D.~Arinkin, D.~Beraldo, L.~Chen, J.~Faergeman, D.~Gaitsgory, K.~Lin, S.~Raskin, and N.~Rozenblyum, \emph{Proof of the geometric {L}anglands conjecture {IV}: ambidexterity}, arXiv:2409.08670.

\bibitem[AF16]{ArFe}
D.~Arinkin and R.~Fedorov, \emph{Partial {F}ourier-{M}ukai transform for integrable systems with applications to {H}itchin fibration}, Duke Math. J. \textbf{165} (2016), no.~15, 2991--3042.

\bibitem[AG15]{AG}
D.~Arinkin and D.~Gaitsgory, \emph{Singular support of coherent sheaves and the geometric {L}anglands conjecture}, Selecta Math. (N.S.) \textbf{21} (2015), no.~1, 1--199.

\bibitem[AGK{\etalchar{+}}]{AGKRRV}
D.~Arinkin, D.~Gaitsgory, D.~Kazhdan, S.~Raskin, N.~Rozenblyum, and Y.~Varshavsky, \emph{The stack of local systems with restricted variation and geometric {L}anglands theory with nilpotent singular support}, arXiv:2010.01906.

\bibitem[AHPS]{AHPS}
E.~Ahlqvist, J.~Hekking, M.~Pernice, and M.~Savvas, \emph{Good {M}oduli {S}paces in {D}erived {A}lgebraic {G}eometry}, arXiv:2309.16574.

\bibitem[AHR20]{AHRLuna}
J.~Alper, J.~Hall, and D.~Rydh, \emph{A {L}una \'etale slice theorem for algebraic stacks}, Ann. of Math. (2) \textbf{191} (2020), no.~3, 675--738.

\bibitem[Alp13]{MR3237451}
J.~Alper, \emph{Good moduli spaces for {A}rtin stacks}, Ann. Inst. Fourier (Grenoble) \textbf{63} (2013), no.~6, 2349--2402.

\bibitem[Ari13]{Ardual}
D.~Arinkin, \emph{Autoduality of compactified {J}acobians for curves with plane singularities}, J. Algebraic Geom. \textbf{22} (2013), no.~2, 363--388.

\bibitem[BB07]{BezBra}
R.~Bezrukavnikov and A.~Braverman, \emph{Geometric {L}anglands correspondence for {D}-modules in prime characteristic: the $\text{GL}(n)$-case}, Pure Appl. Math. Q. \textbf{3} (2007), no.~1, 153--179, Special Issue: In honor of Robert MacPherson.

\bibitem[BBBBJ15]{BBBJ}
O.~Ben-Bassat, C.~Brav, V.~Bussi, and D.~Joyce, \emph{A '{D}arboux {T}heorem' for shifted symplectic structures on derived {A}rtin stacks, with applications}, Geom.~Topol.~ \textbf{19} (2015), 1287--1359.

\bibitem[BBJ19]{BBJ}
C.~Brav, V.~Bussi, and D.~Joyce, \emph{A {D}arboux theorem for derived schemes with shifted symplectic structure}, J. Amer. Math. Soc. \textbf{32} (2019), 399--443.

\bibitem[BD]{BD0}
A.~Beilinson and V.~Drinfeld, \emph{Quantization of {H}itchin's integrable system and {H}ecke eigensheaves}, \url{https://math.uchicago.edu/~drinfeld/langlands/QuantizationHitchin.pdf}.

\bibitem[BDIN{\etalchar{+}}]{BDIKP}
C.~Bu, B.~Davison, A.~{I}báñez {N}úñez, T.~Kinjo, and T.~Pădurariu, \emph{Cohomology of symmetric stacks}, arXiv:2502.04253.

\bibitem[Beh]{Behthesis}
K.~Behrend, \emph{The {L}efschetz {T}race {F}ormula for the {M}oduli {S}tack of {P}rincipal {B}undles}, \url{https://personal.math.ubc.ca/~behrend/thesis.html}.

\bibitem[Beh09]{Beh}
\bysame, \emph{Donaldson-{T}homas type invariants via microlocal geometry}, Ann.~of Math \textbf{170} (2009), 1307--1338.

\bibitem[BFK19]{MR3895631}
M.~Ballard, D.~Favero, and L.~Katzarkov, \emph{Variation of geometric invariant theory quotients and derived categories}, J. Reine Angew. Math. \textbf{746} (2019), 235--303.

\bibitem[BFM05]{BFM}
R.~Bezrukavnikov, M.~Finkelberg, and I.~Mirković, \emph{Equivariant homology and {K}-theory of affine {G}rassmannians and {T}oda lattices}, Compos. Math. \textbf{141} (2005), 746--768.

\bibitem[Bha]{Bhattnote}
B.~Bhatt, \emph{Prismatic {F}-gauges}, \url{https://www.math.ias.edu/~bhatt/teaching/mat549f22/lectures.pdf}.

\bibitem[BHLINK]{BHLINK}
C.~Bu, D.~Halpern-Leistner, A.~{I}báñez {N}úñez, and T.~Kinjo, \emph{Intrinsic {D}onaldson-{T}homas theory. {I}. {C}omponent lattices of stacks}, arXiv:2502.13892.

\bibitem[BHLNK]{bhik}
C.~Bu, D.~Halpern-Leistner, A.~Ib{\'a}{\~n}ez N{\'u}{\~n}ez, and T.~Kinjo, \emph{Intrinsic {D}onaldson-{T}homas theory. {I}. {C}omponent lattices of stacks}, arXiv:2502.13892.

\bibitem[Bla16]{Blanc}
A.~Blanc, \emph{Topological {K}-theory of complex noncommutative spaces}, Compos. Math. \textbf{152} (2016), no.~3, 489--555.

\bibitem[BO]{B-O2}
A.~Bondal and D.~Orlov, \emph{Semiorthogonal decomposition for algebraic varieties}, arXiv:9506012.

\bibitem[BPT]{BPT}
C.~Bu, T.~P\u{a}durariu, and Y.~Toda, \emph{Semiorthogonal decompositions for stacks}, arXiv:2605.25976.

\bibitem[Bri]{Brion0}
M.~Brion, \emph{Linearization of algebraic group actions}, available at \url{https://www-fourier.ujf-grenoble.fr/~mbrion/lin_rev.pdf}.

\bibitem[BRTV18]{MR3877165}
A.~Blanc, M.~Robalo, B.~To\"{e}n, and G.~Vezzosi, \emph{Motivic realizations of singularity categories and vanishing cycles}, J. \'{E}c. polytech. Math. \textbf{5} (2018), 651--747.

\bibitem[BZCHN]{BZCHN}
D.~Ben-Zvi, H.~Chen, D.~Helm, and D.~Nadler, \emph{Between {C}oherent and {C}onstructible {L}ocal {L}anglands {C}orrespondences}, arXiv:2302.00039.

\bibitem[BZN12]{BZNa}
D.~Ben-Zvi and D.~Nadler, \emph{Loop spaces and connections}, J. Topol. \textbf{5} (2012), no.~2, 377–430.

\bibitem[BZN18]{ZN}
\bysame, \emph{Betti geometric {L}anglands}, Algebraic geometry: {S}alt {L}ake {C}ity 2015, Proc. Sympos. Pure Math., vol. 97.2, Amer. Math. Soc., Providence, RI, 2018, pp.~3--41.

\bibitem[CCF{\etalchar{+}}]{GLC3}
J.~Campbell, L.~Chen, J.~Faergeman, D.~Gaitsgory, K.~Lin, S.~Raskin, and N.~Rozenblyum, \emph{Proof of the geometric {L}anglands conjecture {III}: compatibility with parabolic induction}, arXiv:2409.07051.

\bibitem[CD]{ChenDhillon}
H.~Chen and G.~Dhillon, \emph{A {L}anglands dual realization of coherent sheaves on the nilpotent cone}, arXiv:2310.10539.

\bibitem[CDK21]{CDK}
S.~Cautis, C.~Dodd, and J.~Kamnitzer, \emph{Associated graded of {H}odge modules and categorical {$\mathfrak{sl}_2$} actions}, Selecta Math. (N.S.) \textbf{27} (2021), no.~2, Paper No. 22, 55.

\bibitem[Che]{HChen}
H.~Chen, \emph{Categorical cyclic homology and filtered {D}-modules on stacks: {K}oszul duality}, arXiv:2301.06949.

\bibitem[CPT]{CauPT}
S.~Cautis, T.~P\u{a}durariu, and Y.~Toda, \emph{The derived category of commuting stack}, in preparation.

\bibitem[CW]{CW}
S.~Cautis and H.~Williams, \emph{Canonical bases for {C}oulomb branches of 4d gauge theories}, arXiv:2306.03023.

\bibitem[DG13]{MR3037900}
V.~Drinfeld and D.~Gaitsgory, \emph{On some finiteness questions for algebraic stacks}, Geom. Funct. Anal. \textbf{23} (2013), no.~1, 149--294.

\bibitem[DG15]{DGbun}
\bysame, \emph{Compact generation of the category of {D}-modules on the stack of {G}-bundles on a curve}, Cambridge Journal of Mathematics \textbf{3} (2015), no.~1-2, 19--125.

\bibitem[DM20]{DM}
B.~Davison and S.~Meinhardt, \emph{Cohomological {D}onaldson-{T}homas theory of a quiver with potential and quantum enveloping algebras}, Invent. Math. \textbf{221} (2020), no.~3, 777--871.

\bibitem[DP05]{ArPa}
A.~Dey and R.~Parthasarathi, \emph{On {H}arder-{N}arasimhan reductions for {H}iggs principal bundles}, Proc. Indian Acad. Sci. Math. Sci. \textbf{115} (2005), no.~2, 127--146.

\bibitem[DP09]{DPlecture}
R.~Donagi and T.~Pantev, \emph{Geometric {L}anglands and non-abelian {H}odge theory}, Surveys in differential geometry. {V}ol. {XIII}. {G}eometry, analysis, and algebraic geometry: forty years of the {J}ournal of {D}ifferential {G}eometry, Surv. Differ. Geom., vol.~13, Int. Press, Somerville, MA, 2009, pp.~85--116.

\bibitem[DP12]{DoPa}
\bysame, \emph{Langlands duality for {H}itchin systems}, Invent. Math. \textbf{189} (2012), no.~3, 653--735.

\bibitem[DPS22]{DiaconescuPortaSala2022}
D.~E. Diaconescu, M.~Porta, and F.~Sala, \emph{Cohomological {H}all algebras and their representations via torsion pairs}, arXiv:2207.08926, 2022.

\bibitem[Efi18]{Eff}
A.~I. Efimov, \emph{Cyclic homology of categories of matrix factorizations}, Int. Math. Res. Not. IMRN (2018), no.~12, 3834--3869.

\bibitem[Fan]{DFang}
D.~Fang, \emph{On {F}ourier-{M}ukai transforms of upward flows for {H}itchin systems}, arXiv:2504.04309.

\bibitem[FGV02]{FGV}
E.~Frenkel, D.~Gaitsgory, and K.~Vilonen, \emph{On the geometric {L}anglands conjecture}, Journal of the American Mathematical Society \textbf{15} (2002), 367--417.

\bibitem[FR25]{FaeRas}
J.~Faergeman and S.~Raskin, \emph{Non-vanishing of geometric {W}hittaker coefficients for reductive groups}, J. Amer. Math. Soc. \textbf{38} (2025), 919--995.

\bibitem[FSS18]{FeSoSo}
R.~Fedorov, A.~Soibelman, and Y.~Soibelman, \emph{Motivic classes of moduli of {H}iggs bundles and moduli of bundles with connections}, Commun. Number Theory Phys. \textbf{12} (2018), no.~4, 687--766.

\bibitem[Gai13]{MR3136100}
D.~Gaitsgory, \emph{ind-coherent sheaves}, Mosc. Math. J. \textbf{13} (2013), no.~3, 399--528, 553.

\bibitem[Gai15]{Gaitsout}
\bysame, \emph{Outline of the proof of the {G}eometric {L}anglands {C}onjecture for {GL}(2)}, Ast\'erisque \textbf{370} (2015), 1--112.

\bibitem[GRa]{GLC1}
D.~Gaitsgory and S.~Raskin, \emph{Proof of the geometric {L}anglands conjecture {I}: construction of the functor}, arXiv:2405.03599.

\bibitem[GRb]{GLC5}
\bysame, \emph{Proof of the geometric {L}anglands conjecture {V}: the multiplicity one theorem}, arXiv:2405.09856.

\bibitem[GR14]{GaiCry}
D.~Gaitsgory and N.~Rozenblyum, \emph{Crystals and {D}-modules}, Pure Appl. Math. Q. \textbf{10} (2014), no.~1, 57--154.

\bibitem[GR17a]{MR3701352}
\bysame, \emph{A study in derived algebraic geometry. {V}ol. {I}. {C}orrespondences and duality}, Mathematical Surveys and Monographs, vol. 221, American Mathematical Society, Providence, RI, 2017.

\bibitem[GR17b]{MR3701353}
\bysame, \emph{A study in derived algebraic geometry. {V}ol. {II}. {D}eformations, {L}ie theory and formal geometry}, Mathematical Surveys and Monographs, vol. 221, American Mathematical Society, Providence, RI, 2017.

\bibitem[Gro]{GrojnowskiU}
I.~Grojnowski, \emph{Affinising quantum algebras: from {D}-modules to {K}-theory}, available at \url{https://www.dpmms.cam.ac.uk/~ig211/papers.html}.

\bibitem[GWZ20]{GrWy}
M.~Groechenig, D.~Wyss, and P.~Ziegler, \emph{Mirror symmetry for moduli spaces of {H}iggs bundles via p-adic integration}, Invent. Math. \textbf{221} (2020), no.~2, 505--596.

\bibitem[Hau22]{HauICM}
T.~Hausel, \emph{Enhanced mirror symmetry for {L}anglands dual {H}itchin systems}, I{CM}---{I}nternational {C}ongress of {M}athematicians. {V}ol. 3. {S}ections 1--4, EMS Press, Berlin, 2022, pp.~2228--2249.

\bibitem[Hen]{Henn}
L.~Hennecart, \emph{Nonabelian {H}odge isomorphisms for stacks and cohomological {H}all algebras}, arXiv:2307.09920.

\bibitem[HH22]{HauHit}
T.~Hausel and N.~Hitchin, \emph{Very stable {H}iggs bundles, equivariant multiplicity and mirror symmetry}, Invent. math. \textbf{228} (2022), 893--989.

\bibitem[HHRa]{HHR1}
B.~Hennion, J.~Holstein, and M.~Robalo, \emph{Gluing invariants of {D}onaldson--{T}homas type -- {P}art {I}: the {D}arboux stack}, arXiv:2407.08471.

\bibitem[HHRb]{HHR2}
\bysame, \emph{Gluing invariants of {D}onaldson--{T}homas type -- {P}art {II}: {M}atrix factorizations}, arXiv:2503.15198.

\bibitem[Hir17]{Hirano}
Y.~Hirano, \emph{Derived {K}n\"orrer periodicity and {O}rlov's theorem for gauged {L}andau-{G}inzburg models}, Compos. Math. \textbf{153} (2017), no.~5, 973--1007.

\bibitem[HLa]{HalpK32}
D.~Halpern-Leistner, \emph{Derived {$\Theta$}-stratifications and the {$D$}-equivalence conjecture}, arXiv:2010.01127.

\bibitem[HLb]{Halphiggs}
\bysame, \emph{The equivariant {V}erlinde formula on the moduli of {H}iggs bundles}, arXiv:1608.01754.

\bibitem[HLc]{halpinstab}
\bysame, \emph{On the structure of instability in moduli theory}, arXiv:1411.0627.

\bibitem[HL97]{MR1450870}
D.~Huybrechts and M.~Lehn, \emph{The geometry of moduli spaces of sheaves}, Aspects of Mathematics, E31, Friedr. Vieweg \& Sohn, Braunschweig, 1997.

\bibitem[HL12]{HL}
J.~Hu and W.~P. Li, \emph{The {D}onaldson-{T}homas invariants under blowups and flops}, J.~Differential.~Geom.~ (2012), 391--411.

\bibitem[HL15]{halp}
D.~Halpern-Leistner, \emph{The derived category of a {GIT} quotient}, J. Amer. Math. Soc. \textbf{28} (2015), no.~3, 871--912.

\bibitem[HL18]{HalpTheta}
\bysame, \emph{{$\Theta$}-stratifications, {$\Theta$}-reductive stacks, and applications}, Algebraic Geometry: Salt Lake City (2018), 349–379.

\bibitem[HLP20]{HLP}
D.~Halpern-Leistner and D.~Pomerleano, \emph{Equivariant {H}odge theory and noncommutative geometry}, Geom. Topol. \textbf{24} (2020), no.~5, 2361--2433.

\bibitem[HLS20]{hls}
D.~Halpern-Leistner and S.~V. Sam, \emph{Combinatorial constructions of derived equivalences}, J. Amer. Math. Soc. \textbf{33} (2020), no.~3, 735--773.

\bibitem[HMMS]{HMMS}
T.~Hausel, A.~Mellit, A.~Minets, and O.~Schiffmann, \emph{{P}={W} via ${H}_2$}, arXiv:2209.05429.

\bibitem[HMP18]{HaMePe}
T.~Hausel, A.~Mellit, and D.~Pei, \emph{Mirror {S}ymmetry with {B}ranes by {E}quivariant {V}erlinde {F}ormulas}, Geometry and Physics: A Festschrift in honour of Nigel Hitchin \textbf{I} (2018), 189--218.

\bibitem[Hos14]{Hoskins}
V.~Hoskins, \emph{Stratifications associated to reductive group actions on affine spaces}, Q. J. Math. \textbf{65} (2014), no.~3, 1011--1047.

\bibitem[HRS96]{HRS}
D.~Happel, I.~Reiten, and S.~O. Smalo, \emph{Tilting in abelian categories and quasitilted algebras}, Mem.~Amer.~Math.~Soc, vol. 120, American Mathematical Society, 1996.

\bibitem[HT03]{HauTha}
T.~Hausel and M.~Thaddeus, \emph{Mirror symmetry, {L}anglands duality, and the {H}itchin system}, Invent. Math. \textbf{153} (2003), no.~1, 197--229.

\bibitem[Isi13]{I}
M.~U. Isik, \emph{Equivalence of the derived category of a variety with a singularity category}, Int. Math. Res. Not. IMRN (2013), no.~12, 2787--2808.

\bibitem[JS12]{JS}
D.~Joyce and Y.~Song, \emph{A theory of generalized {D}onaldson-{T}homas invariants}, Mem. Amer. Math. Soc. \textbf{217} (2012), no.~1020, iv+199.

\bibitem[Kaw02]{MR1949787}
Y.~Kawamata, \emph{{$D$}-equivalence and {$K$}-equivalence}, J. Differential Geom. \textbf{61} (2002), no.~1, 147--171.

\bibitem[Kin]{Kinjonote}
T.~Kinjo, \emph{Descent of filtered {D}-modules}, personal communication.

\bibitem[Kin94]{Kin}
A.~King, \emph{Moduli of representations of finite-dimensional algebras}, Quart.~J.~Math.~Oxford Ser.(2) \textbf{45} (1994), 515--530.

\bibitem[Kin22]{Kinjo}
T.~Kinjo, \emph{Dimensional reduction in cohomological {D}onaldson-{T}homas theory}, Compositio.~Math \textbf{158} (2022), no.~1, 123--167.

\bibitem[KK]{KinjoKoseki}
T.~Kinjo and N.~Koseki, \emph{Cohomological $\chi$-independence for {H}iggs bundles and {G}opakumar-{V}afa invariants}, arXiv:2112.10053, to appear in JEMS.

\bibitem[KLS06]{KaLeSo}
D.~Kaledin, M.~Lehn, and Ch. Sorger, \emph{Singular symplectic moduli spaces}, Invent.~Math.~ \textbf{164} (2006), 591--614.

\bibitem[KS]{K-S}
M.~Kontsevich and Y.~Soibelman, \emph{Stability structures, motivic {D}onaldson-{T}homas invariants and cluster transformations}, arXiv:0811.2435.

\bibitem[KS11]{MR2851153}
\bysame, \emph{Cohomological {H}all algebra, exponential {H}odge structures and motivic {D}onaldson-{T}homas invariants}, Commun. Number Theory Phys. \textbf{5} (2011), no.~2, 231--352.

\bibitem[KSV17]{KSV}
M.~Kapranov, O.~Schiffmann, and E.~Vasserot, \emph{The {H}all algebra of a curve}, Selecta Math. (N.S.) \textbf{23} (2017), 117--177.

\bibitem[KW07]{KapWit}
A.~N. Kapustin and E.~Witten, \emph{Electric-{M}agnetic {D}uality and the {G}eometric {L}anglands {P}rogram}, Communications in {N}umber {T}heory and {P}hysics \textbf{1} (2007), no.~1, 1--236.

\bibitem[Lau10]{Laufilt}
G.~Laumon, \emph{Sur la categorie derivee des {D}-modules filtres}, J.~Amer.~Math.~Soc.~ \textbf{23} (2010), 405--468.

\bibitem[Li21]{MLi}
M.~Li, \emph{Construction of {P}oincar\'{e} sheaf on stack of rank 2 {H}iggs bundles}, Adv. Math. \textbf{376} (2021), Paper No. 107437, 55.

\bibitem[MMSV]{MMSV}
A.~Mellit, A.~Minets, O.~Schiffmann, and E.~Vasserot, \emph{Coherent sheaves on surfaces, {COHA}s and deformed ${W}_{1+\infty}$-algebras}, arXiv:2311.13415.

\bibitem[MN]{McNe}
K.~McGerty and T.~Nevins, \emph{Morse decomposition for {D}-module categories on stacks}, arXiv:1402.7365.

\bibitem[MRV19]{MRVF2}
M.~Melo, A.~Rapagnetta, and F.~Viviani, \emph{Fourier-{M}ukai and autoduality for compactified {J}acobians. {II}}, Geometry and Topology \textbf{23} (2019), 2335--2395.

\bibitem[MS21]{MSendscopic}
D.~Maulik and J.~Shen, \emph{Endoscopic decompositions and the {H}ausel-{T}haddeus conjecture}, Forum Math. Pi \textbf{9} (2021), Paper No. e8, 49.

\bibitem[MS23]{DMJS}
\bysame, \emph{Cohomological {$\chi$}-independence for moduli of one-dimensional sheaves and moduli of {H}iggs bundles}, Geom. Topol. \textbf{27} (2023), no.~4, 1539--1586.

\bibitem[MT18]{MTK3}
D.~Maulik and R.~P. Thomas, \emph{Sheaf counting on local {K}3 surfaces}, Pure Appl.~Math.~Quart.~ \textbf{2018} (2018), 419--441.

\bibitem[Muk81]{Mu1}
S.~Mukai, \emph{Duality between ${D}({X})$ and ${D}(\hat{X})$ with its application to {P}icard sheaves}, Nagoya Math.~J.~ \textbf{81} (1981), 101--116.

\bibitem[Neg]{Negutthesis}
A.~Negut, \emph{Quantum {A}lgebras and {C}yclic {Q}uiver {V}arieties}, arXiv:1504.06525.

\bibitem[OS22]{OkSm}
A.~Okounkov and A.~Smirnov, \emph{Quantum difference equation for {N}akajima varieties}, Invent. Math. \textbf{229} (2022), 1203--1299.

\bibitem[P{\u{a}}d]{P}
T.~P{\u{a}}durariu, \emph{Generators for {K}-theoretic {H}all algebras of quivers with potential}, arXiv:2108.07919.

\bibitem[P{\u{a}}d23a]{P0}
\bysame, \emph{Categorical and {K}-theoretic {H}all algebras for quivers with potential}, J.~Inst.~Math.~Jussieu \textbf{22} (2023), no.~6, 2717--2747.

\bibitem[P{\u{a}}d23b]{P2}
\bysame, \emph{Generators for categorical {H}all algebras of surfaces}, Math. Z. \textbf{303} (2023), no.~40, 1--25.

\bibitem[Pre]{Preygel}
A.~Preygel, \emph{{T}hom-{S}ebastiani and {D}uality for {M}atrix {F}actorizations}, arXiv:1101.5834.

\bibitem[PS23]{PoSa}
M.~Porta and F.~Sala, \emph{Two dimensional categorified {H}all algebras}, J. Eur. Math. Soc. \textbf{25} (2023), no.~3, 1113--1205.

\bibitem[PTa]{PThiggs}
T.~P\u{a}durariu and Y.~Toda, \emph{Quasi-{BPS} categories for {H}iggs bundles}, arXiv:2408.02168.

\bibitem[PTb]{PTtop}
\bysame, \emph{Topological {K}-theory of quasi-{BPS} categories of symmetric quivers with potential}, arXiv:2309.08432.

\bibitem[PT14]{MR3221298}
R.~Pandharipande and R.~P. Thomas, \emph{13/2 ways of counting curves}, Moduli spaces, London Math. Soc. Lecture Note Ser., vol. 411, Cambridge Univ. Press, Cambridge, 2014, pp.~282--333.

\bibitem[PT23]{PT1}
T.~P\u{a}durariu and Y.~Toda, \emph{Categorical and {K}-theoretic {D}onaldson-{T}homas theory of $\mathbb{C}^3$ (part {II})}, Forum Math. Sigma \textbf{11} (2023), Paper No. e108, 47.

\bibitem[PT24]{PT0}
\bysame, \emph{Categorical and {K}-theoretic {D}onaldson-{T}homas theory of $\mathbb{C}^3$ (part {I})}, Duke Math.~J.~ \textbf{173} (2024), 1973--2038.

\bibitem[PT25a]{PTquiver}
\bysame, \emph{Quasi-{BPS} categories for symmetric quivers with potential}, Compos. Math. \textbf{161} (2025), no.~8, 1799--1854.

\bibitem[PT25b]{PThiggs2}
\bysame, \emph{Topological {K}-theory of quasi-{BPS} categories for {H}iggs bundles}, J. Topol. \textbf{18} (2025), no.~4, Paper No. e70049, 72.

\bibitem[PT26a]{PT2}
\bysame, \emph{The categorical {DT}/{PT} correspondence and quasi-{BPS} categories for local surfaces}, Moduli \textbf{3} (2026), Paper No. e1, 44.

\bibitem[PT26b]{PTK3}
\bysame, \emph{Quasi-{BPS} categories for {K}3 surfaces}, Forum Math. Pi \textbf{14} (2026), Paper No. e11.

\bibitem[PTVV13]{PTVV}
T.~Pantev, B.~To$\ddot{\textrm{e}}$n, M.~Vaquie, and G.~Vezzosi, \emph{Shifted symplectic structures}, Publ.~Math.~IHES \textbf{117} (2013), 271--328.

\bibitem[PV11]{MR3112502}
A.~Polishchuk and A.~Vaintrob, \emph{Matrix factorizations and singularity categories for stacks}, Ann. Inst. Fourier (Grenoble) \textbf{61} (2011), no.~7, 2609--2642.

\bibitem[Sch15]{SchHN}
S.~Schieder, \emph{The {H}arder-{N}arasimhan stratification of the moduli stack of {$G$}-bundles via {D}rinfeld's compactifications}, Selecta Math. (N.S.) \textbf{21} (2015), no.~3, 763--831.

\bibitem[Sch18]{OSc}
O.~Schiffmann, \emph{Kac polynomials and {L}ie algebras associated to quivers and curves}, Proceedings of the International Congress of Mathematicians (2018), 1393--1424.

\bibitem[Sim97]{Simp}
C.~Simpson, \emph{The {H}odge filtration on nonabelian cohomology}, Algebraic geometry--{S}anta {C}ruz 1995, Proc. Sympos. Pure Math., vol. 62, Part 2, Amer. Math. Soc., Providence, RI, 1997, pp.~217--281.

\bibitem[Sim10]{Simp2}
\bysame, \emph{Iterated destabilizing modifications for vector bundles with connection}, Vector bundles and complex geometry, Contemp. Math., vol. 522, Amer. Math. Soc., Providence, RI, 2010, pp.~183--206.

\bibitem[SS20]{SalaSchif}
F.~Sala and O.~Schiffmann, \emph{Cohomological {H}all algebra of {H}iggs sheaves on a curve}, Algebraic Geometry \textbf{7} (2020), no.~3, 346--376.

\bibitem[SV12]{MR2944034}
O.~Schiffmann and E.~Vasserot, \emph{Hall algebras of curves, commuting varieties and {L}anglands duality}, Math. Ann. \textbf{353} (2012), no.~4, 1399--1451.

\bibitem[SV13]{SV}
O.~Schiffmann and E.~Vasserot, \emph{The elliptic {H}all algebra and the {$K$}-theory of the {H}ilbert scheme of {$\mathbb{A}^2$}}, Duke Math. J. \textbf{162} (2013), no.~2, 279--366.

\bibitem[{\v S}VdB17]{SvdB}
{\v S}.~{\v S}penko and M.~{V}an~den {B}ergh, \emph{Non-commutative resolutions of quotient singularities for reductive groups}, Invent. Math. \textbf{210} (2017), no.~1, 3--67.

\bibitem[{\v S}VdB21]{SpVdB}
\bysame, \emph{Semi-orthogonal decompositions of {GIT} quotient stacks}, Selecta Math. (N.S.) \textbf{27} (2021), no.~2, Paper No. 16, 43.

\bibitem[SYZ01]{SYZ}
A.~Strominger, S.~T. Yau, and E.~Zaslow, \emph{Mirror symmetry is {$T$}-duality}, Winter {S}chool on {M}irror {S}ymmetry, {V}ector {B}undles and {L}agrangian {S}ubmanifolds ({C}ambridge, {MA}, 1999), AMS/IP Stud. Adv. Math., vol.~23, Amer. Math. Soc., Providence, RI, 2001, pp.~333--347.

\bibitem[Sze16]{Sz}
B.~Szendrői, \emph{Cohomological {D}onaldson-{T}homas theory}, String-{M}ath 2014, Proc. Sympos. Pure Math., vol.~93, Amer. Math. Soc., Providence, RI, 2016, pp.~363--396.

\bibitem[Tho00]{Thom}
R.~P. Thomas, \emph{A holomorphic {C}asson invariant for {C}alabi-{Y}au 3-folds and bundles on ${K3}$-fibrations}, J.~Differential.~Geom \textbf{54} (2000), 367--438.

\bibitem[Toda]{Toquot2}
Y.~Toda, \emph{Derived categories of {Q}uot schemes of zero-dimensional quotients on curves}, arXiv:2207.09687.

\bibitem[Todb]{TodaA}
\bysame, \emph{The {D}olbeault geometric {L}anglands correspondence for type {A} groups beyond the elliptic locus}, arXiv:2606.28878.

\bibitem[Todc]{TodaGL2}
\bysame, \emph{A proof of {D}olbeault geometric {L}anglands for $\mathrm{GL}_2$ with reduced spectral curves}, arXiv:2602.09359.

\bibitem[Todd]{Totheta}
\bysame, \emph{Semiorthogonal decompositions for categorical {D}onaldson-{T}homas theory via {$\Theta$}-stratifications}, arXiv:2106.05496.

\bibitem[Tod20]{T2}
\bysame, \emph{Hall-type algebras for categorical {D}onaldson-{T}homas theories on local surfaces}, Selecta Math. (N.S.) \textbf{26} (2020), no.~4, 64.

\bibitem[Tod21]{TodaSpringer}
\bysame, \emph{Recent {P}rogress on the {D}onaldson–{T}homas theory: {W}all-{C}rossing and {R}efined {I}nvariants}, vol.~43, SpringerBriefs in Mathematical Physics, 2021.

\bibitem[Tod23a]{T4}
\bysame, \emph{Categorical {D}onaldson-{T}homas theory for local surfaces: {$\mathbb{Z}$}/2-periodic version}, Int. Math. Res. Not. IMRN (2023), no.~13, 11172--11216.

\bibitem[Tod23b]{TodGV}
\bysame, \emph{Gopakumar-{V}afa invariants and wall-crossing}, J. Differential Geom. \textbf{123} (2023), no.~1, 141--193.

\bibitem[Tod24]{T}
\bysame, \emph{Categorical {D}onaldson-{T}homas theory for local surfaces}, Lecture Notes in Mathematics \textbf{2350} (2024), 1--309.

\bibitem[VV22]{VaVa}
M.~Varagnolo and E.~Vasserot, \emph{K-theoretic {H}all algebras, quantum groups and super quantum groups}, Selecta Math. (N.S.) \textbf{28} (2022), no.~7, 46 pages.

\end{thebibliography}

\textsc{\small Tudor P\u adurariu: Sorbonne Université and Université Paris Cité, CNRS, IMJ-PRG, F-75005 Paris, France.}\\
\textit{\small E-mail address:} \texttt{\small padurariu@imj-prg.fr}\\

\textsc{\small Yukinobu Toda: Kavli Institute for the Physics and Mathematics of the Universe (WPI), University of Tokyo, 5-1-5 Kashiwanoha, Kashiwa, 277-8583, Japan.}\\
\textsc{\small 
Inamori Research Institute for Science, 620 Suiginya-cho, Shimogyo-ku, Kyoto 600-8411, Japan.} \\
\textit{\small E-mail address:} \texttt{\small yukinobu.toda@ipmu.jp}\\

 \end{document}